\documentclass[10pt]{article}
\usepackage[all,2cell]{xy} \UseAllTwocells \SilentMatrices
\usepackage{latexsym,amsfonts,amssymb}
\usepackage{amsmath,amsthm,amscd}
\allowdisplaybreaks
\usepackage{hyperref,psfrag}
\usepackage{diagcat}
\usepackage{color}
\usepackage{etoolbox}
\usepackage[dvips]{epsfig}
\usepackage{psfrag}
\usepackage[vcentermath]{youngtab}

\usepackage{graphicx}
\usepackage{a4wide}
\usepackage{geometry}

\usepackage{epigraph,wrapfig}

\normalfont\upshape

\usepackage{fancyhdr}
\theoremstyle{definition}
\newtheorem{thm}{Theorem}[section]
\newtheorem{cor}[thm]{Corollary}

\newtheorem{lem}[thm]{Lemma}
\newtheorem{rem}[thm]{Remark}
\newtheorem{prop}[thm]{Proposition}
\newtheorem{defn}[thm]{Definition}
\newtheorem{example}[thm]{Example}

\newtheorem*{thm*}{Theorem}

\numberwithin{equation}{section}

\usepackage{bbm}
\def\o{\otimes}

\def\N{{\mathbbm N}}

\def\Z{{\mathbbm Z}}
\def\Q{{\mathbbm Q}}

\def\1{{\mathbbm{1}}}

\newcommand{\Hom}{{\rm Hom}}
\newcommand{\HOM}{{\rm HOM}}

\newcommand{\Res}{{\rm Res}}
\newcommand{\End}{{\rm End}}
\newcommand{\END}{{\rm END}}
\newcommand{\Tab}{{\rm Tab}}
\newcommand{\T}{{\rm T}}

\def\dif{\partial}
\def\Res{{\mathrm{Res}}}
\def\Ind{{\mathrm{Ind}}}
\def\lra{{\longrightarrow}}
\def\mod{{\mathrm{mod}}}   
\def\dmod{{\mbox{-}\mathrm{mod}}}   %% finitely-generated modules
\def\gdim{{\mathrm{gdim}}}

\def\rad{{\mathrm{rad}}}

\def\Id{\mathrm{Id}}
\def\mc{\mathcal}
\def\mf{\mathfrak}
\def\shuffle{\,\raise 1pt\hbox{$\scriptscriptstyle\cup{\mskip
               -4mu}\cup$}\,}
\newcommand{\refequal}[1]{\xy {\ar@{=}^{#1}
(-1,0)*{};(1,0)*{}};
\endxy}

\def\e{\varepsilon}

\newcommand{\nh}{\mathrm{NH}}

\newcommand{\mH}{\mathrm{H}} %cohomology
\newcommand{\RHOM}{\mathbf{R}\mathrm{HOM}}

\newcommand{\cwbubble}[2]{
\begin{DGCpicture}
\DGCcoupon*(-0.4,-0.4)(.4,.4){ }
\DGCbubble(0,0){0.35}
\DGCdot*<{0.2,L}
\DGCdot*.{0.1,R}[r]{#2}
\DGCcoupon*(-.4,-.4)(.4,.4){\small{#1}}
\end{DGCpicture}
}

\newcommand{\ccwbubble}[2]{
\begin{DGCpicture}
\DGCcoupon*(-0.4,-0.4)(.4,.4){ }
\DGCbubble(0,0){0.35}
\DGCdot*>{0.2,L}
\DGCdot*.{0.1,R}[r]{#2}
\DGCcoupon*(-.4,-.4)(.4,.4){\small{#1}}
\end{DGCpicture}
}

\newcommand{\bigcwbubble}[2]{
\begin{DGCpicture}
\DGCcoupon*(-0.8,-0.8)(.8,.8){ }
\DGCbubble(0,0){0.5}
\DGCdot*<{0.25,L}
\DGCdot*.{0.25,R}[r]{#2}
\DGCcoupon*(-.6,-.6)(.6,.6){\small{#1}}
\end{DGCpicture}
}

\newcommand{\bigccwbubble}[2]{
\begin{DGCpicture}
\DGCcoupon*(-0.8,-0.8)(.8,.8){ }
\DGCbubble(0,0){0.5}
\DGCdot*>{0.25,L}
\DGCdot*.{0.25,R}[r]{#2}
\DGCcoupon*(-.6,-.6)(.6,.6){\small{#1}}
\end{DGCpicture}
}

\newcommand{\cwcapbubcup}[4]{
\begin{DGCpicture}
\DGCcoupon*(-0.3,-0.1)(1.3,2.1){ }
\ifstrequal{#4}{no}{}{
\DGCcoupon*(0,.8)(.3,1.4){#4}}
\ifstrequal{#1}{no}{}{
\DGCstrand(0,0)(1,0)/d/
\DGCdot*>{0.25,2}
\ifstrequal{#1}{$0$}{}{
\ifstrequal{#1}{$1$}{\DGCdot{0.25,1}}{
\DGCdot{0.25,1}[r]{\mbox{\scriptsize #1}}}}}
\ifstrequal{#3}{no}{}{
\DGCstrand/d/(0,2)(1,2)
\DGCdot*<{1.75,2}
\ifstrequal{#3}{$0$}{}{
\ifstrequal{#3}{$1$}{\DGCdot{1.75,1}}{
\DGCdot{1.75,1}[r]{\mbox{\scriptsize #3}}}}}
\ifstrequal{#2}{no}{}{
\DGCbubble(1,1){.3}
\DGCdot*>{1.2,L}
\DGCcoupon*(.65,.65)(1.35,1.35){\small{#2}}}
\end{DGCpicture}
}

\newcommand{\ccwcapbubcup}[4]{
\begin{DGCpicture}
\DGCcoupon*(-0.3,-0.1)(1.3,2.1){ }
\ifstrequal{#4}{no}{}{
\DGCcoupon*(0,.8)(.3,1.4){#4}}
\ifstrequal{#1}{no}{}{
\DGCstrand(0,0)(1,0)/d/
\DGCdot*<{0.25,2}
\ifstrequal{#1}{$0$}{}{
\ifstrequal{#1}{$1$}{\DGCdot{0.25,1}}{
\DGCdot{0.25,1}[r]{\mbox{\scriptsize #1}}}}}
\ifstrequal{#3}{no}{}{
\DGCstrand/d/(0,2)(1,2)
\DGCdot*>{1.75,2}
\ifstrequal{#3}{$0$}{}{
\ifstrequal{#3}{$1$}{\DGCdot{1.75,1}}{
\DGCdot{1.75,1}[r]{\mbox{\scriptsize #3}}}}}
\ifstrequal{#2}{no}{}{
\DGCbubble(1,1){.3}
\DGCdot*<{1.2,L}
\DGCcoupon*(.65,.65)(1.35,1.35){\small{#2}}}
\end{DGCpicture}
}

\newcommand{\RIII}[5]{
\begin{DGCpicture}[scale=0.85]
\DGCcoupon*(-0.3,-0.3)(2.3,2.3){}
\ifstrequal{#5}{no}{}{
\DGCcoupon*(2.1,.95)(2.3,1.15){#5}}
\DGCstrand(0,0)(2,2)
\DGCdot*>{2}
\ifstrequal{#4}{$0$}{}{
\ifstrequal{#4}{$1$}{\DGCdot{1.7}}{
\DGCdot{1.7}[r]{\mbox{\scriptsize #4}}}}
\DGCstrand(2,0)(0,2)
\DGCdot*>{2}
\ifstrequal{#2}{$0$}{}{
\ifstrequal{#2}{$1$}{\DGCdot{1.7}}{
\DGCdot{1.7}[r]{\mbox{\scriptsize #2}}}}
\ifstrequal{#1}{L}{\DGCstrand(1,0)(0,1)(1,2)}{\DGCstrand(1,0)(2,1)(1,2)}
\DGCdot*>{2}
\ifstrequal{#3}{$0$}{}{
\ifstrequal{#3}{$1$}{\DGCdot{1.7}}{
\DGCdot{1.7}[r]{\mbox{\scriptsize #3}}}}
\end{DGCpicture}
}

\newcommand{\twolines}[3]{
\begin{DGCpicture}
\DGCcoupon*(-.3,-.3)(1.3,1.3){}
\ifstrequal{#3}{no}{}{\DGCcoupon*(1.1,.6)(1.4,.9){#3}}
\DGCstrand(0,0)(0,1)
\DGCdot*>{1}
\ifstrequal{#1}{$0$}{}{
\ifstrequal{#1}{$1$}{\DGCdot{.4}}{
\DGCdot{.4}[r]{\mbox{\scriptsize #1}}}}
\DGCstrand(1,0)(1,1)
\DGCdot*>{1}
\ifstrequal{#2}{$0$}{}{
\ifstrequal{#2}{$1$}{\DGCdot{.4}}{
\DGCdot{.4}[r]{\mbox{\scriptsize #2}}}}
\end{DGCpicture}
}

\newcommand{\twolinesD}[3]{
\begin{DGCpicture}
\DGCcoupon*(-.3,-.3)(1.3,1.3){}
\ifstrequal{#3}{no}{}{\DGCcoupon*(-.4,.6)(-.1,.9){#3}}
\DGCstrand(0,0)(0,1)
\DGCdot*<{0}
\ifstrequal{#1}{$0$}{}{
\ifstrequal{#1}{$1$}{\DGCdot{.6}}{
\DGCdot{.6}[r]{\mbox{\scriptsize #1}}}}
\DGCstrand(1,0)(1,1)
\DGCdot*<{0}
\ifstrequal{#2}{$0$}{}{
\ifstrequal{#2}{$1$}{\DGCdot{.6}}{
\DGCdot{.6}[r]{\mbox{\scriptsize #2}}}}
\end{DGCpicture}
}

\newcommand{\crossing}[5]{
\begin{DGCpicture}
\DGCcoupon*(-.3,-.3)(1.3,1.3){}
\ifstrequal{#5}{no}{}{
\DGCcoupon*(1,.4)(1.3,.7){#5}}
\DGCstrand(0,0)(1,1)
\DGCdot*>{1}
\ifstrequal{#2}{$0$}{}{
\ifstrequal{#2}{$1$}{\DGCdot{.7}}{
\DGCdot{.7}[r]{\mbox{\scriptsize #2}}}}
\ifstrequal{#3}{$0$}{}{
\ifstrequal{#3}{$1$}{\DGCdot{.3}}{
\DGCdot{.3}[r]{\mbox{\scriptsize #3}}}}
\DGCstrand(1,0)(0,1)
\DGCdot*>{1}
\ifstrequal{#1}{$0$}{}{
\ifstrequal{#1}{$1$}{\DGCdot{.7}}{
\DGCdot{.7}[r]{\mbox{\scriptsize #1}}}}
\ifstrequal{#4}{$0$}{}{
\ifstrequal{#4}{$1$}{\DGCdot{.3}}{
\DGCdot{.3}[r]{\mbox{\scriptsize #4}}}}
\end{DGCpicture}
}

\newcommand{\crossingD}[5]{
\begin{DGCpicture}
\DGCcoupon*(-.3,-.3)(1.3,1.3){}
\ifstrequal{#5}{no}{}{\DGCcoupon*(-.3,.4)(0,.7){#5}}
\DGCstrand(0,0)(1,1)
\DGCdot*<{0}
\ifstrequal{#2}{$0$}{}{
\ifstrequal{#2}{$1$}{\DGCdot{.3}}{
\DGCdot{.3}[r]{\mbox{\scriptsize #2}}}}
\ifstrequal{#3}{$0$}{}{
\ifstrequal{#3}{$1$}{\DGCdot{.7}}{
\DGCdot{.7}[r]{\mbox{\scriptsize #3}}}}
\DGCstrand(1,0)(0,1)
\DGCdot*<{0}
\ifstrequal{#1}{$0$}{}{
\ifstrequal{#1}{$1$}{\DGCdot{.3}}{
\DGCdot{.3}[r]{\mbox{\scriptsize #1}}}}
\ifstrequal{#4}{$0$}{}{
\ifstrequal{#4}{$1$}{\DGCdot{.7}}{
\DGCdot{.7}[r]{\mbox{\scriptsize #4}}}}
\end{DGCpicture}
}

\newcommand{\oneline}[2]{
\begin{DGCpicture}
\DGCcoupon*(-.3,-.1)(0.3,2.1){}
\DGCstrand(0,0)(0,2)
\DGCdot*>{2}
\ifstrequal{#2}{no}{}{\DGCcoupon*(.1,1.4)(.4,1.7){#2}}
\ifstrequal{#1}{$0$}{}{
\ifstrequal{#1}{$1$}{\DGCdot{1}}{
\DGCdot{1}[r]{\mbox{\scriptsize #1}}}}
\end{DGCpicture}
}

\newcommand{\onelineD}[2]{
\begin{DGCpicture}
\DGCcoupon*(-.3,-.1)(0.3,2.1){}
\DGCstrand(0,0)(0,2)
\DGCdot*<{0}
\ifstrequal{#2}{no}{}{\DGCcoupon*(-.4,1.4)(-.1,1.7){#2}}
\ifstrequal{#1}{$0$}{}{
\ifstrequal{#1}{$1$}{\DGCdot{1}}{
\DGCdot{1}[r]{\mbox{\scriptsize #1}}}}
\end{DGCpicture}
}

\newcommand{\onelineshort}[2]{
\begin{DGCpicture}
\DGCcoupon*(-.3,-.1)(0.3,1.1){}
\DGCstrand(0,0)(0,1)
\DGCdot*>{1}
\ifstrequal{#2}{no}{}{\DGCcoupon*(.1,.7)(.4,1){#2}}
\ifstrequal{#1}{$0$}{}{
\ifstrequal{#1}{$1$}{\DGCdot{.5}}{
\DGCdot{.5}[r]{\mbox{\scriptsize #1}}}}
\end{DGCpicture}
}

\newcommand{\onelineDshort}[2]{
\begin{DGCpicture}
\DGCcoupon*(-.3,-.1)(0.3,1.1){}
\DGCstrand(0,0)(0,1)
\DGCdot*<{0}
\ifstrequal{#2}{no}{}{\DGCcoupon*(-.4,.7)(-.1,1){#2}}
\ifstrequal{#1}{$0$}{}{
\ifstrequal{#1}{$1$}{\DGCdot{.5}}{
\DGCdot{.5}[r]{\mbox{\scriptsize #1}}}}
\end{DGCpicture}
}

\newcommand{\curl}[5]{
\begin{DGCpicture}
\ifstrequal{#1}{L}{
\DGCcoupon*(-2,-.8)(.3,1.8){}
\DGCstrand(0,-.5)(0,.25)/u/(-1.5,.5)/d/(0,.75)/u/(0,1.5)
\ifstrequal{#4}{no}{}{\DGCcoupon*(-1.4,0)(-.4,1){\small{#4}}}
\ifstrequal{#5}{$0$}{}{\ifstrequal{#5}{$1$}{\DGCdot{.5,4}}{\DGCdot{.5,4}[l]{\mbox{\scriptsize #5}}}}
\ifstrequal{#3}{no}{}{\DGCcoupon*(-1.2,1.1)(-.9,1.4){#3}}
}{
\DGCcoupon*(-.3,-.8)(2,1.8){}
\DGCstrand(0,-.5)(0,.25)/u/(1.5,.5)/d/(0,.75)/u/(0,1.5)
\ifstrequal{#4}{no}{}{\DGCcoupon*(.4,0)(1.4,1){\small{#4}}}
\ifstrequal{#5}{$0$}{}{\ifstrequal{#5}{$1$}{\DGCdot{.5,4}}{\DGCdot{.5,4}[r]{\mbox{\scriptsize #5}}}}
\ifstrequal{#3}{no}{}{\DGCcoupon*(.9,1.1)(1.2,1.4){#3}}
}
\ifstrequal{#2}{D}{\DGCdot*<{-0.25} \DGCdot*<{1.25} \DGCdot*<{.75,2}}{\DGCdot*>{-0.25} \DGCdot*>{1.25} \DGCdot*>{.75,2}}
\end{DGCpicture}
}

\newcommand{\cappy}[4]{
\begin{DGCpicture}
\DGCcoupon*(-.3,-.1)(1.3,.8){}
\DGCstrand(0,0)(1,0)/d/
\ifstrequal{#1}{CCW}{\DGCdot*<{0,2}}{\DGCdot*>{0,1}}
\ifstrequal{#2}{$0$}{}{\ifstrequal{#2}{$1$}{\DGCdot{.3}}{\DGCdot{.3}[r]{\mbox{\scriptsize #2}}}}
\ifstrequal{#3}{no}{}{
\DGCbubble(1.7,.4){0.3}
\ifstrequal{#3}{CCW}{\DGCdot*>{.7,L}}{\DGCdot*<{.7,L}}
\DGCcoupon*(1.35,0)(2.05,.8){\small{$1$}}}
\ifstrequal{#4}{no}{}{\DGCcoupon*(.8,.5)(1.2,.8){#4}}
\end{DGCpicture}
}

\newcommand{\cuppy}[4]{
\begin{DGCpicture}
\DGCcoupon*(-.3,0.2)(1.3,1.1){}
\DGCstrand(0,1)/d/(1,1)/u/
\ifstrequal{#1}{CCW}{\DGCdot*>{1,1}}{\DGCdot*<{1,2}}
\ifstrequal{#2}{$0$}{}{\ifstrequal{#2}{$1$}{\DGCdot{.7}}{\DGCdot{.7}[r]{\mbox{\scriptsize #2}}}}
\ifstrequal{#3}{no}{}{
\DGCbubble(1.7,.6){0.3}
\ifstrequal{#3}{CCW}{\DGCdot*>{.9,L}}{\DGCdot*<{.9,L}}
\DGCcoupon*(1.35,.2)(2.05,1){\small{$1$}}}
\ifstrequal{#4}{no}{}{\DGCcoupon*(.8,.2)(1.2,.5){#4}}
\end{DGCpicture}
}

\newcommand{\bottomcurl}[5]{
\begin{DGCpicture}
\DGCcoupon*(-.3,-.1)(1.3,1.6){}
\DGCstrand(0,0)(1,1)/u/(0,1)/d/(1,0)/d/
\ifstrequal{#1}{CW}{\DGCdot*>{1.3}}{\DGCdot*<{1.3}}
\ifstrequal{#2}{no}{}{\DGCcoupon*(0,0.65)(1,1.3){\mbox{\scriptsize #2}}}
\ifstrequal{#3}{yes}{\DGCdot{.3,2}}{}
\ifstrequal{#4}{no}{}{
\DGCbubble(1.5,0.6){.3}
\ifstrequal{#4}{CCW}{\DGCdot*>{.8,L}}{\DGCdot*<{.8,L}}
\DGCcoupon*(1.3,0.4)(1.7,0.8){\small{$1$}}}
\ifstrequal{#5}{no}{}{\DGCcoupon*(1.1,0.9)(1.5,1.5){#5}}
\end{DGCpicture}
}

\newcommand{\crossingR}[4]{
\begin{DGCpicture}
\DGCcoupon*(-.3,-.3)(1.3,1.3){}
\DGCstrand(0,0)(1,1)
\DGCdot*>{1}
\ifstrequal{#1}{no}{}{\DGCdot{.65}}
\DGCstrand(1,0)(0,1)
\DGCdot*<{0}
\ifstrequal{#2}{no}{}{\DGCdot{.65}}
\ifstrequal{#3}{no}{}{
\DGCbubble(1.2,.5){0.3}
\DGCdot*<{.7,R}
\DGCcoupon*(0.9,0.2)(1.5,.8){\small{$1$}}}
\ifstrequal{#4}{no}{}{
\DGCcoupon*(-.3,.2)(.3,.8){#4}}
\end{DGCpicture}
}

\newcommand{\crossingL}[4]{
\begin{DGCpicture}
\DGCcoupon*(-.3,-.3)(1.3,1.3){}
\DGCstrand(0,0)(1,1)
\DGCdot*<{0}
\ifstrequal{#1}{no}{}{\DGCdot{.35}}
\DGCstrand(1,0)(0,1)
\DGCdot*>{1}
\ifstrequal{#2}{no}{}{\DGCdot{.35}}
\ifstrequal{#3}{no}{}{
\DGCbubble(-.2,.5){0.3}
\DGCdot*<{.7,L}
\DGCcoupon*(-.5,0.2)(.1,.8){\small{$1$}}}
\ifstrequal{#4}{no}{}{
\DGCcoupon*(.7,.2)(1.3,.8){#4}}
\end{DGCpicture}
}

\newstrandstyle{Green}{type=ribbon,style=solid,ribboncolor=green,color=green}

\title{$p$-DG cyclotomic nilHecke algebras II}
\author{You Qi and Joshua Sussan}

\date{November 11, 2018}

\begin{document}
%
% ==============================================================================

\maketitle

\begin{abstract}
We categorify tensor products of the fundamental representation of quantum $\mathfrak{sl}_2$ at prime roots of unity, building upon earlier work where a tensor product of two Weyl modules was categorified.
\end{abstract}

\setcounter{tocdepth}{2} \tableofcontents

\section{Introduction}
\subsection{Background}
A categorification of tensor products of the natural representation of
the Lie algebra $\mathfrak{sl}_2$ was constructed in \cite{BFK}. The authors used maximally singular blocks of the Bernstein-Gelfand-Gelfand (BGG) category $\mathcal{O}$ for $\mathfrak{gl}_l$ to categorify the module $V_1^{\otimes l}$, where $V_1$ is the standard two-dimensional module. Furthermore, they showed projective functors between the blocks categorify the action of the enveloping algebra of $\mf{sl}_2$.  
This construction was extended to a categorification of the action of the corresponding quantum group at a generic value of the quantum parameter using graded versions of category $\mathcal{O}$ in \cite{FKS}.  In addition, a categorification of the action of $\mathcal{U}_q(\mathfrak{sl}_2)$ on tensor products of arbitrary finite-dimensional representations $V_{d_1} \otimes \cdots \otimes V_{d_n}$ was also obtained by using special subcategories equivalent to certain subcategories of Harish-Chandra bimodules.

The above work serves as the algebraic motivation for the celebrated quantum topological knot and link invariant known as Khovanov homology \cite{KhJones}. Hopfological algebra (and more specifically $p$-DG theory) was introduced in \cite{Hopforoots} and further developed in \cite{QYHopf}, with the goal of extending the results on categorifications of quantum group structures and quantum topological invariants at generic values to prime roots of unity cases. The realization of this goal would in turn define categorical three-manifold invariants lifting the Witten-Reshetikhin-Turaev (WRT) invariants.
A survey containing various successful implementations of this framework could be found in \cite{QiSussan2}.
For approaches to categorification at non-prime roots of unity see \cite{LaQi, Mir}.

Cyclotomic KLR algebras give rise to a categorification of irreducible representations of various quantum groups \cite{KK} \cite{Webcombined} at generic $q$ values.  A special case of these algebras is the cyclotomic nilHecke algebra $\nh_n^l$.  This algebra categorifies a weight space of the irreducible representation $V_l$ of quantum $\mathfrak{sl}_2$.  Over a field of positive characteristic $p$, $\nh_n^l$ comes equipped with a derivation $\partial$ such that $\partial^p=0$.  We say that $(\nh_n^l,\partial)$ is a $p$-DG algebra.  This $p$-DG algebra categorifies a weight space of an irreducible representation of the small quantum group for $\mathfrak{sl}_2$ \cite{EQ1} building upon earlier work \cite{KQ} where half of the small quantum group was categorified.

In \cite{KQS}, we categorified a tensor product of two Weyl modules $V_r \otimes V_s$ of quantum $\mathfrak{sl}_2$ defined over the ring $\mathbb{O}_p = \mathbb{Z}[q]/(\Psi_p(q^2))$, where $\Psi_p(q^2)=1+q^2+\cdots+q^{2(p-1)}$.
We utilized an endomorphism algebra of certain $p$-DG modules over the $p$-DG  algebra $\nh_n^l$.  As an endomorphism algebra over a $p$-DG algebra, it naturally inherits a $p$-DG structure.  Our motivation comes from work of Hu and Mathas \cite{HuMathas} who studied endomorphism algebras of certain cyclically generated modules over general cyclotomic KLR algebras, which they called \emph{quiver Schur algebras}\footnote{Similar constructions have been considered in many works before, including \cite[Section 7.6]{ChuangRouquier}  where (ungraded) $q$-Schur algebras were studied in connection to categorical representation theory.}.  They proved that these quiver Schur algebras are cellular. 
% When the input cyclotomic KLR algebra is just $\nh_n^l$, it follows by a Morita equivalence result of Webster \cite{Webcombined} that the Grothendieck group of the quiver Schur algebra is a weight space of generic quantum group representation $V_1^{\otimes l}$.  
In order to categorify the tensor product $V_r \otimes V_s$, we considered the endomorphism algebras of a smaller collection of $p$-DG cyclotomic nilHecke modules which we called \emph{two-tensor quiver Schur algebra}. We showed that this collection of $p$-DG modules were preserved by functors generating the categorical quantum $\mf{sl}_2$ action.  Using an explicit decomposition of these $\nh_n^l$-modules into indecomposable objects, we were able to identify the Grothendieck group of the corresponding endomorphism algebra with a weight space of $V_r \otimes V_s$.  It is, however, more difficult to obtain this decomposition of an arbitrary object in the full collection of modules considered by Hu and Mathas.  

In the current work, we return to the quasi-hereditary cellular quiver Schur algebra studied in \cite{HuMathas}, for the $\mf{sl}_2$-case, in the presence of a $p$-differential $\partial$, in order to categorify $V_1^{\otimes l}$ at a prime root of unity.
We introduce the notion of a $p$-DG cellular algebra. Roughly speaking, this is a cellular algebra equipped with a $p$-differential such that cellular ideals have filtrations whose subquotients are $p$-DG isomorphic to some given cellular $p$-DG module.  In a similar vein, we define a $p$-DG quasi-hereditary cellular algebra. Importing ideas from the theory of quasi-hereditary algebras, and stratifications of derived categories due to Orlov, we are able to determine the Grothendieck group of $p$-DG quasi-hereditary algebras.   We show, in particular, that quiver Schur algebra $S_n^l$ for $\nh_n^l$ is $p$-DG quasi-hereditary cellular, allowing us to deduce that the Grothendieck group is of rank $\binom{l}{n}$ over $\mathbb{O}_p$.

For a fixed $l$, there are induction and restriction functors between categories of modules over $\nh_n^l$ for varying $n$ generating a quantum $\mf{sl}_2$-action.  We show here that these functors preserve the class of cyclic modules used in the definition of the quiver Schur algebra.  Using machinery developed in \cite{KQS}, we are able to conclude that these functors lift to categories of $p$-DG modules between quiver Schur algebras $S_n^l$ for varying $n$, and that they satisfy categorical quantum group relations at a prime root of unity introduced in \cite{EQ1,EQ2}. This places the study of $p$-DG quiver Schur algebras inside the realm of $2$-representation theory at prime roots of unity.

Using the general machinery developed here and in \cite{KQS}, we would in the future like to categorify an arbitrary tensor product of finite dimensional representations at a prime root of unity.  We will no longer be in the setting of quasi-hereditary algebras as we are in here, but we expect that a $p$-DG analogue of standardly stratified algebras will help us in determining the Grothendieck group of the relevant $p$-DG algebras.

We would also like to construct a $p$-DG enhancement of the categorified WRT tangle invariant en route to a categorification of the WRT $3$-manifold invariant. In particular, we will be studying a $p$-DG analogue of Khovanov homology in an upcoming work.

In another parallel direction, the odd version of the nilHecke algebra (defined over $\Z$) was equipped with a DG structure in \cite{EllisQ} where it was shown that the Grothendieck group is isomorphic to the positive half of quantum $\mf{sl}_2$ at a fourth root of unity.  This DG structure was extended to the entire odd $\mf{sl}_2$ category in \cite{LaudaEgilmez}, which gives rise to a categorification of the entire quantum group for $\mf{sl}_2$ at a fourth root of unity, as well as a certain subalgebra of $\mf{gl}(1|1)$.  These structures should give rise to an algebraic explanation of the close relationship between special values of Alexander polynomial and the Jones polynomial.

\subsection{Summary}
Let us summarize our main goal for this paper.

\begin{thm*}[\ref{thmtensorcatfn}]
There is an action of the derived $p$-DG category of $\dot{\mathcal{U}}$ on $\oplus_{n=0}^l  \mathcal{D}^c(S_n^l)$. The action induces an identification of the Grothendieck groups 
\begin{equation*}
K_0\left(\bigoplus_{n=0}^l  \mathcal{D}^c(S_n^l)\right) \cong V_1^{\otimes l},
\end{equation*}
where the right hand side is the tensor product of the quantum $\mf{sl}_2$ Weyl module $V_1$ at a primitive $p$th root of unity.
\end{thm*}

Closely related to the quiver Schur algebra $S_n^l$ is a diagrammatically defined algebra $W_n^l$ introduced by Webster \cite[Proposition 9.7]{Webcombined}.  He showed that $S_n^l$ and $W_n^l$ are Morita equivalent. Hu and Mathas \cite[Theorem C]{HuMathas} also proved that the quiver Schur algebras in type $A$ were Morita equivalent to the Webster algebras by a chain of Morita equivalences established in the work of Stroppel and Webster \cite[Theorem A]{SWSchur}. We enhance this Morita equivalence to the $p$-DG setting.

\begin{thm*}[\ref{thmqschurwebmoritaequiv}]
The $p$-DG algebras $S_n^l$ and $W_n^l$ are $p$-DG Morita equivalent.
\end{thm*}
Note that this is an analogue of \cite[Theorem 9.10]{KQS} which deals with the case of a two-tensor quiver Schur algebra.

Stroppel and Webster introduced a cellular basis of $W_n^l$ in \cite{SWSchur}. Our final main result adapts their construction to give a diagrammatic description of the cellular basis due to Hu and Mathas for the smaller algebra $S_n^l$. It turns out, under the above isomorphism, the cellular basis of Hu-Mathas agrees with the that of Stroppel-Webster by an idempotent truncation (Corollary \ref{corcellbasesagree}). This is particularly surprising and pleasing to us, as cellular bases are highly non-unique.

\subsection{Organization}
In Section \ref{sec-small-qgroup} we define the small quantum group and the BLM form of the quantum group for $\mathfrak{sl}_2$ at a prime root of unity.  We also describe certain representations that we categorify.

A review of basic elements of hopfological algebra is given in Section \ref{sec-hopfo-algebra}.  A summary of some of the relevant machinery developed in \cite{KQS} is also presented.

In Section \ref{sec-cellular} we review the definition of a cellular algebra first introduced by Graham and Lehrer \cite{GrLe} along with its coordinate-free reformulation due to K\"oenig and Xi \cite{KoXi2}. We also recall some of its important properties developed in these works.  Quasi-hereditary cellular algebras form a special subfamily of these algebras, and they play a prominent role throughout representation theory.  We define $p$-DG versions of these algebras and provide some basic examples.  Then we extend some general results to the $p$-DG setting.  Finally we explain how to stratify the derived category of a $p$-DG quasi-hereditary cellular algebra, allowing for computation of its Grothendieck group.

Section \ref{sec-cat-sl(2)} provides a review of a categorification of quantum $\mf{sl}_2$ for generic values of the quantum parameter first introduced by Lauda, which is then extended to the root of unity setting by the first author and Elias.  A review of a categorification of a Weyl module via the $p$-DG nilHecke algebra is also given.

Some combinatorial structures arising from the nilHecke algebra are presented in Section \ref{sec-cyclic-mod} along with an important class of cyclically generated modules. We review the basic properties of these modules established by Hu and Mathas, and exhibit the compatibility of these properties with the natural $p$-differential.  

In Section \ref{sec-quiverschur} we review the quiver Schur algebras associated to nilHecke algebras first defined by Hu and Mathas.  These are defined using the cyclic modules discussed in Section \ref{sec-cyclic-mod}.  These algebras naturally inherit a $p$-DG structure from the $p$-DG structure on the cyclic modules.
We show that these quiver Schur algebras are $p$-DG quasi-hereditary algebras allowing us to compute their Grothendieck groups.  Furthermore, we show that induction and restriction functors on categories of modules for nilHecke algebras extend to a categorical action on quiver Schur algebras.  A combination of these results provides a categorification of the representation $V_1^{\otimes l}$ for the quantum group at a prime root of unity.

The Webster algebra for $\mf{sl}_2$ is presented in Section \ref{sec-Web}.  We review a cellular basis of this algebra constructed by Stroppel and Webster \cite{SWSchur}.  Inspired by this construction we give a diagrammatic presentation of the cellular basis defined by Hu and Mathas.  While the Webster algebra and the quiver Schur algebra are not isomorphic, they are Morita equivalent by work of Webster \cite[Proposition 9.7]{Webcombined}, Stroppel-Webster \cite[Theorem A]{SWSchur} and Hu-Mathas \cite[Theorem C]{HuMathas}. However, the quiver Schur algebras are computationally more manageable, as their dimensions are much smaller that the Webster algebras in general.  In the case of $\mf{sl}_2$, we give a more direct and simpler proof of the Morita equivalence in this section, which also adapts to the $p$-DG setting. We show that these two $p$-DG algebras are in fact $p$-DG Morita equivalent. 

Throughout the paper we provide examples illustrating various concepts.  In Section \ref{sec-appendix} we present the example of the quiver Schur algebra $S_2^4$.  It is the most complicated example we consider in this work, and so we highlight some of its features separately in the appendix.

\paragraph{Acknowledgements.}
The authors would like to thank Catharina Stroppel and Daniel Tubbenhauer for their helpful and enlightening discussions on cellular algebras, and in particular, for teaching us the cellular basis for Webster algebras described in \cite{SWSchur}.
The authors would also like to thank Mikhail Khovanov and Aaron Lauda for helpful conversations.
%for his continuing encouragement and support.
They would like to thank Ben Elias for his helpful suggestions and allowing them to use some diagrams in his joint works with the first author. They would also like to thank the referee for the detailed comments and corrections.
Y.~Q.~benefited greatly from discussions and ongoing projects with Peter J.~McNamara.

Y.~Q.~is partially supported by the NSF grant DMS-1763328. 
J.~S.~is supported by NSF grant DMS-1407394, PSC-CUNY Award 67144-0045, and Simons Foundation Collaboration Grant 516673.

\section{The quantum group \texorpdfstring{$\mathfrak{sl}_2$}{sl(2)} at a root of unity}\label{sec-small-qgroup}
\subsection{The small quantum group \texorpdfstring{$\mathfrak{sl}_2$}{sl(2)}}
Let $ \zeta_{2 N} $ be a primitive $2 N$th root of unity where $N$ is $2$ or odd.  The quantum group $ u_{\Q(\zeta_{2 N})}(\mathfrak{sl}_2) $, which we will denote simply by $ u_{\Q(\zeta_{2 N})} $, is the $\Q(\zeta_{2 N})$-algebra generated by $ E, F, K^{\pm 1} $ subject to relations:
\begin{enumerate}
\item[(1)] $ KK^{-1} = K^{-1}K=1$,
\item[(2)] $ K^{\pm 1}E = \zeta_{2 N}^{\pm 2} EK^{\pm 1} $, \quad $ K^{\pm 1} F = \zeta_{2 N}^{\mp 2} FK^{\pm 1} $,
\item[(3)] $ EF-FE = \frac{K-K^{-1}}{\zeta_{2 N}-\zeta_{2 N}^{-1}} $,
\item[(4)] $ E^{N} = F^{N} = 0 $.
\end{enumerate}

The quantum group is a Hopf algebra whose comultiplication map 
$$ \Delta \colon u_{\Q(\zeta_{2 N})} \longrightarrow u_{\Q(\zeta_{2 N})} \otimes_{\Q(\zeta_{2 N})} u_{\Q(\zeta_{2 N})} $$
is given on generators by
\begin{equation} \label{comultgen}
\Delta(E) = E \otimes 1 + K^{-1} \otimes E, \quad\quad \Delta(F) = 1 \otimes F + F \otimes K^{}, \quad\quad \Delta(K^{\pm 1}) = K^{\pm 1} \otimes K^{\pm 1}.
\end{equation}

Now let $ N=p $ be a prime number.  Define the auxiliary ring
\begin{equation}\label{eqn-aux-ring}
\mathbb{O}_p = \mathbb{Z}[q]/(\Psi_p(q^2)),
\end{equation}  where $ \Psi_p(q) $ is the $p$-th cyclotomic polynomial.
For the purpose of categorification, it is more convenient to use an idempotented
version of the quantum $\mf{sl}_2$ over $\mathbb{O}_p$
which we will denote simply by  $ \dot{u}_{\mathbb{O}_p} $. 

\begin{defn}
The $\mathbb{O}_p$-integral quantum group $ \dot{u}_{\mathbb{O}_p} $ is a non-unital algebra generated by $ E, F $ and idempotents $1_m$ ($m \in \Z$), subject to relations:
\begin{enumerate}
\item[(1)] $ 1_m 1_n=\delta_{m, n}1_m$,
\item[(2)] $ E1_m=1_{m+2}E $, \quad $ F1_{m+2} = 1_{m} F $,
\item[(3)] $ EF1_{m}-FE1_{m} = [m]_{\mathbb{O}_p}1_m $,
\item[(4)] $ E^p = F^p = 0 $.
\end{enumerate}
Here $ [m]_{\mathbb{O}_p} =\sum_{i=0}^{m-1} q^{1-m+2i} $ is the usual quantum integer specialized in $\mathbb{O}_p$.
\end{defn}

%The quantum group $\dot{u}_{\Q[\zeta_{2l}]}$ has an integral lattice subalgebra which we now recall. 
For any integer $n\in \{0,1,\dots, p-1\}$, let $ E^{(n)} = \frac{E^n}{[n]!} $, and $ F^{(n)} = \frac{F^n}{[n]!} $.
%The elements $ E^{(n)}$, $F^{(n)}$ ($0\leq n \leq l-1$), and $ 1_{m}$ ($m \in \Z$) generate an algebra over the ring of cyclotomic integers $ \mathcal{O}_{2l} = \mathbb{Z}[\zeta_{2l}] $. Denote this integral form by $ \dot{u}_{\zeta_{2l}} $.
%We can define in a similar fashion an integral form  $ {u}_{\mathbb{O}_p} $ and its dotted version $ \dot{u}_{\mathbb{O}_p} $ for the small quantum $\mf{sl}_2$ over $ \mathbb{O}_p $.   
Let the lower half of $ \dot{u}_{\mathbb{O}_p} $ be the subalgebra generated by the $ F^{(n)} $ ($0\leq n \leq p-1$) and denote it by $ u_{\mathbb{O}_p}^- $. Likewise, write the upper half as $u_{\mathbb{O}_p}^+$.
For more details about these algebras, see \cite[Section 3.3]{KQ}.

\subsection{The BLM integral form}
We next recall the Beilinson-Lusztig-MacPherson \cite{BLM} (BLM) integral form of quantum $\mf{sl}_2$ at a prime root of unity. The small quantum sits inside the BLM form as an (idempotented) Hopf algebra.

\begin{defn}\label{def-BLM}
The non-unital associative quantum algebra $\dot{U}_{\mathbb{O}_p}(\mathfrak{sl}_2)$, or just $\dot{U}_{\mathbb{O}_p}$, is the $\mathbb{O}_p$-algebra generated by a family of orthogonal idempotents $\{1_n|n \in \Z\}$, a family of raising operator $E^{(a)}$ and a family of lowering operator $F^{(b)}$ ($a,b\in \N$), subject to the following relations.
\begin{itemize}
\item[(1)] $1_m 1_n=\delta_{m,n}1_m$ for any $m,n \in \Z$.
\item[(2)] $ E^{(a)}1_m=1_{m+2a}E^{(a)} $,  $ F^{(a)}1_{m+2a} = 1_{m} F^{(a)}$, for any $a\in \N$ and $m\in \Z$.
\item[(3)] For any $a,b\in \N$ and $\lambda\in \Z$, 
\begin{equation}\label{eqn-divided-power}
E^{(a)}E^{(b)}1_m={ a+b \brack a}_{\mathbb{O}_p}E^{(a+b)}, \quad \quad \quad F^{(a)}F^{(b)}1_m={ a+b \brack a}_{\mathbb{O}_p}F^{(a+b)}1_m.
\end{equation}
\item[(4)] The divided power $E$-$F$ relations, which state that
\begin{subequations}
\begin{align}\label{eqn-higher-Serre-1}
E^{(a)}F^{(b)}1_m=\sum_{j=0}^{\mathrm{min}(a,b)}{a-b+m\brack j}_{\mathbb{O}_p}F^{(b-j)}E^{(a-j)}1_m, \\
F^{(b)}E^{(a)}1_m=\sum_{j=0}^{\mathrm{min}(a,b)}{b-a-m\brack j}_{\mathbb{O}_p}E^{(b-j)}F^{(a-j)}1_m. \label{eqn-higher-Serre-2}
\end{align}
\end{subequations}
\end{itemize}
The elements $E^{(a)}1_m$, $F^{(a)}1_{m}$ will be referred to as the \emph{divided power elements}.
\end{defn}

The small version $\dot{u}_{\mathbb{O}_p}$ sits inside the BLM form $\dot{U}_{\mathbb{O}_p}$ by identifying $E^{(a)}1_m$, $F^{(a)}1_m$ ($1\leq a \leq p-1$) with the elements of the same name in $\dot{U}_{\mathbb{O}_p}$.

\subsection{Representations}
\label{subsec-rep}
Let $ {V}_l $ be the unique (up to isomorphism) Weyl module for $ \dot{U}_{\mathbb{O}_p} $ of rank $ l+1 $.  It has a basis $ \lbrace v_0, \ldots, v_{l} \rbrace $ such that
\begin{equation}
\label{irreddef}
 F^{(a)} 1_m v_i =\delta_{m, l-2i} {i+a \brack a}_{\mathbb{O}_p} v_{i+a}\quad\quad E^{(a)}1_{m} v_i = \delta_{m, l-2i} {l-i+a \brack a}_{\mathbb{O}_p} v_{i-a}.
\end{equation}

On $ V_1^{\otimes l} $ there is an action of $ \dot{u}_{\mathbb{O}_p} $ given via the comultiplication map $ \Delta $.  There is also a commuting action of the braid group that factors through the Hecke algebra.  
The subspace $ 1_{l-2k}V_1^{\otimes l}$ is spanned by
vectors $ v_{i_1} \otimes \cdots \otimes v_{i_l} $ where $ i_r \in \lbrace 0, 1 \rbrace $ and $k$ of the $ i_r$ are $ 1 $. Thus $ 1_{l-2k}V_1^{\otimes l} $
is free of rank ${l \choose k}$ over $\mathbb{O}_p$.

We record the following result for later use.

%\begin{lem}
%Let $v$ be a weight vector of $V_1^{\otimes (l-1)}$ and $v_0$ be the highest %weight basis vector of $V_1$. Then, for any $a\in \N$, the following identity 
%\[
%\sum_{i=0}^a (-1)^i q^{?}F^{(i)}(v_0\otimes F^{ (a-i) } v) = 0.
%\]
%holds inside $V_1^{\otimes l}$.
%\end{lem}
%\begin{proof}
%\end{proof}

\begin{lem}\label{lemSerrelikerelation}
Let $v$ be a weight vector of $V_1^{\otimes (l-1)}$ and $v_0$ be the highest weight basis vector of $V_1$. Then, for any $a\in \N$, the following identity 
\[
\sum_{i=0}^a (-1)^i q^{-i(a-2)}F^{(a-i)}(F^{ (i) } v \otimes v_0) = 0.
\]
holds inside $V_1^{\otimes l}$.
\end{lem}
\begin{proof}
The following comultiplication formulas may derived from \eqref{comultgen} using induction
\begin{equation}
\label{gencomultform}
\Delta(E^{(t)}) = \sum_{j=0}^t q^{-j(t-j)} E^{(t-j)} K^{-j} \otimes E^{(j)} \quad \quad
\Delta(F^{(t)}) = \sum_{j=0}^t q^{-j(t-j)}  F^{(t-j)} \otimes F^{(j)} K^{t-j}.
\end{equation}

The lemma follows after applying the above formula for $\Delta(F^{(a-i)})$ and using the quantum binomial identity
\begin{equation}
\sum_{j=0}^N (-1)^j q^{\pm(N-1)j} {N \brack j}_{\mathbb{O}_p}=0.
\end{equation}
We leave the details to the reader as an exercise.
\end{proof}

\section{Elements of hopfological algebra}
\label{sec-hopfo-algebra}
In this section, we gather some necessary background material on hopfological algebra as developed in \cite{QYHopf}.

\paragraph{Convention.}We fix some notations concerning $\Z$-graded vector spaces over a ground field $\Bbbk$. 

Let $M=\oplus_{i\in \Z} M^i$ and $N=\oplus_{j\in \Z}N^j$ be $\Z$-graded vector spaces over  $\Bbbk$. We set $M\otimes_{\Bbbk} N$, or simply $M\otimes N$, to be the graded vector space
\[
M\otimes N:=\bigoplus_{k\in \Z}(M\otimes N)^k,\quad \quad  (M\otimes N)^k:=\bigoplus_{i+j=k}M^i\otimes N^j.
\]
For any integer $k\in \Z$, we denote by $M\{k\}$  the graded vector space $M$ with its grading shifted up by $k$:
$
(M\{k\})^i=M^{i-k}.
$
The morphisms space $\Hom_\Bbbk^0(M,N)$ consists of homogeneous $\Bbbk$-linear maps from $M$ to $N$:
\[
\Hom_\Bbbk^0(M,N):=\left\lbrace f:M\longrightarrow N\middle| f(M^i)\subseteq N^i\right\rbrace.
\]
Writing $\Hom^i_\Bbbk (M,N):=\Hom^0_\Bbbk(M,N\{i\})=\left\lbrace f:M\longrightarrow N\middle| f(M^j)\subseteq N^{i+j}\right\rbrace$, we set the graded hom space to be
\[
\HOM_\Bbbk(M,N):=\bigoplus_{i\in \Z} \Hom^i_\Bbbk (M,N).
\]
If no confusion can be caused, we will simplify $\HOM_\Bbbk(M,N)$ to $\HOM(M,N)$. 
A special case is the graded dual $M^*=\HOM(M,\Bbbk)$.

Graded module categories considered in this work will be enriched over $\Z$-graded vector spaces. Thus if $A$ is a $\Z$-graded algebra, and $M$ and $N$ are graded $A$-modules, we will denote by $\Hom_A(M,N)$ the space of degree preserving $A$-module homomorphisms, and
\[
\HOM_A(M,N):=\bigoplus_{k\in \Z}\Hom_{A}(M,N\{k\}).
\]

\subsection{\texorpdfstring{$p$}{p}-DG derived categories}
As a matter of notation for the rest of the paper, the undecorated tensor product symbol $\otimes$ will always denote tensor product over the ground field $\Bbbk$. All of our algebras will be graded so $A\dmod$ will denote the category of graded $A$-modules.

\begin{defn}\label{def-p-DGA}Let $\Bbbk$ be a field of positive characteristic $p$. A $p$-DG algebra $A$ over $\Bbbk$ is a $\Z$-graded $\Bbbk$-algebra equipped with a degree-two\footnote{In general one could define the degree of $\dif_A$ to be one. We adopt this degree only to match earlier grading conventions in categorification. One may adjust the gradings of the algebras we consider so as to make the degree of $\dif_A$ to be one, but we choose not to do so.} endomorphism $\dif_A$, such that, for any elements $a,b\in A$, we have
\[
\dif_A^p(a)=0, \quad \quad \dif_A(ab)=\dif_A(a)b+a\dif_A(b).
\]
\end{defn}

Compared with the usual DG case, the lack of the sign in the Leibniz rule is because of the fact that the Hopf algebra $H:=\Bbbk[\dif]/(\dif^p)$ is a genuine Hopf algebra, not a Hopf super-algebra.

As in the DG case, one has the notion of left and right $p$-DG modules.

\begin{defn}\label{def-p-DG-module}Let $(A,\dif_A)$ be a $p$-DG algebra. A left $p$-DG module $(M,\dif_M)$ is a $\Z$-graded $A$-module endowed with a degree-two endomorphism $\dif_M$, such that, for any elements $a\in A$ and $m\in M$, we have
\[
\dif_M^p(m)=0, \quad \quad \dif_M(am)=\dif_A(a)m+a\dif_M(m).
\]
Similarly, one has the notion of a right $p$-DG module.
\end{defn}

It is readily checked that the category of left (right) $p$-DG modules, denoted $(A,\dif)\dmod$ ($(A^{\circ},\dif)\dmod$, where $A^{\circ}$ denotes the opposite algebra), is abelian, with morphisms being those grading-preserving $A$-module maps that also commute with differentials. When no confusion can be caused, we will drop all subscripts in differentials.

\begin{rem}
\label{rmksmashproductalgebra} A $p$-DG algebra is a
graded $H$-module algebra, where $H$ stands, as above, for the graded Hopf algebra $\Bbbk[\dif]/(\dif^p)$. The $p$-derivation on
$A$ allows us to define the \emph{smash product ring $A_{\dif}$} as
follows. As a $\Bbbk$-vector space, $A_{\dif}\cong A\otimes H$, and
the multiplication on $A_\dif$ is given by $(a_1\otimes 1)\cdot(a_2\otimes 1)=(a_1a_2)\otimes
1$, $(1\otimes h_1)\cdot (1\otimes h_2)=1\otimes(h_1h_2)$, $(a_1\otimes 1)(1\otimes
h_1)=a_1\otimes h_1$ for any $a_1,~a_2\in A$, $h_1, h_2 \in H$, and by the
rule for commuting elements of the form $a\otimes 1$, $1\otimes \dif$:
$$(1\otimes \dif)(a\otimes 1)=(a\otimes 1)(1\otimes \dif)+\dif(a)\otimes 1.$$
Note that $1\o H \subset A_\dif$ is a subalgebra. Moreover, since
$A$, $H$ are compatibly graded, i.e.,
$\mathrm{deg}(\dif(a))=\textrm{deg}(a)+2$ for any homogeneous $a\in
A$, the commutator equation is homogeneous so that $A_\dif$ is
graded. It can be easily checked that the category of left (resp. right) $p$-DG modules is equivalent
to the category of graded left (resp. right) $A_\dif$-modules.
\end{rem}

If $M$, $N$ are two left $p$-DG module over $A$, then the graded homomorphism space
\[
\HOM_A(M,N):=\bigoplus_{i\in \Z}\Hom_A(M,N\{ i \})
\]
is an $H$-module by declaring the action
\begin{equation}\label{eqnHmodstructureonmorphismspace}
\dif(f)(m):=\dif(f(m))-f(\dif(m)).
\end{equation}
This action equips the category of $p$-DG modules over $A$ an enriched structure by the category of graded $H$-modules. The space of degree-zero $H$-invariants under this action agrees with the space of $(A,\dif)$-module, or, equivalently, $A_\dif$-module homomorphisms.

\begin{defn} \label{def-null-homotopy}
Let $M$ and $N$ be two $p$-DG modules. A morphism $f:M\lra N$ in $(A,\dif)\dmod$ is called \emph{null-homotopic} if there is an $A$-module map $h$ of degree $2-2p$ such that
\[
f=\dif^{p-1}(h):=\sum_{i=0}^{p-1}\dif_N^{i}\circ h \circ \dif_M^{p-1-i}.
\]
\end{defn}

It is an easy exercise to check that null-homotopic morphisms form an ideal in $(A, \dif)\dmod$. The resulting quotient category, denoted $\mathcal{H}(A)$, is called the \emph{homotopy category} of left $p$-DG modules over $A$, and it is a triangulated category.

The simplest $p$-DG algebra is the ground field $\Bbbk$ equipped with the trivial differential, whose homotopy category is denoted $\mc{H}(\Bbbk)$\footnote{This is usually known as the graded stable category of $\Bbbk[\dif]/(\dif^p)$. }. Modules over $(\Bbbk,\dif)$ are usually referred to as \emph{$p$-complexes}. In general, given any $p$-DG algebra $A$, one has a forgetful functor
\begin{equation}
\mathrm{For}:\mathcal{H}(A)\lra \mc{H}(\Bbbk)
\end{equation}
by remembering only the underlying $p$-complex structure up to homotopy of any $p$-DG module over $A$. A morphism between two $p$-DG modules $f:M\lra N$ (or its image in the homotopy category) is called a \emph{quasi-isomorphism} if $\mathrm{For}(f)$ is an isomorphism in $\mc{H}(\Bbbk)$. Denoting the class of quasi-isomorphisms in $\mathcal{H}$ by $\mathcal{Q}$, we define the $p$-DG derived category of $A$ to be
\begin{equation}\label{eqn-derived-cat}
\mathcal{D}(A):=\mathcal{H}(A)[\mathcal{Q}^{-1}],
\end{equation}
the localization of $\mathcal{H}(A)$ at quasi-isomorphisms. By construction, $\mathcal{D}(A)$ is triangulated.

\subsection{Hopfological properties of \texorpdfstring{$p$}{p}-DG modules}
Many constructions in the usual homological algebra of DG-algebras translate over into the $p$-DG context without any trouble. For a starter, it is easy to see that the homotopy category of the ground field coincides with the derived category: $\mc{D}(\Bbbk)\cong \mc{H}(\Bbbk)$. We will see a few more illustrations of the similarities in what follows.

We first recall the following definitions.

\begin{defn}\label{def-finite-cell} Let $A$ be a $p$-DG algebra, and $K$ be a (left or right) $p$-DG module. 
\begin{enumerate}
\item[(1)] The module $K$ is said to satisfy \emph{property P} if there exists an increasing, possibly infinite, exhaustive $\dif_K$-stable filtration $F^\bullet$, such that each subquotient $F^{\bullet}/F^{\bullet -1}$ is isomorphic to a direct sum of $p$-DG direct summands of $A$. 
\item[(2)] A $p$-DG module $K$ is called a \emph{cofibrant} if it is a $p$-DG direct summand of a property-P module.\item[(3)]The module $K$ is called a \emph{finite cell module}, if it is cofibrant, and as an $A$-module, it is finitely generated (necessarily projective by the cofibrant requirement).
\end{enumerate}
\end{defn}

We have the following alternative characterization of a cofibrant module.

\begin{lem}\label{lemcofchar}
Let $A$ be a $p$-DG algebra. A $p$-DG module $M$ over $A$ is cofibrant if and only if $M$ is $A$-projective, and for any acyclic $p$-DG $A$-module $N$, the $p$-complex $\HOM_A(M,N)$ is contractible.
\end{lem}
\begin{proof}
See \cite[Corollary 6.9]{QYHopf}.
\end{proof}

Cofibrant modules are the analogues of K-projective resolutions in usual homological algebra. For instance, the morphism spaces from a cofibrant module $M$ to any $p$-DG module $N$ coincide in both the homotopy and derived categories:
\begin{equation}\label{eqnmorphismspaceoutofcofibrantmod}
\Hom_{\mc{D}(A)}(M,N)  \cong  \Hom_{\mc{H}(A)}(M,N)
\cong
\dfrac{\textrm{Ker}(\dif:\Hom^0_A(M,N)\lra
\Hom_A^{2}(M,N))}{\textrm{Im}(\dif^{p-1}:\Hom_A^{2-2p}(M,N)\lra
\Hom_A^0(M,N))}.
\end{equation}

It is a theorem \cite[Theorem 6.6]{QYHopf} that there are always sufficiently many property-P modules: for any $p$-DG module $M$, there is a surjective quasi-isomorphism 
\begin{equation}\label{eqn-bar-resolution}
\mathbf{p}(M)\lra M
\end{equation} 
of $p$-DG modules, with $\mathbf{p}(M)$ satisfying property P.  We will usually refer to such a property-P replacement $\mathbf{p}(M)$ for $M$ as a \emph{bar resolution}. The proof of its existence is similar to that of the usual (simplicial) bar resolution for DG modules over DG algebras. 

In a similar vein, finite cell modules play the role of finitely-generated projective modules in usual homological algebra.

\subsection{\texorpdfstring{$p$}{p}-DG functors}
A $p$-DG bimodule $_AM_B$ over two $p$-DG algebras $A$ and $B$ is a $p$-DG module over $A\otimes B^{\mathrm{op}}$. One has the associated tensor and (graded) hom functors
\begin{equation}\label{eqn-def-tensor}
M\otimes_B(-): (B,\dif)\dmod\lra (A,\dif)\dmod, \quad X\mapsto M\otimes_B X,
\end{equation}
\begin{equation}\label{eqn-def-hom}
\HOM_A(M,-): (A,\dif)\dmod\lra (B,\dif)\dmod, \quad Y\mapsto \HOM_A(M,Y),
\end{equation}
which form an adjoint pair of functors. In fact, we have the following enriched version of the adjunction.

\begin{lem}\label{lem-adjunction-tensor-hom}
Let $A$, $B$ be $p$-DG algebras and $M$ a $p$-DG bimodule over $A\otimes B^{\mathrm{op}}$. Then, for any $p$-DG $A$-module $Y$ and $B$-module $X$, there is an isomorphism of $p$-complexes
\[
\HOM_A(M\otimes_B X, Y)\cong \HOM_B(X,\HOM_A(M,Y)).
\]
\end{lem}
\begin{proof}
See \cite[Lemma 8.5]{QYHopf}.
\end{proof}

The functors descend to derived categories once appropriate property-P replacements are utilized. For instance, the derived tensor functor is given as the composition
 \begin{equation}\label{eqnderivedtensor}
M\otimes^{\mathbf{L}}_B(\mbox{-}):\mc{D}(B)\lra \mc{D}(A), X\mapsto M\otimes_B\mathbf{p}_B(X)
\end{equation}
where $\mathbf{p}_B(X)$ is a bar resolution for $X$ as a $p$-DG module over $B$. 
Likewise, the derived $\HOM$, denoted $\RHOM_A(M,-)$, is given by the functor
\begin{equation}\label{eqnderivedhom}
\RHOM_A(M,\mbox{-}) :\mc{D}(A)\lra\mc{D}(B),\quad Y\mapsto  \HOM_{A}(\mathbf{p}_A(M),Y).
\end{equation}
The functors form an adjoint pair, in the sense that
\begin{equation}\label{eqntensorhomadjunction}
\Hom_{\mc{D}(A)}(M\otimes_B^{\mathbf{L}}X,Y)\cong \Hom_{\mc{D}(B)}(X,\RHOM_A(M,Y)).
\end{equation}

\subsection{Compact objects and Grothendieck groups}
We next recall the notion of compact modules, which takes place in the derived category.

\begin{defn} \label{def-compact-mod} Let $A$ be a $p$-DG algebra. A $p$-DG module $M$ over $A$ is called \emph{compact} (in the derived category $\mc{D}(A)$) if and only if, for any family of $p$-DG modules $N_i$ where $i$ takes value in some index set $I$, the natural map
\[
\bigoplus_{i\in I}\Hom_{\mc{D}(A)}(M,N_i)\lra \Hom_{\mc{D}(A)}(M,\bigoplus_{i\in I} N_i)
\]
is an isomorphism of $\Bbbk$-vector spaces.
\end{defn}

We recall the following characterization of compact objects.

\begin{thm}[Characterization of compact modules]
\label{thmcompactmodules} The compact
objects in $\mc{D}(A)$ are those which are isomorphic in the
derived category to a direct summand of a finite cell module.
\end{thm}
\begin{proof} See Theorem 7.14 and Corollary
7.15 of \cite{QYHopf}.
\end{proof}

The strictly full subcategory of $\mc{D}(A)$ consisting of compact modules will be denoted by $\mc{D}^c(A)$.  It is triangulated and will be referred to as the \emph{compact derived category}. The triangulated category $\mc{D}(A)$ is \emph{compactly generated} by the set of $p$-DG modules $\{A\{k\}|k\in \Z\}$ in the sense that
\[
\Hom_{\mc{D}(A)}(A\{k\}, M)=0
\]
if and only if $M=0$.

The following result will be used later.

\begin{prop}\label{propfullyfaithfulness}
Let $A$, $B$ be $p$-DG algebras, and $_AM_B$ a $p$-DG bimodule that descends to a compact object in $\mc{D}(A)$.
\begin{enumerate}
\item[(i)] The derived tensor functor $M\otimes^{\mathbf{L}}_B(\mbox{-}):\mc{D}(B)\lra \mc{D}(A)$ is fully faithful if and only if the canonical unit map
\[
\mathrm{Id}_{\mc{D}(B)}\Rightarrow \mathbf{R}\HOM_A(M,M\otimes_B^\mathbf{L}(\mbox{-}))
\]
is an isomorphism on the set of compact generators $\{B\{k\}|k\in \Z\}$.
\item[(ii)]The derived hom functor $\RHOM_A(M,\mbox{-}) :\mc{D}(A)\lra\mc{D}(B)$ is fully faithful if and only if the canonical counit map
\[
M\otimes_B^\mathbf{L}(\mathbf{R}\HOM_A(M,(\mbox{-})))\Rightarrow \mathrm{Id}_{\mc{D}(A)} 
\]
is an isomorphism on the set of compact generators $\{A\{k\}|k\in \Z\}$.
\end{enumerate}
\end{prop}
\begin{proof}
It is well-known that the fully-faithfulness of adjoint functors is equivalent to the corresponding unit/counit maps being isomorphisms. On compactly generated triangulated categories, the isomorphisms of functors that commute with direct sums can be detected on the set of compact objects. See, for instance, \cite[Lemma 7.18]{QYHopf} for a proof.
\end{proof}

As in the DG case, in order to avoid trivial cancellations in the Grothendieck group, one should restrict the class of objects used to define $K_0(A)$. It turns out that the correct condition is that of compactness. 
So we let $K_0(A):=K_0(\mc{D}^c(A))$.
What we gain as dividend in the current situation is that, since $\mathcal{D}(A)$ is a ``categorical module'' over $\mc{D}(\Bbbk)$, the abelian group $K_0(A)$ naturally has a module structure over the auxiliary cyclotomic ring  at a $p$th root of unity, which was defined in equation \eqref{eqn-aux-ring} of the previous section:
\begin{equation}\label{eqn-aux-ring-2}
\mathbb{O}_p \cong K_0(\mc{D}^c(\Bbbk)).
\end{equation}
The Grothendieck group $K_0(A)$ will be the primary algebraic invariant of the triangulated category $\mc{D}(A)$ that will interest us in this work.

%
%A class of examples for which the Grothendieck group is relatively easy to compute is the following. The notion is introduced in \cite[Section 2]{EQ1}.
%
%\begin{defn}\label{def-positive-p-dga}
% A $p$-DG algebra $A$ is called \emph{positive} if the following three conditions hold:
%\begin{itemize}
%\item[(1)] $A$ is supported on non-negative degrees: $A= \oplus_{k\in \N}A^k$, and it is finite dimensional in each degree.
%\item[(2)] The homogeneous degree zero part $A^0$ is semisimple.
%\item[(3)] The differential $\dif_A$ acts trivially on $A^0$.
%\end{itemize}
%\end{defn}
%
%\begin{thm}
%\label{thm-K-group-positive}Let $A$ be a positive $p$-DG algebra. Then there is an isomorphism of Grothendieck groups
%\[
%K_0(A)\cong K_0^\prime(A)\o_{\Z[q,q^{-1}]}\mathbb{O}_p, 
%\]
%where $K_0^\prime(A)$ stands for the usual Grothendieck group of graded projective $A$-modules.
%\end{thm}
%\begin{proof}
%See \cite[Corollary 2.18]{EQ1}.
%\end{proof}

\begin{rem}[Grading shift]
rmk-uniqueness-of-dLet $M$ be a $p$-DG module and $i\in \Z$ be an integer. In what follows, we will abuse notation by writing $q^i M$ for the degree-shifted module $M\{i\}$. More generally, if $g(q)=\sum_{i} a_iq^i\in \N[q,q^{-1}]$, then we will write
\[
g(q)M:=\bigoplus_{i} (q^iM)^{\oplus a_i}.
\]
On the level of Grothendieck groups, the symbol of $g(q)M$ is equal to the multiplication of $g(q)$, regarded here as an element of $\mathbb{O}_p$, with the symbol of $M$.

Likewise, if $\mc{E}$ is a $p$-DG functor given by tensoring with the $p$-DG bimodule $E$ over $(A,B)$,
\[
\mc{E}:(B,\dif)\dmod\lra (A,\dif)\dmod,\quad \quad M\mapsto E\otimes_B M,
\]
then we will write $q^i \mc{E}$ as the functor represented by $q^iE$.
\end{rem}

\subsection{A double centralizer property}\label{sec-double}
In this subsection we review a context where a categorical action on module categories for one family of algebras induces a categorical action for another family when the two families are related by a double centralizer property.
For more details see \cite[Section 4]{KQS}.

In this subsection, we consider $M_A$ to be a $p$-DG right module over $(A,\partial)$.  Then
$B=\END_A(M)$ is a $p$-DG algebra, and $M$ becomes a $p$-DG bimodule over $(B,A)$. 
We define 
\begin{equation}
{^\vee M}:= \HOM_B({_BM},B), \quad \quad \quad M^\vee = \HOM_A(M_A,A),
\end{equation}
as $p$-DG $(A,B)$-bimodules.

A $p$-DG algebra $(A,\partial)$ is said to be a symmetric Frobenius $p$-DG algebra if there is a non-degenerate pairing $\epsilon \colon A \rightarrow \Bbbk$ such that $\epsilon(ab)=\epsilon(ba)$.

\begin{defn}
\label{def-of-commutant}
Let $A$ be a symmetric Frobenius $p$-DG algebra and $M$ be a faithful $p$-DG right $A$-module.
\begin{enumerate}
\item[(1)] Define the {\it generalized Soergel functor}
\[
\mc{V}: (B,\partial) \dmod\lra (A,\partial) \dmod,\quad Y\mapsto \HOM_B(M,Y).
\]
\item[(2)] The right adjoint of $\mc{V}:B\dmod\lra A\dmod$ is called the {\it projection functor}
and is given by
\[
\mc{I}: (A,\partial) \dmod\lra (B,\partial) \dmod, \quad X\mapsto \HOM_A({^\vee M}, X).
\]
\end{enumerate}
\end{defn}

\begin{thm}\label{thm-double-centralizer}
Let $A$ be a graded symmetric Frobenius algebra, and $M_A$ a faithful finitely-generated graded right $A$-module. Set $B=\END_A(M_A)$ so that $M$ is a $(B,A)$-bimodule. Then there is a canonical isomorphism of graded algebras:
\[
A \cong \END_B({_BM}), \quad a\mapsto (m\mapsto ma)
\]
for any $a\in A$ and $m\in M$. Furthermore, the module $M$ is projective over $B$, and
the generalized Soergel functor $\mc{V}$ is fully-faithful on the class of projective $B$-modules.

\end{thm}
\begin{proof}
See Theorem 4.6 and Theorem 4.12 of \cite{KQS}.
\end{proof}

\begin{defn}\label{def-filtered-envelope}
Let $A$ be a $p$-DG algebra, and $X$ be a left (or right) $p$-DG module. The \emph{filtered $p$-DG envelope of $X$}, consists of direct summands of left (or right) $p$-DG modules which have a finite $A$-split filtration, whose subquotients are isomorphic to grading shifts of $X$ as $p$-DG modules.
\end{defn}

A crucial property that we will use about the filtered $p$-DG envelope is the following.

\begin{prop}\label{propfilteredenvelope2outof3}
Let $A$ be a symmetric Frobenius $p$-DG algebra, and $M_A$ be a finitely generated faithful $A$-module. Suppose 
\begin{equation} \label{Aspliteq}
0\lra X\lra Y \lra Z\lra 0
\end{equation}
is a short exact sequence of finitely generated right $p$-DG modules over $A$ that splits when forgetting the differential. If two terms out of $X$, $Y$ and $Z$ lie in the filtered $p$-DG envelope of $M$, then so does the third one.

More generally, suppose $\{X_i|i=0,\dots, n\}$ is a family that fits into an exact sequence of $p$-DG $A$-modules
\[
\xymatrix{0\ar[r]& X_0 \ar[r]^{d_0} & X_1\ar[r]^{d_1} &\cdots \ar[r]^{d_{n-1}} & X_n\ar[r] & 0}
\]
which is a null-homotopic complex of $A$-modules (in the usual sense, not the $p$-dg sense). If all but one of $X_i$, $i=0,\dots, n$, lie in the filtered $p$-DG envelope of $M_A$, then so is the remaining one.
\end{prop}
\begin{proof}
Since the sequence \eqref{Aspliteq} is $A$-split, applying the functor $\HOM_A(M,\mbox{-})$ to the sequence results in the exact sequence of $p$-DG $B$-modules
\[
0\lra \HOM_A(M,X)\lra \HOM_A(M,Y)\lra \HOM_A(M,Z)\lra 0
\]
that is $B$-split. If two out of the three modules, say $X$ and $Y$, are in the filtered $p$-DG envelope of $M$, then $\HOM_A(M,X)$ and $\HOM_A(M,Y)$ are both cofibrant $p$-DG modules over $B$ by Definition \ref{def-finite-cell}. Therefore, using Lemma \ref{lemcofchar}, we see that the third term $\HOM_A(M,Z)$ is also cofibrant. Again by Definition \ref{def-finite-cell},  since $\HOM_A(M,Z)$ is finitely generated over $B$, it lies in the filtered $p$-DG envelope of the free module $B$.

Now we apply the exact functor $\HOM_B({^\vee M_B},\mbox{-})$ (Theorem \ref{thm-double-centralizer}) to $\HOM_A(M,Z)$. We then obtain
\begin{align*}
\HOM_B({^\vee M},\HOM_A(M,Z)) & \cong \HOM_A({^\vee M} \otimes_B M, Z)\cong \HOM_A(\END_B(M),Z)\\
 & \cong \HOM_A(A, Z)\cong Z.
\end{align*}
Here the first isomorphism follows from the tensor-hom adjunction, while the second and third follow from the $B$-projectivity of $M$ and the double centralizer property (Theorem \ref{thm-double-centralizer}). This implies that $Z$ appears in the filtered $p$-DG envelope of $\HOM_B({^\vee M}, B)\cong M$. 

The last statement follows readily from the first one by breaking up the long exact sequence into $A$-splitting short exact sequences
\[
0\lra \mathrm{Ker}(d_i)\lra X_i \lra \mathrm{Im}(d_i)\lra 0 ,\quad \quad (i=1,\dots, n-1)
\]
and an induction on $n$. The proposition follows.
\end{proof}

Our next goal is to determine a condition under which categorical constructions on the $A$-module level will automatically translate into categorical constructions on the $B$-module level. To do so, let us consider the following situation. Let $A_i$, $(i=1,2,3)$, be three finite-dimensional symmetric Frobenius $p$-DG algebras, and $M_i$, $(i=1,2,3)$, be faithful left $p$-DG modules over respective $A_i$'s.  Set $B_i=\END_{A_i}(M_i)$ for each $i$. Suppose we now have $p$-DG functors between $A_i$-module categories given by bimodules $_{A_2}{E_1}_{A_1}$, $_{A_3}{E_2}_{A_2}$:
\[
\mc{E}_1:(A_1,\partial) \dmod\lra (A_2,\partial) \dmod,\quad X\mapsto E_1 \otimes_{A_1} X ,
\]
\[
\mc{E}_2: (A_2,\partial) \dmod\lra (A_3, \partial) \dmod,\quad Y\mapsto E_2 \otimes_{A_2} Y .
\]

Together with the generalized Soergel and projection functors, we have a diagram:
\begin{equation}
\begin{gathered}
\xymatrix{
(A_1,\partial) \dmod \ar[r]^{\mc{E}_1}\ar@<-0.7ex>[d]_{\mc{I}_1} & (A_2, \partial) \dmod \ar[r]^{\mc{E}_2}\ar@<-0.7ex>[d]_{\mc{I}_2} & (A_3,\partial) \dmod\ar@<-0.7ex>[d]_{\mc{I}_3}\\
(B_1,\partial) \dmod \ar@<-0.7ex>[u]_{\mc{V}_1} & (B_2,\partial) \dmod \ar@<-0.7ex>[u]_{\mc{V}_2} & (B_3,\partial) \dmod \ar@<-0.7ex>[u]_{\mc{V}_3}
}
\end{gathered} \ .
\end{equation}
Composing functors gives rise to
\[
\mc{E}_1^{!}:= \mc{I}_2\circ \mc{E}_1\circ \mc{V}_1,\quad
\mc{E}_2^{!}:= \mc{I}_3\circ \mc{E}_2\circ \mc{V}_2
\]
and their composition
\begin{equation}
\mc{E}_2^! \circ \mc{E}_1^{!}: (B_1,\partial) \dmod \lra (B_3,\partial) \dmod.
\end{equation}
It is a natural question to ask whether this functor agrees with the composition
\begin{equation}
\mc{I}_3\circ \mc{E}_2\circ \mc{E}_1\circ \mc{V}_1:(B_1,\partial) \dmod\lra (B_3,\partial) \dmod.
\end{equation}

For the statement of the next theorem, we will utilize the graded dual of a graded module. If $M=\oplus_{i\in \Z} M^i$ is a graded right $A$ module, then set
\begin{equation}
    M^*:= \bigoplus_{i\in \Z} (M^i)^* =\bigoplus_{i\in \Z} \Hom_\Bbbk(M^i,\Bbbk),
\end{equation}
which inherits a graded left $A$-module structure.

\begin{thm} \label{thm-pdg-extension}
Suppose $\mc{E}_{ij}:(A_j,\dif)\dmod\lra (A_i,\dif)\dmod$, $i,j\in I$, are $p$-DG functors given by tensoring with $p$-DG bimodules $E_{ij}$ over $(A_i,A_j)$:
$$
\mc{E}_{ij}:(A_j,\dif)\dmod\lra (A_i,\dif)\dmod,\quad N\mapsto {E_{ij}}\otimes_{A_j} N,
$$ 
and the collection of functors $\mc{E}_{ij}$ sends the filtered $p$-DG envelope of $M_j^*$ into that of $M_i^*$. 
Then the following statements hold.
\begin{enumerate}
\item[(i)]  There is an isomorphism of functors
\begin{equation*}
\mc{E}_{kj}^!\circ \mc{E}_{ji}^{!}\cong\left( \mc{E}_{kj}\circ \mc{E}_{ji}\right)^!:=\mc{I}_k \circ \mc{E}_{kj}\circ \mc{E}_{ji} \circ \mc{V}_i: (B_i,\partial) \dmod \lra (B_k,\partial) \dmod.
\end{equation*}
\item[(ii)]
The extended $p$-DG functors $\mc{E}_{ij}^!$ preserve compact cofibrant $p$-DG modules in $\oplus_{i\in I}(B_i,\dif)\dmod$.
\end{enumerate}
\end{thm}
\begin{proof}
See Theorem 4.16 and Proposition 4.21 of \cite{KQS}.
\end{proof}

Finally, to conclude this section, let us record a special $p$-DG Morita equivalence result for later use.

\begin{prop} \cite[Proposition 4.22]{KQS}\label{prop-p-dg-Morita}
Let $M_1$ and $M_2$ be $p$-DG right modules over $A$, and let $M_1^\prime$ be in the filtered $p$-DG envelope of $M_1\oplus M_2$, whose associated graded has the form $M_1\oplus g(q)M_2$ for some $g(q)\in \N[q,q^{-1}]$. Set $B_1:=\END_A(M_1\oplus M_2)$ and $B_2:=\END_A( M_1^\prime\oplus M_2)$. Then there is a derived equivalence between $\mc{D}(B_1)$ and $\mc{D}(B_2)$.
\end{prop} 

\section{Cellular algebras under \texorpdfstring{$p$}{p}-differential} \label{sec-cellular}
In the greatly influential work \cite{GrLe}, Graham and Lehrer introduced the notion of cellular algebras. An equivalent and coordinate-free reformulation of cellular algebras was later defined by K\"onig and Xi \cite{KoXi2}, with further homological properties of cellular algebras studied in greater detail.

In this section, we first recall the two equivalent definitions of a cellular algebra and their graded adaptions as in Hu and Mathas \cite{HuMathasKLR}. Then we introduce the notion of $p$-DG cellular algebras and study their hopfological properties.

\subsection{Graded cellular algebras}
We start by recording the combinatorial description due to Graham and Lehrer \cite{GrLe} (see \cite[Section 2.2]{HuMathasKLR} for the graded version). 

\begin{defn}
\label{defcellalgebra}
Let $A$ be a finite-dimensional $\Z$-graded algebra over $\Bbbk$.  A \emph{graded cell datum} for $A$ is an ordered quadruple $(\mathcal{P}, \T,C, \deg)$ where $(\mathcal{P}, \geq )$ is the \emph{weight poset}, $\T(\lambda)$ is a finite set for
$\lambda \in \mathcal{P}$, and
\begin{equation*}
C \colon \bigsqcup_{\lambda \in \mathcal{P}} \T(\lambda) \times \T(\lambda) \longrightarrow A,~~(\mathfrak{s}, \mathfrak{t}) \mapsto c_{\mathfrak{s} \mathfrak{t}}^{\lambda}; \quad \quad
\deg \colon \bigsqcup_{\lambda \in \mathcal{P}} \T(\lambda) \longrightarrow \Z
\end{equation*}
are two functions such that $C$ is injective and
\begin{enumerate}
\item[(1)] Each basis element $c_{\mathfrak{s} \mathfrak{t}}^{\lambda}$ is homogeneous of degree $\deg c_{\mathfrak{s} \mathfrak{t}}^{\lambda} = \deg \mathfrak{s} + \deg \mathfrak{t}$, for $\lambda \in \mathcal{P}$ and $\mathfrak{s},\mathfrak{t} \in \T(\lambda)$.
\item[(2)] The set $\{ c_{\mathfrak{s} \mathfrak{t}}^{\lambda} | \mathfrak{s}, \mathfrak{t} \in \T(\lambda), \lambda \in \mathcal{P} \} $ constitutes a basis of $A$.
\item[(3)] If $\mf{s},\mf{t} \in \T(\lambda)$, for some $\lambda \in \mathcal{P}$ and $a \in A$, then there exist scalars $r_{\mathfrak{t} \mathfrak{v}}(a)$ which do not depend on $\mathfrak{s}$ such that 
\begin{equation*}
a c_{\mathfrak{s} \mathfrak{t}}^{\lambda}  = \sum_{\mathfrak{v} \in \T(\lambda)}
r_{\mathfrak{s} \mathfrak{v}}(a) c_{\mathfrak{v} \mathfrak{t}}^{\lambda} \quad (\mod~A^{> \lambda}),
\end{equation*}
where $A^{> \lambda}$ is the subspace of $A$ spanned by 
$\{ c^{\mu}_{\mathfrak{a} \mathfrak{b}} | \mu > \lambda; \mathfrak{a}, \mathfrak{b} \in \T(\mu) \}$.
\item[(4)] The linear map $* \colon A \rightarrow A$ determined by 
$(c_{\mathfrak{s} \mathfrak{t}}^{\lambda})^* = c_{\mathfrak{t} \mathfrak{s}}^{\lambda}$ for all $\lambda \in \mathcal{P}$ and all $\mathfrak{s}, \mathfrak{t} \in \mathcal{P}$, is an anti-involution of $A$.
\end{enumerate}
A \emph{graded cellular algebra} is a graded algebra which is equipped with a graded cell datum. The
basis
$\{c_{\mathfrak{s} \mathfrak{t}}^{\lambda} | \lambda \in \mathcal{P}; \mathfrak{s}, \mathfrak{t} \in \T(\lambda) \}$ is called a \emph{graded cellular basis} of $A$.
\end{defn}

 In what follows, for a graded algebra $A$ with an anti-involution $*$ and a left $A$-module $\Delta$, we always set $\Delta^\circ$ as the same underlying vector space as $\Delta$ but equipped with the \emph{right} $A$-module action
\begin{equation}
x\cdot a:=a^*\cdot x
\end{equation}
for any $a\in A$ and $x\in \Delta$.

\begin{defn}
\label{defcellmodule}
Fix a graded cellular algebra $A$ with graded cellular basis $\{c_{\mathfrak{s} \mathfrak{t}} \}$. If $\lambda \in \mathcal{P}$, then the graded cell module $\Delta(\lambda)$ is the left $A$-module with basis $\{c_{\mathfrak{s}} | \mathfrak{s} \in \T(\lambda) \} $ 
such that for $a \in A$,
%with $A$-action
\begin{equation*}
a c_{\mathfrak{s}}  = \sum_{\mathfrak{v} \in \T(\lambda)} r_{\mathfrak{s} \mathfrak{v}}(a) c_{\mathfrak{v}}
\ ,
\end{equation*}
where the scalars $r_{\mathfrak{t} \mathfrak{v}}(a)$ are the scalars from Definition \ref{defcellalgebra}.

Via the convention before the defintion, $\Delta^\circ(\lambda)$ is the right $A$-module with the same basis as $\Delta(\lambda)$ but with $A$-action given by $ c_{\mathfrak{t}} \cdot a := a^* c_{\mathfrak{t}} $.
\end{defn}

We next recall the basis-free definition due to K\"onig and Xi \cite{KoXi2} in the graded setting.

\begin{defn}
\label{defcellalgebra2}
Let $A$ be a graded algebra equipped with an anti-involution $*: A\lra A$. A graded ideal $J\subset A$ is called a \emph{graded cell ideal} of $A$ if 
\begin{itemize}
\item[(1)] The ideal $J$ is stable under the anti-involution: $a^*\in J$ for all $a\in J$.
\item[(2)] There is a graded left $A$-module $\Delta$ and an $(A,A)$-bimodule isomorphism $\phi:\Delta\otimes \Delta^\circ \lra J$ making the diagram below commute:
\begin{equation}\label{eqncellideal}
\begin{gathered}
\xymatrix{
 \Delta \otimes \Delta^\circ \ar[r]^-\phi \ar[d]_{\tau} & J \ar[d]^{*} 
\\  \Delta \otimes \Delta^\circ \ar[r]^-\phi & J
}
\end{gathered} \ .
\end{equation}
Here the map $\tau$ is the permutation map
$$\tau:\Delta\otimes \Delta^\circ \lra \Delta\otimes \Delta^\circ, \quad \quad \tau(x\otimes y):=y\otimes x.$$
\end{itemize}

The algebra $A$ is called \emph{graded cellular} if there is a $*$-stable filtration of $A$ by graded ideals
\begin{equation}\label{eqncellchain}
0= J_{n+1}\subset J_{n} \subset \cdots \subset J_1 =A
\end{equation}
which makes $J_{i}/J_{i+1}$ a graded cellular ideal of $A/J_{i+1}$ with respect to the quotient involution ($i=1,\cdots, n$).
\end{defn}

The equivalence of Definitions \ref{defcellalgebra} and \ref{defcellalgebra2} is readily seen. Starting with a cellular datum on $A$ as in Definition \ref{defcellalgebra}, we fix a total ordering on $\mc{P}=\{\mu_1 < \mu_2 < \dots < \mu_n \}$ refining the original partial ordering $\geq$, where $n=|\mc{P}|$. Let $J_i:=A^{\geq \mu_i}$ be the subspace with basis
$\sqcup_{\mu\geq \mu_i}\{ c_{\mathfrak{s} \mathfrak{t}}^{\mu} | \mathfrak{s}, \mathfrak{t} \in \T(\mu) \}$.
Then the space $A^{\geq \mu_i}$ is a two-sided ideal of $A$, and there are isomorphisms
\begin{equation*}
\phi_{i}:\Delta(\mu_i) \otimes \Delta^\circ(\mu_i)\cong A^{\geq \mu_i} / A^{> \mu_i} 
\end{equation*}
so that the collection $\{J_{i} \}$ is a cellular chain as in \eqref{eqncellchain}.  Conversely, pick a basis $B$ for each layer of $\Delta_i$ in Definition \ref{defcellalgebra2} with $\mf{s}, \mf{t} \in B$, and define $c_{\mf{s}\mf{t}}^i:=\phi_i(\mf{s},\mf{t})$. It is easily seen to result in a graded cellular basis for $A$ satisfying the conditions of Definition \ref{defcellalgebra}.

When no confusion can be caused, we will usually identify $\mc{P}=\{\mu_1,\dots, \mu_n\}$, once a total ordering of refining the partial order is chosen, with the set $\{1,\dots, n\}$ with respect to the natural order on integers.

\begin{example}\label{egrunningegwithoutdif}
We record a few running examples for this section.
\begin{enumerate}
\item[(1)] Let $A$ be the truncated polynomial algebra $\Bbbk[x]/(x^{n})$ with the trivial involution $*=\mathrm{Id}_A$. The natural inclusion of ideals
\[
0\subset (x^{n-1})\subset (x^{n-2}) \subset \cdots \subset (x) \subset A
\]
is a cell chain where we identify $\Delta(i)\cong \Delta^\circ(i)\cong \Bbbk x^{i-1}~\mathrm{mod}~J_{i+1}$, $i=1,\dots, n$. 
\item[(2)] Let $\Bbbk^n$ be an $n$-dimensional graded vector space and $A=M_n(\Bbbk)$ be its graded endomorphism algebra. Fix $*:A\lra A$ to be the usual matrix transpose operation. Choose $k,l$ two arbitrary integers in $\{1,\dots,n\}$. If we identify $\Bbbk^n=(E_{ik})_{i=1}^n$ with the $k$th column submodule of $A$, and $\Delta^\circ=(E_{lj})_{j=1}^n$ with the $l$th row submodule inside $A$. The map
\[
\Delta \otimes \Delta^\circ\lra A,\quad (E_{ik},E_{lj})\mapsto E_{ij},
\]
up to a grading shift, clearly exhibits $A$ as a cellular ideal in itself. Therefore $0\subset A$ is a graded cellular chain, and $A$ is a cellular algebra.
\item[(3)]Let $n$ be a natural number greater than or equal to two, and $Q_n$ be the following quiver:
\begin{equation}\label{quiver-Q}
\xymatrix{
 \overset{1}{\circ} &
 \cdots \ltwocell{'}&
 \overset{i-1}{~\circ~}\ltwocell{'}&
 \overset{i}{\circ}\ltwocell{'}&
 \overset{i+1}{~\circ~}\ltwocell{'}&
 \cdots \ltwocell{'}&
 \overset{n}{\circ}\ltwocell{'}
 }
\end{equation}
Let $\Bbbk Q_n$ be the path algebra\footnote{We adopt the following convention when multiplying paths. For instance, the symbol $(i|j|k)$, where $i,j,k$ are vertices of the quiver $Q_n$, denotes the path which starts at a vertex $i$, then goes through $j$ (necessarily $j=i\pm 1$) and ends at $k$.
The composition of paths is given defined by
\begin{equation*}
(i_i|i_2|\cdots|i_r)\cdot (j_1|j_2|\cdots|j_s)=
\left\{
\begin{array}{ll}
(i_i|i_2|\cdots|i_r|j_2|\cdots|j_s) & \textrm{if $i_r=j_1$,}\\
0 & \textrm{otherwise,}
\end{array}
\right.
\end{equation*}
} associated to $Q_n$ over the ground field. The algebra $A_n^!$ is the quotient of the path algebra $\Bbbk Q_n$ by the relations
\[
(i|i-1|i)=(i|i+1|i)~~ (i=2,\dots, n-1),~~ (1|2|1)=0,
\]
with a choice of gradings on generators so that the relations above are homogeneous. Define $*:A\lra A$ by reversing the direction of arrows.

Define the spaces
\[
\Delta(i)= \Bbbk \left\langle (n-i+1), (n-i+2|n-i+1), \dots, (n|\dots|n-i+1)\right\rangle, 
\]
i.e., all paths which begin at a vertex greater than or equal to $n-i+1$, only travel to the left and end at the vertex $n-i+1$. Likewise, set 
\[
 \Delta^\circ(i)= \Bbbk\left\langle (n-i+1), (n-i+1|n-i+2), \dots, (n-i+1|\dots|n)\right\rangle.
\]
Consider
\[
J_i:=\{
\textrm{paths that pass through at least one of the vertices $(1)$, \dots , $(n-i+1)$}
\}.
\]
It is clear that there are inclusions of two-sided ideals $J_{i+1}\subset J_i$, and each $J_i$ is $*$-stable. It is now an easy exercise to show that $J_i/J_{i+1}$ is isomorphic to the bimodule $\Delta(i)\otimes \Delta^\circ(i)$ over $A_n^!/J_{i+1}$ via the map
\begin{align*}
\phi_i:\Delta(i)\otimes \Delta^\circ(i) & \lra J_i/J_{i+1},\\
(k|k-1|\cdots|n-i+1)\otimes (n-i+1|n-i+2|\cdots|l)& \mapsto (k|k-1|\cdots|n-i+1|n-i+2|\cdots|l),
\end{align*}
where $k,l\geq n-i+1$. We leave the details to the reader as an exercise.
\end{enumerate}
\end{example}

Given a graded cellular algebra $A$, we recall the construction of certain homogeneous bilinear forms on its cell modules. 

\begin{defn}\label{defcellpairing}
Let $J\cong \Delta\otimes \Delta^\circ$ be a graded cellular ideal in a graded algebra $A$. There is a $\Bbbk$-bilinear homogeneous form\begin{equation}
\langle\mbox{-,\mbox{-}}\rangle:\Delta \times \Delta\lra \Bbbk
\end{equation}
characterized by 
\begin{equation}\label{eqnbilinearform}
\langle x, y \rangle z  :=\phi(z,x)\cdot y 
\end{equation}
for any homogeneous elements $x,y, z\in \Delta$ and $z\neq 0$.
\end{defn}

\begin{lem}\label{lem-bilinearpairingproperty}
The bilinear form $\langle \mbox{-},\mbox{-}\rangle$ is symmetric and $*$-associative in the sense that
\[
\langle x,y \rangle= \langle y,x\rangle,\quad \quad
\langle ax,y\rangle=\langle x, a^*y\rangle
\]
for any $a\in A$ and $x,y,z\in \Delta$.
\end{lem}
\begin{proof}
Since $\phi$ is an isomorphism, there exists $z,w\in \Delta$ such that $\phi(w,z)\neq 0\in J$. Fix such a pair of elements. By \eqref{eqnbilinearform}, we have, for any $x,y\in\Delta$, that
\[
\langle x,y\rangle\phi(z,w)=\phi(\langle x,y\rangle z, w)= \phi(\phi(z,x)y,w)=\phi(z,x)\phi(y,w),
\]
the last equality holding since $\phi$ is left $A$-linear. Applying the anti-involution $*$ to both sides of the above equation and using \eqref{eqncellideal}, we obtain
\[
\langle x,y\rangle \phi(w,z)=\phi(w,y)\phi(x,z)=\phi(\phi(w,y)x,z)=\phi(\langle y,x\rangle w,z)=\langle y,x\rangle \phi(w,z).
\]
The result follows by comparing the coefficients on both sides.

For associativity, we have
\[
\langle ax, y \rangle z=\phi(z,ax)y=\phi(z,x\cdot a^*)y=\phi(z,x)a^*y =\langle x,a^*y\rangle z.
\]
Here in the third equality, we have used the right $A$-linearity of $\phi$. The result follows by comparing coefficients in front of any pre-chosen non-zero $z\in \Delta$.
\end{proof}

\begin{defn}\label{defradical}
The \emph{radical} of the bilinear form $\langle\mbox{-},\mbox{-}\rangle$ is the subspace of $\Delta$ defined by
\[
\mathrm{rad}\Delta:=\{x\in \Delta| \langle x,y\rangle=0,~\forall
y\in \Delta\}.
\]
\end{defn}

The previous lemma implies that $\mathrm{rad}\Delta$ is a submodule of $\Delta$. It is a proper submodule if and only if the bilinear pairing is non-zero.

Via the equivalence of Definitions \ref{defcellalgebra} and \ref{defcellalgebra2}, one obtains, using induction, a bilinear pairing on each $\Delta(\lambda)$, denoted $\langle \mbox{-},\mbox{-}\rangle_\lambda$, once a graded cellular datum $(\mc{P},T,C,\mathrm{deg})$ is chosen. It is alternatively characterized on cellular basis elements by $\langle\mbox{-},\mbox{-} \rangle_{\lambda}$ on $\Delta(\lambda)$ determined by
\begin{equation*}
c^{\lambda}_{\mathfrak{a} \mathfrak{s}} c^{\lambda}_{\mathfrak{t} \mathfrak{b}}
=
\langle c^{\lambda}_{\mathfrak{s}}, c^{\lambda}_{\mathfrak{t}} \rangle_{\lambda}
c^{\lambda}_{\mathfrak{a} \mathfrak{b}} \quad \mod  ~A^{> \lambda},
\end{equation*}
for any $\mathfrak{s}, \mathfrak{t}, \mathfrak{a}, \mathfrak{b} \in T(\lambda)$.

\begin{defn}
Let $A$ be a graded cellular algebra with a graded cellular datum $(\mc{P},\T,C,\mathrm{deg})$.  Define the left graded $A$-module
$L(\lambda)=\Delta(\lambda) / \rad \Delta(\lambda)$ and
set $\mathcal{P}_0=\{\lambda \in \mathcal{P} | L(\lambda) \neq 0 \}$.
\end{defn}

The following result summarizes some basic properties of graded cellular algebras and their modules.

\begin{thm}\label{thmmoduleproperties}
Let $A$ be a graded cellular algebra with a graded cellular datum $(\mc{P},\T,C,\mathrm{deg})$.
\begin{enumerate}
\item[(i)] If $\lambda \in \mathcal{P}_0$ then $L(\lambda)$ is an absolutely irreducible graded $A$-module.
\item[(ii)] The graded Jacobson radical of $\Delta(\lambda)$ is $\rad \Delta(\lambda)$ for $\lambda \in \mathcal{P}_0$.
\item[(iii)] If $\lambda \in \mathcal{P}$ and $M$ is a graded $A$-submodule of $\Delta(\lambda)$, then $\HOM_A(\Delta(\mu), \Delta(\lambda)/M) \neq 0$ only if $\lambda \geq \mu$.  Moreover, if $\lambda=\mu$ then
\begin{equation*}
\HOM_A(\Delta(\lambda),\Delta(\lambda)/M)\cong \Bbbk.
\end{equation*}
%\item[(iv)] If $\lambda \in \mathcal{P}_0$, then $L(\lambda)$ is an absolutely irreducible graded $A$-module.
\item[(iv)] If $\lambda, \mu \in \mathcal{P}_0$, then $L(\lambda) \cong L(\mu) \{ k \}$ for some $ k \in \Z$, if and only if $\lambda=\mu$ and $k=0$.
\item[(v)] The complete set of pairwise non-isomorphic graded simple $A$-modules is given by
$\{ L(\lambda) \{ k \} | \lambda \in \mathcal{P}_0, k \in \Z \}$.
\end{enumerate}
\end{thm}
\begin{proof}
See \cite[Section 3]{GrLe} for the ungraded case and \cite[Section 2]{HuMathasKLR} for the graded adaption.
\end{proof}

Let $M$ be a graded $A$-module.  The multiplicity of the simple module $L\{ k \}$ as a graded composition factor of $M$ is denoted by $[M:L \{ k \}]$.

\begin{defn}
The graded decomposition matrix of $A$ is the matrix $D_A(q)=(d_{\lambda \mu}(q))$, where
\begin{equation*}
d_{\lambda \mu}(q)= \sum_{k \in \Z} [\Delta(\lambda):L(\mu) \{ k \}] q^k,
\end{equation*}
for $\lambda \in \mathcal{P}$ and $\mu \in \mathcal{P}_0$.
\end{defn}

\begin{lem}\label{lemcellsimplemultiplicity}
Suppose that $ \mu \in \mathcal{P}_0$ and $\lambda \in \mathcal{P}$.  Then
\begin{enumerate}
\item[(i)] $d_{\lambda \mu}(q) \in \N[q,q^{-1}]$,
\item[(ii)] $d_{\mu \mu}(q)=1$ and $d_{\lambda \mu}(q) \neq 0$ only if $\lambda \geq \mu$.
\end{enumerate}
\end{lem}
\begin{proof}
See \cite[Lemma 2.13]{HuMathasKLR}.
\end{proof}

\begin{defn}
A graded $A$-module $M$ has a graded cell module filtration if there exists a filtration
\begin{equation*}
0=M_0 \subset M_1 \subset M_2 \subset \cdots \subset M_k=M
\end{equation*}
such that each $M_i$ is a graded submodule of $M$ and if $1 \leq i \leq k$ then 
$M_i / M_{i-1} \cong \Delta(\lambda_i) \{ j_i \}$ for some $\lambda_i \in \mathcal{P}$ and some $j_i \in \Z$.  
\end{defn}

\begin{prop}\label{propprojcellfiltration}
Suppose that $P$ is a projective $A$-module.  Then $P$ has a graded cell module filtration.
\end{prop}
\begin{proof}
See 
\cite[Proposition 2.14]{HuMathasKLR}.
\end{proof}

For each $\lambda \in \mathcal{P}_0$, let $P(\lambda)$ be the projective cover of $L(\lambda)$.  

\begin{lem}

Suppose that $\lambda \in \mathcal{P}$ and $\mu \in \mathcal{P}_0$.  Then
\begin{enumerate}
\item[(i)] $d_{\lambda \mu}(q)=\gdim( \HOM_A(P(\mu),\Delta(\lambda)))$.
\item[(ii)] $\HOM_A(P(\mu),\Delta(\lambda))=P(\mu) \otimes_A \Delta^{\circ}(\lambda)$
as graded vector spaces.
\end{enumerate}
\end{lem}
\begin{proof}
This follows from general reciprocity properties of (graded) cellular algebras. See \cite[Theorem 3.7]{GrLe} and \cite[Lemma 2.15]{HuMathasKLR}.
\end{proof}

\begin{defn}
The graded Cartan matrix of $A$ is the matrix $C_A(q)=(c_{\lambda \mu}(q))$, where
\begin{equation*}
c_{\lambda \mu}(q)= \sum_{k \in \Z} [P(\lambda):L(\mu) \{ k \}] q^k,
\end{equation*}
for $\lambda, \mu \in \mathcal{P}$.
\end{defn}

\begin{prop}
The decomposition and Cartan matrices of a graded cellular algebra $A$ are related by: $C_A(q)=D_A^{tr}(q) D_A(q)$ where $D_A^{tr}(q)$ is the transpose of $D_A(q)$.
\end{prop}
\begin{proof}
\cite[Theorem 2.17]{HuMathasKLR}
\end{proof}

Later on this section, we will study quasi-hereditary $p$-DG algebras.
For now, we record a result about quasi-hereditary algebras in the setting of classical graded cellular algebras.

\begin{prop}
A graded cellular algebra is quasi-hereditary if and only if $\mathcal{P}=\mathcal{P}_0$.
\end{prop}
\begin{proof}
See \cite[Section 2.2]{HuMathasKLR} or
\cite[Remark 3.10]{GrLe}.
\end{proof}

\subsection{\texorpdfstring{$p$}{p}-DG cellular algebras}
In this section, we define the notion of a cellular algebra object in the category of $p$-complexes. 

\begin{defn}
\label{defpdgcellalgebra}
Let $(A,\dif)$ be a $p$-DG algebra equipped with an anti-involution $*: A\lra A$. A graded ideal $J\subset A$ is called a \emph{$p$-DG cellular ideal} of $A$ if 
\begin{itemize}
\item[(1)] The ideal $J$ is stable under the differential and anti-involution: $\partial(J)\subset J$ and $*(J)\subset J$.
\item[(2)] There is a left $p$-DG $A$-module $\Delta$ and a $p$-DG $(A,A)$-bimodule isomorphism $\phi:\Delta\otimes \Delta^\circ \lra J$ making the diagram below commute:
\begin{equation}\label{eqncellidealpdg}
\begin{gathered}
\xymatrix{
 \Delta \otimes \Delta^\circ \ar[r]^-\phi \ar[d]_{\tau} & J \ar[d]^{*} 
\\ \Delta \otimes \Delta^\circ \ar[r]^-\phi & J
}
\end{gathered} \ .
\end{equation}
Here the map $\tau$ is the permutation map
$$\tau:\Delta\otimes \Delta^\circ\lra \Delta\otimes \Delta^\circ, \quad \quad \tau(x\otimes y):=y\otimes x.$$
\end{itemize}

The algebra $A$ is called \emph{$p$-DG cellular} if there is a filtration of $A$ by graded ideals, 
\begin{equation}\label{eqncellchainpdg}
0=J_{n+1} \subset J_{n}\subset \cdots \subset J_1 =A
\end{equation}
 stable under $\partial$ and $*$, that makes $J_{i}/J_{i+1}$ a $p$-DG cellular ideal of $A/J_{i+1}$ ($i=1,\cdots, n$).
\end{defn}

We remark that, although $\Delta^\circ$ has the same underlying vector space as $\Delta$, we do not require them to be isomorphic as $p$-complexes (despite Example \ref{egrunningegwithdif} which seems to show so). In fact, by definition, $\Delta^\circ$ is isomorphic to a $p$-DG \emph{left} $A$-module with the conjugate differential $*\dif*$, which is not necessarily equal to $\dif$. An illustrating instance where $\Delta$ and $\Delta^\circ$ have distinct $p$-complex structures is exhibited in Example \ref{eg-Delta-Delta-op} later.

\begin{lem}\label{lemcellmodulebasicproperty}
Let $(A,\partial)$ be a $p$-DG cellular algebra. Then there exists $(\mc{P},\T,C, \deg)$ a graded cellular datum for $A$ such that the following properties hold.
\begin{enumerate}
\item[(i)] For each $\lambda\in \mc{P}$, there exists a $p$-nilpotent matrix $(a_{\mf{s}\mf{t}})_{\mf{s},\mf{t}\in T(\lambda)}$ such that
\[
\dif(c_{\mf{s}\mf{t}}^\lambda) = \sum_{\mf{v}}a_{\mf{v}\mf{s}}c_{\mf{v}\mf{t}}^\lambda+ a_{\mf{v}\mf{t}}c^\lambda_{\mf{s}\mf{v}} \quad ~\mathrm{mod}~A^{>\lambda}.
\]
for all $\mf{s},\mf{t}\in T(\lambda)$.
In particular, the differential respects the partial ordering $\geq$ in the sense that $\partial(A^{\geq \lambda}) \subset A^{\geq \lambda}$,
\item[(ii)] Let $\lambda \in \mathcal{P}$ and 
$T(\lambda)=\{\mathfrak{s}_1, \ldots, \mathfrak{s}_k \}$.  Then the left $p$-DG module
$A^{\geq \lambda} /A^{> \lambda}$ has a filtration
\begin{equation*}
0=M_0 \subset M_1 \subset M_2 \subset \cdots \subset M_k=A^{\geq  \lambda} / A^{> \lambda}
\end{equation*}
such that all subquotients are isomorphic as $p$-DG modules and
$M_i / M_{i-1} \cong \Delta(\lambda) \langle \deg \mathfrak{s}_i \rangle $ for $i=1,\ldots, k$.
\end{enumerate}
\end{lem}
\begin{proof}
For each cell module $\Delta(\lambda)$, choose a graded basis $B$ such that
for $\mf{s} \in B$,
$\dif(\mf{s})=\sum_{\mf{v}}a_{\mf{v}\mf{s}}\mf{v}$ for some elements $\mf{v} \in B$. Then the result follows from the fact that
\[\phi_\lambda: \Delta(\lambda)\otimes \Delta^\circ (\lambda)\lra A^{\geq \lambda}/A^{>\lambda}\]
is an isomorphism of $p$-complexes.

The second statement follows since $\phi_\lambda$ is an isomorphism of $p$-DG bimodules over $A/A^{>\lambda}$. Any filtration of $\Delta^\circ(\lambda)$ by $p$-subcomplexes whose associated graded pieces are one-dimensional leads to a filtration of $\Delta(\lambda)\otimes \Delta^\circ(\lambda)$ satisfying the conditions of (ii). The result follows.
\end{proof}

We call the common subquotients in Definition \ref{defpdgcellalgebra}, up to grading shifts,  the \emph{$p$-DG cell module} $\Delta(\lambda)$.

%\begin{prop}
%Let $(A,\partial)$ be a graded cellular $p$-DG algebra.  Then each graded cellular %module $C(\lambda)$ is a $p$-DG modules.
%\end{prop}

\begin{example}\label{egrunningegwithdif}
We illustrate the definition with some examples, the last three continued numerically from the corresponding ones in Example \ref{egrunningegwithoutdif}.
\begin{enumerate}
\item[(0)] As one can easily examine, any graded cellular algebra with the zero $p$-differential is a $p$-DG cellular algebra.
\item[(1)] Let $A$ be the truncated polynomial algebra $(\Bbbk[x]/(x^{n}))$ with the trivial involution $*=\mathrm{Id}_A$. We equip it with the $p$-differential defined on the generator $x$ by $\dif(x):=x^2$ and extended to the entire algebra by the Leibniz rule. It is then clear that $\dif$ preserves the cellular chain constructed in Example \ref{egrunningegwithoutdif}~(1). Each of the one-dimensional cellular module inherits the zero differential.
\item[(2)]  Let us equip $\Bbbk^n$ with a $p$-complex structure, so that $A=\END_\Bbbk(\Bbbk^n)\cong \mathrm{M}(n,\Bbbk)$ be equipped with the induced $p$-DG algebra structure determined by (c.f.~equation \eqref{eqnHmodstructureonmorphismspace})
\begin{equation}
\dif(f)(x):=\dif(f(x))-f(\dif(x))
\end{equation}
for any $x\in\Bbbk^n$ and $f\in \mathrm{M}(n,\Bbbk)$. Then, by identifying $\Bbbk^n$ with a $\partial$-stable column module in $A$ and $(\Bbbk^n)^*$ with a $\partial$-stable row module, the natural matrix map defined in Example \ref{egrunningegwithoutdif}~(2) makes $A$ itself into a $p$-DG cellular ideal satisfying Definition \ref{defpdgcellalgebra}.
\item[(3)] In \cite{QiSussan}, the algebra $A_n^!$ is equipped with a $p$-DG structure given on the generators by
\begin{equation}\label{eqndifonpaths}
\dif(i+1|i)=0,\quad \quad \dif(i|i+1)=(i|i+1|i|i+1) \quad \quad (i=1,\dots, n-1),
\end{equation}
and extended to the full algebra by the Leibniz rule. We claim that this $p$-differential makes $(A_n^!,\dif)$ into a $p$-DG cellular algebra. 

Equip both $\Delta(i)$ and $\Delta^\circ(i)$, $i=1,\dots, n$, with the zero $p$-differential. We will show that the map from Example \ref{egrunningegwithoutdif}~(3),
$\phi_i:\Delta(i)\otimes\Delta^\circ(i)\lra J_i/J_{i+1}$, is a $p$-DG isomorphism of bimodules over $A_n^!/J_{i+1}$.
Note that $\Delta(i)$ consists of paths starting from vertices greater than or equal to $n-i+1$ and ending at the vertex $(n-i+1)$.  Let $x$ be such a path. Then the differential acts as zero on $\phi_i(x \otimes (n-i+1))$. On the other hand, for those paths in $\Delta^\circ(i)$, a simple computation using \eqref{eqndifonpaths} shows that, for any $k\in \{1,\dots, i\}$,
\begin{equation}
\dif(n-i+1|\cdots | n-i+k)=(k-1)(n-i+1|n-i|n-i+1|\cdots |n-i+k) \in J_{i+1}.
\end{equation}
Let $y$ be such a path in $\Delta^\circ(i)$.
Then the differential also acts trivially on $\phi_i((n-i+1) \otimes y)$.
The claim follows.
\end{enumerate}
\end{example}

\begin{lem}\label{lemcellbilinearformdifinv}
Let $A$ be a $p$-DG algebra with an anti-involution $*$ and $ J\subset A$ be a $p$-DG cellular ideal with a $p$-DG bimodule isomorphism $\phi:\Delta \otimes \Delta^\circ \cong J$. Then the bilinear pairing
\[
\langle \mbox{-},\mbox{-} \rangle:\Delta\otimes \Delta\lra \Bbbk
\]
is a $p$-DG pairing in the sense that, for any $x,y\in \Delta$,
\[
\langle \dif(x),y\rangle+\langle x,\dif(y)\rangle = 0.
\]
\end{lem}
\begin{proof}
By Definition \ref{defcellpairing}, the bilinear pairing is characterized by
\[
\langle x, y\rangle z= \phi(z,x)\cdot y
\]
on homogeneous elements $x,y,z\in \Delta$. Fix $z\neq 0$. Differentiating the above equation on both sides, we have that
\[
\langle x,y\rangle \partial(z) = \partial(\phi(z,x))\cdot y+\phi(z,x)\cdot \partial(y).
\]
Since $\phi$ is a $p$-DG bimodule homomorphism, we have 
\[
\partial(\phi(z,x))=\phi(\partial z,x)+\phi(z,\partial x).
\]
Combining these equations gives us
\begin{align*}
\langle x,y \rangle \partial(z) & = \phi(\partial z,x)\cdot y+\phi(z,\partial x)\cdot y+ \phi(z,x)\cdot \partial(y) \\
& =\langle x, y \rangle \partial(z)+ \langle \partial(x), y \rangle z + \langle x, \partial (y)\rangle z,
\end{align*}
which in turn implies $\langle \partial(x), y \rangle z + \langle x, \partial (y)\rangle z=0$, giving us the lemma.
\end{proof}

\begin{cor}\label{cor-simple-mod-pdg}
Let $(A,\dif)$ be a $p$-DG cellular algebra. Then, every simple $A$-module carries a $p$-DG module structure over $A$.
\end{cor}
\begin{proof}
By Theorem \ref{thmmoduleproperties}, simple $A$ modules correspond to $L(\lambda)=\Delta(\lambda)/\rad \Delta(\lambda)$ with $\lambda\in \mc{P}_0$. It suffices to show that $\rad \Delta(\lambda)$ is a $p$-DG submodule of $\Delta(\lambda)$. In other words, it is closed under the differential action. This is clear from the $\dif$-invariance of the bilinear pairing, for, if $x\in \rad \Delta(\lambda)$ and $y\in \Delta(\lambda)$, then,
\[
\langle \dif(x), y\rangle =-\langle x, \dif(y) \rangle=0,
\]
The result follows.
\end{proof} 

\begin{thm}
Let $(A,\dif)$ be a $p$-DG cellular algebra. Then, every simple $A$-module $L(\lambda)$ admits a unique simple $p$-DG $A$-module structure. Furthermore, any simple $p$-DG module over $A$ is isomorphic, up to grading shifts, to a unique $L(\lambda)$ for some $\lambda\in \mc{P}_0$.
\end{thm}
\begin{proof}
Let $L$ be a simple $A_\dif$-module (see Remark \ref{rmksmashproductalgebra} for the definition of the smash product algebra $A_\dif$). Consider its restriction as a graded $A$-module $\Res^{A_\dif}_A(L)$. Let $M$ be a simple submodule of the restricted module. Then by adjunction, there is a non-zero morphism from the induced module $\Ind_A^{A_\dif}(M)$ to $L$. This is surjective since the target is simple.
Therefore every simple $A_\dif$-module is a quotient of $\Ind_A^{A_\dif}(M)$ for a simple $A$-module $M$.

Now let $M$ be a simple $A$-module. The module $\Ind_A^{A_\dif}(M)$ is generated by any vector of lowest degree. Let this degree be $d$. Thus every submodule $N$ of $\Ind_A^{A_\dif}(M)$ satisfies either $N_d=0$ or $N_d=M_d
=(\Ind_A^{A_\dif}(M))_d$ and in this latter case $N=\Ind_A^{A_\dif}(M)$. Therefore $\Ind_A^{A_\dif}(M)$ has a maximal proper submodule $N$ and hence a unique simple quotient. This also shows that there is at most one $p$-DG $A$-module structure associated with any simple $A$-module. It follows that the simple lift of the $p$-DG module structure on each $L(\lambda)$ in Corollary \ref{cor-simple-mod-pdg} is unique.

It remains to see that $L$ up to grading shifts must agree with one of $L(\lambda)$'s. Again, let $M=L(\lambda)$ be a simple submodule of $\Res_A^{A_\dif}(L)$. We now analyze the Jordan-Holder factors of $\Res^{A_\dif}_{A}\Ind_A^{A_\dif}(L(\lambda))$.
Let $H^i$ be the $\Bbbk$-span of $\{1,\dif, \dots, \dif^{i-1}\}$ for $1\leq  i\leq p$, ($H^1=\Bbbk$). As with the usual Mackey's formula, one has that, as an $(A,A)$-bimodule, there is an increasing filtration on $A_\dif$ given by\footnote{Here we identify $A\cong A\otimes \Bbbk$ inside $A_\dif$ as in Remark \ref{rmksmashproductalgebra}. The left $A$-module structure on the filtration is clear, while the right $A$-module structure utilizes the commutator relations in that remark. }
\[
0 \subset A\otimes H^1 \subset A\otimes H^2 \subset \dots \subset A\otimes H^p.
\]
The subquotients of this filtration are just grading shifts of $A$. Therefore $\Res_A^{A_\dif}\Ind_A^{A_\dif}(L(\lambda))$ has only $L(\lambda)$ appearing in a Jordan-Holder filtration. Thus if $L$ is the simple quotient of $\Ind_A^{A_\dif}(L(\lambda))$, then $\Res_A^{A_\dif} L$ only has $L(\lambda)$ appearing in a Jordan-Holder filtration. By the discussion in the previous paragraph, the simple $A_\dif$-module $L$ can only have one copy of $L(\lambda)$ in its Jordan-Holder filtration of $\Res_A^{A_\dif} L$. It follows that $L\cong L(\lambda)$ as an $A_\dif$-module. The result follows.
\end{proof}

A similar uniqueness result holds for cell modules concerning their $p$-DG structures.

\begin{prop}\label{propuniquepdgcellmodstructure}
Let $(A,\dif)$ be a $p$-DG cellular algebra. Then, every $p$-DG module cell $\Delta$ over $A$ has a unique $p$-DG structure.
\end{prop}
\begin{proof}
Suppose $\Delta$ and $\Delta^\prime$ are two $p$-DG cell modules which are isomorphic as $A$-modules. Then the internal morphism space
$$
\HOM_A(\Delta,\Delta^\prime)\cong \END_A(\Delta)\cong \Bbbk
$$
is a one-dimensional $p$-complex (Theorem \ref{thmmoduleproperties}~(iii)), and thus must carry the zero differential. Any non-zero element $\HOM_A(\Delta,\Delta^\prime)$ is then an isomorphism of $A$-modules and commutes with $p$-differentials. By equation \eqref{eqnHmodstructureonmorphismspace}, it is a $p$-DG isomorphism.
\end{proof}

\subsection{Quasi-hereditary \texorpdfstring{$p$}{p}-DG cellular algebras}
In this subsection, we analyze the analogue of quasi-hereditary cellular algebras in the $p$-DG setting. We first make the following definition.

\begin{defn}\label{defcellalgquasihereditary}
Let $(A,\dif)$ be a $p$-DG cellular algebra equipped with an anti-involution $*:A\lra A$. A $p$-DG cellular ideal $J\subset A$, as in Definition \ref{defpdgcellalgebra}, is called a \emph{$p$-DG quasi-hereditary ideal} if the cell module $\Delta$ is a cofibrant $p$-DG module over $A$.

A $p$-DG cellular algebra $(A,\dif, *)$ is called \emph{$p$-DG quasi-hereditary} if there is a chain of $*$-stable $p$-DG ideals 
\[ 
0= J_{n+1} \subset J_{n}\subset \dots \subset J_1=A
\]
such that $J_i/J_{i+1}$ is a $p$-DG quasi-hereditary ideal in $A/J_{i+1}$ for $i=1,\dots, n$.
\end{defn}

Forgetting the differential, the cofibrance of $\Delta$ implies that $\Delta$ is projective as an $A$-module (see part (2) of Definition \ref{def-finite-cell}). In particular, this shows that the cellular ideal is quasi-hereditary in the classical sense (c.f.~\cite{KoXi2}).

\begin{lem}\label{lemcellularidealqh}
The bilinear pairing \eqref{eqnbilinearform} associated to a $p$-DG quasi-hereditary cellular ideal is non-zero. Consequently, the ideal $J$ is generated by an indecomposable idempotent $e\in J$ such that $\Delta=Ae$ and $J=AeA$.
The idempotent can be chosen to satisfy $\dif(e)\in Ae$.
\end{lem}
\begin{proof}
Since $\Delta$ is projective and indecomposable (Theorem \ref{thmmoduleproperties} (iii)), it is the projective cover of a unique simple module. Therefore the radical of the module, which agrees with the radical of the bilinear form (Theorem \ref{thmmoduleproperties} (ii)), is a proper submodule of $\Delta$. It follows that the bilinear form is non-zero. 

The second statement follows from the fact that cellular ideals are either square-zero or idempotent. See, for instance, \cite[Lemma 2.1]{KoXi2}. We just indicate how such an idempotent is constructed. Since the bilinear pairing \eqref{eqnbilinearform} is non-zero in this case, there exists an $x \in \Delta$ such that $\langle x,y \rangle=1$ for some other element $y$. Choose a homogeneous maximal degree such $x$. Then necessarily $\dif(x)$ lies in the radical of $\Delta$, so that $\langle \dif(x),y\rangle=0$. By Lemma \ref{lemcellbilinearformdifinv}, we also have $\langle x,\dif(y)\rangle=0$.

The element
$e:=\phi(y,x)$
satisfies 
\begin{equation}
e^2=\phi(y,x)\phi(y,x)=\phi(\phi(y,x)y,x)=\phi(\langle x,y\rangle y, x)= \phi(y,x)=e.
\end{equation}

Differentiating $e=e^2$, we obtain
    $\dif(e) = \dif(e)e+e\dif(e)$. Thus, in order to show $\dif(e)\in Ae$, it suffices to show that $e\dif(e)=0$. 
We compute
\begin{align}
    e\dif(e) & =\phi(y,x)\dif (\phi(y,x))= \phi(y,x)\left( \phi(\dif(y),x)+\phi(y,\dif(x))\right) \nonumber \\
    & = \phi(\phi(y,x)\dif(y),x)+\phi(y, \phi(x,y)\dif(x)) =
    \phi(\langle x , \dif(y) \rangle y,x)+  \phi( y , \langle y , \dif(x) \rangle x) = 0.
\end{align}
Here in the equalities, we have used Lemma \ref{lem-bilinearpairingproperty}, Lemma \ref{lemcellbilinearformdifinv}, as well as the relationship between $\phi$, $\tau$ and the anti-automorphism $*$ (Definition \ref{defpdgcellalgebra}). The result follows.
\end{proof}

We want to emphasize that, in the context of the previous lemma, finding an idempotent $e$ such that $\Delta = Ae$ and $\dif(e)\in Ae$ does not automatically guarantee the cofibrantness of $\Delta$. This can be seen from (2) of Example \ref{egrunningegwithdifquasihereditary} below.

\begin{example}\label{egrunningegwithdifquasihereditary}
We continue with the examples discussed in Examples \ref{egrunningegwithoutdif} and \ref{egrunningegwithdif}.
\begin{enumerate}
\item[(0)] Any graded cellular algebra with the zero $p$-differential is a $p$-DG quasi-hereditary if and only if $A$ is quasi-hereditary.
\item[(1)] The truncated polynomial algebra $(\Bbbk[x]/(x^{n}))$ is not $p$-DG quasi-hereditary if $n\geq 2$, since one can easily check that the one-dimensional module is never cofibrant.
\item[(2)] Consider the $p$-DG matrix algebra $(\END_\Bbbk(\Bbbk^n),\dif)$ with the $p$-DG structure induced from a $p$-complex structure on $\Bbbk^n$. Then $\Delta=\Bbbk^n$ is a cofibrant $p$-DG module over $\END(\Bbbk^n)$ if and only if $\Bbbk^n$ is a non-acyclic $p$-complex. The proof of this result can be found in \cite[Lemma 4.21]{EQ2}. Therefore, $(\END_\Bbbk(\Bbbk^n),\dif)$ is $p$-DG quasi-hereditary if and only if $\Bbbk^n$ is non-acyclic.
\item[(3)] The $p$-DG cellular algebra $(A_n^!,\dif)$ is always $p$-DG quasi-hereditary. Inductively, there is a $p$-DG isomorphism $A_n^!/J_n\cong A_{n-1}^!$, and thus it suffices to show that $\Delta(n)$ is cofibrant. This is the case since the idempotent $(1)\in A_n^!$ satisfies $\dif(1)=0$, and thus
$\Delta(n)=A_n^!\cdot (1)$
is a $p$-DG direct summand of $A_n^!$. 
\end{enumerate}
\end{example}

When a $p$-DG quasi-hereditary ideal $J$ is actually generated by an idempotent $e$ such that $e=e^*$ and $ \dif(e)=0$, then the cellular modules $\Delta$ and $\Delta^\circ$ can naturally be chosen to be
\[
\Delta=Ae \quad \quad \Delta^\circ:=eA.
\]
This can be seen from Example \ref{egrunningegwithdifquasihereditary} (3) above. On the contrary, the condition is not, in general, satisfied for the matrix $p$-DG algebra in case (2).

\begin{prop}\label{propdifkilledidempotentideals}
Let $A$ be a $p$-DG quasi-hereditary cellular algebra with a chain of cellular ideals 
$$0=J_{n+1}\subset J_n\subset \cdots \subset J_1=A.$$ Suppose that the chain of ideals are generated by idempotents $\{e_i\in A|i=1,\dots n\}$ in the sense that
\[
J_k=\sum_{i=k}^{n} Ae_i A.
\]
Further assume that $e_i=e_i^*$ and that the idempotents are annihilated by the differential: $\dif(e_i)=0$, $i=1,\dots, n$. Then the cellular modules are isomorphic, as $p$-DG modules, to
\[
\Delta_k\cong \dfrac{Ae_k+J_{k+1}}{J_{k+1}}, \quad \quad
\Delta^\circ_k \cong \dfrac{e_kA+J_{k+1}}{J_{k+1}} \quad \quad (k=1,\dots, n).
\]
\end{prop}
\begin{proof}
It is clear, by cellularity, that the map
\[
\Delta_k \otimes \Delta^\circ_k \lra J_k/J_{k+1}, \quad \quad
\overline{ae_k}\otimes \overline{e_ka^\prime}\mapsto \overline{ae_ka^\prime}
\]
is an isomorphism for all $a,a^\prime\in A$ and $k=1,\dots, n$. Here the bar notation indicates the coset space of the corresponding element. Since $\dif(e_k)=0$, it is clear that this map is a $p$-DG map. It follows, by the uniqueness of $p$-DG structures on $\Delta_k$ and $\Delta^\circ_k$ (Proposition \ref{propuniquepdgcellmodstructure}), that any choice of $p$-DG module structures on the cellular modules are isomorphic to these canonically defined ones.
\end{proof}

The requirement for the cell module to be non-acyclic in Example \ref{egrunningegwithdifquasihereditary} (2) is not a special property for the matrix $p$-DG algebra, as shown by the next result.

\begin{lem}\label{lemcompactcell}
Let $J\subset A$ be a $p$-DG quasi-hereditary cellular ideal. Then the cell module $\Delta$ is non-acyclic, and it descends to a compact object in the derived category $\mc{D}(A)$.
\end{lem}
\begin{proof}
By the morphism space formula \eqref{eqnmorphismspaceoutofcofibrantmod} for cofibrant modules and the cellular property (Theorem \ref{thmmoduleproperties} (iii)), we have that
\[
\End_{\mc{D}(A)}(\Delta)\cong \End_{\mc{H}(A)}(\Delta)\cong \Bbbk.
\]
Thus $\Delta$ must be non-acyclic. The morphism space formula \eqref{eqnmorphismspaceoutofcofibrantmod} also shows that the functor $\Hom_{\mc{D}(A)}(\Delta,\mbox{-})$ commutes with infinite direct sums since $\Delta$ is a finitely-generated $A$-module, the second claim follows.
\end{proof}

\begin{cor}\label{corJnonacyclic}
Let $\phi:\Delta\otimes \Delta^\circ \cong J\subset A$ be a $p$-DG cellular ideal. Then $\Delta$ is a cofibrant $p$-DG module if and only if $J$ is cofibrant and non-acyclic.
\end{cor}
\begin{proof}
The ``only if'' part is clear by the cofibrance of $\Delta$ and Lemma \ref{lemcellmodulebasicproperty} (ii). Conversely, since $J\cong \Delta\otimes \Delta^\circ$ is non-acyclic, both $\Delta$ and $\Delta^\circ$ are non-acyclic. We have
\[
\END_A(J)\cong \END_A(\Delta\otimes \Delta^\circ)\cong \END_\Bbbk(\Delta^\circ)
\]
using Theorem \ref{thmmoduleproperties}~(iii). By Example \ref{egrunningegwithdifquasihereditary}~(2) above, the $p$-DG module $\Delta^\circ$ is cofibrant over $\END_{\Bbbk}(\Delta^\circ)$ and thus so is its graded dual. Hence we have that
\[
\Delta\cong \Delta\otimes \left(\Delta^\circ\otimes_{\END_\Bbbk(\Delta^\circ)} (\Delta^\circ)^*\right)\cong J\otimes_{\END_\Bbbk(\Delta^\circ)} (\Delta^\circ)^*
\]
is a cofibrant $p$-DG module over $A$ since $J$ is.
\end{proof}

The next result is a $p$-DG version of Proposition \ref{propprojcellfiltration}

\begin{prop}
\label{pdgcellfiltration}
Let $A$ be a $p$-DG cellular algebra and let $P$ be a cofibrant $A$-module.
Then $P$ has a filtration whose subquotients are $p$-DG cell modules.
\end{prop}

\begin{proof}
Extend the partial ordering $>$ on $\mathcal{P}$ to a total ordering so that
\begin{equation*}
\mathcal{P}=\{\lambda_1 < \lambda_2 < \cdots < \lambda_n \}.
\end{equation*}
Let $A(\lambda_j)=\cup_{j \leq i} A^{\geq \lambda_i}$.  Then Definition \ref{defpdgcellalgebra} implies that there is a a filtration of $A$ by two-sided $p$-DG ideals
\begin{equation*}
0 \subset A(\lambda_n) \subset A(\lambda_{n-1}) \subset \cdots \subset A(\lambda_1)= A.
\end{equation*}
Tensoring with $P$ gives
\begin{equation*}
0 \subset  A(\lambda_n)\otimes_A P \subset  A(\lambda_{n-1})\otimes_A P \subset \cdots \subset A(\lambda_1)\otimes_A P=P.
\end{equation*}
Note that there is a short exact sequence of $p$-DG modules
\begin{equation*}
0 \longrightarrow A(\lambda_{i+1}) \longrightarrow A(\lambda_i) \longrightarrow
A^{\geq \lambda_i} / A^{> \lambda_i} \longrightarrow 0.
\end{equation*}
Since $P$ is cofibrant, tensoring with $P$ is exact so the subquotients in the filtration of $P$ above are
\begin{equation*}
 (A(\lambda_i) \otimes_A P) / (A(\lambda_{i+1}) \otimes_A P)
\cong
(A^{\geq \lambda_i} / A^{> \lambda_i})\otimes_A P.
\end{equation*}
We also know by Definition \ref{defpdgcellalgebra}, condition (2), that there is a $p$-DG isomorphism
\begin{equation*}
(A^{\geq \lambda_i} / A^{> \lambda_i}) \otimes_A P\cong (\Delta(\lambda_i) \otimes \Delta^{\circ}(\lambda_i)) \otimes_A P \cong \Delta(\lambda_i) \otimes (\Delta^{\circ}(\lambda_i) \otimes_A P )
\ .
\end{equation*}
This completes the proof.
%We know from earlier that $A^{\geq \lambda_i} / A^{> \lambda_i}$ has a filtration whose subquotients are isomorphic to the $p$-DG cell module $\Delta(\lambda_i)$ completing the proof.
\end{proof}

\begin{prop}
\label{cellmodulescompact}
Let $A$ be a $p$-DG quasi-hereditary cellular algebra. 
Then each $p$-DG cell module is a compact object in the derived category $\mc{D}(A)$.
\end{prop}

\begin{proof}
Let $0=J_{n+1}\subset J_{n}\subset \dots \subset J_1=A$ be a $p$-DG quasi-hereditary chain as in Definition \ref{defcellalgquasihereditary}, and let
\[
\phi_i : \Delta_i\otimes \Delta_i^\circ \cong J_{i}/J_{i+1}
\]
be isomorphisms of $p$-DG bimodules over $A/J_{i+1}$, $i=1,2,\dots,n$. We will establish the proposition by induction on the length of the chain. The length-one case follows from Lemma \ref{lemcompactcell}.

Inductively, assume that each $\Delta_{i}$, $i=1,2,\dots n-1$, is compact inside $\mc{D}(A/J_{n})$, and we would like to show that it is compact in $\mc{D}(A)$ as well. By the characterization of compact objects (Theorem \ref{thmcompactmodules}), it suffices to show that $A/J_{n}$ is compact as a $p$-DG module over $A$. The short exact sequence of left $p$-DG modules
$
0\lra \Delta_{n}\otimes \Delta_{n}^\circ \lra A\lra A/J_{n}\lra 0
$
results in a distinguished triangle in $\mc{D}(A)$
\[
\Delta_{n}\otimes \Delta_{n}^\circ \lra A\lra A/J_{n}\stackrel{[1]}{\lra} \Delta_{n}\otimes \Delta_{n}^\circ[1].
\]
The left two terms of the triangle are clearly compact, and thus so is the third term by the usual ``two-out-of-three'' property.
\end{proof}

\begin{rem}
We remark that $p$-DG quasi-hereditary is only a sufficient condition for cell modules to be compact, but is not a necessary one. For a non-example, consider the special case $\Bbbk[x]/(x^{p+1})$ of Example \ref{egrunningegwithdifquasihereditary}~(1). Under the differential $\dif(x)=x^2$, it is clear that  
the $\Bbbk[x]/(x^{p+1})\lra \Bbbk$, $ x\mapsto 0$ is a quasi-isomorphism of $p$-DG algebras. Therefore, any finite-dimensional $p$-DG module over $\Bbbk[x]/(x^{p+1})$ is compact, but the $p$-DG algebra itself is not $p$-DG quasi-hereditary.
\end{rem}

\begin{thm}
\label{thmsimplemodulescompact}
Let $A$ be a $p$-DG quasi-hereditary cellular algebra. 
Then every simple $p$-DG module is compact in the derived category $\mc{D}(A)$. Consequently, the $p$-DG algebra $A$ is hopfologically finite.
\end{thm}

\begin{proof}
When $\lambda$ is the minimal element of $\mathcal{P}$, the simple module $L(\lambda)$ is isomorphic to the cell module $\Delta(\lambda)$.  Then $\Delta(\lambda)$ is compact by Lemma \ref{lemcompactcell}.

Assume by induction that $L(\mu)$ is compact for $\mu < \lambda$.  
We know that $\rad \Delta(\lambda)$ is a filtered $p$-DG module with finitely many subquotients coming from graded shifts of $L(\mu)$ for $\mu < \lambda$.  Thus $\rad \Delta(\lambda)$ is compact.
Consider the distinguished triangle
\begin{equation*}
\rad \Delta(\lambda) \longrightarrow \Delta(\lambda) \longrightarrow L(\lambda)\stackrel{[1]}{\lra} \rad \Delta(\lambda)[1].
\end{equation*}
Since $\rad \Delta(\lambda)$ and $\Delta(\lambda)$ are compact, it follows that $L(\lambda)$ is compact.
\end{proof}

\begin{rem}
The theorem applied to Example \ref{egrunningegwithdifquasihereditary}~(iii) gives an alternative proof that $(A_n^!,\dif)$ is hopfologically finite. This is proven explicitly in \cite[Section 5]{QiSussan} by constructing finite cofibrant replacements of simple $p$-DG modules.
\end{rem}

\subsection{Stratifying derived categories}
In this subsection, we exhibit a stratified structure on the derived category of a $p$-DG quasi-hereditary cellular algebra. 
Let us abbreviate $\mc{D}:=\mc{D}(A)$ or $\mc{D}:=\mc{D}^c(A)$, where $A$ is a $p$-DG algebra. We will use the fact that $\mc{D}$ is a triangulated module category over $\mc{D}^c(\Bbbk)$, the homotopy category of finite-dimensional graded vector spaces. 

We first perform some well-known constructions of triangulated categories in the $p$-DG derived setting, loosely following the treatment in \cite[Section 1]{Orlov}. Throughout, let $\mc{T}$ be a full triangulated subcategory in $\mc{D}$ that is closed under direct summands and tensor product action by $\mc{D}^c(\Bbbk)$.

\begin{defn}\label{deforthcat}
The \emph{right orthogonal} of $\mc{T}$ inside $\mc{D}$, denoted $\mc{T}^\perp$, is the full-triangulated subcategory in $\mc{D}$ consisting of objects $M$ such that
\begin{equation}\label{eqnrightorthogonal}
\Hom_{\mc{D}}(L,M)=0
\end{equation}
for all $L\in \mc{T}$. Similarly, the notion \emph{left orthogonal} $^\perp\mc{T}$ of $\mc{T}$ is defined by switching the place of $L$ and $M$ in  \eqref{eqnrightorthogonal}.
\end{defn}

\begin{lem}
The right (resp.~left) orthogonal complement of $\mc{T}$ in $\mc{D}$ is triangulated and closed under tensor product action by $\mc{D}^c(\Bbbk)$.
\end{lem}
\begin{proof}
The fact that $\mc{T}^\perp$ (resp.~$^\perp\mc{T}$ ) is triangulated is an easy exercise. To see that it is closed under the tensor product action by $\mc{D}^c(\Bbbk)$, we use the tensor-hom adjunction \eqref{eqntensorhomadjunction} (taking $B=\Bbbk$ and $M=U$):
\begin{equation}
\Hom_{\mc{D}}(U\otimes L,M)\cong \Hom_{\mc{D}}(L,U^*\otimes M)
\end{equation} 
for all objects $L,M\in \mc{D}(A)$ . The result follows since $\mc{T}$ is closed under the action by $\mc{D}^c(\Bbbk)$.
\end{proof}

\begin{defn}\label{defadmissible}
A full-triangulated subcategory $i_*:\mc{T}\subset \mc{D}$ is called \emph{right admissible} (resp.~\emph{left admissible}) if the natural inclusion functor $i_*$ has a right adjoint $i^*:\mc{D}\lra \mc{T}$ (resp.~left adjoint $i^!:\mc{D}\lra \mc{T}$). The subcategory $\mc{T}$ is called \emph{admissible} if it is both left and right admissible.
\end{defn}

\begin{lem}\label{lemadmissiblecate}
\begin{enumerate}
\item[(i)]The triangulated subcategory $\mc{T}\subset \mc{D}$ is right (resp. left) admissible if and only if, for any object $X\in \mc{D}$, there exists a distinguished triangle 
\[
M\lra X\lra N\stackrel{[1]}{\lra} M[1] 
\]
with $M\in \mc{T}$ and $N\in \mc{T}^\perp$ (resp.~$N\in \mc{T}$ and $M\in{^\perp\mc{T}}$).
\item[(ii)]The functor $i^*$ (resp. $i^!$) commutes with the tensor product action by $\mc{D}^c(\Bbbk)$.
\end{enumerate}
\end{lem}
\begin{proof}
We only show the right admissible case. The left admissible case follows by taking the opposite categories.

Suppose $\mc{T}$ is right admissible. Then the adjunction between the pair $(i_*,i^*)$ gives us a distinguished triangle
\[
i_*i^*X\lra X\lra N  \stackrel{[1]}{\lra} i_* i^*X[1]
\]
for any object $X\in \mc{D}$, where $N$ is the cone of the canonical map $i_*i^*X\lra X$. Then $N\in \mc{T}^\perp$ by applying $\Hom_{\mc{D}}(i_*L,\mbox{-})$ to the distinguished triangle and using adjunction.

Conversely, for any $X\in \mc{D}$ with the given distinguished triangle, we define $M:= i^*X$. If 
\[
M_1\stackrel{f_1}{\lra} X\stackrel{g_1}{\lra} N_1\stackrel{[1]}{\lra} M_1[1], \quad \quad 
M_2\stackrel{f_2}{\lra} X \stackrel{g_2}{\lra}N_2\stackrel{[1]}{\lra} M_2[1]
\]
are two such triangles, then the map $g_2\circ f_1=0$, since $\Hom_{\mc{D}}(M_1,N_2)=0$. Apply $\Hom_{\mc{D}}(M_1,\mbox{-})$ to the second distinguished triangle, and the resulting long exact sequence shows there exists $h_1$ such that $f_1=f_2\circ h_1$. The axioms of triangulated categories allow us to complete the diagram into a morphism of triangles:
\[
\begin{gathered}
\xymatrix{
M_1\ar[r]^{f_1} \ar@{-->}[d]_{h_1} & X\ar[r]^{g_1}\ar@{=}[d] & N_1 \ar[r]^-{[1]}\ar@{-->}[d]^{h_2} & M_1[1]\ar@{-->}[d]\\
M_2\ar[r]^{f_2} & X \ar[r]^{g_2} & N_2  \ar[r]^-{[1]} & M_2[1]
}
\end{gathered} \ .
\]
The map $h_1$ is an isomorphism since 
$$\Hom_{\mc{T}}(L,M_1)\cong \Hom_{\mc{D}}(L,X)\cong \Hom_{\mc{T}}(L,M_2) $$
for any $L\in \mc{T}$. Thus $i^*$ is well-defined up to isomorphism.

 Similarly, if $h:X\lra Y$ is a morphism, and we have triangles
\[
\begin{gathered}
\xymatrix{
M_1\ar[r]^{f_1} \ar@{-->}[d]_{h_1} & X\ar[r]^{g_1}\ar[d]^h & N_1 \ar[r]^-{[1]}\ar@{-->}[d]^{h_2} & M_1[1]\ar@{-->}[d]\\
M_2\ar[r]^{f_2} & Y\ar[r]^{g_2} & N_2  \ar[r]^-{[1]} & M_2[1]
}
\end{gathered} \ ,
\] 
then there is a morphism of triangles again since $g_2\circ h \circ f_1\in \Hom_{\mc{D}}(M_1,N_2)$ must be zero.
Set $i^*(h):=h_1$. Then $i^*$ is easily checked to be well defined on morphisms in $\mc{D}$ and respect compositions. We have thus shown that $i^*$ is a functor. The converse now follows.

To prove the second claim, it suffices to notice that tensor product by finite-dimensional $p$-complexes is exact on $\mc{D}$, so that
\[
U\otimes M \lra U\otimes X \lra U\otimes N\stackrel{[1]}{\lra} U\otimes M[1]
\]
is a distinguished triangle with $U\otimes M\in \mc{T}$ ($\mc{T}$ is closed under action by $\mc{D}^c(\Bbbk)$) and $U \otimes N \in \mc{T}^{\perp}$. Hence
\[
i^*(U\otimes X)=U\otimes M= U\otimes i^*(X)
\]
holds for all $X\in \mc{D}$.
\end{proof}

It follows from the proof of the lemma that, given a right admissible $\mc{T}\subset \mc{D}$, there is also a functor
\begin{equation}\label{eqnfunctorpilowerstar}
j^!:\mc{D}\lra \mc{T}^\perp,\quad X\mapsto j^!(X)=:N,
\end{equation}
where $N$ is the object associated to $X$ in Lemma \ref{lemadmissiblecate}. Clearly, $\mc{T}^\perp$ is left admissible, and the functor $j^!$ is left adjoint to the natural inclusion functor $j_!:\mc{T}^\perp\lra \mc{D}$. Clearly $j_!$ commutes with the action of $\mc{D}^c(\Bbbk)$, while the same proof as in the lemma also shows that $j^!$ does so as well.

Recall that the \emph{quotient category} $\mc{D}/\mc{T}$ is the Verdier localization of $\mc{D}$ by the class of morphisms whose cones are isomorphic to objects of $\mc{T}$. Since $\mc{T}$ is closed under the action by $\mc{D}^c(\Bbbk)$, this class of morphisms is also closed under the action of $\mc{D}^c(\Bbbk)$. The quotient category is triangulated, and, in our setting, inherits a module-category structure over $\mc{D}^c(\Bbbk)$. 

\begin{lem}\label{lemTorthequivlocalication}
Let $\mc{T}$ be a triangulated subcategory of $\mc{D}$ closed under direct summands and the tensor action of $\mc{D}^c(\Bbbk)$. Denote by $i:\mc{T}\lra\mc{D}$ the inclusion functor. If $\mc{T}$ is right admissible (resp.~left admissible), then $\mc{T}^\perp\cong \mc{D}/\mc{T}$ (resp.~$^\perp\mc{T}\cong \mc{D}/\mc{T}$) as module categories over $\mc{D}^c(\Bbbk)$.
\end{lem}
\begin{proof}
The projection functor $j^!:\mc{D}\lra \mc{T}^\perp$ from  \eqref{eqnfunctorpilowerstar} factors through $\mc{T}$ and induces a quotient functor which is still denoted $j^!:\mc{D}/\mc{T}\lra \mc{T}^\perp$. It is then an easy exercise to check that $j_!$ composed with the quotient functor $\mc{D}\lra \mc{D}/\mc{T}$ serves as the quasi-inverse of $j^!:\mc{D}/\mc{T}\lra \mc{T}^\perp$.
\end{proof}

Now let us apply the general results above to the situation we are interested in. Let $\phi:\Delta\otimes \Delta^\circ\cong J\subset A$ be a quasi-hereditary $p$-DG cellular ideal. Let us denote by $\mc{T}$ the full triangulated subcategory generated by $\Delta$ inside $\mc{D}=\mc{D}(A)$. In this case, $\Delta$ is cofibrant. Then, by equation \eqref{eqnmorphismspaceoutofcofibrantmod} and Theorem \ref{thmmoduleproperties} (iii), we have
\begin{equation}
    \bigoplus_{i\in \Z}\Hom_{\mc{D}(A)}(\Delta,\Delta[i]) = \Hom_{\mc{D}(A)}(\Delta, \Delta) = \Bbbk.
\end{equation}

\begin{lem}\label{lemTembedding}
There is a fully-faithful embedding of $p$-DG derived categories
\[
\mc{D}(\Bbbk)\lra\mc{D}(A),\quad U\mapsto \Delta\otimes U,
\]
whose image coincides with $\mc{T}$. The embedding commutes with tensor product actions by $\mc{D}^c(\Bbbk)$ on both sides.
\end{lem} 
\begin{proof}
It follows from the previous discussion that the functor
 $ \mc{D}(\Bbbk)\lra \mc{D}(A)$, $ U \mapsto \Delta\otimes U$
is fully-faithful with image equal to $\mc{T}$. The second statement is clear.
\end{proof}

We next show that $\mc{T}$ is right admissible. To do this let us represent any object of $\mc{D}$ by a cofibrant replacement. Any cofibrant $p$-DG module $P$ over $A$ fits into a short exact sequence
\begin{equation}
0\lra JP \lra P \lra P/JP\lra 0.
\end{equation}
Furthermore, since $P$ is projective as an $A$-module, we have 
$$JP\cong \Delta \otimes \HOM_A(\Delta,P)\cong J\otimes_A P.$$
Also $(A/J)\otimes_A P=P/JP$ is clearly a cofibrant $p$-DG module over $A/J$. If $P$ is acyclic, it follows then that $JP$ is also acyclic since $\Delta$ is cofibrant, and $P/JP$ is acyclic via the above sequence. It follows that, given any object $M\in \mc{D}$, there exists a distinguished triangle
\begin{equation}
J\otimes_A^\mathbf{L} M\lra M\lra (A/J)\otimes_A^{\mathbf{L}}M\lra J\otimes_A^{\mathbf{L}}M[1].
\end{equation}

\begin{lem}\label{lemTadmissible}
The category $\mc{T}$ is right admissible.
\end{lem} 
\begin{proof}
By Lemma \ref{lemadmissiblecate}, it amounts to showing that, for any $M\in \mc{D}$, the object $(A/J)\otimes_A^{\mathbf{L}}M\in \mc{T}^\perp$. To do this it suffices to show that $\Hom_{\mc{D}}(\Delta\{k\}, A/J)=0$ for all $k\in \Z$. 

First off, notice that
$\HOM_A(J\{k\}, A/J)$ is zero because $J^2=J$ (Lemma \ref{lemcellularidealqh}). The graded hom space is a right $p$-DG module over the $p$-DG endomorphism algebra
\[
E:=\END_A(J\{k\})\cong \END_{A}(\Delta\otimes \Delta^\circ)\cong \END_\Bbbk(\Delta^\circ).
\]
Since $\Delta^\circ$ is non-acyclic, $\Delta^\circ$ is cofibrant over $\END_\Bbbk(\Delta^\circ)$ (see the proof of Corollary \ref{corJnonacyclic}). Thus we have
\[
0=\HOM_A(J\{k\},A/J)\otimes_E (\Delta^\circ)^*\cong \HOM_A(J\{k\}\otimes_E (\Delta^\circ)^*, A/J) \cong \HOM_A(\Delta\{k\},A/J).
\]
The result now follows from the morphism space formula \eqref{eqnmorphismspaceoutofcofibrantmod}.
\end{proof}

Now let us give a more direct description of the category $\mc{T}^\perp$. The natural $p$-DG algebra homomorphism $\pi: A\lra A/J$ induces a derived tensor product functor (see \eqref{eqnderivedtensor})
\begin{equation}
j^!:=(A/J)\otimes_A^{\mathbf{L}}(\mbox{-}): \mc{D}(A)\lra \mc{D}(A/J),
\end{equation}
which is left adjoint to the derived hom functor (see  \eqref{eqntensorhomadjunction}) 
\begin{equation}
j_!:=\mathbf{R}\HOM_{A/J}(A/J,\mbox{-}):\mc{D}(A/J)\lra \mc{D}(A).
\end{equation}
The latter is no other than the natural restriction functor along $\pi$, i.e., regarding a $p$-DG module over $A/J$ as that over $A$, since $A/J$ is left cofibrant over itself. 

\begin{lem}\label{lemTperp}
The restriction functor $j_!:\mc{D}(A/J)\lra \mc{D}(A)$ is a fully-faithful embedding of $p$-DG derived categories whose image coincides with $\mc{T}^\perp$.
\end{lem}
\begin{proof}
By Proposition \ref{propfullyfaithfulness}, the fully-faithfulness of the restriction functor is equivalent to the canonical adjunction of derived $p$-DG functors
\[
j^!\circ j_! \Rightarrow \mathrm{Id}_{\mc{D}(A/J)}
\]
being an isomorphism on the set $\{A/J\{k\}|k\in \Z\}$ of compact generators for $\mc{D}(A/J)$. 

We compute the effect of the left-hand term on $A/J\{k\}$. The $p$-DG module $A/J\{k\}$ has a cofibrant replacement as the cone of the $p$-DG inclusion map $\iota: J\{k\}\hookrightarrow A\{k\}$:
\begin{equation}\label{eqncofresolutionAmodJ}
\underbrace{ J\{k\}\stackrel{=}{\lra}\cdots \stackrel{=}{\lra} J\{k\}}_{p-1~\textrm{terms}}\stackrel{\iota}{\longrightarrow} A\{k\},
\end{equation}
which is treated as a finite-cell module (Definition \ref{def-finite-cell}) with both internal differentials (on $J\{k\}$ and $A\{k\}$) and external ones (indicated on the arrows). Thus
\[
j^!\circ j_! (A/J\{k\})\cong j^!(A/J\{k\})\cong A/J\otimes_A\left( \underbrace{ J\{k\}=\cdots = J\{k\}}_{p-1~\textrm{terms}}\longrightarrow A\{k\} \right) \cong A/J\{k\},
\]
where, in the last isomorphism, we have used the fact that $J^2=J$ (Lemma \ref{lemcellularidealqh}) so that $A/J\otimes_A J\cong 0$. The result follows.
\end{proof}

\begin{rem}
Note that Lemma \ref{lemTembedding}--Lemma \ref{lemTperp} also hold when one restricts to the full subcategories of compact objects in the corresponding derived categories, as long as $\Delta$ is assumed to be compact and cofibrant. For instance, to see that $j_!$ sends compact objects in $\mc{D}(A/J)$ to compact objects in $\mc{D}(A)$, it suffices to test it on the $p$-DG module $A/J\in \mc{D}^c(A/J)$ by Theorem \ref{thmcompactmodules}. Then the module $A/J$ has the cofibrant replacement in \eqref{eqncofresolutionAmodJ}, and therefore it is  compact in $\mc{D}(A)$.
\end{rem}

\begin{defn}\label{defsemiorthdecomp}
A collection of admissible full triangulated subcategories 
\[
\langle \mc{T}_1,\dots, \mc{T}_n \rangle \subset \mc{D}
\]
is a \emph{semi-orthogonal decomposition of $\mc{D}$} if there is a sequence of left admissible subcategories
\[
0=\mc{D}_0\subset \mc{D}_1 \subset \mc{D}_2 \subset \cdots \subset \mc{D}_n=\mc{D}
\]
such that $\mc{T}_i$ is left orthogonal to $\mc{D}_{i-1}$ in $\mc{D}_i$, $i=1,\dots, n$.
\end{defn}

Via this definition, a $p$-DG quasi-hereditary cellular ideal gives rise to a semi-orthogonal decomposition of the (compact) derived category.

\begin{thm}\label{thmorthdecomp}
Let $\phi:\Delta\otimes \Delta^\circ\cong J\subset A$ be a $p$-DG quasi-hereditary ideal. Then there is a semi-orthogonal decomposition of triangulated categories
\[
\mc{D}(A)=\left\langle \mc{D}(A/J), \mc{T} \right\rangle, \quad \quad \mc{D}^c(A)=\left\langle  \mc{D}^c(A/J) ,\mc{T}\right\rangle.
\]
\end{thm}
\begin{proof}
The result now follows directly from Lemmas \ref{lemadmissiblecate},  \ref{lemTembedding},  \ref{lemTadmissible} and \ref{lemTperp}. 
\end{proof}

The semi-orthogonal decompsition translates into a direct sum decomposition on the Grothendieck group level.

\begin{cor}\label{corqhidealcase}
Let $\phi:\Delta\otimes \Delta^\circ\cong J\subset A$ be a $p$-DG quasi-hereditary ideal. Then there is a direct sum decomposition of Grothendieck groups
\[
K_0(A)\cong K_0(A/J)\oplus \mathbb{O}_p[\Delta].
\]
\end{cor}
\begin{proof}
By Lemma \ref{lemTorthequivlocalication}, we have a short exact sequence of triangulated categories
\[
\mc{T}\stackrel{i_*}{\lra} \mc{D}^c(A)\stackrel{j^!}{\lra} \mc{D}^c(A/J).
\]
It follows that there is a split exact sequence of Grothendieck groups, as $\mathbb{O}_p$ modules,
\begin{equation}\label{eqnKgroupsequence}
K_0(\mc{T})\stackrel{[i_*]}{\lra}K_0(A)\stackrel{[j^!]}{\lra} K_0(A/J)\lra 0,
\end{equation}
where the splitting is provided by the symbol of the functor $j_!$.

Since there is a $p$-DG derived equivalence $\mc{D}^c(\Bbbk)\cong\mc{T}$, we see that
\[
 K_0(\Bbbk) \cong  K_0(\mc{T}), \quad \quad [\Bbbk]\mapsto [\Delta]
\] 
is an isomorphism of free rank-one $\mathbb{O}_p$-modules. We are then reduced to showing that 
\[
K_0(\mc{T})\stackrel{[i_*]}{\lra} K_0(A),\quad \quad [\Delta]\mapsto [\Delta],
\]
is injective. If there is a relation $g(q)[\Delta]=0$ in $K_0(A)$, we would then obtain that
\[
0=[\HOM_A(\Delta,g(q) \Delta)]=g(q)[\HOM_A(\Delta,\Delta)]=g(q).
\]
The corollary now follows.
\end{proof}

\begin{cor}\label{corquasihereditarycase}
Let $A$ be a $p$-DG quasi-hereditary cellular algebra with a $p$-DG quasi-hereditary chain
 $$0=J_{n+1}\subset J_{n}\subset \dots \subset J_1=A,$$ 
and let
\[
\phi_i : \Delta_i\otimes \Delta_i^\circ \cong J_{i}/J_{i+1}
\]
be the corresponding isomorphisms of $p$-DG bimodules over $A/J_{i+1}$, $i=1,2,\dots,n$. Set $\mc{T}_i$ be the full-triangulated subcategory 
generated by $\Delta_i$ inside $\mc{D}$. Then the collection $\mc{T}_i$, $i=1,\dots, n$, gives a semi-orthogonal decomposition of $\mc{D}$. Furthermore, there is an isomorphism of Grothendieck groups
\[
K_0(A)\cong \bigoplus_{i=1}^n \mathbb{O}_p[\Delta_i].
\]
\end{cor}

\begin{proof}
This follows by an easy induction on $n$ via Theorem \ref{thmorthdecomp} and Corollary \ref{corqhidealcase}.
\end{proof}

\begin{rem}
We make two simple remarks about $p$-DG quasi-hereditary cellular algebras to conclude this subsection.
\begin{enumerate}
\item[(1)] The collection of $p$-DG cell modules $\{\Delta_i|i=1,\dots,n\}$ constitutes a \emph{full exceptional sequence} in $\mc{D}(A)$ in the sense that 
\[
\RHOM_A(\Delta_i,\Delta_j)=
\begin{cases}
\Bbbk & i=j,\\
0 & i>j,
\end{cases}
\]
and their grading shifts constitute a set of (compact) generators for $\mc{D}(A)$.
\item[(2)] Another basis of the Grothendieck group $K_0(A)$ is given by the images of simple $p$-DG modules $\{L_i|i=1,\dots, n\}$. This follows directly from Theorem \ref{thmsimplemodulescompact} and Lemma \ref{lemcellsimplemultiplicity}. 
\end{enumerate}
\end{rem}

%\subsection{Cofibrant covers}
%
%\begin{lem}
%Let $\phi:\Delta\otimes \Delta^\circ\cong J\subset A$ be a $p$-DG quasi-hereditary ideal, and let $Q$ be a cofibrant $p$-DG module over $A/J$. Set $U:=\Ext_{A_\dif}^1(P,\Delta)$. If
%\[
%0\lra \Delta\otimes U^*\lra P \lra Q\lra 0
%\]
%is the extension sequence corresponding to the identity map in 
%$$\Hom(U,U)\cong \Hom(U,\Ext^1_{A_\dif}(Q,\Delta))\cong \Ext^1_{A_\dif}(Q,\Delta\otimes U^*).$$
%Then $P$ is a cofibrant $p$-DG module over $A$.
%\end{lem}

%%%%%%%%%%%%%%%%%%%%%%%%%%%%%%%%%
%%%%%%%%%%%%%%%%%%%%%%%%%%%%%%%%%

%%%%%%%%%%%%%%%%%%%%%%%%%%%%%%%%%%%%%%%%%%%%%
%%%%%%%%%%%%%%%%%%%%%%%%%%%%%%%%%%%%%%%%%%%%%
\section{A categorification of quantum \texorpdfstring{$\mathfrak{sl}_2$}{sl(2)} at prime roots of unity}\label{sec-cat-sl(2)}
%%%%%%%%%%%%%%%%%%%%%%%%%%%%%%%%%%%%%%%%%%%%%
%%%%%%%%%%%%%%%%%%%%%%%%%%%%%%%%%%%%%%%%%%%%%

In this section, we review various forms of categorified quantum $\mf{sl}_2$ at a prime root of unity. 

\subsection{The \texorpdfstring{$p$}{p}-DG 2-category \texorpdfstring{$\mathcal{U}$}{U}}\label{sec-2-cat-U}
We begin this section by recalling the diagrammatic definition of the $2$-category $\mathcal{U}$ introduced by Lauda in \cite{Lau1}. Then we will recall a specific $p$-differential on $\mc{U}$ introduced in \cite{EQ1}.

\begin{defn}\label{def-u-dot} The $2$-category $\mathcal{U}$ is an additive graded $\Bbbk$-linear category whose objects $m$ are elements of the weight lattice of $\mathfrak{sl}_2$. The $1$-morphisms are (direct sums of grading shifts of) composites of the generating $1$-morphisms $\1_{m+2} \mathcal{E} \1_m$ and $\1_m \mathcal{F} \1_{m+2}$, for each $m \in \Z$. 
%Each $\1_{m+2} \mathcal{E} \1_m$ will be drawn the same, regardless of the object $m$.

\begin{align*}
\begin{tabular}{|c|c|c|}
	\hline
	$1$-\textrm{Morphism Generator} &
	\begin{DGCpicture}
	\DGCstrand(0,0)(0,1)
	\DGCdot*>{0.5}
	\DGCcoupon*(0.1,0.25)(1,0.75){$^m$}
	\DGCcoupon*(-1,0.25)(-0.1,0.75){$^{m+2}$}
    \DGCcoupon*(-0.25,1)(0.25,1.15){}
    \DGCcoupon*(-0.25,-0.15)(0.25,0){}
	\end{DGCpicture}&
	\begin{DGCpicture}
	\DGCstrand(0,0)(0,1)
	\DGCdot*<{0.5}
	\DGCcoupon*(0.1,0.25)(1,0.75){$^{m+2}$}
	\DGCcoupon*(-1,0.25)(-0.1,0.75){$^m$}
    \DGCcoupon*(-0.25,1)(0.25,1.15){}
    \DGCcoupon*(-0.25,-0.15)(0.25,0){}
	\end{DGCpicture} \\ \hline
	\textrm{Name} & $\1_{m+2}\mathcal{E}\1_m$ & $\1_m\mathcal{F}\1_{m+2}$ \\
	\hline
\end{tabular}
\end{align*}

The weight of any region in a diagram is determined by the weight of any single region. When no region is labeled, the ambient weight is irrelevant.

The $2$-morphisms will be generated by the following pictures.
	
\begin{align*}
\begin{tabular}{|c|c|c|c|c|}
  \hline
  \textrm{Generator} &
  \begin{DGCpicture}
  \DGCstrand(0,0)(0,1)
  \DGCdot*>{0.75}
  \DGCdot{0.45}
  \DGCcoupon*(0.1,0.25)(1,0.75){$^m$}
  \DGCcoupon*(-1,0.25)(-0.1,0.75){$^{m+2}$}
  \DGCcoupon*(-0.25,1)(0.25,1.15){}
  \DGCcoupon*(-0.25,-0.15)(0.25,0){}
  \end{DGCpicture}&
  \begin{DGCpicture}
  \DGCstrand(0,0)(0,1)
  \DGCdot*<{0.25}
  \DGCdot{0.65}
  \DGCcoupon*(0.1,0.25)(1,0.75){$^{m}$}
  \DGCcoupon*(-1,0.25)(-0.1,0.75){$^{m-2}$}
  \DGCcoupon*(-0.25,1)(0.25,1.15){}
  \DGCcoupon*(-0.25,-0.15)(0.25,0){}
  \end{DGCpicture} &
  \begin{DGCpicture}
  \DGCstrand(0,0)(1,1)
  \DGCdot*>{0.75}
  \DGCstrand(1,0)(0,1)
  \DGCdot*>{0.75}
  \DGCcoupon*(1.1,0.25)(2,0.75){$^m$}
  \DGCcoupon*(-1,0.25)(-0.1,0.75){$^{m+4}$}
  \DGCcoupon*(-0.25,1)(0.25,1.15){}
  \DGCcoupon*(-0.25,-0.15)(0.25,0){}
  \end{DGCpicture} &
  \begin{DGCpicture}
  \DGCstrand(0,0)(1,1)
  \DGCdot*<{0.25}
  \DGCstrand(1,0)(0,1)
  \DGCdot*<{0.25}
  \DGCcoupon*(1.1,0.25)(2,0.75){$^m$}
  \DGCcoupon*(-1,0.25)(-0.1,0.75){$^{m-4}$}
  \DGCcoupon*(-0.25,1)(0.25,1.15){}
  \DGCcoupon*(-0.25,-0.15)(0.25,0){}
  \end{DGCpicture} \\ \hline
  \textrm{Degree}  & 2   & 2 & -2 & -2 \\
  \hline
\end{tabular}
\end{align*}

\begin{align*}
\begin{tabular}{|c|c|c|c|c|}
  \hline
  \textrm{Generator} &
  \begin{DGCpicture}
  \DGCstrand/d/(0,0)(1,0)
  \DGCdot*>{-0.25,1}
  \DGCcoupon*(1,-0.5)(1.5,0){$^m$}
  \DGCcoupon*(-0.25,0)(1.25,0.15){}
  \DGCcoupon*(-0.25,-0.65)(1.25,-0.5){}
  \end{DGCpicture} &
  \begin{DGCpicture}
  \DGCstrand/d/(0,0)(1,0)
  \DGCdot*<{-0.25,2}
  \DGCcoupon*(1,-0.5)(1.5,0){$^m$}
  \DGCcoupon*(-0.25,0)(1.25,0.15){}
  \DGCcoupon*(-0.25,-0.65)(1.25,-0.5){}
  \end{DGCpicture}&
  \begin{DGCpicture}
  \DGCstrand(0,0)(1,0)/d/
  \DGCdot*<{0.25,1}
  \DGCcoupon*(1,0)(1.5,0.5){$^m$}
  \DGCcoupon*(-0.25,0.5)(1.25,0.65){}
  \DGCcoupon*(-0.25,-0.15)(1.25,0){}
  \end{DGCpicture}&
  \begin{DGCpicture}
  \DGCstrand(0,0)(1,0)/d/
  \DGCdot*>{0.25,2}
  \DGCcoupon*(1,0)(1.5,0.5){$^m$}
  \DGCcoupon*(-0.25,0.5)(1.25,0.65){}
  \DGCcoupon*(-0.25,-0.15)(1.25,0){}
  \end{DGCpicture}  \\ \hline
  \textrm{Degree} & $1+m$ & $1-m$ & $1+m$ & $1-m$ \\
  \hline
\end{tabular}
\end{align*}
\end{defn}

Before giving the full list of relations for $\mc{U}$, let us introduce some abbreviated notation. For a product of $r$ dots on a single strand, we draw a single dot labeled by $r$. Here is the case when $r=2$.
\begin{align*}
\begin{DGCpicture}
\DGCstrand(0,0)(0,1)
\DGCdot{.35}
\DGCdot{.65}
\DGCdot*>{.95}
\end{DGCpicture}
~=~
\begin{DGCpicture}
\DGCstrand(0,0)(0,1)
\DGCdot{.5}[r]{$^2$}
\DGCdot*>{.95}
\end{DGCpicture}
\end{align*}

A \emph{closed diagram} is a diagram without boundary, constructed from the generators above. The simplest non-trivial closed diagram is a \emph{bubble}, which is a closed diagram without any other closed diagrams inside. Bubbles can be oriented clockwise or counter-clockwise.
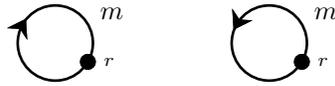
\begin{align*}
\begin{DGCpicture}
\DGCbubble(0,0){0.5}
\DGCdot*<{0.25,L}
\DGCdot{-0.25,R}[r]{$_r$}
\DGCdot*.{0.25,R}[r]{$m$}
\end{DGCpicture}
\qquad \qquad
\begin{DGCpicture}
\DGCbubble(0,0){0.5}
\DGCdot*>{0.25,L}
\DGCdot{-0.25,R}[r]{$_r$}
\DGCdot*.{0.25,R}[r]{$m$}
\end{DGCpicture}
\end{align*}

A simple calculation shows that the degree of a bubble with $r$ dots in a region labeled $m$ is $2(r+1-m)$ if the bubble is clockwise, and $2(r+1+m)$ if the bubble is counter-clockwise. Instead of keeping track of the number $r$ of dots a bubble has, it will be much more illustrative to keep track of the degree of the bubble, which is in $2\Z$. We will use the following shorthand to refer to an oriented bubble of degree $2k$.

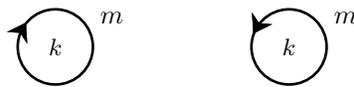
\begin{align*}
\bigcwbubble{$k$}{$m$}
\qquad \qquad
\bigccwbubble{$k$}{$m$}
\end{align*}
This notation emphasizes the fact that bubbles have a life of their own, independent of their presentation in terms of caps, cups, and dots.

Note that $m$ can be any integer, but $r\geq 0$ because it counts dots. Therefore, we can only construct a clockwise (resp. counter-clockwise) bubble of degree $k$ when $k \geq 1-m$ (resp. $k \geq 1+m$). These are called \emph{real bubbles}. Following Lauda, we also allow bubbles drawn as above with arbitrary $k \in \Z$. Bubbles with $k$ outside of the appropriate range are not yet defined in terms of the generating maps; we call these \emph{fake bubbles}. One can express any fake bubble in terms of real bubbles (see Remark \ref{rmk-inf-Grass-relation}).

Now we list the relations. Whenever the region label is omitted, the relation applies to all ambient weights.

\begin{itemize}
\item[(1)] {\bf Biadjointness and cyclicity relations.} These relations require that the generating endomorphisms of $\mc{E}$ and $\mc{F}$ are biadjoint to each other:
\begin{subequations} \label{biadjoint}
\begin{align} \label{biadjoint1}
\begin{DGCpicture}[scale=0.85]
\DGCstrand(0,0)(0,1)(1,1)(2,1)(2,2)
\DGCdot*>{0.75,1}
\DGCdot*>{1.75,1}
\end{DGCpicture}
~=~
\begin{DGCpicture}[scale=0.85]
\DGCstrand(0,0)(0,2)
\DGCdot*>{0.5}
\end{DGCpicture}
~=~
\begin{DGCpicture}[scale=0.85]
\DGCstrand(2,0)(2,1)(1,1)(0,1)(0,2)
\DGCdot*>{0.75}
\DGCdot*>{1.75}
\end{DGCpicture} \ ,
\qquad \qquad
\begin{DGCpicture}[scale=0.85]
\DGCstrand(0,0)(0,1)(1,1)(2,1)(2,2)
\DGCdot*<{0.75,1}
\DGCdot*<{1.75,1}
\end{DGCpicture}
~=~
\begin{DGCpicture}[scale=0.85]
\DGCstrand(0,0)(0,2)
\DGCdot*<{0.5}
\end{DGCpicture}
~=~
\begin{DGCpicture}[scale=0.85]
\DGCstrand(2,0)(2,1)(1,1)(0,1)(0,2)
\DGCdot*<{0.75}
\DGCdot*<{1.75}
\end{DGCpicture} \ ,
\end{align}

\begin{align} \label{biadjointdot}
\begin{DGCpicture}
\DGCstrand(0,0)(0,.5)(1,.5)/d/(1,0)/d/
\DGCdot*<{1}
\DGCdot{.3,1}
\end{DGCpicture}
~=~
\begin{DGCpicture}
\DGCstrand(0,0)(0,.5)(1,.5)/d/(1,0)/d/
\DGCdot*<{1}
\DGCdot{.3,2}
\end{DGCpicture} \ ,
\qquad \qquad \qquad \qquad
\begin{DGCpicture}
\DGCstrand(0,0)(0,.5)(1,.5)/d/(1,0)/d/
\DGCdot*>{1}
\DGCdot{.3,1}
\end{DGCpicture}
~=~
\begin{DGCpicture}
\DGCstrand(0,0)(0,.5)(1,.5)/d/(1,0)/d/
\DGCdot*>{1}
\DGCdot{.3,2}
\end{DGCpicture} \ ,
\end{align}

\begin{align} \label{biadjointcrossing}
\begin{DGCpicture}[scale=0.75]
\DGCstrand(0,0)(1,1)/u/(2,1)/d/(2,0)/d/
\DGCdot*<{1.5}
\DGCstrand(1,0)(0,1)/u/(3,1)/d/(3,0)/d/
\DGCdot*<{2.5}
\end{DGCpicture}
~=~
\begin{DGCpicture}[scale=0.75]
\DGCstrand(0,0)(0,1)(3,1)/d/(2,0)/d/
\DGCdot*<{2.5}
\DGCstrand(1,0)(1,1)(2,1)/d/(3,0)/d/
\DGCdot*<{1.5}
\end{DGCpicture} \ ,
\qquad \qquad
\begin{DGCpicture}[scale=0.75]
\DGCstrand(0,0)(1,1)/u/(2,1)/d/(2,0)/d/
\DGCdot*>{1.5}
\DGCstrand(1,0)(0,1)/u/(3,1)/d/(3,0)/d/
\DGCdot*>{2.5}
\end{DGCpicture}
~=~
\begin{DGCpicture}[scale=0.75]
\DGCstrand(0,0)(0,1)(3,1)/d/(2,0)/d/
\DGCdot*>{2.5}
\DGCstrand(1,0)(1,1)(2,1)/d/(3,0)/d/
\DGCdot*>{1.5}
\end{DGCpicture} \ .
\end{align}
\end{subequations}

\item[(2)] {\bf Positivity and Normalization of bubbles.} Positivity states that all bubbles (real or fake) of negative degree should be zero.
\begin{subequations} \label{negzerobubble}
\begin{align} \label{negbubble}
\bigcwbubble{$k$}{}~=~0~=~\bigccwbubble{$k$}{}
\qquad
\textrm{if}~k<0.
\end{align}

Normalization states that degree $0$ bubbles are equal to the empty diagram (i.e., the identity $2$-morphism of the identity $1$-morphism).

\begin{align} \label{zerobubble}
\bigcwbubble{$0$}{}~=~1~=~\bigccwbubble{$0$}{}~.
\end{align}
\end{subequations}

\item[(3)] {\bf NilHecke relations.} The upward pointing strands satisfy nilHecke relations. 
%Note that, diagrammatically, far-away elements commuting with each other become isotopy relations, and are thus built into the diagrammatics by default.
\begin{subequations} \label{NHrels}
\begin{align}
\begin{DGCpicture}
\DGCstrand(0,0)(1,1)(0,2)
\DGCdot*>{2}
\DGCstrand(1,0)(0,1)(1,2)
\DGCdot*>{2}
\end{DGCpicture}
=0\ , \quad \quad
\RIII{L}{$0$}{$0$}{$0$}{no} = \RIII{R}{$0$}{$0$}{$0$}{no},
\label{NHrelR3}
\end{align}
\begin{align}
\crossing{$0$}{$0$}{$1$}{$0$}{no} - \crossing{$0$}{$1$}{$0$}{$0$}{no} = \twolines{$0$}{$0$}{no} = \crossing{$1$}{$0$}{$0$}{$0$}{no} -
\crossing{$0$}{$0$}{$0$}{$1$}{no}. \label{NHreldotforce}
\end{align}
\end{subequations}
\item[(4)] {\bf Reduction to bubbles.} The following equalities hold for all $m \in \Z$.
\begin{subequations} \label{bubblereduction}
\begin{align}
\curl{R}{U}{$m$}{no}{$0$} = -\sum_{a+b=-m} \oneline{$b$}{$m$} \bigcwbubble{$a$}{}\ ,
\end{align}
\begin{align}
\curl{L}{U}{$m$}{no}{$0$} = \sum_{a+b=m} \bigccwbubble{$a$}{$m$} \oneline{$b$}{no} \ .
\end{align}
\end{subequations}
These sums only take values for $a,b \geq 0$. Therefore, when $m \neq 0$, either the right curl or the left curl is zero.
\item[(5)] {\bf Identity decomposition.} The following equations hold for all $m \in \Z$.
\begin{subequations} \label{IdentityDecomp}
\begin{align}
\begin{DGCpicture}
\DGCstrand(0,0)(0,2)
\DGCdot*>{1}
\DGCdot*.{1.25}[l]{$m$}
\DGCstrand(1,0)(1,2)
\DGCdot*<{1}
\DGCdot*.{1.25}[r]{$m$}
\end{DGCpicture}
~=~-~
\begin{DGCpicture}
\DGCstrand(0,0)(1,1)(0,2)
\DGCdot*>{0.25}
\DGCdot*>{1}
\DGCdot*>{1.75}
\DGCstrand(1,0)(0,1)(1,2)
\DGCdot*.{1.25}[l]{$m$}
\DGCdot*<{0.25}
\DGCdot*<{1}
\DGCdot*<{1.75}
\end{DGCpicture}
~+~
\sum_{a+b+c=m-1}~
\cwcapbubcup{$a$}{$b$}{$c$}{$m$} \ , \label{IdentityDecompPos}
\end{align}
\begin{align}
\begin{DGCpicture}
\DGCstrand(0,0)(0,2)
\DGCdot*<{1}
\DGCdot*.{1.25}[l]{$m$}
\DGCstrand(1,0)(1,2)
\DGCdot*>{1}
\DGCdot*.{1.25}[r]{$m$}
\end{DGCpicture}
~=~-~
\begin{DGCpicture}
\DGCstrand(0,0)(1,1)(0,2)
\DGCdot*<{0.25}
\DGCdot*<{1}
\DGCdot*<{1.75}
\DGCstrand(1,0)(0,1)(1,2)
\DGCdot*.{1.25}[l]{$m$}
\DGCdot*>{0.25}
\DGCdot*>{1}
\DGCdot*>{1.75}
\end{DGCpicture}
~+~
\sum_{a+b+c=-m-1}~
\ccwcapbubcup{$a$}{$b$}{$c$}{$m$} \ . \label{IdentityDecompNeg}
\end{align}
\end{subequations}
The sum in the first equality vanishes for $m \leq 0$, and the sum in the second equality vanishes for $m \geq 0$.

The terms on the right hand side form a collection of orthogonal idempotents. 
\end{itemize}

\begin{rem}[Infinite Grassmannian relations]\label{rmk-inf-Grass-relation}  This family of relations, which follows from the above defining relations, can be expressed most succinctly in terms of generating functions.

\begin{equation}\label{eqn-infinite-Grassmannian}
\left( \cwbubble{$0$}{}+t~\cwbubble{$1$}{}+t^2~\cwbubble{$2$}{}+\ldots \right) \cdot
\left( \ccwbubble{$0$}{}+t~\ccwbubble{$1$}{}+t^2~\ccwbubble{$2$}{}+\ldots \right)  =  1~.
\end{equation}

The cohomology ring of the ``infinite dimensional Grassmannian" is the ring $\Lambda$ of symmetric functions. Inside this ring, there is an analogous relation $\mathtt{e}(t)\mathtt{h}(t)=1$, where $\mathtt{e}(t) = \sum_{i \ge 0} (-1)^i \mathtt{e}_i t^i$ is the total Chern class of the tautological bundle, and $\mathtt{h}(t) = \sum_{i \ge 0} \mathtt{h}_i t^i$ is the total Chern class of the dual bundle. Lauda has proved that the bubbles in a single region generate an algebra inside $\mathcal{U}$ isomorphic to $\Lambda$.

Looking at the homogeneous component of degree $m$, we have the following equation.
\begin{align}
\sum_{a + b = m} \bigcwbubble{$a$}{} \bigccwbubble{$b$}{} = \delta_{m,0}. \label{infgrass}
\end{align}
Because of the positivity of bubbles relation, this equation holds true for any $m \in \Z$, and the sum can be taken over all $a,b \in \Z$.

Using these equations one can express all (positive degree) counter-clockwise bubbles in terms of clockwise bubbles, and vice versa.
Consequentially, all fake bubbles can be expressed in terms of real bubbles.
\end{rem}

\begin{defn}\label{def-special-dif}
Let $\dif$ be the derivation defined on the $2$-morphism generators of $\mathcal{U}$ as follows,
\[
\begin{array}{c}
\dif \left( \onelineshort{$1$}{no} \right) =  \onelineshort{$2$}{no} , \quad \quad \quad
\dif \left( \crossing{$0$}{$0$}{$0$}{$0$}{no} \right) =  \twolines{$0$}{$0$}{no} -2 \crossing{$1$}{$0$}{$0$}{$0$}{no} ,\\ \\
\dif \left( \onelineDshort{$1$}{no} \right) =  \onelineDshort{$2$}{no} , \quad \quad \quad
\dif \left( \crossingD{$0$}{$0$}{$0$}{$0$}{no} \right) = - \twolinesD{$0$}{$0$}{no} -2 \crossingD{$0$}{$1$}{$0$}{$0$}{no},
\end{array}
\]
\[
\begin{array}{c}
\dif \left( \cappy{CW}{$0$}{no}{$m$} \right)  =  \cappy{CW}{$1$}{no}{$m$}- \cappy{CW}{$0$}{CW}{$m$},\quad \quad \quad
\dif \left( \cuppy{CW}{$0$}{no}{$m$} \right) = (1-m) \cuppy{CW}{$1$}{no}{$m$},\\ \\
\dif \left( \cuppy{CCW}{$0$}{no}{$m$} \right) =  \cuppy{CCW}{$1$}{no}{$m$} + \cuppy{CCW}{$0$}{CW}{$m$},\quad \quad \quad
\dif \left( \cappy{CCW}{$0$}{no}{$m$} \right) = (m+1 ) \cappy{CCW}{$1$}{no}{$m$} ,
\end{array}
\]
and extended to the entire $\mc{U}$ by the Leibniz rule.
\end{defn}

\begin{thm}\label{thm-itsanalghom}
There is an isomorphism of $\mathbb{O}_p$-algebras
$$\dot{u}_{\mathbb{O}_p} \lra K_0(\mathcal{D}^c(\mathcal{U})),$$
sending $1_{m+2}E1_m $ to $[\1_{m+2}\mathcal{E}\1_m]$ and $1_m F 1_{m+2} $ to $ [\1_m \mathcal{F}\1_{m+2}]$ for any weight $m \in \Z$. 
\end{thm}
\begin{proof}
This is \cite[Theorem 6.11]{EQ1}.
\end{proof}

\begin{rem}\label{rmk-uniqueness-of-d}
Let us also record some basic properties of the differential established in \cite{EQ1}.
\begin{enumerate}
\item[(1)] In \cite[Definition 4.6]{EQ1}, a multi-parameter family of $p$-differentials is defined on $\mc{U}$ which preserve the defining relations of $\mc{U}$. This specific differential above is essentially the only one, up to conjugation by automorphisms of $\mc{U}$, that allows a \emph{fantastic filtration} to exist on $(\mc{U},\dif)$. The fantastic filtration in turn decategorifies to the quantum $E$-$F$ relations for  $\mf{sl}_2$ at a prime root of unity.
\item[(2)]In the definition of the multi-parameter family of $p$-differentials, there are, though, certain redundancies set forth by the constraints on the parameters thereof (see \cite[equation (4.8)]{EQ1}). Indeed, once one fixes the differential on the upward pointing nilHecke generators, clockwise cups and counter-clockwise caps, then the differential is uniquely specified on the entire $\mc{U}$. Therefore, in order to obtain the differential in Definition \ref{def-special-dif}, it is enough to just specify
\[
\dif \left( \onelineshort{$1$}{no} \right) =  \onelineshort{$2$}{no} , \quad \quad \quad
\dif \left( \crossing{$0$}{$0$}{$0$}{$0$}{no} \right) =  \twolines{$0$}{$0$}{no} -2 \crossing{$1$}{$0$}{$0$}{$0$}{no} ,
\]
\[
\dif \left( \cuppy{CW}{$0$}{no}{$m$} \right) = (1-m) \cuppy{CW}{$1$}{no}{$m$},\quad \quad \quad
\dif \left( \cappy{CCW}{$0$}{no}{$m$} \right) = (m+1) \cappy{CCW}{$1$}{no}{$m$}.
\]
\end{enumerate} 
\end{rem}

\subsection{Thick calculus with the differential}\label{sec-thick-cal}
We begin by recording some notation concerning symmetric polynomials. Let $\mc{P}(a,b)$ be the set of partitions which fit into an $a \times b$ box (height $a$, width $b$).  An element $\mu=(\mu_1, \ldots, \mu_a) \in \mc{P}(a,b)$ with $ \mu_1 \geq \cdots \geq \mu_a \geq 0$ gives rise to a \emph{Schur polynomial} $ \pi_{\mu}$ defined as follows:
\begin{equation}\label{eqn-Schur-pol}
\pi_{\mu}:= \frac{\mathrm{det}(M_{\mu})}{\prod_{1 \leq i<j \leq a}(y_i-y_j)}
\hspace{.5in}
(M_{\mu})_{ij}:=y_i^{a+\mu_j-j}.
\end{equation}

%The following are some special examples of Schur polynomials.
%\begin{itemize}
%\item If $\mu=(1^c)$ with $ c \leq a$ then $\pi_{\mu}=\sum_{1\leq i_1<\cdots <i_c %\leq b}y_{i_1}\cdots y_{i_c}$ is the usual degree-$c$ elementary symmetric %function.

%\item If $\mu=(c)$ with $ c \leq b$ then $\pi_{\mu}=\pi_{\mu}=\sum_{1\leq i_1\leq %\cdots \leq i_c \leq b}y_{i_1}\cdots y_{i_c}$ is the usual degree-$c$ complete %symmetric function.

%\item If $\mu=(c^a)$ with $ c \leq b $ then $\pi_{\mu}= y_1^c \cdots y_a^c$.
%\end{itemize}

For a partition $\mu \in \mc{P}(a,b)$ we will form a complementary partition $ \hat{\mu} \in \mc{P}(b,a)$.
First define the sequence
\begin{equation*}
\mu^c=(b-\mu_a, \ldots, b-\mu_1).
\end{equation*}
Then set
\begin{equation*}
\hat{\mu}=(\hat{\mu}_1, \ldots, \hat{\mu}_b) 
\hspace{.5in}
\hat{\mu}_{j}=|\{i | \mu^c_i \geq j \}      |
\end{equation*}
which we will refer to as the \emph{complementary partition} of $\mu$, with the underlying $\mc{P}(a,b)$ implicitly taken into account.

The $2$-category $\mc{U}$ introduced in the previous section has a ``thick enhancement'' $\dot{\mc{U}}$ defined by Khovanov, Lauda, Mackaay and Sto\v{s}i\'{c} (\cite{KLMS}). The thick calculus $\dot{\mc{U}}$ is the Karoubi envelope of $\mc{U}$, and thus is Morita equivalent to $\mc{U}$.  

To construct $\dot{\mc{U}}$ it suffices to adjoin to $\mc{U}$ the image of the idempotents $e_r\in \END_{\mc{U}}(\mc{E}^r)$, one for each $r\in \N$, where $e_r$ is diagrammatically given by
\begin{equation}
e_{r}:=
\begin{DGCpicture}[scale=0.9]
\DGCstrand(0,0)(4,4)
\DGCdot>{3.95}
\DGCstrand(4,0)(0,4)
\DGCdot>{3.95}
\DGCdot{3.5}[r]{$_{r-1}$}
\DGCstrand(3,0)(4,1)(1,4)
\DGCdot>{3.95}
\DGCdot{3.5}[r]{$_{r-2}$}
\DGCstrand(2,0)(4,2)(2,4)
\DGCdot>{3.95}
\DGCdot{3.5}[r]{$_{r-3}$}
\DGCcoupon*(2.2,3.6)(3.8,4){$\cdots$}
\DGCcoupon*(0.2,0)(1.8,0.4){$\cdots$}
\end{DGCpicture}
\ .
\end{equation}
As in \cite{KLMS}, we can depict the 1-morphism representing the pair $(\mc{E}^r,e_r)$ by an upward pointing ``thick arrow'' of thickness $r$ (when the thickness equals one, the above $1$-morphism agrees with the single black strand for $\mc{E}$):
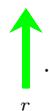
\[
\begin{DGCpicture}
\DGCstrand[Green](0,0)(0,1)[$^r$]
\DGCdot>{0.95}
\end{DGCpicture} \ .
\]
The new object of thickness $r$, for each $r\in \N$, is called the \emph{$r$th divided power} of $\mc{E}$, and is usually denoted as $\mc{E}^{(r)}$. The object $\mc{E}^{(r)}$ has symmetric polynomials in $r$ variables as its endomorphism algebra. We express the multiplication of these endomorphism algebra elements by vertically concatenating pictures labeled by symmetric polynomials:
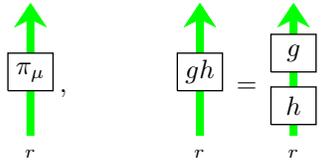
\[
\begin{DGCpicture}
\DGCstrand[Green](0,-0.35)(0,1.35)[$^r$]
\DGCdot>{1.3}
\DGCcoupon(-0.3,0.25)(0.3,0.75){$g$}
\end{DGCpicture}\ , \quad \quad \quad \quad
\begin{DGCpicture}
\DGCstrand[Green](0,-0.35)(0,1.35)[$^r$]
\DGCdot>{1.3}
\DGCcoupon(-0.3,0.25)(0.3,0.75){$gh$}
\end{DGCpicture}
~=~
\begin{DGCpicture}
\DGCstrand[Green](0,-0.35)(0,1.35)[$^r$]
\DGCdot>{1.3}
\DGCcoupon(-0.3,0.5)(0.3,1){$g$}
\DGCcoupon(-0.3,-0.2)(0.3,0.3){$h$}
\end{DGCpicture}
\ .
\]

There are generating morphisms from $\mc{E}^r$ to $\mc{E}^{(r)}$ and backwards, depicted as
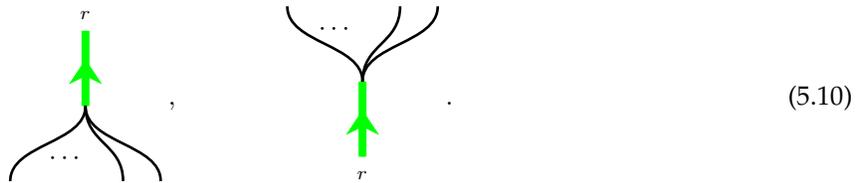
\begin{equation}
\begin{DGCpicture}
\DGCstrand[Green](0,0)(0,1)[`$_r$]
\DGCdot>{0.5}
\DGCstrand/u/(-1,-1)(0,0)/u/
\DGCstrand/u/(0.5,-1)(0,0)/u/
\DGCstrand/u/(1,-1)(0,0)/u/
\DGCcoupon*(-0.8,-0.8)(0.3,-0.6){$\cdots$}
\end{DGCpicture} 
\ ,
\quad \quad \quad \quad
\begin{DGCpicture}
\DGCstrand[Green](0,-1)(0,0)[$^r$]
\DGCdot>{-0.5}
\DGCstrand/d/(-1,1)(0,0)/d/
\DGCstrand/d/(0.5,1)(0,0)/d/
\DGCstrand/d/(1,1)(0,0)/d/
\DGCcoupon*(-0.6,0.8)(-0.1,0.6){$\cdots$}
\end{DGCpicture}
\ .
\end{equation}
More generally, for any $a,b\in \N$, we have generating morphisms between $\mc{E}^{(a)}\mc{E}^{(b)}$ and $\mc{E}^{(a+b)}$ (see next subsection for more details) which are drawn respectively as
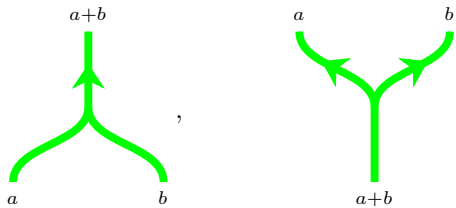
\[
\begin{DGCpicture}
\DGCstrand[Green](0,0)(0,1)[`$_{a+b}$]
\DGCdot>{0.45}
\DGCstrand[Green]/u/(-1,-1)(0,0)/u/[$^a$]
\DGCstrand[Green]/u/(1,-1)(0,0)/u/[$^b$]
\end{DGCpicture} \ ,
\quad \quad \quad \quad
\begin{DGCpicture}
\DGCPLstrand[Green](0,-1)(0,0)[$^{a+b}$]
\DGCstrand[Green]/u/(0,0)(-1,1)/u/[`$_a$]\DGCdot>{0.55}
\DGCstrand[Green]/u/(0,0)(1,1)/u/[`$_b$]\DGCdot>{0.55}
\end{DGCpicture} \ .
\]
Such thick diagrams, carrying symmetric polynomials, satisfy certain diagrammatic identities which are consequences of relations in the $2$-category $\mc{U}$. Let us recall a few relations for later use.

The relation below is known as the \emph{associativity of splitters} (\cite[Proposition 2.2.4]{KLMS}):
\begin{equation}\label{eqn-thick-relation-splitter-asso}
\begin{DGCpicture}[scale=0.85]
\DGCPLstrand[Green](0,0)(1,1)[$^a$]
\DGCPLstrand[Green](1,0)(0.5,0.5)[$^b$]
\DGCPLstrand[Green](2,0)(1,1)[$^c$]
\DGCPLstrand[Green](1,1)(1,2)[`$_{a+b+c}$]
\DGCdot>{1.5}
\end{DGCpicture}
~=~
\begin{DGCpicture}[scale=0.85]
\DGCPLstrand[Green](0,0)(1,1)[$^a$]
\DGCPLstrand[Green](1,0)(1.5,0.5)[$^b$]
\DGCPLstrand[Green](2,0)(1,1)[$^c$]
\DGCPLstrand[Green](1,1)(1,2)[`$_{a+b+c}$]
\DGCdot>{1.5}
\end{DGCpicture} \ ,
\quad \quad \quad
\begin{DGCpicture}[scale=0.85]
\DGCPLstrand[Green](1,0)(1,1)[$^{a+b+c}$]
\DGCdot>{0.5}
\DGCPLstrand[Green](1,1)(0,2)[`$_a$]
\DGCPLstrand[Green](0.5,1.5)(1,2)[`$_b$]
\DGCPLstrand[Green](1,1)(2,2)[`$_c$]
\end{DGCpicture}
~=~
\begin{DGCpicture}[scale=0.85]
\DGCPLstrand[Green](1,0)(1,1)[$^{a+b+c}$]
\DGCdot>{0.5}
\DGCPLstrand[Green](1,1)(0,2)[`$_a$]
\DGCPLstrand[Green](1.5,1.5)(1,2)[`$_b$]
\DGCPLstrand[Green](1,1)(2,2)[`$_c$]
\end{DGCpicture} \ .
\end{equation}

There is the \emph{sliding relation} (\cite[Proposition 2.2.5]{KLMS}):
\begin{equation}\label{eqn-thick-relation-sliding}
\begin{DGCpicture}[scale=0.8]
\DGCstrand[Green](0,0)(2,2.5)[$^c$]
\DGCstrand[Green](2,0)(2,0.5)[$^{a+b}$]
\DGCstrand[Green](2,0.5)(0,3)[`$_a$]
\DGCdot>{2.75}
\DGCstrand[Green](2,0.5)(2.5,1.5)(2,2.5)
\DGCstrand[Green](2,2.5)(2,3)[`$_{b+c}$]
\DGCdot>{2.75}
\end{DGCpicture}
~=~
\begin{DGCpicture}[scale=0.8]
\DGCstrand[Green](0,0)(1,1.15)[$^c$]
\DGCstrand[Green](2,0)(1,1.15)[$^{a+b}$]
\DGCstrand[Green](1,1.15)(1,1.85)
\DGCstrand[Green](1,1.85)(0,3)[`$_a$]
\DGCdot>{2.75}
\DGCstrand[Green](1,1.85)(2,3)[`$_{b+c}$]
\DGCdot>{2.75}
\end{DGCpicture} \ ,
\quad \quad \quad
\begin{DGCpicture}[scale=0.8]
\DGCstrand[Green](0,0)(-2,2.5)[$^c$]
\DGCstrand[Green](-2,0)(-2,0.5)[$^{a+b}$]
\DGCstrand[Green](-2,0.5)(0,3)[`$_a$]
\DGCdot>{2.75}
\DGCstrand[Green](-2,0.5)(-2.5,1.5)(-2,2.5)
\DGCstrand[Green](-2,2.5)(-2,3)[`$_{b+c}$]
\DGCdot>{2.75}
\end{DGCpicture}
~=~
\begin{DGCpicture}[scale=0.8]
\DGCstrand[Green](0,0)(-1,1.15)[$^c$]
\DGCstrand[Green](-2,0)(-1,1.15)[$^{a+b}$]
\DGCstrand[Green](-1,1.15)(-1,1.85)
\DGCstrand[Green](-1,1.85)(0,3)[`$_a$]
\DGCdot>{2.75}
\DGCstrand[Green](-1,1.85)(-2,3)[`$_{b+c}$]
\DGCdot>{2.75}
\end{DGCpicture}~.
\end{equation}

We have the following \emph{pairing relation} (\cite[Corollary 2.4.2]{KLMS}):
\begin{equation} \label{eqn-thick-pairing}
\begin{DGCpicture}
\DGCstrand[Green](0,0)(0,0.25)[$^{a+b}$]
\DGCstrand[Green](0,0.25)(-0.5,1)(0,1.75)
\DGCstrand[Green](0,0.25)(0.5,1)(0,1.75)
\DGCstrand[Green](0,1.75)(0,2)
\DGCcoupon(-0.75,0.8)(-0.25,1.2){$\pi_{\alpha}$}
\DGCcoupon(0.25,0.8)(0.75,1.2){$\pi_{\beta}$}
\DGCcoupon*(-0.75,1.2)(-0.25,1.7){$_a$}
\DGCcoupon*(0.25,1.2)(0.75,1.7){$_b$}
\DGCdot>{1.95}
\end{DGCpicture}
=
\begin{cases}
(-1)^{|\hat{\alpha}|}
\begin{DGCpicture}
\DGCstrand[Green](0,0)(0,2)[$^{a+b}$]
\DGCdot>{1.95}
\end{DGCpicture} & \textrm{if}~\alpha = \hat{\beta} ,\\
0 & \textrm{otherwise}.
\end{cases}
\end{equation}
Here $\pi_{\alpha}$ (respectively $\pi_{\beta}$) stands for the Schur polynomial corresponding to a partition $\alpha\in \mc{P}(a,b)$, (respectively $\beta\in \mc{P}(b,a)$). 
The notation $\hat{\alpha}$ stands for the complementary partition of $\alpha$ in $\mc{P}(b,a)$. 

Finally, there is a family of idempotents satisfying the following \emph{identity decomposition} relation:
\begin{equation} \label{eq-identitydecomp}
\begin{DGCpicture}[scale=0.9]
\DGCstrand[Green](0,0)(0,3)[$^a$`$_a$]
\DGCdot>{1.5}
\DGCstrand[Green](2,0)(2,3)[$^b$`$_b$]
\DGCdot>{1.5}
\end{DGCpicture}
=\sum_{\alpha\in \mc{P}(a,b) }(-1)^{|\hat{\alpha}|}
\begin{DGCpicture}[scale=0.9]
\DGCPLstrand[Green](0,0)(1,1.15)[$^a$]
\DGCPLstrand[Green](2,0)(1,1.15)[$^b$]
\DGCPLstrand[Green](1,1.15)(1,1.85)
\DGCdot>{1.5}
\DGCPLstrand[Green](1,1.85)(0,3)[`$_a$]
\DGCPLstrand[Green](1,1.85)(2,3)[`$_b$]
\DGCcoupon(1.15,0.35)(1.85,0.85){$_{\pi_{\hat{\alpha}}}$}
\DGCcoupon(0.15,2.15)(0.85,2.65){$_{\pi_\alpha}$}
\end{DGCpicture}~.
\end{equation}
This relation and the pairing relation \eqref{eqn-thick-pairing} imply that the right hand side of \eqref{eq-identitydecomp} constitutes a full orthogonal family of idempotents in $\END_{\dot{\mc{U}}}(\mc{E}^{(a)}\mc{E}^{(b)})$
We refer the reader to \cite{KLMS} for the details.
 
In \cite{EQ2}, a $p$-differential is defined on $\dot{\mc{U}}$ extending that of $\mc{U}$:
\begin{subequations}
\begin{equation}\label{eqn-d-action-mod-generator}
\dif\left(~
\begin{DGCpicture}[scale=0.85]
\DGCPLstrand[Green](1,0)(1,1)[$^{a+b}$]
\DGCdot>{0.5}
\DGCPLstrand[Green](1,1)(0,2)[`$_a$]
\DGCPLstrand[Green](1,1)(2,2)[`$_b$]
\end{DGCpicture}
~\right)
=
-b\begin{DGCpicture}[scale=0.85]
\DGCPLstrand[Green](1,0)(1,1)[$^{a+b}$]
\DGCdot>{0.5}
\DGCPLstrand[Green](1,1)(0,2)[`$_a$]
\DGCPLstrand[Green](1,1)(2,2)[`$_b$]
\DGCcoupon(0.2,1.25)(0.8,1.75){$_{\pi_1}$}
\end{DGCpicture} \ ,
\quad \quad \quad
\dif\left(~
\begin{DGCpicture}[scale=0.85]
\DGCPLstrand[Green](0,0)(1,1)[$^a$]
\DGCPLstrand[Green](2,0)(1,1)[$^b$]
\DGCPLstrand[Green](1,1)(1,2)[`$_{a+b}$]
\DGCdot>{1.5}
\end{DGCpicture}
~\right)=
-a
\begin{DGCpicture}[scale=0.85]
\DGCPLstrand[Green](0,0)(1,1)[$^a$]
\DGCPLstrand[Green](2,0)(1,1)[$^b$]
\DGCPLstrand[Green](1,1)(1,2)[`$_{a+b}$]
\DGCdot>{1.5}
\DGCcoupon(1.2,0.25)(1.8,0.75){$_{\pi_1}$}
\end{DGCpicture} \ ,
\end{equation}

\begin{equation}\label{eqn-dif-thick-cup-cap-1}
\dif\left(~
\begin{DGCpicture}[scale =0.5]
\DGCstrand[Green]/u/(0,0)(0.75,1)(1.5,0)/d/
\DGCdot<{0.5,1}[r]{$\scriptstyle{a}$}
\DGCcoupon*(1.4,0.7)(1.8,1.1){$m$}
\end{DGCpicture}
~\right)
=
(m+a)
\begin{DGCpicture}[scale =0.5]
\DGCstrand[Green]/u/(0,0)(0.75,1)(1.5,0)/d/
\DGCdot<{0.5,1}[r]{$\scriptstyle{a}$}
\DGCcoupon*(1.4,0.8)(1.8,1){$m$}
\DGCcoupon(1.2,0.2)(1.7,0.5){$\scriptstyle{\pi_1}$}
\end{DGCpicture}
\ ,
\quad \quad
\dif\left(~
\begin{DGCpicture}[scale =0.5]
\DGCstrand[Green]/d/(0,0)(0.75,-1)(1.5,0)/u/
\DGCdot<{-0.5,1}[r]{$\scriptstyle{a}$}
\DGCcoupon*(1.4,-0.8)(1.8,-1){$m$}
\end{DGCpicture}
~\right)
=(a-m)
\begin{DGCpicture}[scale =0.5]
\DGCstrand[Green]/d/(0,0)(0.75,-1)(1.5,0)/u/
\DGCdot<{-0.5,1}[r]{$\scriptstyle{a}$}
\DGCcoupon*(1.4,-0.8)(1.8,-1){$m$}
\DGCcoupon(1.2,-0.2)(1.7,-0.5){$\scriptstyle{\pi_1}$}
\end{DGCpicture}
\ ,
\end{equation}

\begin{equation}\label{eqn-dif-thick-cup-cap-2}
\dif\left(~
\begin{DGCpicture}[scale =0.5]
\DGCstrand[Green]/u/(0,0)(0.75,1)(1.5,0)/d/
\DGCdot>{0.5,1}[r]{$\scriptstyle{a}$}
\DGCcoupon*(1.4,0.7)(1.8,1.1){$m$}
\end{DGCpicture}
~\right)
=
a~~~
\begin{DGCpicture}[scale =0.5]
\DGCstrand[Green]/u/(0,0)(0.75,1)(1.5,0)/d/
\DGCdot>{0.5,1}[r]{$\scriptstyle{a}$}
\DGCcoupon*(1.4,0.8)(1.8,1){$m$}
\DGCcoupon(1.2,0.2)(1.7,0.5){$\scriptstyle{\pi_1}$}
\end{DGCpicture}
-a~~~
\begin{DGCpicture}[scale =0.5]
\DGCstrand[Green]/u/(0,0)(0.75,1)(1.5,0)/d/
\DGCdot>{0.5,1}[r]{$\scriptstyle{a}$}
\DGCcoupon*(1.4,0.8)(1.8,1){$m$}
\DGCbubble(2.25,0.65){0.3}
\DGCdot<{0.65,L}
\DGCcoupon*(2,0.5)(2.5,0.8){$\scriptsize{1}$}
\end{DGCpicture}
\ ,
\end{equation}

\begin{equation}\label{eqn-dif-thick-cup-cap-3}
\dif\left(~
\begin{DGCpicture}[scale =0.5]
\DGCstrand[Green]/d/(0,0)(0.75,-1)(1.5,0)/u/
\DGCdot>{-0.5,1}[r]{$\scriptstyle{a}$}
\DGCcoupon*(1.4,-0.8)(1.8,-1){$m$}
\end{DGCpicture}
~\right)
=
a~~~
\begin{DGCpicture}[scale =0.5]
\DGCstrand[Green]/d/(0,0)(0.75,-1)(1.5,0)/u/
\DGCdot>{-0.5,1}[r]{$\scriptstyle{a}$}
\DGCcoupon*(1.4,-0.8)(1.8,-1){$m$}
\DGCcoupon(1.2,-0.2)(1.7,-0.5){$\scriptstyle{\pi_1}$}
\end{DGCpicture}
+a~~~
\begin{DGCpicture}[scale =0.5]
\DGCstrand[Green]/d/(0,0)(0.75,-1)(1.5,0)/u/
\DGCdot>{-0.5,1}[r]{$\scriptstyle{a}$}
\DGCcoupon*(1.4,-0.8)(1.8,-1){$m$}
\DGCbubble(2.25,-0.65){0.3}
\DGCdot<{-0.65,L}
\DGCcoupon*(2,-0.5)(2.5,-0.8){$\scriptsize{1}$}
\end{DGCpicture}
\ .
\end{equation}
Here $\pi_1$ stands for the first elementary symmetric function in the number of variables labeled by the thickness of the strand, and the clockwise ``bubble''
\begin{equation}\label{equ-degree-2-bubble}
\begin{DGCpicture}
\DGCbubble(2.25,0.65){0.5}
\DGCdot<{0.65,L}
\DGCcoupon*(2,0.5)(2.5,0.8){$\scriptsize{1}$}
\end{DGCpicture}
:=
\begin{DGCpicture}
\DGCbubble(0,0){0.5}
\DGCdot<{0,L}
\DGCdot{-0.25,R}[r]{$_m$}
\DGCdot*.{0.25,R}[r]{$m$}
\end{DGCpicture}
\end{equation}
agrees with the one defined in $\mc{U}$ in the previous subsection.
\end{subequations}

The thick cups and caps give rise to right and left adjoints of the $1$-morphism 
$\mathcal{E}^{(r)} \1_m$ in $\dot{\mc{U}}$. Taking into account of degrees, they are given by
\begin{equation}
\label{Fadjointformula}
(\mathcal{E}^{(r)} \1_m)_R = \1_{m} \mathcal{F}^{(r)} \{-r(m+r)  \} , 
\quad \quad \quad \quad
(\mathcal{E}^{(r)} \1_m)_L = \1_{m} \mathcal{F}^{(r)} \{r(m+r)  \} .
\end{equation}

By construction, in the non-$p$-DG setting, $\dot{\mathcal{U}}$ is Morita equivalent to ${\mathcal{U}}$, and they both categorify quantum $\mathfrak{sl}_2$ at generic $q$ values. However, unlike the abelian case, the $p$-DG derived categories are drastically different. 

There is a natural embedding of $p$-DG $2$-categories
\[
\mathcal{J}: (\mc{U},\dif)\lra (\dot{\mathcal{U}},\dif),
\]
which is given by tensor product with $\dot{\mc{U}}$ regarded as a $(\dot{\mathcal{U}},{\mathcal{U}})$-bimodule with a compatible differential. This functor, not surprisingly, induces an equivalence of abelian categories of $p$-DG modules, and further descends to an equivalence of the corresponding homotopy categories. However, under localization, it is no longer an equivalence, but instead is a fully-faithful embedding of derived categories:
\[
\mathcal{J}: \mathcal{D}(\mathcal{U}) \lra \mathcal{D}(\dot{\mathcal{U}}).
\]
The embedding also plays an important role when categorifying the quantum Frobenius map, see \cite{QYFrob}.

\begin{thm}
The derived embedding $\mc{J}$ categorifies the embedding of $\dot{u}_{\mathbb{O}_p}$ into the BLM form $\dot{U}_{\mathbb{O}_p}$ for quantum $\mathfrak{sl}_2$.
\end{thm}
\begin{proof}
See \cite{EQ2}. 
\end{proof}

In the absence of $\dif$, the Sto\v{s}i\'{c} formula in \cite{KLMS} gives rise to a direct sum decomposition of $1$-morphism $\mc{E}^{(a)}\mc{F}^{(b)}\1_m$ for various $a,b\in \N$ and $m\in \Z$. In the $p$-DG setting, the direct sum decomposition is replaced by a fantastic filtration on the corresponding $1$-morphisms.

\begin{prop}\label{prop-higher-Serre}
In the $p$-DG category $(\dot{\mc{U}},\dif)$, the following statements hold.
\begin{enumerate}
\item[(i)] The $1$-morphisms in the collection
\[
\{\mc{E}^{(a)}\mc{E}^{(b)}\1_m,~\mc{F}^{(a)}\mc{F}^{(b)}\1_m|a,b\in \N,~m\in \Z\}
\]
are equipped with a fantastic filtration, whose subquotients are isomorphic to grading shifts of $\mc{E}^{(a+b)}\1_m$ and $\mc{F}^{(a+b)}\1_m$ respectively. Consequently, the defining relation \eqref{eqn-divided-power} for $\dot{U}_{\mathbb{O}_p}$ holds in the Grothendieck group of $\mc{D}(\dot{\mc{U}})$. 
\item[(ii)] The $1$-morphisms in the collection
\[
\{\mc{E}^{(a)}\mc{F}^{(b)}\1_m,~\mc{F}^{(a)}\mc{E}^{(b)}\1_m|a,b\in \N,~m\in \Z\}
\]
admit natural $\dif$-stable fantastic filtrations. The direct summands in the St\v{o}si\'{c} formula consitute the associated graded pieces of the filtrations. Consequently, in the Grothendieck group of $\mc{D}(\dot{\mc{U}})$, the divided power $E$-$F$ relations (equation \eqref{eqn-higher-Serre-1} and \eqref{eqn-higher-Serre-2}) hold.
\end{enumerate}

\end{prop}
\begin{proof}
See \cite[Section 3 and Section 6]{EQ2}.
\end{proof}

%%%%%%%%%%%%%%%%%%%%%%%%%%%%%%%%%%%%%%%%%%%%%%%%%%%%%%%%%%%%%%%%%%%%%%%%%%%%%%%%%%%%%%%%%%%%%%%%%%%%%%%%%%%%%%%%%%%%%%
%%%%%%%%%%%%%%%%%%%%%%%%%%%%%%%%%%%%%%%%%%%%%%%%%%%%%%%%%%%%%%%%%%%%%%%%%%%%%%%%%%%%%%%%%%%%%%%%%%%%%%%%%%%%%%%%%%%%%%

%\section{The nilHecke algebra}
%\label{sec-nilHecke}
\subsection{The nilHecke algebra}
\label{subsec-def-nilHecke}
Recall that $\Bbbk$ is a field of characteristic $p>0$.
Let $l $ and $n$ be integers such that $ l \geq n \geq 0$.
Define the nilHecke algebra $\nh_n$ to be the $\Bbbk$-algebra generated by $ y_1, \ldots, y_n$
and $\psi_1, \ldots, \psi_{n-1} $ with relations

\begin{equation}\label{eqn-NH-relation}
\begin{gathered}
y_iy_j=y_jy_i, \quad y_i\psi_j=\psi_jy_i~(i \neq j,j+1),\quad  y_i\psi_i-\psi_iy_{i+1}=1=\psi_iy_i-y_{i+1}\psi_i,\\
\psi_i^2=0,
\quad \quad
\psi_i\psi_j=\psi_j\psi_i~(|i-j|>1),\quad \quad \psi_{i}\psi_{i+1}\psi_i=\psi_{i+1}\psi_i\psi_{i+1}.
\end{gathered}
\end{equation}
The cyclotomic nilHecke algebra $\nh_n^l$ is the quotient of the nilHecke algebra $\nh_n$ by the \emph{cyclotomic relation}
\begin{equation}\label{eqn-NH-cyclotomicrelation}
\begin{gathered}
y_1^l=0.
\end{gathered}
\end{equation}
The (cyclotomic) nilHecke algebra is a graded algebra where the degree of $y_i$ is $2$ and the degree of $\psi_i$ is $-2$.

The relations above translate into planar diagrammatic relations for the upward pointing strands in the 2-category $\mc{U}$ (see Section \ref{sec-2-cat-U}), with the orientation labels dropped:
\begin{equation} \label{nicheckepicrelations}
\begin{gathered} 
\begin{DGCpicture}[scale=0.55]
\DGCstrand(1,0)(0,1)(1,2)
\DGCstrand(0,0)(1,1)(0,2)
\end{DGCpicture}
~= 0 \ ,  \quad \quad \quad \quad
\begin{DGCpicture}[scale=0.55]
\DGCstrand(0,0)(2,2)
\DGCstrand(1,0)(0,1)(1,2)
\DGCstrand(2,0)(0,2)
\end{DGCpicture}
~=~
\begin{DGCpicture}[scale=0.55]
\DGCstrand(0,0)(2,2)
\DGCstrand(1,0)(2,1)(1,2)
\DGCstrand(2,0)(0,2)
\end{DGCpicture}
\ , \\
\begin{DGCpicture}
\DGCstrand(0,0)(1,1)
\DGCdot{0.25}
\DGCstrand(1,0)(0,1)
\end{DGCpicture}
-
\begin{DGCpicture}
\DGCstrand(0,0)(1,1)
\DGCdot{0.75}
\DGCstrand(1,0)(0,1)
\end{DGCpicture}
~=~
\begin{DGCpicture}
\DGCstrand(0,0)(0,1)
\DGCstrand(1,0)(1,)
\end{DGCpicture}
~=~
\begin{DGCpicture}
\DGCstrand(1,0)(0,1)
\DGCdot{0.75}
\DGCstrand(0,0)(1,1)
\end{DGCpicture}
-
\begin{DGCpicture}
\DGCstrand(1,0)(0,1)
\DGCdot{0.25}
\DGCstrand(0,0)(1,1)
\end{DGCpicture} \ ,
\end{gathered}
\end{equation}
while the cyclotomic relation means that a black strand carrying $l$ consecutive dots and appearing to the left of the rest of a diagram annihilates the entire picture:
\begin{equation}\label{eqn-cyclotomic}
\begin{DGCpicture}
\DGCstrand(1,0)(1,1)
\DGCdot{0.5}[ur]{$_l$}
\DGCcoupon*(1.25,0.25)(1.75,0.75){$\cdots$}
\end{DGCpicture}
~=~0.
\end{equation}

There is a graded anti-automorphism on the (cyclotomic) nilHecke algebras $ * \colon \nh_n^l \rightarrow \nh_n^l$ defined by $\psi_i^*=\psi_i$
and $ y_i^*=y_i$. Diagrammatically, it is interpreted as flipping a diagram upside down about a horizontal axis.

Let us recall some special elements of (cyclotomic) nilHecke algebras that correspond to symmetric group elements. Fix a reduced decomposition of $w \in \mf{S}_n$,
$w=s_{i_1} \cdots s_{i_r}$.
This gives rise to an element $\psi_w= \psi_{i_1} \cdots \psi_{i_r} \in \nh_n^l$ 
which is independent of the expression for $w$ by the second group of relations in \eqref{eqn-NH-relation}. For instance, if $w_0\in \mf{S}_n$ is the usual longest element with respect to the usual Coxeter length function $\ell:\mf{S}_n\lra \N$, then the corresponding (cyclotomic) nilHecke element is unambiguously depicted as the following $n$-stranded element:
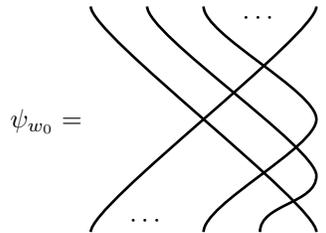
\[
\psi_{w_0}=
\begin{DGCpicture}[scale=0.75]
\DGCstrand(0,0)(4,4)
\DGCstrand(4,0)(0,4)
\DGCstrand(3,0)(4,1)(1,4)
\DGCstrand(2,0)(4,2)(2,4)
\DGCcoupon*(2.2,3.6)(3.8,4){$\cdots$}
\DGCcoupon*(0.2,0)(1.8,0.4){$\cdots$}
\end{DGCpicture} \ .
\]
The element is symmetric with respect to the $*$ anti-automorphism. 

%\subsection{Idempotents}
Let $w_0$ be the longest element in the symmetric group $\mf{S}_n$.  This gives rise to an indecomposable idempotent $e_n \in \nh_n^l$
\begin{subequations}
\begin{equation}\label{eqnidempotenten}
e_n := y_1^{n-1} \cdots y_n^0 \psi_{w_0}=\begin{DGCpicture}[scale=0.75]
\DGCstrand(0,0)(4,4)
\DGCstrand(4,0)(0,4)
\DGCdot{3.65}[ur]{$_{_{n-1}}$}
\DGCstrand(3,0)(4,1)(1,4)
\DGCdot{3.65}[ur]{$_{_{n-2}}$}
\DGCstrand(2,0)(4,2)(2,4)
\DGCdot{3.65}[ur]{$_{_{n-3}}$}
\DGCcoupon*(2.2,3.3)(3.8,3.7){$\cdots$}
\DGCcoupon*(0.2,0)(1.8,0.4){$\cdots$}
\end{DGCpicture} \ .
\end{equation}
In the notation of Section \ref{sec-thick-cal}, we will also depict the above idempotent as an unoriented thick strand
\begin{equation}\label{eqn-thick-idemp}
e_n=
\begin{DGCpicture}
\DGCstrand[Green](0,0)(0,1)[$^n$`{\ }]
\end{DGCpicture} \ .
\end{equation}
\end{subequations}
Throughout the rest of the paper, we will use the notation in \eqref{eqn-thick-idemp} to denote the element in \eqref{eqnidempotenten}.

One can show that
\[
\psi_{w_0}=
\begin{DGCpicture}
\DGCstrand[Green](0,0)(0,0.5)
\DGCdot.{0.25}[r]{$_n$}
\DGCstrand/u/(-1,-1)(0,0)/u/
\DGCstrand/u/(0.5,-1)(0,0)/u/
\DGCstrand/u/(1,-1)(0,0)/u/
\DGCstrand/u/(0,0.5)(-1,1.5)/u/
\DGCstrand/u/(0,0.5)(0.5,1.5)/u/
\DGCstrand/u/(0,0.5)(1,1.5)/u/
\DGCcoupon*(-0.8,-0.8)(0.3,-0.6){$\cdots$}
\DGCcoupon*(-0.8,1.1)(0.3,1.3){$\cdots$}
\end{DGCpicture}
\ ,
\]
which follows from the simple computation that
\begin{equation}
\psi_{w_0} e_n = \psi_{w_0}y_1^{n-1}\cdots y_{n-1}^1 y_n^0 \psi_{w_0}= \psi_{w_0}.
\end{equation}

The diagram is thick in the middle, indicating that it is a morphism factoring through the middle part, the latter representing the image of the idempotent. In this notation, the diagram is the concatenation its two halves:
\begin{equation}\label{eqn-half-diag}
\begin{DGCpicture}
\DGCstrand[Green](0,0)(0,0.5)
\DGCdot.{0.25}[r]{$_n$}
\DGCstrand/u/(0,0.5)(-1,1.5)/u/
\DGCstrand/u/(0,0.5)(0.5,1.5)/u/
\DGCstrand/u/(0,0.5)(1,1.5)/u/
\DGCcoupon*(-0.8,1.1)(0.3,1.3){$\cdots$}
\end{DGCpicture}
:=
\begin{DGCpicture}
\DGCstrand(0,0)(0,1.5)
\DGCstrand(0.5,0)(0.5,1.5)
\DGCstrand(1.5,0)(1.5,1.5)
\DGCcoupon(-0.25,0.5)(1.75,1){$\psi_{w_0}$}
\DGCcoupon*(0.6,0.1)(1.4,0.4){$\cdots$}
\DGCcoupon*(0.6,1.1)(1.4,1.4){$\cdots$}
\end{DGCpicture} \ ,
\quad \quad \quad \quad
\begin{DGCpicture}
\DGCstrand[Green](0,0)(0,0.5)
\DGCdot.{0.25}[r]{$_n$}
\DGCstrand/u/(-1,-1)(0,0)/u/
\DGCstrand/u/(0.5,-1)(0,0)/u/
\DGCstrand/u/(1,-1)(0,0)/u/
\DGCcoupon*(-0.8,-0.8)(0.3,-0.6){$\cdots$}
\end{DGCpicture}
:=
\begin{DGCpicture}
\DGCstrand(0,0)(0,1.5)
\DGCstrand(0.5,0)(0.5,1.5)
\DGCstrand(1.5,0)(1.5,1.5)
\DGCcoupon(-0.25,0.5)(1.75,1){$e_n$}
\DGCcoupon*(0.6,0.1)(1.4,0.4){$\cdots$}
\DGCcoupon*(0.6,1.1)(1.4,1.4){$\cdots$}
\end{DGCpicture} \ ,
\end{equation}

Analogously, there are thick \emph{splitters} and \emph{mergers}, for any pair $(a,b)\in \N^2$, that are built out of the idempotents:
\begin{equation}
\begin{DGCpicture}[scale=0.85]
\DGCPLstrand[Green](1,0)(1,1)[$^{a+b}$]
\DGCPLstrand[Green](1,1)(0,2)[`$_a$]
\DGCPLstrand[Green](1,1)(2,2)[`$_b$]
\end{DGCpicture}
:=
\begin{DGCpicture}[scale=0.85]
\DGCstrand(0,0)(2,2)(2,2.3)[$^1$]
\DGCstrand(0.5,0)(2.5,2)(2.5,2.3)[$^2$]
\DGCstrand(1.5,0)(3.5,2)(3.5,2.3)[$^b$]
\DGCstrand(2.5,0)(0,2)(0,2.3)[$^{b+1}$]
\DGCstrand(3.5,0)(1,2)(1,2.3)[$^{b+a}$]
\DGCcoupon*(0.8,0.1)(1.4,0.3){$\cdots$}
\DGCcoupon*(2.5,1.55)(3.2,1.7){$\cdots$}
\DGCcoupon*(2.6,0.1)(3.2,0.3){$\cdots$}
\DGCcoupon*(0.4,1.55)(1,1.7){$\cdots$}
\DGCcoupon(1.8,1.8)(3.7,2.2){$e_b$}
\DGCcoupon(-0.2,1.8)(1.2,2.2){$e_{a}$}
\end{DGCpicture} \ ,
\quad \quad \quad
\begin{DGCpicture}[scale=0.85]
\DGCPLstrand[Green](0,0)(1,1)[$^a$]
\DGCPLstrand[Green](2,0)(1,1)[$^b$]
\DGCPLstrand[Green](1,1)(1,2)[`$_{a+b}$]
\end{DGCpicture} 
:=
\begin{DGCpicture}[scale=0.85]
\DGCstrand(0,0)(0,2.3)[$^1$]
\DGCstrand(0.5,0)(0.5,2.3)[$^2$]
\DGCstrand(1.5,0)(1.5,2.3)[$^b$]
\DGCstrand(2.5,0)(2.5,2.3)[$^{b+1}$]
\DGCstrand(3.5,0)(3.5,2.3)[$^{b+a}$]
\DGCcoupon*(0.7,0.1)(1.3,0.3){$\cdots$}
\DGCcoupon*(2.7,1.55)(3.3,1.7){$\cdots$}
\DGCcoupon*(2.7,0.1)(3.3,0.3){$\cdots$}
\DGCcoupon*(0.7,1.55)(1.3,1.7){$\cdots$}
\DGCcoupon(-0.2,0.9)(3.7,1.4){$e_{a+b}$}
\end{DGCpicture} \ .
\end{equation} 

Let ${\bf i}=(i_1, \ldots, i_r)$ be a tuple of natural numbers such that $i_1+ \cdots + i_r=n$.
Then we set the idempotent
\begin{equation}\label{eqn-sequence-idempotent}
e_{{\bf i}} = e_{i_1} \otimes \cdots \otimes e_{i_r}.
\end{equation}
This corresponds to putting the diagrams for $e_{i_1}$,..., $e_{i_r}$ side by side next to one another:
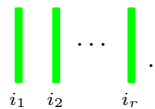
\[
\begin{DGCpicture}
\DGCstrand[Green](0,0)(0,1)[$^{i_1}$]
\DGCstrand[Green](0.5,0)(0.5,1)[$^{i_2}$]
\DGCstrand[Green](1.5,0)(1.5,1)[$^{i_r}$]
\DGCcoupon*(0.6,0.1)(1.4,0.9){$\cdots$}
\end{DGCpicture} \ .
\]
The thick generators obey all the relation satisfied by the upward oriented ones in Section \ref{sec-thick-cal}, such as equations \eqref{eqn-thick-relation-splitter-asso}, \eqref{eqn-thick-relation-sliding} and \eqref{eqn-thick-pairing}.

Rotating a diagram $180^\circ$ turns $e_n$ into a quasi-idempotent
\begin{equation}\label{eqnenprime}
e_n' := \psi_{w_0} \cdot y_1^0 \cdots y_n^{n-1}.
\end{equation}
To obtain a genuine idempotent one needs to correct the element with the sign $(-1)^{n(n-1)/2}$.
%
%Let $a+b=n$.
%For $\mu \in P(a,b) $ we define a minimal idempotent $ e^{\mu}_{(a,b)} $ of $\nh_n^l$ as follows.
%First set
%\begin{equation}
%\psi_{a,b} = (\psi_b \cdots \psi_{a+b-1}) \cdots (\psi_2 \cdots \psi_{a+1})(\psi_1 \cdots \psi_{a-1}),
%\end{equation}
%which is diagrammatically depicted as
%\begin{equation}
%\psi_{a,b} =
%\begin{DGCpicture}[scale=0.85]
%\DGCstrand(0,0)(2,2)[$^1$]
%\DGCstrand(0.5,0)(2.5,2)[$^2$]
%\DGCstrand(1.5,0)(3.5,2)[$^a$]
%\DGCstrand(2.5,0)(0,2)[$^{a+1}$]
%\DGCstrand(3.5,0)(1,2)[$^{a+b}$]
%\DGCcoupon*(0.8,0.1)(1.4,0.3){$\cdots$}
%\DGCcoupon*(2.5,1.55)(3.2,1.7){$\cdots$}
%\DGCcoupon*(2.6,0.1)(3.2,0.3){$\cdots$}
%\DGCcoupon*(0.4,1.55)(1,1.7){$\cdots$}
%\end{DGCpicture} \ .
%\end{equation}
%Then we define
%\begin{equation}\label{eqn-general-idempotent}
%e^{\mu}_{(a,b)} = (-1)^{|\hat{\mu}|}(\pi_{\mu}(y_1, \cdots, y_a)) (e_{(a,b)}) (\psi_{a,b})  (e_{a+b}) (\pi_{\hat{\mu}}(y_{a+1},\cdots,y_{a+b})).
%\end{equation}
%The idempotent is diagrammatically depicted as (c.f.~equation \eqref{eq-identitydecomp} and also \cite{KLMS})
%\begin{equation}
%(-1)^{|\hat{\mu}|}
%\begin{DGCpicture}[scale=0.9]
%\DGCPLstrand[Green](0,0)(1,1.15)[$^a$]
%\DGCPLstrand[Green](2,0)(1,1.15)[$^b$]
%\DGCPLstrand[Green](1,1.15)(1,1.85)
%\DGCPLstrand[Green](1,1.85)(0,3)[`$_a$]
%\DGCPLstrand[Green](1,1.85)(2,3)[`$_b$]
%\DGCcoupon(1.15,0.35)(1.85,0.85){$_{\pi_{\hat{\mu}}}$}
%\DGCcoupon(0.15,2.15)(0.85,2.65){$_{\pi_\mu}$}
%\end{DGCpicture} \ ,
%\end{equation}
%with it understood that we may let $\nh_{n}^l$ elements be multiplied from above and below of the element.

%\subsection{$p$-DG structure}
The cyclotomic nilHecke algebra $\nh^l_n$ has a $p$-DG structure inherited from that of $\nh_n$:
\begin{subequations}
\begin{equation}
\partial(y_i)=y_i^2 \hspace{1in}
\partial(\psi_i)=-y_i \psi_i - \psi_i y_{i+1},
\end{equation}
which is diagrammatically expressed as
\begin{equation}\label{eqndifactionnilHeckegen}
\dif\left(~
\begin{DGCpicture}
\DGCstrand(0,0)(0,1)
\DGCdot{0.5}
\end{DGCpicture}
~\right)=
\begin{DGCpicture}
\DGCstrand(0,0)(0,1)
\DGCdot{0.5}[r]{$_2$}
\end{DGCpicture} \ ,
\quad \quad \quad \quad
\dif\left(~
\begin{DGCpicture}
\DGCstrand(1,0)(0,1)
\DGCstrand(0,0)(1,1)
\end{DGCpicture}
~\right)
=-
\begin{DGCpicture}
\DGCstrand(1,0)(0,1)
\DGCdot{0.75}
\DGCstrand(0,0)(1,1)
\end{DGCpicture}
-
\begin{DGCpicture}
\DGCstrand(1,0)(0,1)
\DGCdot{0.25}
\DGCstrand(0,0)(1,1)
\end{DGCpicture} \ .
\end{equation}
\end{subequations}

We will need later the easily verified fact below.
\begin{equation}
\label{enprime}
\partial(e_n)=-\sum_{i=1}^n (i-1)  e_n y_i, \quad \quad \quad \partial(e_n^\prime)=-\sum_{i=1}^n (n-i) y_i e_n^\prime.
\end{equation}

\begin{prop}
\label{pdgidempotprop}
Assume $ {\bf i}=(i_1, \ldots, i_r)$ with $i_1+\cdots+i_r=n$.  Then the projective $\nh_n^l$-module $e_{{\bf i}}\nh_n^l$ is a right $p$-DG module over $\nh_n^l$.
\end{prop}
\begin{proof}
See \cite[Proposition 6.1]{KQS}.
\end{proof}

\begin{prop}
For $0 \leq n \leq p-1$,
the $p$-DG module $e_{n}\nh_n^l$ is compact and cofibrant.
\end{prop}
\begin{proof}
This is \cite[Proposition 6.2]{KQS}.
\end{proof}

\subsection{A categorification of simples}\label{subsec-cat-simples}
We first review a categorification of the irreducible representation $V_l$ of quantum $\mathfrak{sl}_2$ for a generic value of $q$ using cyclotomic nilHecke algebras due to Kang and Kashiwara \cite{KK}.  
Then we enhance it to a categorification of cyclically generated modules over the small quantum group $u_{\mathbb{O}_p}$ and over the BLM quantum group $\dot{U}_{\mathbb{O}_p}$.

For any $a \in \Z_{\geq 0}$ there is an embedding $\nh_n^l \hookrightarrow \nh_{n+a}^l$ given by 
\begin{equation*}
y_i \mapsto y_i \quad (i=1,\ldots,n), \quad  \quad \quad  \quad \psi_i \mapsto \psi_i \quad (i=1, \ldots, n-1).
\end{equation*}
We use this embedding to produce functors between categories of nilHecke modules.

\begin{defn}\label{def-E-and-F}
There is an induction functor
\begin{equation*}
\mf{F}^{(a)} \colon (\nh_n^l,\partial) \dmod \longrightarrow (\nh^l_{n+a},\partial) \dmod
\end{equation*}
given by tensoring with the $p$-DG bimodule over $(\nh_n^l,\nh_{n+a}^l)$
\begin{equation*}
e_{(1^n,a)} \nh_{n+a}^l, \quad \quad \partial(e_{(1^n,a)}):=-\sum_{i=1}^{a}(1-i)e_{(1^n,a)}y_{n+i}.
\end{equation*}
Notice that the differential action on the bimodule generator $e_{(1^n,a)}$ arises from the differential action on the idempotent $e_a$ (Proposition \ref{pdgidempotprop}).

There is a restriction functor
\begin{equation*}
\mf{E}^{(a)} \colon (\nh_{n+a}^l, \partial) \dmod \longrightarrow (\nh_n^l,\partial) \dmod
\end{equation*}
given by tensoring with the $p$-DG bimodule over $(\nh_{n+a}^l,\nh_n^l)$
\begin{equation*}
\nh_{n+a}^l e_{(1^n,a)}^\star, \quad \quad \partial( e_{(1^n,a)}^\star):=\sum_{i=1}^a (2n+i-l)y_{n+i}e_{(1^n,a)}^\star.
\end{equation*}
Here the differential action on the bimodule generator $e_{(1^n,a)}^\star$ is twisted from the natural $\dif$-action on $e_a^\prime$ (see equation \eqref{enprime}) by the symmetric function $(a-l+2n)(y_{n+1}+\cdots+y_{n+a})$.

For simplicity, set $\mf{F}=\mf{F}^{(1)}$ and $\mf{E}=\mf{E}^{(1)}$.
\end{defn}

\begin{rem}
In the absence of the $p$-DG structure, the functors $\mf{E}$ and $\mf{F}$ give rise to an action of Lauda's $2$-category $\mc{U}$ on 
$ \oplus_{n = 0}^l \nh_n^l \dmod$. See, for instance, \cite{Rou2} and \cite{KK}.  
\end{rem}

There is an adjunction map of the functors $\cap \colon \mf{F} \mf{E}\Rightarrow \Id$ given by the following bimodule homomorphism
\begin{equation}\label{eqn-rep-cap}
   \nh_n^le_{(1^{n-1},1)}^\star \otimes_{\nh^l_{n-1}}e_{(1^{n-1},1)} \nh_n^l \longrightarrow  \nh_n^l ,\quad \quad \quad \alpha \otimes \beta \mapsto \alpha \beta,
\end{equation}
and similarly an adjunction map $\cup:\Id\Rightarrow \mf{E} \mf{F}$ arising from
\begin{equation}\label{eqn-rep-cup}
  \nh_n^l \longrightarrow  \nh_{n+1}^le_{(1^n,1)}^\star, \quad \quad \quad \alpha \mapsto \alpha e_{(1^n,1)}^\star.
\end{equation}
There is a ``dot'' natural transformation 
\begin{equation}\label{eqn-rep-dot}
Y \colon \mf{E}\Rightarrow \mf{E}, \quad \quad \quad \left(\nh_n^l \longrightarrow \nh_n^l \quad  \quad \alpha \mapsto \alpha y_n\right).
\end{equation}
There is also a ``crossing''
\begin{equation}\label{eqn-rep-cross}
\Psi \colon \mf{E} \mf{E}  \Rightarrow \mf{E} \mf{E}, \quad \quad \quad \left(\nh_n^l \lra \nh_n^l \quad \quad \alpha \mapsto \alpha \psi_{n-1}\right).
\end{equation}

\begin{thm}
\label{catofVl}
For any $l\in \N$, there is a $2$-representation of the $p$-DG $2$-category $(\mathcal{U},\dif)$ on 
$ \oplus_{n=0}^l (\nh_n^l,\dif)\dmod $ defined as follows.

On $0$-morphisms we have
\begin{equation*}
{m} \mapsto 
\begin{cases}
(\nh_n^l,\dif)\dmod & \textrm{ if } m = l-2n \\
0 & \textrm{ otherwise }.
\end{cases}
\end{equation*}

On $1$-morphisms we have
\begin{equation*}
\mathcal{E} \1_{m} \mapsto
\begin{cases}
\mf{E} & \textrm{ if } m = l-2n \\
0 & \textrm{ otherwise }
\end{cases}
\quad \quad \quad
\mathcal{F} \1_{m} \mapsto
\begin{cases}
\mf{F} & \textrm{ if } m = l-2n \\
0 & \textrm{ otherwise }.
\end{cases}
\end{equation*}

On $2$-morphisms we have
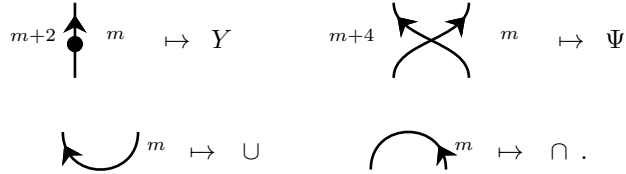
\begin{equation*}
  \begin{DGCpicture}
  \DGCstrand(0,0)(0,1)
  \DGCdot*>{0.75}
  \DGCdot{0.45}
  \DGCcoupon*(0.1,0.25)(1,0.75){$^m$}
  \DGCcoupon*(-1,0.25)(-0.1,0.75){$^{m+2}$}
  \DGCcoupon*(-0.25,1)(0.25,1.15){}
  \DGCcoupon*(-0.25,-0.15)(0.25,0){}
  \end{DGCpicture}
  ~\mapsto~
  ~Y~
  \quad \quad \quad 
  \begin{DGCpicture}
  \DGCstrand(0,0)(1,1)
  \DGCdot*>{0.75}
  \DGCstrand(1,0)(0,1)
  \DGCdot*>{0.75}
  \DGCcoupon*(1.1,0.25)(2,0.75){$^m$}
  \DGCcoupon*(-1,0.25)(-0.1,0.75){$^{m+4}$}
  \DGCcoupon*(-0.25,1)(0.25,1.15){}
  \DGCcoupon*(-0.25,-0.15)(0.25,0){}
  \end{DGCpicture} 
 ~ \mapsto~
 ~\Psi~
  \end{equation*}

\begin{equation*}
\begin{DGCpicture}
  \DGCstrand/d/(0,0)(1,0)
  \DGCdot*<{-0.25,2}
  \DGCcoupon*(1,-0.5)(1.5,0){$^m$}
  \DGCcoupon*(-0.25,0)(1.25,0.15){}
  \DGCcoupon*(-0.25,-0.65)(1.25,-0.5){}
  \end{DGCpicture}
  ~\mapsto~
  ~\cup~
  \quad \quad \quad
  \begin{DGCpicture}
  \DGCstrand(0,0)(1,0)/d/
  \DGCdot*<{0.25,1}
  \DGCcoupon*(1,0)(1.5,0.5){$^m$}
  \DGCcoupon*(-0.25,0.5)(1.25,0.65){}
  \DGCcoupon*(-0.25,-0.15)(1.25,0){}
  \end{DGCpicture}
  ~\mapsto~
  ~\cap~.
\end{equation*}

Furthermore, for $0 \leq l \leq p-1$, there is an isomorphism
$K_0(\oplus_{n=0}^l \mathcal{D}^c(\nh_n^l)) \cong V_l $
as (irreducible) modules over $\dot{u}_{\mathbb{O}_p}$.
\end{thm}
\begin{proof}
This is \cite[Theorem 6.5]{KQS}.
\end{proof}

\begin{rem}
\begin{enumerate}
\item[(1)]The biadjointness of the functors $\mf{E}$ and $\mf{F}$ follows from \cite{Rou2, KK} in conjunction with \cite{Brundan2KM}.
Kashiwara \cite{KaBi} showed directly that the functors $\mf{E}$ and $\mf{F}$ are biadjoint.
\item[(2)] The above categorical small quantum group action naturally extends to $(\dot{\mc{U}},\dif)$. See the diagrammatic discussion below.
\end{enumerate}
\end{rem}

Let us give a diagrammatic interpretation of the differential actions in Definition \ref{def-E-and-F}. 

Consider the \emph{cyclotomic quotient category} $\mathcal{V}_l$ to be the quotient category of $\1_l\dot{\mc{U}}$ by morphisms in the two-sided ideal which is right monoidally generated by
\begin{itemize}
\item[(1)] Any morphism that contains the following subdiagram on the far left:
$$\begin{DGCpicture}
\DGCstrand(0,0)(0,1)
\DGCdot*<{0.5}
\DGCcoupon*(-0.4,0.4)(-0.7,0.8){$l$}
\end{DGCpicture}~.
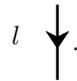
$$
\item[(2)] All positive degree bubbles on the far left region labeled $l$.
\end{itemize}
Here by ``two-sided'' we mean concatenating diagrams vertically from top and bottom to those in the relations, while by ``right monoidally generated'' we mean composing pictures from $\dot{\mc{U}}$ to the right of those generators. Schematically we depict elements in the ideal as follows.

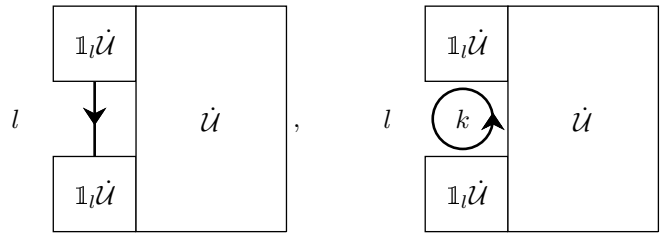
\[
\begin{DGCpicture}
\DGCstrand(-0.55,1)(-0.55,2)
\DGCdot*<{1.5}
\DGCcoupon(0,0)(-1.1,1){$\1_l\dot{\mc{U}}$}
\DGCcoupon(0,2)(-1.1,3){$\1_l\dot{\mc{U}}$}
\DGCcoupon(2,0)(0,3){$\dot{\mc{U}}$}
\DGCcoupon*(-1.2,1)(-2,2){$l$}
\end{DGCpicture}~, \quad \quad
\begin{DGCpicture}
\DGCbubble(-0.6,1.5){0.4}
\DGCdot*>{1.5,R}
\DGCcoupon*(-0.4,1.3)(-0.8,1.7){$k$}
\DGCcoupon(0,0)(-1.1,1){$\1_l\dot{\mc{U}}$}
\DGCcoupon(0,2)(-1.1,3){$\1_l\dot{\mc{U}}$}
\DGCcoupon(2,0)(0,3){$\dot{\mc{U}}$}
\DGCcoupon*(-1.2,1)(-2,2){$l$}
\end{DGCpicture}~.
\]

One implication of these relations is that
\begin{equation}\label{cyclotomiccurl}
0 = \curl{L}{U}{$l$}{no}{$0$} = \sum_{a+b=l}
~\bigccwbubble{$a$}{$l$}\oneline{$b$}{no} .
\end{equation}
Moreover, every bubble with $a>0$ is also in the ideal, so that
\begin{equation}\label{eqn-cyclotomic-rel}
\begin{DGCpicture}
\DGCstrand(0,0)(0,1)
\DGCdot{.5}[r]{$^l$}
\DGCdot*>{1}
\DGCcoupon*(-0.2,0.4)(-0.8,0.8){$l$}
\end{DGCpicture}~=0.
\end{equation}
Therefore, it follows that the cyclotomic nilHecke algebra $\nh_n^l$ maps onto $\END_{\mc{V}_l}(\1_l \mc{E}^n\1_{l-2n})$. Further, Lauda  \cite[Section 7]{Lau1} has proven that this map is an isomorphism:
\begin{equation}\label{eqn-iso-to-nilHecke}
\nh_n^l\cong \END_{\mc{V}_l}(\1_l\mc{E}^n\1_{l-2n})
=\left\{
\begin{DGCpicture}
\DGCstrand(0,0)(0,1.5)
\DGCdot*>{1.5}
\DGCstrand(0.5,0)(0.5,1.5)
\DGCdot*>{1.5}
\DGCstrand(1.5,0)(1.5,1.5)
\DGCdot*>{1.5}
\DGCcoupon(-0.1,0.4)(1.6,1.1){$\nh_n^l$}
\DGCcoupon*(-0.7,0.4)(-0.1,1.1){$_l$}
\DGCcoupon*(1.6,0.4)(2.4,1.1){$_{l-2n}$}
\DGCcoupon*(0.6,0.1)(1.4,0.3){$\cdots$}
\DGCcoupon*(0.6,1.2)(1.4,1.4){$\cdots$}
\end{DGCpicture}
\right\} \ ,
\end{equation}
and the center of $\nh_n^l$ is isomorphic to
\begin{equation}\label{eqn-iso-to-nilHecke-center}
Z(\nh_n^l)\cong \END_{\mc{V}_l}(\1_l\mc{E}^{(n)}\1_{l-2n})
=\left\{
\begin{DGCpicture}
\DGCstrand[Green](0,0)(0,1.5)[$^n$]
\DGCdot*>{1.5}
\DGCcoupon*(-1,0.4)(-0.5,1.1){$_l$}
\DGCcoupon*(0.5,0.5)(1.5,1.1){$_{l-2n}$}
\DGCcoupon(-0.4,0.4)(0.4,1.1){$\pi_{\alpha}$}
\end{DGCpicture}
 \Bigg| \alpha\in P(n,l-n)\right\}\ .
\end{equation}
Since the ideal used in the definition of $\mc{V}_l$ is clearly $\dif$-stable, $\mc{V}_l$ carries a natural quotient $p$-differential which is still denoted $\dif$. The isomorphism \eqref{eqn-iso-to-nilHecke} is then an isomorphism of $p$-DG algebras by Definition \ref{def-special-dif}. In this way, $(\mc{V}_l,\dif)$ is a $p$-DG module-category over $(\dot{\mc{U}},\dif)$. 

The assignment on 0-morphisms in Theorem \ref{catofVl} can now be seen as reading off the weights on the far right for the diagrams in equation \eqref{eqn-iso-to-nilHecke}.

To see the necessity of twisiting the differential on the restriction functor $\mc{E}^{(a)}$, it is readily seen that the restriction bimodule $\nh_n^le_{(1^{n-a},a)}^\star$, regarded as a functor
\[
\mc{E}^{(a)}=(\mbox{-})\otimes_{\nh_n^l}\left(\nh_n^le_{(1^{n-a},a)}^\star\right): (\nh_n^l,\dif)\lra (\nh_{n-a}^l,\dif),
\]
may be identified with the space of diagrams
\begin{equation}\label{eqnrestrictionfunctor}
\nh_n^le_{(1^{n-a},a)}^\star\cong 
\left\{
\begin{DGCpicture}
\DGCstrand(0,-0.2)(0,1.5)
\DGCdot*>{1.5}
\DGCstrand(0.5,-0.2)(0.5,1.5)
\DGCdot*>{1.5}
\DGCstrand/ul/(1.75,0.3)(1.5,1)/u/(1.5,1.5)
\DGCdot*>{1.5}
\DGCstrand/ur/(1.75,0.3)(2,1)/u/(2,1.5)
\DGCdot*>{1.5}
\DGCstrand[Green]/d/(1.75,0.3)(2.5,0.3)(2.5,1.5)/u/
\DGCcoupon(-0.1,0.5)(2.1,1.1){$\nh_n^l$}
\DGCcoupon*(-0.7,0.4)(-0.1,1.1){$_l$}
\DGCcoupon*(2.8,0.4)(4.1,1.1){$_{l-2n+2a}$}
\DGCcoupon*(0.6,0.1)(1.4,0.3){$\cdots$}
\DGCcoupon*(0.6,1.2)(1.4,1.5){$\cdots$}
\end{DGCpicture}
\right\}
\end{equation}
whose diagrammatic generator satisfies the differential formula
\begin{equation}
\dif\left(~
\begin{DGCpicture}
\DGCstrand(0,-0.2)(0,1.5)
\DGCdot*>{1.5}
\DGCstrand(0.5,-0.2)(0.5,1.5)
\DGCdot*>{1.5}
\DGCstrand/ul/(1.75,0.3)(1.25,1)/u/(1.25,1.5)
\DGCdot*>{1.5}
\DGCstrand(1.75,0.3)(1.75,1.5)
\DGCdot*>{1.5}
\DGCstrand/ur/(1.75,0.3)(2.25,1)/u/(2.25,1.5)
\DGCdot*>{1.5}
\DGCstrand[Green]/d/(1.75,0.3)(2.75,0.3)(2.75,1.5)/u/
\DGCcoupon*(-0.7,0.4)(-0.1,1.1){$_l$}
\DGCcoupon*(2.8,0.4)(4.1,1.1){$_{l-2n+2a}$}
\DGCcoupon*(0.5,0.7)(1.25,1){$\cdots$}
\DGCcoupon*(1.25,0.7)(1.75,1){$\cdots$}
\DGCcoupon*(1.75,0.7)(2.25,1){$\cdots$}
\end{DGCpicture}
~\right)
=\sum_{i=1}^a
(2n+i-l)
\begin{DGCpicture}
\DGCstrand(0,-0.2)(0,1.5)
\DGCdot*>{1.5}
\DGCstrand(0.5,-0.2)(0.5,1.5)
\DGCdot*>{1.5}
\DGCstrand/ul/(1.75,0.3)(1.25,1)/u/(1.25,1.5)
\DGCdot*>{1.5}
\DGCstrand(1.75,0.3)(1.75,1.5)[`$_{_{n-a+i}}$]
\DGCdot*>{1.5}
\DGCdot{1.25}
\DGCstrand/ur/(1.75,0.3)(2.25,1)/u/(2.25,1.5)
\DGCdot*>{1.5}
\DGCstrand[Green]/d/(1.75,0.3)(2.75,0.3)(2.75,1.5)/u/
\DGCcoupon*(-0.7,0.4)(-0.1,1.1){$_l$}
\DGCcoupon*(2.8,0.4)(4.1,1.1){$_{l-2n+2a}$}
\DGCcoupon*(0.5,0.7)(1.25,1){$\cdots$}
\DGCcoupon*(1.25,0.7)(1.75,1){$\cdots$}
\DGCcoupon*(1.75,0.7)(2.25,1){$\cdots$}
\end{DGCpicture} \ .
\end{equation}
The differential action is induced from the differential action on a thick cup given in equation \eqref{eqn-dif-thick-cup-cap-1}.

On the other hand, the induction functor $\mf{F}^{(a)}$ has an obvious diagrammatic interpretation by identifying the bimodule
$_{\nh_{n-a}^l}\left({e_{(1^{n-a},a)}\nh_n^l}\right)_{\nh_n^l}$ as the space
\begin{equation}
\nh_n^le_{(1^{n-a},a)}^\star\cong 
\left\{
\begin{DGCpicture}
\DGCstrand(0,-1.5)(0,0.2)
\DGCdot*>{0.2}
\DGCstrand(0.5,-1.5)(0.5,0.2)
\DGCdot*>{0.2}
\DGCstrand(1.5,-1.5)(1.5,-1)(1.75,-0.3)/ur/
\DGCdot*>{-1.25}
\DGCstrand(2,-1.5)(2,-1)(1.75,-0.3)/ul/
\DGCdot*>{-1.25}
\DGCstrand[Green](1.75,-0.3)(1.75,0.2)
\DGCdot*>{0.2}
\DGCcoupon(-0.1,-0.5)(2.1,-1){$\nh_n^l$}
\DGCcoupon*(-0.7,-0.4)(-0.1,-1.1){$_l$}
\DGCcoupon*(2.1,-0.4)(3.2,-1.1){$_{l-2n}$}
\DGCcoupon*(0.6,-0.1)(1.4,-0.3){$\cdots$}
\DGCcoupon*(0.6,-1.2)(1.4,-1.4){$\cdots$}
\end{DGCpicture}
\right\} \ ,
\end{equation}
 but the left $\nh_{n-a}^l$ acts only through the left-most $n-a$ strands on the top.

\section{Some cyclic modules}
\label{sec-cyclic-mod}
In this section, we will study a collection of combinatorially defined nilHecke modules introduced in the work of Hu-Mathas \cite{HuMathasKLR, HuMathas} under the action of $p$-differentials. These modules will be utilized to define a $p$-DG quiver Schur algebra.

\subsection{Cellular combinatorics}
\begin{defn}\label{def-multipartition}
A \emph{$\nh^l_n$-multipartition} (or simply a \emph{partition} for short) is an $l$-tuple $\mu=(\mu^{1},\ldots,\mu^{l})$ 
such that $ \mu^i \in \{0, 1\} $ and $ \mu^1 + \cdots + \mu^l=n$. We may also think of a partition as a sequence of empty slots and boxes.

The set of $\nh_n^l$-partitions will be denoted by $\mathcal{P}_n^l$. 
\end{defn}

\begin{example}\label{eg-23partitions}
As an example, the following partitions constitute the full list of all $\nh_{2}^3$-multipartitions:
\[
\left(~\yng(1)~,~\yng(1)~,~\emptyset~\right),\quad \quad
\left(~\yng(1)~,~\emptyset~,~\yng(1)~\right),\quad \quad
\left(~\emptyset~,~\yng(1)~,~\yng(1)~\right).
\]
They correspond to the numerical notation of $(1,1,0)$, $(1,0,1)$ and $(0,1,1)$ respectively.
\end{example}

\begin{defn}\label{def-order-on-nh-partition}
For two elements $\lambda, \mu \in \mathcal{P}_n^l$, declare $\lambda \geq \mu$ if
\begin{equation*}
\lambda^1+\cdots+\lambda^k \geq \mu^1+\cdots+\mu^k 
\end{equation*}
for $k=1, \ldots, l$. We will say $\lambda> \mu$ if $\lambda \geq \mu$ and $\lambda \neq \mu$. This defines a partial order on the set of partitions $\mc{P}_n^l$ called the \emph{dominance order}.
\end{defn}

Combinatorially, when regarded as partitions, $\lambda \geq \mu$ if $\mu$ can be obtained from $\lambda$ by a sequence of moves which exchange a box and an empty space immediately right to the box. It is easily seen that there is always a unique partition that is minimal with respect to the dominance order, namely the one with all boxes on the right. We will denote the unique minimal element under this partial ordering by $\lambda_0$.

 The following is an example of incomparable partitions:
 \[
\left(~\emptyset~, ~\yng(1)~,~\yng(1)~,~\emptyset~\right),\quad \quad
\left(~\yng(1)~,~\emptyset~,~\emptyset~,~\yng(1)~\right),
\]
which are both greater than
\[
\left(~\emptyset~, ~\yng(1)~,~\emptyset~,~\yng(1)~\right).
\]

We next introduce the notion of tableaux and a partial order on them as well.

\begin{defn}\label{def-tableau}
Given a partition $\mu \in \mc{P}_n^l$, suppose $ \mu^{j_1}=\cdots=\mu^{j_n}=1$ and $ j_1 < \cdots < j_n$.
A \emph{tableau of shape $\mu$} is a bijection 
\begin{equation*}
\mf{t} \colon \{j_1, \ldots, j_n \} \longrightarrow \{1, \ldots, n \}.
\end{equation*}
Denote the set of $\mu$-tableaux by $\mathrm{Tab}(\mu)$, and write for any $\mf{t}\in \mathrm{Tab}(\mu)$ that $\mathrm{shape}(\mf{t})=\mu$.
\end{defn}

Given a tableau, we may think of it as a filling of its underlying partition labeled by the set of natural numbers $\{1,\dots, n\}$. This, in turn, gives us a sequence of subtableaux in order of which the tableaux is built up by adding at the $k$th step the box labeled by $k$ ($1\leq k \leq n$).

\begin{example}
For the partition $\mu:=\left(~\yng(1)~,~\yng(1)~,~\emptyset~\right)$, we have its set of tableaux equal to
\[
\mathrm{Tab}(\mu)=\left\{
\left(~\young(1)~,~\young(2)~,~\emptyset~\right),~\left(~\young(2)~,~\young(1)~,~\emptyset\right)
\right\}.
\]
In these examples, the corresponding tableaux can be regarded as built up in two steps:
\[
 \left(~\young(1)~,~\emptyset~,~\emptyset~\right)\rightarrow
\left(~\young(1)~,~\young(2)~,~\emptyset~\right) , 
\quad \quad
\left(~\emptyset~,~\young(1)~,~\emptyset~\right)\rightarrow
\left(~\young(2)~,~\young(1)~,~\emptyset~\right) .
\]
Another example of the process can be read from
\[
\left(~\emptyset~, ~\young(1)~,~\emptyset~,~\emptyset~\right)\rightarrow \left(~\emptyset~, ~\young(1)~,~\young(2)~,~\emptyset~\right)\rightarrow \left(~\young(3)~, ~\young(1)~,~\young(2)~,~\emptyset~\right).
\]
\end{example}

\begin{defn}\label{def-standard-tableaux}
\begin{enumerate}
\item[(1)] For a partition $\mu$ let $\mf{t}^{\mu}$ be the tableau given by
\begin{equation*}
\mf{t}^{\mu}(j_k)=k \hspace{.5in} k=1,\ldots, n.
\end{equation*}
We will refer to the tableau as the \emph{standard tableau} of shape $\mu$. 
\item[(2)] Any tableau $\mf{t}\in \mathrm{Tab}(\mu)$ can be obtained from $\mf{t}^{\mu}$ by a unique permutation $w_\mf{t}\in \mf{S}_n$. We will call $w_\mf{t}$ the \emph{permutation determined by $\mf{t}$}.
\end{enumerate}
\end{defn}

\begin{rem}[On notation]\label{ntntableauxsymmetricgroup}
The set $\Tab(\mu)$ constitutes an $\mf{S}_n$-set with a simple transitive action. Therefore, part (2) of Definition \ref{def-standard-tableaux} relies on the fact we assign to the identity element $e\in \mf{S}_n$ the maximal tableaux $\mf{t}^\mu$. 

In what follows, when talking about tableaux of a fixed partition $\mu$, we will abuse notation and use the tableau $\mf{t}$ and the permutation associated with it $w_\mf{t}\in \mf{S}_n$ interchangeably. In this notation $e$ will always stand for the standard tableau $\mf{t}^\mu$. We will also use the usual Coxeter length function on the symmetric group $\ell: \mf{S}_n\lra \N$, and transport it to $\Tab(\mu)$ under this identification.
\end{rem}

\begin{example}\label{eg-two-tableaux}
The tableau $\left(~\young(1)~,~\young(2)~,~\emptyset~\right)$ is the standard one of its shape, while $\left(~\young(2)~,~\young(1)~,~\emptyset~\right)$ is non-standard. The corresponding permutations are the identity element $e$ and non-identity element $s_1$ of the symmetric group $\mf{S}_2$.
\end{example}

\begin{defn}
Given a tableau $\mf{t}$, let $\mf{t}_{\downarrow k}$ be the subtableau defined by
\begin{equation*}
\mf{t}_{\downarrow k} \colon \mf{t}^{-1}\{1, \ldots, k \} \subset \{j_1, \ldots, j_n \} \rightarrow \{1, \ldots, k \}.
\end{equation*}
\end{defn}

Note that $\mf{t}_{\downarrow k}$ is the subtableau of $\mf{t}$ built in the first $k$ steps, and it is a filling of a partition in the set $\mathcal{P}_k^l$.

We are now ready to introduce a partial order on tableaux.

\begin{defn}\label{def-order-on-nh-tableaux}
Let $\mf{s}$ be a $\mu$-tableau and $\mf{t}$ a $\lambda$-tableau.  We write $ \mf{s} \geq \mf{t}$ if the following conditions on the shapes of the subtableaux are satisfied (Definition \ref{def-order-on-nh-partition}):
\begin{equation*}
\mathrm{shape}(\mf{s}_{\downarrow k}) \geq \mathrm{shape}(\mf{t}_{\downarrow k}), \hspace{.5in} \text{ for all }k=1, \ldots, n.
\end{equation*}
Moreover, if $\mf{s}\geq \mf{t}$ and $\mf{s}\neq \mf{t}$, we then write $\mf{s}>\mf{t}$.
\end{defn}

\begin{rem}\label{rmktableauxcomments1}
We make two simple notes.
\begin{enumerate}
\item[(1)] 
For later use it is convenient to observe that a $\mu$-tableau $\mf{s}$ is greater than or equal to the the standard tableau $\mf{t}^{\lambda}$ of a partition $\lambda$, if each element in the filling $\mu$ corresponding to $\mf{s}$ appears to the left of the same element in  $\mf{t}^\lambda$.
%By definition, we may think of a $\mu$-tableau $\mf{s}$ to be greater than or equal to another $\lambda$-tableau $\mf{t}$ if each element in the filling of $\mu$ corresponding to $\mf{s}$ appears (non-strictly) to the left of the same element in the filling of $\lambda$ corresponding to $\mf{t}$. Using this description, it is readily seen that any two tableaux in $\Tab(\mu)$ are incomparable.
\item[(2)] One may easily verify that, $\mu \geq\lambda$ if and only if the standard filling $\mf{t}^\mu$ of $\mu$ is greater than or equal to the standard filling $\mf{t}^\lambda$ of $\lambda$.
\end{enumerate} 
\end{rem}

In view of part (1) of Remark \ref{rmktableauxcomments1}, it will be useful to introduce a different partial order for later use. 
\begin{defn}
Let $\mf{s}$ be a $\mu$-tableau and $\mf{t}$ be a $\lambda$-tableau. We say $\mf{s} \leq_{LR} \mf{t}$, if each element in the filling of $\mf{s}$ appears to the right of the corresponding element in the filling $\mf{t}$.
\end{defn}

\begin{defn}\label{deftableuaxdegree}
The degree of a $\mu$-tableau $\mf{t}$ is defined by
\begin{equation*}
\mathrm{deg}(\mf{t})=nl-(j_1+\cdots+j_n)-2\ell(w_\mf{t}),
\end{equation*}
where $\ell(w_\mf{t})$ is the length of the permutation $w_\mf{t}$.
\end{defn}

The next definition is used later for describing cellular structures.

\begin{defn}\label{defTablambdamu}
Let $\mu\in \mc{P}_n^l$ be a partition. Relative to a second partition $\lambda$, we define two subsets of $\mu$-tableaux as follows.
\begin{subequations}
\begin{equation}\label{eqnuppertableaux}
\mathrm{Tab}^{\lambda}(\mu) := \{ \mf{s} \in \mathrm{Tab}(\mu) | \mf{s} \geq  \mf{t}^{\lambda} \},
\end{equation}
\begin{equation}\label{eqnlowertableaux}
\Tab_\lambda(\mu):=\{\mf{s}\in \Tab(\mu)|
\mf{s} \leq_{LR} \mf{t}^\lambda \}
%\mf{t}\leq \mf{t}^{\lambda}\}.
\end{equation}
\end{subequations}
Thus a $\mu$-tableau $\mf{s}$ is in the set $\Tab_\lambda(\mu)$ if each element
in $\mf{t}^\lambda$ appears to the left of the corresponding element in the filling $\mf{s}$ of $\mu$.
\end{defn}

\begin{rem}\label{remtableauxcoments2}
\begin{enumerate}
\item[(1)] Note that while $\mathrm{Tab}^{\lambda}(\mu)$ is the same as 
$\mathrm{Std}^{\lambda}(\mu)$ used in \cite[Section 4.1]{HuMathas},
the definitions of $\Tab_\lambda(\mu)$ and  $\mathrm{Std}_\lambda(\mu)$ are different. 
\item[(2)] By Remark \ref{rmktableauxcomments1} (1), it is clear that $\Tab_\lambda(\lambda)=\Tab^\lambda(\lambda)=\{\mf{t}^\lambda\}$.
\end{enumerate}
\end{rem}

\begin{example} \label{eg-23-tableaux}
We continue to consider the case of $l=3,n=2$ as in Example \ref{eg-23partitions}, where the three partitions are
\begin{equation*}
\lambda = (~\yng(1)~, ~\yng(1)~, ~\emptyset~)  \quad \quad
\mu = (~\yng(1)~, ~\emptyset~, ~\yng(1)~) \quad \quad
\nu = (~\emptyset~, ~\yng(1)~, ~\yng(1)~).
\end{equation*}
For $\lambda$, we have, relative to all partitions, that
\[
\mathrm{Tab}^{\lambda}(\lambda)=\left\{
\left(~\young(1)~,~\young(2)~,~\emptyset~\right)
\right\}, \quad \quad
\mathrm{Tab}_{\lambda}(\lambda)=\left\{
\left(~\young(1)~,~\young(2)~,~\emptyset~\right)
\right\},
\]
\[
\mathrm{Tab}^{\mu}(\lambda)=\left\{
\left(~\young(1)~,~\young(2)~,~\emptyset~\right),
\right\} ,\quad \quad
\mathrm{Tab}_{\mu}(\lambda)=\emptyset,
\]
\[
\mathrm{Tab}^{\nu}(\lambda)=\left\{
\left(~\young(1)~,~\young(2)~,~\emptyset~\right), \left(~\young(2)~,~\young(1)~,~\emptyset~\right)
\right\}, \quad \quad
\mathrm{Tab}_{\nu}(\lambda)=\emptyset.
\]
For $\mu$, we have that
\[
\mathrm{Tab}^{\lambda}(\mu)=\emptyset, \quad \quad
\mathrm{Tab}_{\lambda}(\mu)=\left\{
\left(~\young(1)~,~\emptyset~,~\young(2)~\right)
\right\},
\]
\[
\mathrm{Tab}^{\mu}(\mu)=\left\{
\left(~\young(1)~,~\emptyset~,~\young(2)~\right)
\right\}, \quad \quad
\mathrm{Tab}_{\mu}(\mu)=\left\{
\left(~\young(1)~,~\emptyset~,~\young(2)~\right)
\right\},
\]
\[
\mathrm{Tab}^{\nu}(\mu)=\left\{
\left(~\young(1)~,~\emptyset~,~\young(2)~\right)
\right\}, \quad \quad
\mathrm{Tab}_{\nu}(\mu)=
\emptyset.
\]
Finally, for $\nu$, we have
\[
\mathrm{Tab}^{\lambda}(\nu)=\emptyset,
 \quad \quad
\mathrm{Tab}_{\lambda}(\nu)=\left\{\left(~\emptyset~,
~\young(1)~,~\young(2)~\right) , \left(~\emptyset~,
~\young(2)~,~\young(1)~\right)
\right\} ,
\]
\[
\mathrm{Tab}^{\mu}(\nu)=\emptyset,
 \quad \quad
\mathrm{Tab}_{\mu}(\nu)=\left\{\left(~\emptyset~,
~\young(1)~,~\young(2)~\right) 
\right\} ,
\]
\[
\mathrm{Tab}^{\nu}(\nu)=\left\{\left(~\emptyset~,
~\young(1)~,~\young(2)~\right) 
\right\},
 \quad \quad
\mathrm{Tab}_{\nu}(\nu)=\left\{\left(~\emptyset~,
~\young(1)~,~\young(2)~\right) 
\right\} .
\]
\end{example}

Observe that the cardinality of the sets $\Tab^\lambda(\mu)$ and $\Tab_\mu(\lambda)$ are equal in the previous example. We next show that there is a canonical bijection of the two sets in general. To do this, we use that elements of $\Tab(\lambda)$ and $\Tab(\mu)$ are in bijection with the symmetric group $\mf{S}_n$ via Remark \ref{ntntableauxsymmetricgroup}. 

\begin{lem} \label{bijlem}
For any $\alpha,\beta\in \mc{P}_n^l$, there is a bijection between the sets of tableaux
$\Tab^\alpha(\beta)$ and $\Tab_\beta(\alpha)$.
\end{lem}
\begin{proof}
By Definition \ref{defTablambdamu}, an element $\mf{t}$ of $\Tab^\alpha(\beta)$ is a filling $\mathfrak{t}$ of $\beta$ with entries $\{1, \ldots, n\}$ which is greater than or equal to $\mf{t}^{\alpha}$ in the dominance order of Definition \ref{def-order-on-nh-tableaux}.  Let us depict this element diagrammatically as follows. We draw the standard filling of $\alpha$ as labeled boxes and empty slots on an upper horizontal line, and on a lower horizontal line, we draw the the filling of $\mf{t}\in \Tab(\beta)$ in a similar fashion. For instance, consider, as in Example \ref{eg-23-tableaux}, the tableau $\left(~\young(2)~,~\young(1)~,~\emptyset~\right)\in \mathrm{Tab}^{\nu}(\lambda)$, which will be depicted as
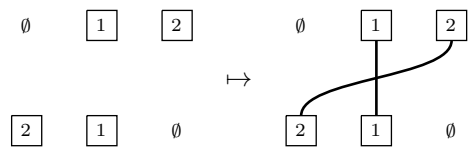
\[
\begin{DGCpicture}
\DGCcoupon(0,0)(0.4,-0.4){$_2$}
\DGCcoupon(1,0)(1.4,-0.4){$_1$}
\DGCcoupon*(2,0)(2.4,-0.4){$_\emptyset$}
\DGCcoupon*(0,1.4)(0.4,1){$_\emptyset$}
\DGCcoupon(1,1.4)(1.4,1){$_1$}
\DGCcoupon(2,1.4)(2.4,1){$_2$}
\end{DGCpicture}
\quad \mapsto \quad
\begin{DGCpicture}
\DGCcoupon(0,0)(0.4,-0.4){$_2$}
\DGCcoupon(1,0)(1.4,-0.4){$_1$}
\DGCcoupon*(2,0)(2.4,-0.4){$_\emptyset$}
\DGCcoupon*(0,1.4)(0.4,1){$_\emptyset$}
\DGCcoupon(1,1.4)(1.4,1){$_1$}
\DGCcoupon(2,1.4)(2.4,1){$_2$}
\DGCstrand(1.2,0)(1.2,1)
\DGCstrand(0.2,0)(2.2,1)
\end{DGCpicture}
\ .
\]
Then we connect the corresponding boxes with matching numbers, as shown in the second step of the above example. The trajectory of the diagram, read from bottom to top, gives us a permutation $w_\mf{t}\in \mf{S}_n$ similarly as in Remark \ref{ntntableauxsymmetricgroup} (despite the picture being slanted). Therefore it is clear that we may identify $\Tab^\alpha(\beta)$ with the collection of symmetric group elements in $w \in \mf{S}_n$, which satisfy that $w(i)$ appears to the upper right of $i$ for all $i\in \{1,\dots, n\}$.

On the other hand, for tableaux in $\mf{t}\in \Tab_\beta(\alpha)$, we perform a similar procedure, but instead drawing the standard filling $\mf{t}^\beta$ on the bottom horizontal line and that of $\mf{t}$ on top, to obtain $w_\mf{t}\in \mf{S}_n$. For the example $\left(~\emptyset~,~\young(2)~,~\young(1)~\right)\in \mathrm{Tab}_\lambda(\nu)$, we have:
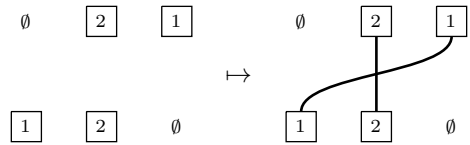
\[
\begin{DGCpicture}
\DGCcoupon(0,0)(0.4,-0.4){$_1$}
\DGCcoupon(1,0)(1.4,-0.4){$_2$}
\DGCcoupon*(2,0)(2.4,-0.4){$_\emptyset$}
\DGCcoupon*(0,1.4)(0.4,1){$_\emptyset$}
\DGCcoupon(1,1.4)(1.4,1){$_2$}
\DGCcoupon(2,1.4)(2.4,1){$_1$}
\end{DGCpicture}
\quad \mapsto \quad
\begin{DGCpicture}
\DGCcoupon(0,0)(0.4,-0.4){$_1$}
\DGCcoupon(1,0)(1.4,-0.4){$_2$}
\DGCcoupon*(2,0)(2.4,-0.4){$_\emptyset$}
\DGCcoupon*(0,1.4)(0.4,1){$_\emptyset$}
\DGCcoupon(1,1.4)(1.4,1){$_2$}
\DGCcoupon(2,1.4)(2.4,1){$_1$}
\DGCstrand(1.2,0)(1.2,1)
\DGCstrand(0.2,0)(2.2,1)
\end{DGCpicture}
\]
It is now clear that, following this procedure, the elements in $\Tab_\beta(\alpha)$ gives rise to the same collection of symmetric group elements in $\mf{S}_n$. After all, once the top and bottom partitions are fixed, the condition $\sigma(i)$ appears to the upper  right of $i$ does not depend on the ordering on the set $\{1,\dots, n\}$ entered into the lower partition. The desired bijection between $\Tab^\alpha(\beta)$ and $\Tab_\beta(\alpha)$ follows by matching them with this subset of $\mf{S}_n$.
\end{proof}

\begin{defn}\label{defallowablepermutation}
Let $\alpha,\beta\in \mc{P}_n^l$ be two partitions. A symmetric group element $w\in \mf{S}_n$ is called \emph{${\alpha \choose \beta}$-permissible} if $w=w_{\mf{t}}$ for some $\mf{t}\in \Tab_\beta(\alpha)$, or equivalently $w=w_\mf{s}$ for some $\mf{s}\in \Tab^\alpha(\beta)$. The set of all ${\alpha \choose \beta}$-permissible permutations will be denoted as $\mf{S}^\alpha_\beta$. 
\end{defn}

\begin{example} 
For $l=4,n=3$, there are four partitions, two of which are
\begin{equation*}
\lambda = (~\yng(1)~, ~\yng(1)~, ~\yng(1)~  ~\emptyset~)  \quad \quad
\mu = (~\emptyset~, ~\yng(1)~, ~\yng(1)~, ~\yng(1)~) 
\ .
\end{equation*}
We then have
\[
\mathrm{Tab}^{\mu}(\lambda)=
\left\{
\left(~\young(1)~,~\young(2)~,~\young(3)~, ~\emptyset~\right), 
\left(~\young(2)~,~\young(1)~,~\young(3)~, ~\emptyset~\right),
\left(~\young(1)~,~\young(3)~,~\young(2)~, ~\emptyset~\right),
\left(~\young(3)~,~\young(1)~,~\young(2)~, ~\emptyset~\right)
\right\}  ,
\]
\[
\mathrm{Tab}_{\lambda}(\mu)=
\left\{
\left(~\emptyset~,~\young(1)~,~\young(2)~,~\young(3)~\right), 
\left(~\emptyset~,~\young(2)~,~\young(1)~,~\young(3)~\right), 
\left(~\emptyset~,~\young(1)~,~\young(3)~,~\young(2)~\right), 
\left(~\emptyset~,~\young(2)~,~\young(3)~,~\young(1)~\right)
\right\}  .
\]
The corresponding permissible symmetric group elements $\mf{S}^\mu_\lambda$ are drawn as: 
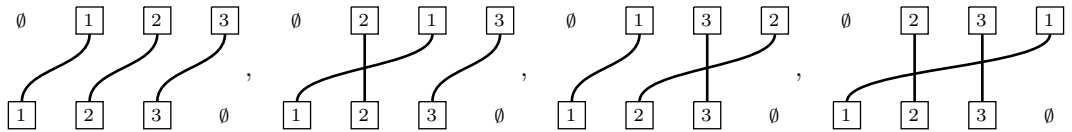
\[
\begin{DGCpicture}[scale=0.9]
\DGCcoupon*(0,0.4)(0.4,0){$_\emptyset$}
\DGCcoupon(1,0.4)(1.4,0){$_1$}
\DGCcoupon(2,0.4)(2.4,0){$_2$}
\DGCcoupon(3,0.4)(3.4,0){$_3$}
\DGCcoupon(0,-1)(0.4,-1.4){$_1$}
\DGCcoupon(1,-1)(1.4,-1.4){$_2$}
\DGCcoupon(2,-1)(2.4,-1.4){$_3$}
\DGCcoupon*(3,-1)(3.4,-1.4){$_\emptyset$}
\DGCstrand/d/(1.2,0)(0.2,-1)/d/
\DGCstrand/d/(2.2,0)(1.2,-1)/d/
\DGCstrand/d/(3.2,0)(2.2,-1)/d/
\end{DGCpicture} \ ,
\quad
\begin{DGCpicture}[scale=0.9]
\DGCcoupon*(0,0.4)(0.4,0){$_\emptyset$}
\DGCcoupon(1,0.4)(1.4,0){$_2$}
\DGCcoupon(2,0.4)(2.4,0){$_1$}
\DGCcoupon(3,0.4)(3.4,0){$_3$}
\DGCcoupon(0,-1)(0.4,-1.4){$_1$}
\DGCcoupon(1,-1)(1.4,-1.4){$_2$}
\DGCcoupon(2,-1)(2.4,-1.4){$_3$}
\DGCcoupon*(3,-1)(3.4,-1.4){$_\emptyset$}
\DGCstrand/d/(1.2,0)(1.2,-1)/d/
\DGCstrand/d/(2.2,0)(0.2,-1)/d/
\DGCstrand/d/(3.2,0)(2.2,-1)/d/
\end{DGCpicture} \ ,
\quad
\begin{DGCpicture}[scale=0.9]
\DGCcoupon*(0,0.4)(0.4,0){$_\emptyset$}
\DGCcoupon(1,0.4)(1.4,0){$_1$}
\DGCcoupon(2,0.4)(2.4,0){$_3$}
\DGCcoupon(3,0.4)(3.4,0){$_2$}
\DGCcoupon(0,-1)(0.4,-1.4){$_1$}
\DGCcoupon(1,-1)(1.4,-1.4){$_2$}
\DGCcoupon(2,-1)(2.4,-1.4){$_3$}
\DGCcoupon*(3,-1)(3.4,-1.4){$_\emptyset$}
\DGCstrand/d/(1.2,0)(0.2,-1)/d/
\DGCstrand/d/(2.2,0)(2.2,-1)/d/
\DGCstrand/d/(3.2,0)(1.2,-1)/d/
\end{DGCpicture}
\ ,
\quad
\begin{DGCpicture}[scale=0.9]
\DGCcoupon*(0,0.4)(0.4,0){$_\emptyset$}
\DGCcoupon(1,0.4)(1.4,0){$_2$}
\DGCcoupon(2,0.4)(2.4,0){$_3$}
\DGCcoupon(3,0.4)(3.4,0){$_1$}
\DGCcoupon(0,-1)(0.4,-1.4){$_1$}
\DGCcoupon(1,-1)(1.4,-1.4){$_2$}
\DGCcoupon(2,-1)(2.4,-1.4){$_3$}
\DGCcoupon*(3,-1)(3.4,-1.4){$_\emptyset$}
\DGCstrand/d/(1.2,0)(1.2,-1)/d/
\DGCstrand/d/(2.2,0)(2.2,-1)/d/
\DGCstrand/d/(3.2,0)(0.2,-1)/d/
\end{DGCpicture}
\ .
\]
This illustrates the bijection in Lemma \ref{bijlem} with $\Tab^\mu(\lambda)$:
\[
\begin{DGCpicture}[scale=0.9]
\DGCcoupon*(0,0.4)(0.4,0){$_\emptyset$}
\DGCcoupon(1,0.4)(1.4,0){$_1$}
\DGCcoupon(2,0.4)(2.4,0){$_2$}
\DGCcoupon(3,0.4)(3.4,0){$_3$}
\DGCcoupon(0,-1)(0.4,-1.4){$_1$}
\DGCcoupon(1,-1)(1.4,-1.4){$_2$}
\DGCcoupon(2,-1)(2.4,-1.4){$_3$}
\DGCcoupon*(3,-1)(3.4,-1.4){$_\emptyset$}
\DGCstrand/d/(1.2,0)(0.2,-1)/d/
\DGCstrand/d/(2.2,0)(1.2,-1)/d/
\DGCstrand/d/(3.2,0)(2.2,-1)/d/
\end{DGCpicture} \ ,
\quad
\begin{DGCpicture}[scale=0.9]
\DGCcoupon*(0,0.4)(0.4,0){$_\emptyset$}
\DGCcoupon(1,0.4)(1.4,0){$_1$}
\DGCcoupon(2,0.4)(2.4,0){$_2$}
\DGCcoupon(3,0.4)(3.4,0){$_3$}
\DGCcoupon(0,-1)(0.4,-1.4){$_2$}
\DGCcoupon(1,-1)(1.4,-1.4){$_1$}
\DGCcoupon(2,-1)(2.4,-1.4){$_3$}
\DGCcoupon*(3,-1)(3.4,-1.4){$_\emptyset$}
\DGCstrand/d/(1.2,0)(1.2,-1)/d/
\DGCstrand/d/(2.2,0)(0.2,-1)/d/
\DGCstrand/d/(3.2,0)(2.2,-1)/d/
\end{DGCpicture} \ ,
\quad
\begin{DGCpicture}[scale=0.9]
\DGCcoupon*(0,0.4)(0.4,0){$_\emptyset$}
\DGCcoupon(1,0.4)(1.4,0){$_1$}
\DGCcoupon(2,0.4)(2.4,0){$_2$}
\DGCcoupon(3,0.4)(3.4,0){$_3$}
\DGCcoupon(0,-1)(0.4,-1.4){$_1$}
\DGCcoupon(1,-1)(1.4,-1.4){$_3$}
\DGCcoupon(2,-1)(2.4,-1.4){$_2$}
\DGCcoupon*(3,-1)(3.4,-1.4){$_\emptyset$}
\DGCstrand/d/(1.2,0)(0.2,-1)/d/
\DGCstrand/d/(2.2,0)(2.2,-1)/d/
\DGCstrand/d/(3.2,0)(1.2,-1)/d/
\end{DGCpicture}
\ ,
\quad
\begin{DGCpicture}[scale=0.9]
\DGCcoupon*(0,0.4)(0.4,0){$_\emptyset$}
\DGCcoupon(1,0.4)(1.4,0){$_1$}
\DGCcoupon(2,0.4)(2.4,0){$_2$}
\DGCcoupon(3,0.4)(3.4,0){$_3$}
\DGCcoupon(0,-1)(0.4,-1.4){$_3$}
\DGCcoupon(1,-1)(1.4,-1.4){$_1$}
\DGCcoupon(2,-1)(2.4,-1.4){$_2$}
\DGCcoupon*(3,-1)(3.4,-1.4){$_\emptyset$}
\DGCstrand/d/(1.2,0)(1.2,-1)/d/
\DGCstrand/d/(2.2,0)(2.2,-1)/d/
\DGCstrand/d/(3.2,0)(0.2,-1)/d/
\end{DGCpicture}
\ .
\]
\end{example}

\subsection{Cellular structure on nilHecke algebras}
We now consider certain important elements in the cyclotomic nilHecke algebra. Recall from Definition \ref{def-standard-tableaux} that if  $\mf{t}$ is a tableau of shape $\mu$, then $w_\mf{t}$ is a permutation, which in turn defines a nilHecke element $\psi_\mf{t}:=\psi_{w_\mf{t}}\in \nh_n^l$ (see Section \ref{subsec-def-nilHecke}). For instance, for the standard and non-standard tableaux $e$ and $s$ in Example \ref{eg-two-tableaux} (see also Remark \ref{ntntableauxsymmetricgroup}), we obtain their corresponding nilHecke elements :
\[
\psi_e=
\begin{DGCpicture}
\DGCstrand(1,0)(1,1)
\DGCstrand(0,0)(0,1)
\end{DGCpicture} \ , 
\quad \quad \quad
\psi_{s}=
\begin{DGCpicture}
\DGCstrand(1,0)(0,1)
\DGCstrand(0,0)(1,1)
\end{DGCpicture} \ .
\]

\begin{defn}\label{def-y-mu-psi-mu}
Again suppose that for the partition $\mu$ that $ \mu^{j_1}=\cdots=\mu^{j_n}=1$ and $ j_1 < \cdots < j_n$.
For two $\mu$-tableaux $\mf{s}$ and $\mf{t}$ let 

\begin{equation}\label{eqn-psi-st}
y^{\mu}:=y_1^{l-j_1} \cdots y_n^{l-j_n}, \quad \quad \quad \psi_{\mf{s}\mf{t}}^{\mu}:=\psi_\mf{s}^* y^{\mu} \psi_\mf{t}.
\end{equation}
When the composition $ \mu $ is clear from context we will often abbreviate $ \psi_{\mf{s}\mf{t}}^{\mu} $ by $ \psi_{\mf{s}\mf{t}}$.
\end{defn}

\begin{thm}\label{thm-HMcellularNH}
The set $ \{ \psi^{\mu}_{\mf{s}\mf{t}} | \mf{s},\mf{t} \in \mathrm{Tab}(\mu), \mu \in \mathcal{P}_n^l \}$ is a graded cellular basis of $\nh_n^l$.
More precisely
\begin{enumerate}
\item[(i)] The degree of $\psi_{\mf{s}\mf{t}}$ is the sum $\mathrm{deg}(\mf{s})+\mathrm{deg}(\mf{t})$.
\item[(ii)] For $\mu \in \mathcal{P}_n^l$ and $\mf{s},\mf{t} \in \mathrm{Tab}(\mu)$, there are scalars $r_{\mf{t}\mf{v}}(x)$ which do not depend on $\mf{s}$ such that,
\begin{equation*}
\psi_{\mf{s}\mf{t}}x = \sum_{\mf{v} \in \mathrm{Tab}(\mu)} r_{\mf{t}\mf{v}}(x) \psi_{\mf{s}\mf{v}} \hspace{.15in} \text{ mod } (\nh_n^l)^{>\mu},
\end{equation*}
where 
\begin{equation*}
(\nh_n^l)^{>\mu} = \Bbbk \left\langle \psi_{\mf{a}\mf{b}}^{\lambda} | \lambda > \mu, \text{ and } \mf{a},\mf{b} \in \mathrm{Tab}(\lambda) \right\rangle.
\end{equation*}
\item[(iii)] The anti-automorphism $ * \colon \nh_n^l \longrightarrow \nh_n^l$ sends $\psi_{\mf{s}\mf{t}}$ to $\psi_{\mf{s}\mf{t}}^*=\psi_{\mf{t}\mf{s}}$.
\end{enumerate}
\end{thm}
\begin{proof}
This can be found in \cite[Theorem 5.8 and 6.11]{HuMathasKLR}. See also \cite{Lige}.
\end{proof}

The next result shows that $\nh_n^l$ is a symmetric Frobenius algebra.  

\begin{prop}
\label{symm}
 There is a non-degenerate homogeneous trace $\tau \colon \nh_n^l \longrightarrow \Bbbk$ of degree $-2n(l-n)$.
\end{prop}
\begin{proof}
See \cite[Theorem 6.17]{HuMathasKLR}.
\end{proof}

\begin{prop}
\label{idealispdg}
The cyclotomic nilHecke algebra $\nh_n^l$ is a $p$-DG cellular algebra.
\end{prop}
\begin{proof}
By \cite[Proposition 7.18]{KQS}, 
the two-sided ideal $(\nh_n^l)^{> \mu}$ is preserved by the differential.

Recall that 
$\psi_{\mf{s} \mf{t}^{\lambda}}^{\lambda}=\psi^*_{\mf{s}} y^{\lambda}$
and 
$\psi_{\mf{t}^{\lambda} \mf{t}}^{\lambda}=y^{\lambda} \psi_{\mf{t}}$. For $\lambda \in \mathcal{P}_n^l$, the left and right cell modules have bases given by
\begin{equation*}
\Delta(\lambda)=\{
\psi_{\mf{s} \mf{t}^{\lambda}}^{\lambda} | \mf{s} \in \mf{S}_n
\}
\quad \quad
\Delta^{\circ}(\lambda)=\{
\psi_{\mf{t}^{\lambda} \mf{t}}^{\lambda} | \mf{t} \in \mf{S}_n.
\}
\end{equation*}
The $p$-DG module structure on them is determined by
\begin{equation}
  \dif(\psi_{\mf{s} \mf{t}^{\lambda}}^{\lambda})  = \dif(\psi_\mf{s}^*)y^\lambda \quad \quad
    \dif(\psi_{\mf{t}^{\lambda}\mf{t} }^{\lambda})  = y^\lambda \dif(\psi_\mf{t}).
\end{equation}

We claim that the map
\begin{equation} \label{nilheckepdgcellmap}
\Delta(\lambda) \otimes \Delta^{\circ}(\lambda) 
\lra
(\nh_n^l)^{\geq \lambda} / (\nh_n^l)^{> \lambda}
\quad \quad
\psi_{\mf{s} \mf{t}^{\lambda}}^{\lambda} 
\otimes
\psi_{\mf{t}^{\lambda} \mf{t}}^{\lambda}
\mapsto
\psi_{\mf{s} \mf{t}}^{\lambda}
\end{equation}
is a $p$-DG isomorphism.

It suffices to show that the image of the bimodule generator
$\psi^\lambda_{\mf{t}^\lambda \mf{t}^\lambda}\otimes \psi^\lambda_{\mf{t}^\lambda \mf{t}^\lambda} $, which is $y^\lambda$,  has trivial differential action in $\mathrm{NH}_n^l$ modulo $(\mathrm{NH}_n^l)^{>\lambda}$. First note that
$\dif(y^\lambda) = x y^\lambda$ for some linear polynomial $x$ in $y$'s. By part (ii) of Theorem \ref{thm-HMcellularNH}, \begin{equation}
xy^\lambda=\sum_{\mf{v}} r_{\mf{v}\mf{t}^\lambda}\psi_{\mf{v}\mf{t}^\lambda}^\lambda 
=\sum_{\mf{v}} r_{\mf{v}\mf{t}^\lambda}\psi_{\mf{v}}^*y^\lambda \quad \mathrm{mod}~(\mathrm{NH}_n^l)^{>\lambda} .
\end{equation}
Comparing degrees of both sides, it follows $xy^\lambda \in(\mathrm{NH}_n^l)^{>\lambda} $, since $x$ has positive degree and $\psi_{\mf{v}}^*$ has non-positive degree. The claim now follows.
\end{proof}

\begin{rem}
While $\nh_n^l$ is a $p$-DG cellular algebra, it is not in general a $p$-DG quasi-hereditary algebra.  By Example \ref{egrunningegwithdifquasihereditary}~
(1), $\nh_1^l$ is not a $p$-DG quasi-hereditary algebra for $l \geq 2$.
Likewise, $\nh_l^l$ is a $p$-DG matrix algebra of size $l! \times l!$ over the ground field.  It is quasi-hereditary if and only if it is not acyclic, which is equivalent to requiring $l \leq p-1$.
\end{rem}

Let us also record the following useful fact concerning another basis for $\nh_n^l$ for later use.

\begin{prop} \label{dotbottombasis}
\begin{enumerate}
\item[(i)] The set of elements
\[\{\psi_w  y_1^{r_1} \cdots y_k^{r_k} \cdots y_n^{r_n}  | w \in \mf{S}_n, 0 \leq r_k \leq l-k \}\]
is a basis for the algebra $\nh_n^l$.
\item[(ii)] As a right module over $\nh_{n-1}^{l}$, the cyclotomic nilHecke algebra
$\nh_n^{l}$ is free of rank $n(l-n+1)$: 
\begin{equation} \label{NHnResNHn-1}
\begin{DGCpicture}
\DGCstrand(0,-0.5)(0,1.5)[`$_1$]
\DGCstrand(0.5,-0.5)(0.5,1.5)[`$_{2}$]
\DGCstrand(3.5,-0.5)(3.5,1.5)[`$_{n}$]
\DGCcoupon(-0.25,0)(3.75,1){$\nh_{n}^{l}$}
\DGCcoupon*(0.5,-0.5)(3.5,-0.25){$\cdots$}
\DGCcoupon*(0.5,1.25)(3.5,1.5){$\cdots$}
\end{DGCpicture}\
\cong
\bigoplus_{i=1}^n\bigoplus_{r=0}^{l-n}~
\begin{DGCpicture}
\DGCstrand(0,-0.5)(0,1.5)[`$_1$]
\DGCstrand(1,-0.5)(1,1.5)[`$_{i-1}$]
\DGCstrand(3,-.5)(3,0.75)(3.5,1.5)[`$_n$]
\DGCstrand(1.5,-.5)(1.5,0.75)(2,1.5)[`$_{i+1}$]
\DGCstrand(3.5,-0.5)(3.5,0.75)(1.5,1.5)[`$_i$]
\DGCdot{0.3}[urr]{$_r$}
\DGCcoupon(-0.25,-0.1)(3.25,0.7){$\nh_{n-1}^{l}$}
\DGCcoupon*(1.5,-0.5)(3,-0.25){$\cdots$}
\DGCcoupon*(0,-0.5)(1,-0.25){$\cdots$}
\DGCcoupon*(0,1.25)(1,1.5){$\cdots$}
\DGCcoupon*(1.75,1.3)(3.75,1.5){$\cdots$}
\end{DGCpicture} \ .
\end{equation}
\end{enumerate}
\end{prop}

\begin{proof}
It suffices to prove the first statement, and the second one follows readily.

We first show that $\nh_n^l$ has a spanning set
\begin{equation} \label{dotsonbottomeq}
\{\psi_w  y_1^{r_1} \cdots y_k^{r_k} \cdots y_n^{r_n}  | w \in \mf{S}_n, 0 \leq r_k \leq l-k \}.
\end{equation}
Using the relations of the nilHecke algebra, it is clear that elements of the form
$\psi_w y_1^{r_1} \cdots y_n^{r_n} $ span $\nh_n^l$, where $w\in \mf{S}_n$ and $r_i\in \N$.
By the cyclotomic relation, we know that $r_1 \leq l-1$.
We verify that the set in \eqref{dotsonbottomeq} is a spanning set by induction on $k$.
Assume that we may we rewrite $y_k^{l-k+1}$ as follows
\begin{equation} \label{yktopower}
y_k^{l-k+1} = \sum c_{r_1\dots r_k} y_1^{r_1} \cdots y_k^{r_k}
\end{equation}
for some constants $c_{r_1\dots r_k}$ with $r_k \leq l-k$.  Thus
\begin{equation} \label{yktopower2}
y_k^{l-k+1} - \sum c_{\gamma} y_1^{\gamma_1} \cdots y_k^{\gamma_k} = 0.
\end{equation}
By \cite[Lemma 4.2]{KK}, applying the $k$th divided difference operator to the left hand side of \eqref{yktopower2} produces zero.  On the other hand, using the inductive hypothesis \eqref{yktopower}, applying the $k$th divided difference operator yields
\begin{equation} \label{yktopower3}
y_{k+1}^{l-k} - \sum d_{\nu} y_1^{\nu_1} \cdots y_{k+1}^{\nu_{k+1}} = 0
\end{equation}
for come constants $d_{\nu}$ and exponents $\nu_i$ such that $\nu_i \leq l-i$.  
This verifies the inductive hypothesis and shows that \eqref{dotsonbottomeq} is indeed a spanning set.  Since $\nh_n^l$ has dimension $\binom{l}{n}n!n!$ and \eqref{dotsonbottomeq} has cardinality $\binom{l}{n}n!n!$, the spanning set 
\eqref{dotsonbottomeq} is in fact a basis of $\nh_n^l$.

Now write an element $\psi_w y_1^{\gamma_1} \cdots y_n^{\gamma_n}$ in \eqref{dotsonbottomeq} as
\begin{equation*}
\psi_{i+1} \cdots \psi_{n-1} 
(\psi_{w'} y_1^{\gamma_1} \cdots y_{n-1}^{\gamma_{n-1}})
y_n^{\gamma_n}
\end{equation*}
where $w' \in \mf{S}_{n-1}$.
Since elements of the form 
$\psi_{w'} y_1^{\gamma_1} \cdots y_{n-1}^{\gamma_{n-1}} $
constitute a basis of $\nh_{n-1}^l$, the second part of the proposition follows.
\end{proof}

\begin{rem}
There is another proof of the above result avoiding the use of \cite{KK}.
It uses the action of $\nh_n^l$ on $\Bbbk[y_1,\ldots, y_n]/(h_a | a \geq l)$.
One checks directly that these basis elements act linear-independently on this faithful module.
\end{rem}

\subsection{Specht modules}
The modules over the cyclotomic nilHecke algebra which we are about to define are known as Specht modules.  These modules have been considered for the classical cyclotomic Hecke algebras in earlier literature.  Their graded lifts (as modules over graded KLR algebras) have been constructed by Brundan, Kleshchev, and Wang~\cite{BKW}.  Here we once again specialize their general KLR construction to the nilHecke case, corresponding to $\mf{sl}_2$-categorification.

Let $\mu \in \mathcal{P}_n^l$ and $ \mf{t} \in \mathrm{Tab}(\mu)$.  
One defines the right Specht module $S^{\mu}_\mf{t}$ in terms of the cellular structure 
on $\nh_n^l$ defined in the previous section.
Let $S_\mf{t}^{\mu}$ be the submodule of $ \nh_n^l / (\nh_n^l)^{> \mu} $ generated by the coset $\psi_{\mf{t} \mf{t}^\mu}^{\mu} + (\nh_n^l)^{> \mu}$. For the ease of notation, we will usually write $e$ for $\mf{t}^{\mu}$ (see Remark \ref{ntntableauxsymmetricgroup}), so that, for instance, $\psi_{\mf{t}^\mu\mf{t}^\mu}^\mu$ will be denoted $\psi^\mu_{ee}$ as well.

\begin{prop}
The Specht module $S_{e}^{\mu}$ is a right $p$-DG module, where $e$ the identity tableaux.
\end{prop}
\begin{proof}
See \cite[Proposition 7.19]{KQS}.
\end{proof}

\begin{rem}
To see that the above proposition is not true for an arbitrary tableau, consider the case $l=n=2$.
Then there is only one partition; call it $\mu$.  There are two tableaux: the identity $e$ and the transposition $s$.
Then $S^{\mu}_e$ has basis $\{ y_1, y_1 \psi_1 \}$ and as in the proposition is obviously a $p$-DG module.
However $S^{\mu}_{s}$ has a basis $\{ \psi_1 y_1, \psi_1 y_1 \psi_1 \} $ and is clearly not stable under $\partial$.
\end{rem}
The module $S_\mf{t}^{\mu} $ has a basis
\begin{equation*}
\{ \psi_{\mf{t}\mf{s}}^{\mu} + (\nh_n^l)^{> \mu}  | \mf{s} \in \mathrm{Tab}(\mu) \}.
\end{equation*}

\begin{prop}
\label{spechtsareiso}
For any $ \mf{t} \in \mathrm{Tab}(\mu)$ there is an isomorphism of $\nh_n^l$-modules 
$S^{\mu}_\mf{t} \cong S^{\mu}_e \langle -2\ell(\mf{t}) \rangle$.
\end{prop}
\begin{proof}
This is \cite[Proposition 7.21]{KQS}.
\end{proof}

\begin{rem}
Proposition \ref{spechtsareiso} cannot be a statement of $p$-DG modules since there is only one preferred representative in the isomorphism class which actually even carries a $p$-DG structure.
\end{rem}

Since $S^{\mu}_e$ is the preferred Specht module which carries a $p$-DG structure, we write $S^{\mu} := S^{\mu}_{e} $.
We now give another realization of this Specht module.
Let $ v^{\mu} $ be a formal basis vector spanning a ($p$-DG) $\Bbbk[y_1,\ldots,y_n]$-module where $v^{\mu}y_i=0$ for $i=1,\ldots,n$ and $\partial(v^{\mu})=0$.
%$= y_1^{l-j_1} \cdots y_n^{l-j_n} $
Define the module
\begin{equation*}
\widetilde{S}^{\mu} = \Bbbk v^{\mu} \otimes_{\Bbbk[y_1,\ldots,y_n]} \nh_n^l.
\end{equation*}

\begin{prop}
\label{spechtinducediso}
\begin{enumerate}
\item[(i)]The $\nh_n^l$-module $\widetilde{S}^{\mu}$ is isomorphic to the Specht module $S^{\mu}$.
Furthermore, this is an isomorphism of $p$-DG modules.
\item[(ii)]The Specht module $S^{\mu}$ is irreducible. In particular, up to isomorphism and grading shifts, the isomorphism class of the $p$-DG module $S^\mu$ is independent of the partition $\mu\in \mc{P}_n^l$.
\end{enumerate}
\end{prop}
\begin{proof}
This is proven in \cite[Proposition 7.23]{KQS} and \cite[Proposition 7.24]{KQS}
\end{proof}

\subsection{The modules \texorpdfstring{$G(\lambda)$}{G(lambda)}}
\label{subset-Glambda}
In this subsection, we recall the properties of a family of $p$-DG modules over $\nh_n^l$ generated by monomials, which are established in \cite{HuMathas} and \cite{KQS}. 

Once again, we fix a partition $\lambda$ with $ \lambda^{j_1}=\cdots=\lambda^{j_n}=1$ and $ j_1 < \cdots < j_n$.
Recall that
\begin{equation*}
y^{\lambda} = y_1^{l-j_1} \cdots y_n^{l-j_n}.
\end{equation*}
As in \cite{HuMathas} we define for each partition a cyclic module generated by this element.

\begin{defn}\label{def-G-lambda}
Let $\lambda\in \mc{P}_n^l$ be a partition. Define the module
\begin{equation*}
G(\lambda):=q^{ -nl+ j_1 + \cdots + j_n }y^{\lambda} \nh_n^l.
\end{equation*}
\end{defn}

\begin{prop}
\label{basisG}
The module $G(\lambda)$ has a basis
\begin{equation*}
\left\{ \psi_{\mf{ts}} | \mf{t} \in \mathrm{Tab}^{\lambda}(\mu), \mf{s} \in \mathrm{Tab}(\mu), \mu \in \mathcal{P}_n^l \right\}.
\end{equation*}
\end{prop}
\begin{proof}
See \cite[Theorem 4.9]{HuMathas}.
\end{proof}
Since $ \partial(y^{\lambda})=y^{\lambda}a$ for some $ a \in \Bbbk[y_1,\ldots,y_n]$,
the module $G(\lambda)$ is naturally a right $p$-DG module.

\begin{rem}\label{rmk-Spect-equals-G}
If $\lambda=(1^n 0^{l-n})$ then $G(\lambda) \cong q^{-nl+j_1+\cdots+j_n} S^{\lambda} $ as $p$-DG modules.
For further details see \cite[Corollary 7.28]{KQS}.
%By Proposition \ref{basisG}, since $\lambda$ is the maximal partition in the dominance order, the $p$-DG module $G(\lambda)$ has a basis labeled by $\mathrm{Tab}(\lambda)$. Thus it has dimension $n!$ and must be the Specht module up to grading shift. The claim now follows from the uniqueness of the $p$-DG structure on Specht modules (Proposition \ref{spechtinducediso}).
\end{rem}

\begin{prop}
\label{spechtfiltofG}
The module $G(\lambda)$ has a filtration whose subsequent quotients are isomorphic to Specht modules.
More specifically let $ \sqcup_{\mu} \mathrm{Tab}^{\lambda}(\mu) = \{ \mf{w}_1, \ldots, \mf{w}_m   \} $ be such that $ \mf{w}_i \geq \mf{w}_j$ whenever $i \leq j$.  Suppose the tableau $\mf{w}_i$ corresponds to a partition $\nu_i$.  Then $ G(\lambda)$ has a filtration
\begin{equation*}
G(\lambda)=G_m \supset G_{m-1} \supset \cdots \supset G_0=0
\end{equation*}
such that 
\begin{equation*}
G_i/G_{i-1} \cong q^{\ell(\mf{w}_i) -nl+(j_1+\cdots+j_n)}S^{\nu_i}.
\end{equation*}
\end{prop}
\begin{proof}
This is \cite[Corollary 4.11]{HuMathas}.
\end{proof}

\begin{prop}
The Specht filtration of $G(\lambda)$ in Proposition \ref{spechtfiltofG} is a filtration of $p$-DG modules.
\end{prop}
\begin{proof}
See \cite[Proposition 7.29]{KQS}.
\end{proof}

\begin{prop}\label{prop-G-self-dual}
There is a non-degenerate bilinear form on $G(\lambda)$. In particular, the graded dual of $G(\lambda)$ is isomorphic to itself.
\end{prop}

\begin{proof}
This is proved in \cite[Theorem 4.14]{HuMathas}.  It utilizes the Frobenius structure of $\nh_n^l$.
\end{proof}

%\begin{defn}\label{def-partition-idemp}
%For $\lambda=(0^{a_1} 1^{b_1} \dots 0^{a_k} 1^{b_k})\in \mc{P}_n^l$, we associate with it the sequence ${\bf b}:=(b_1,\dots, b_k)$ and define the idempotent $e_\lambda \in \nh_n^l$ by (see \eqref{eqn-sequence-idempotent})
%$$
%e_\lambda:=e_{\bf b}=
%\begin{DGCpicture}
%\DGCstrand[Green](0,0)(0,1)[$^{b_1}$]
%\DGCstrand[Green](0.5,0)(0.5,1)[$^{b_2}$]
%\DGCstrand[Green](1.5,0)(1.5,1)[$^{b_k}$]
%\DGCcoupon*(0.6,0.1)(1.4,0.9){$\cdots$}
%\end{DGCpicture} 
%\ .
%$$
%\end{defn}
%
%\begin{prop}
%\label{truncatedcyclicprop}
%Suppose $\lambda=(0^{a_1} 1^{b_1} \ldots 0^{a_r} 1^{b_r})\in \mc{P}_n^l$.
%Let $ \sum_{i=1}^r (a_i+b_i) =l$ and $\sum_{i=1}^r b_i =n$.
%Then the right module $e_{\lambda} y^{\lambda} \nh_n^l$ is a $p$-DG submodule of $G(\lambda)$.
%\end{prop}
%\begin{proof}
%This is \cite[Proposition 7.32]{KQS}.
%\end{proof}

%
%
%We call the module from Proposition \ref{truncatedcyclicprop} a {\it truncated cyclic module}.  These modules are in general decomposable.  Abstractly it is known that the cyclic modules $G(\lambda)$ could be written as a direct sum of indecomposable modules $Y(\mu)$ where $\mu \in \mc{P}_n^l$.
%
%\begin{prop}
%\label{defofY(lambda)}
%For each $\lambda \in \nh_n^l$, there is an indecomposable $\nh_n^l$-module $Y(\lambda)$ such that
%$$G(\lambda) \cong Y(\lambda) \bigoplus \left( \bigoplus_{\mu > \lambda} z_{\mu \lambda}(q) Y(\mu)\right)$$
%such that $ z_{\mu \lambda}(q) \in \N[q,q^{-1}]$.
%\end{prop}
%\begin{proof}
%This is \cite[Proposition 5.6]{HuMathas}.
%\end{proof}

\section{The quiver Schur algebra}
\label{sec-quiverschur}
The quiver Schur algebra associated to a KLR for a linear quiver was shown by Hu and Mathas \cite{HuMathas} to be quasi-hereditary.  We extend this result, for the $\mf{sl}_2$ case, in the context of a $p$-differential and show that the resulting $p$-DG quiver Schur algebra is $p$-DG quasi-hereditary.  By establishing the necessary conditions in Section \ref{sec-double}, we realize a categorification of tensor products of the fundamental representation for quantum $\mf{sl}_2$ at a prime root of unity via $p$-DG quiver Schur algebras.

\subsection{Definition and properties}
In this subsection we recall the definition of quiver Schur algebra for $\mf{sl}_2$ in the sense of Hu and Mathas \cite{HuMathas}. 

\begin{defn}\label{def-Schur-algebra}
The \emph{(graded) quiver Schur algebra} is by definition
\begin{equation}
S_n^l:=\END_{\nh_n^l}\left(\bigoplus_{\lambda\in \mc{P}_n^l} G(\lambda)\right).
\end{equation}
\end{defn}

We now recall a basis of this algebra from \cite[Section 4.2]{HuMathas}.
\begin{defn}\label{defcellularbasisofSchur}
For a pair of tableaux $\mf{t} \in \mathrm{Tab}^{\mu}(\lambda), \mf{s} \in \mathrm{Tab}^{\nu}(\lambda)$, there is a map of $\nh_n^l$-modules determined by
\begin{equation}\label{eqndefPsielements}
\Psi_{\mf{t}\mf{s}}^{\mu \nu} \colon G(\nu) \longrightarrow G(\mu)  \quad \quad
\Psi_{\mf{t}\mf{s}}^{\mu \nu}(y^{\nu}x):=\psi^{\lambda}_{\mf{t}\mf{s}}x.
\end{equation} 
for any $x\in \nh_n^l$.
\end{defn}

Note that there is a hidden dependence of $\Psi_{\mf{t}\mf{s}}^{\mu \nu}$ on $\lambda$ because $\mf{t}$ and $\mf{s}$ are tableaux of shape $\lambda$.  When we want to stress the dependence on $\lambda$ we write $\Psi^{\mu \nu}_{\mf{t}\mf{s} \lambda}$ instead.

The fact that $\Psi_{\mf{t}\mf{s}}^{\mu \nu}$ is a homomorphism of right $\nh_n^l$-modules follows from the structure of the cellular basis of cyclic modules (Proposition \ref{basisG}) and the fact that $\nh_n^l$ is a symmetric algebra (Proposition \ref{symm}) since, in this case,
\begin{equation*}
\HOM_{\nh_n^l}(G(\nu),G(\mu)) \cong G(\mu) \cap G^*(\nu).
\end{equation*}
For more details see \cite[Section 4.2]{HuMathas}.

Let us introduce some combinatorial parameter sets that play an important role in the study of cellular structure of $S_n^l$. Recall from Definition \ref{defTablambdamu} that $\Tab_\lambda(\mu)$ consists of $\mu$-tableaux that are greater than or equal to the standard tableaux $\mf{t}^\lambda$, and $\Tab_{\lambda}(\mu)$ consists of those that are less than or equal to the standard tableaux $\mf{t}^\lambda$.

\begin{defn}\label{defcellularparameterset}
For each $\lambda \in \mathcal{P}_n^l$, we set 
\begin{equation}
\label{qschurT(lambda)}
\T^\lambda=\bigsqcup_{\mu\in \mc{P}_n^l} \Tab^\mu(\lambda)=\{(\mu, \mf{s}) | \mu\in \mc{P}_n^l,~ \mf{s} \in \Tab^{\mu}(\lambda)   \},
\end{equation}
\begin{equation}
\label{qschurT(lambda)2}
\T_\lambda=\bigsqcup_{\mu\in \mc{P}_n^l} \Tab_\lambda(\mu)=\{\mf{s} |\mf{s} \in \Tab(\mu),~ \mu\in \mc{P}_n^l,~  \mf{s} \leq_{LR} \mf{t}^\lambda     \},
\end{equation}
where $\Tab^\mu(\lambda)$ and $\Tab_\lambda(\mu)$ are defined in Definition \ref{defTablambdamu}. 
\end{defn}

\begin{defn}\label{defcellularideals}
\begin{enumerate}
\item[(1)] Let $(S_n^l)^{\geq \lambda}$ be the $\Bbbk$-vector space spanned by
$$(S_n^l)^{\geq \lambda}:=\Bbbk\left\langle \Psi_{\mf{t}\mf{s}\gamma}^{\mu\nu} | \mu,\nu\in \T^\gamma,~\gamma\geq \lambda \right\rangle.$$
%$\{ \Psi_{\mf{t}\mf{s}}^{\mu \nu} \}$ for $\mf{t} \in \Tab^{\mu}(\gamma)$ and $\mf{s} \in \Tab^{\nu}(\gamma)$ with $\gamma \geq \lambda$.
\item[(2)] Likewise, let $(S_n^l)^{> \lambda}$ be the $\Bbbk$-vector space spanned by
$$(S_n^l)^{> \lambda}:=\Bbbk\left\langle \Psi_{\mf{t}\mf{s}\gamma}^{\mu\nu} | \mu,\nu\in \T^\gamma,~\gamma>\lambda \right\rangle.$$
%$\{ \Psi_{\mf{t}\mf{s}}^{\mu \nu} \}$ for $\mf{t} \in \Tab^{\mu}(\gamma)$ and $\mf{s} \in \Tab^{\nu}(\gamma)$ with $\gamma > \lambda$.
\end{enumerate}
\end{defn}

\begin{thm}
\label{cellularschuralgprop}
The algebra $S_n^l$ is a graded cellular algebra with cellular basis
\begin{equation*}
\left\{ \Psi_{\mf{t}\mf{s} \lambda}^{\mu \nu} | \mf{t} \in \Tab^{\mu}(\lambda),~ \mf{s} \in \Tab^{\nu}(\lambda),~ \lambda \in \mathcal{P}_n^l \right\}.
\end{equation*}
Here the parameter set $\mc{P}_n^l$ is equipped with the dominance order ``$>$'' of Definition \ref{def-order-on-nh-partition}. The collections of ideals $(S_n^l)^{\geq \lambda}$, $\lambda \in \mc{P}_n^l$ constitutes a chain of cellular ideals in $S_n^l$.
\end{thm}
\begin{proof}
This is a special case of \cite[Theorem 4.19]{HuMathas}. 
\end{proof}

The cellular structure will be recalled and studied in more detail in the next subsection in the $p$-DG setting. 

Using Lemma \ref{bijlem} one may give this collection of maps a different parametrization. This parametrization becomes more natural in terms of diagrammatics to be discussed in Section \ref{sec-Web} below.  

\begin{defn}\label{defcellularbasisofSchur2}
For a pair of tableaux $\mf{t}\in \Tab_\lambda(\mu)$, $\mf{s}\in \Tab_\lambda(\nu)$ , we define, as in Definition \ref{defcellularbasisofSchur}, a map of a map of $\nh_n^l$-modules
\begin{equation}
\Phi_{\mf{t}\mf{s}} \colon G(\nu)\lra G(\mu) \quad \quad \Phi_{\mf{t}\mf{s}}(y^\nu x):= \psi_{\mf{t}\mf{s}}^\lambda x
\end{equation}
for any $x\in \nh_n^l$. 
\end{defn}
Note that a $\mu$-tableau $\mf{t}$ might belong to multiple $\T_\lambda$.  When we want to emphasize the dependence of $\Phi_{\mf{t}\mf{s}} $ we will write $\Phi_{\mf{t}\mf{s}}^{\lambda}$ instead.

Also recall the notion of permissible permutations from Definition \ref{defallowablepermutation}.

\begin{lem}\label{lemPhiequalsPsi}
\begin{enumerate}
\item[(i)]Let $\mf{u}\in \Tab^\mu(\lambda)$ $\mf{v}\in \Tab^\nu(\lambda)$ and $\mf{t}\in \Tab_\lambda(\mu)$ $\mf{s}\in \Tab_\lambda(\nu)$. The maps $\Psi_{\mf{u}\mf{v}}^{\mu\nu}$ is equal to $\Phi^\mf{t}_\mf{s}$ if and only if $(\mf{u},\mf{v})$ defines the same pair of permissible permutations as $(\mf{t},\mf{s})$.
\item[(ii)] The ideal $(S_n^l)^{\geq \lambda}$ (resp.~$(S_n^l)^{>\lambda}$) is equal to the $\Bbbk$-linear span by $\{\Phi_{\mf{t}\mf{s}}|\mf{s},\mf{t}\in \T_\gamma,~\gamma\geq \lambda\}$ (resp.~$\{\Phi_{\mf{t}\mf{s}}|\mf{s},\mf{t}\in \T_\gamma,~\gamma\geq \lambda\}$).
\end{enumerate}

\end{lem}
\begin{proof}
The two maps agree if and only if $\psi_{\mf{u}}=\psi_{\mf{t}}$ and $\psi_{\mf{v}}=\psi_{\mf{s}}$, which happens if and only if their corresponding permissible permutations match. The second statement is just a combinatorial translation of Definition \ref{defcellularideals}.
\end{proof}

Hu and Mathas proved that $S_n^l$ is a quasi-hereditary algebra \cite[Theorem 5.24]{HuMathas} and so for each $\lambda \in \mathcal{P}_n^l$, there is a unique simple module up to isomorphism and grading shift.

\begin{defn}
\label{defofLandP}
The set of isomorphism classes of simple modules over $S_n^l$ (up to shift) is
\begin{equation*}
\{L(\lambda) | \lambda \in \mc{P}_n^l \}. 
\end{equation*}
Let $P(\lambda)$ be the projective cover of $L(\lambda)$ for each $\lambda \in \mc{P}_n^l$.
\end{defn}

\begin{example}
\label{n=1generall}
When $n=1$, the quiver Schur algebra $S_1^l$ is isomorphic to $A_l^!$ which is discussed as Case (3) of Examples \ref{egrunningegwithoutdif}, \ref{egrunningegwithdif} and \ref{egrunningegwithdifquasihereditary}.  This algebra has been studied with its $p$-DG structure in \cite{QiSussan}.  It is shown that there is a braid group action on the derived category of compact modules.
For more details on the isomorphism between $S_1^l$ and $A_l^!$ see \cite[Example 8.4]{KQS}.

We now describe the cellular structure on $S_1^l$.  The set 
$\mathcal{P}$ consists of multipartitions $\lambda_i$ for $i=1,\ldots,l$ where $\lambda_i$ has a single box in the $i$-th position, whose only tableaux will be denoted as
\[
i:=
\left(~\emptyset~, \ldots, ~\young(1)~, \ldots, ~\emptyset~ \right).
\]
Then $\T_{\lambda_k}=\{i|i\geq k\}$. The cellular basis of $S_1^l$ is given by
\[
\bigsqcup_{k=1}^n\left\{ \Phi_{ij}^{\lambda_k} | i,j\geq k \right\}, 
\]
The element $\Phi_{ij}^{\lambda_k}$ can be identified with $(i|i-1|\dots|k|\dots|j-1|j)$ in the quiver notation of Example \ref{egrunningegwithoutdif}.
\end{example}

\begin{rem}
\label{n=2l=3}
The quiver Schur algebra $S_2^3$ is the first example of a quiver Schur algebra which is not basic.  For more details see \cite[Example 8.5]{KQS}.
\end{rem}

\begin{defn}\label{defZlambda}
For each $\lambda \in \mc{P}_n^l$ there is a projective left $S_n^l$-module defined by
\begin{equation*}
Z(\lambda):= \HOM_{\nh_n^l}\left(G(\lambda), \bigoplus_{\mu \in \mc{P}_n^l} G(\mu)\right).
\end{equation*}
\end{defn}

There is a decomposition of each $Z(\lambda)$ in terms of indecomposable projective modules $P(\mu)$:
\begin{equation*}
Z(\lambda)=P(\lambda) \bigoplus \left(\bigoplus_{\mu > \lambda} z_{\mu \lambda}(q) P(\mu)\right)
\end{equation*}
for some Laurent polynomials $z_{\mu \lambda} \in \N[q,q^{-1}]$.

\subsection{\texorpdfstring{$p$}{p}-DG cellular structure}
Recall from Section \ref{subset-Glambda} that the cyclic modules $G(\lambda)$, $\lambda\in \mc{P}_n^l$, have natural $p$-DG structures inherited as right ideals of $\nh_n^l$. Therefore, the algebra $S_n^l$ inherits a $p$-DG structure from that of the $\nh_n^l$-modules $G(\lambda)$, i.e., if $ \Phi \in S_n^l$ then
\begin{equation}
(\partial \Phi)(x)=\partial(\Phi(x))-\Phi(\partial x).
\end{equation}
for all $x\in \oplus_{\lambda\in \mc{P}_n^l}G(\lambda)$. In this section, we will show that $S_n^l$ equipped with this differential makes it into a $p$-DG cellular algebra.

For each $\lambda$, we have that 
$$\Phi^\lambda_{\mf{t}^\lambda \mf{t}^\lambda}: G(\lambda)\lra G(\lambda)$$
is the identity $p$-DG map, and thus $\dif(\Phi^\lambda_{\mf{t}^\lambda \mf{t}^\lambda})=0$.  We will repeated use the element $\Phi^\lambda_{\mf{t}^\lambda \mf{t}^\lambda}$ below, and thus we simplify it as
\begin{equation}
\Phi^\lambda_{ee}:=\Phi^\lambda_{\mf{t}^\lambda \mf{t}^\lambda},
\end{equation}
where we have used $e$ as the identity permutation in the sense of Remark \ref{ntntableauxsymmetricgroup}.

As two-sided ideals, we have
\begin{equation}
(S_n^l)^{\geq \lambda}=\sum_{\mu \geq \lambda} S_n^l \cdot \Phi^{\mu}_{ee}\cdot S_n^l\quad \quad
(S_n^l)^{> \lambda}=\sum_{\mu > \lambda} S_n^l \cdot \Phi^{\mu}_{ee} \cdot S_n^l.
\end{equation}
Then both $(S_n^l)^{\geq \lambda}$ and $(S_n^l)^{> \lambda}$ are $\dif$-stable ideals.

\begin{defn} For each $\lambda\in \mc{P}_n^l$, set
\[
\Delta(\lambda):= \dfrac{S_n^l\cdot \Phi^\lambda_{ee}+ (S_n^l)^{>\lambda}}{(S_n^l)^{>\lambda}} \ ,
\quad \quad
\Delta^\circ(\lambda):= \dfrac{\Phi^\lambda_{ee} \cdot S_n^l +(S_n^l)^{>\lambda}}{ (S_n^l)^{>\lambda}} \ .
\]
\end{defn}

Since $(S_n^l)^{>\lambda}$ is $\dif$-stable and $\dif(\Phi^\lambda) = 0$, $\Delta(\lambda)$ is a left $p$-DG module over $S_n^l/(S_n^l)^{>\lambda}$ while $\Delta^\circ(\lambda)$ is a right $p$-DG module.

\begin{prop}
\label{qschurispdg}
The quiver Schur algebra $S_n^l$ is a $p$-DG cellular algebra.
\end{prop}

\begin{proof}
This now is a special case of Proposition \ref{propdifkilledidempotentideals}, since the chain of cellular ideals $\{S_n^{\geq \lambda}|\lambda\in \mc{P}_n^l\}$ are generated by $\dif$-annihilated idempotents $\{\Phi^\lambda|\lambda \in \mc{P}_n^l\}$.
\end{proof}

\begin{example}
In the notation there, the $p$-DG structure from the quiver Schur algebra in Example~\ref{n=1generall} is given by
\begin{equation}
\label{pdgforln=1}
\partial(\Phi_{ij}^{\lambda_k})=
(j-k) \Phi_{ij}^{\lambda_{k-1}}.
\end{equation}
\end{example}

Recall that for $\lambda \in \mathcal{P}_n^l$ a basis of $\Delta(\lambda)$ is given by the image of $\Phi^\lambda_{\mf{s}\mf{t}^\lambda}$ modulo $(S_n^l)^{>\lambda}$. Let us denote
\begin{equation}
\bar{\Phi}_\mf{s}^\lambda := \Phi^{\lambda}_{\mf{s}\mf{t}^\lambda} + (S_n^l)^{>\lambda} \in \Delta(\lambda).
\end{equation}

\begin{prop}
\label{cellnoncontract}
The $p$-DG $S_n^l$-cell module $\Delta(\lambda)$ is non-acyclic.
\end{prop}

\begin{proof}
For tableaux $\mf{s}, \mf{t}\in \T_\lambda$,
$\Phi_{\mf{t}\mf{s}}^\lambda$ is a map between certain cyclic $p$-DG $\nh_n^l$-modules labeled by the underlying partitions of $\mf{t}$ and $\mf{s}$, and thus $\dif(\Phi_{\mf{t}\mf{s}}^\lambda)$ remains a map between these same $p$-DG modules.

Since $\T_\lambda$ has only a single tableau $\mf{t}^\lambda$ whose partition is $\lambda$ (Remark \ref{rmktableauxcomments1} (ii)) and $\dif(\Phi^\lambda_{ee})=0$, $\Phi^\lambda_{ee}$ is the only cellular basis element for $\Delta(\lambda)$ of shape $\lambda$ and is also killed by the differential.  Thus this element is not in the image of $\partial$, and $\Delta(\lambda)$ is not acyclic.
\end{proof}

\begin{prop}
\label{Deltacofibrant}
The cell module
$\Delta(\lambda)$ is cofibrant over $S_n^l/(S_n^l)^{> \lambda}$.
\end{prop}

\begin{proof}
The elements in
\[
\left\{ \Phi^\lambda_{\mf{t}} | \mf{t} \in \T_\lambda \right\}
\]
constitute a basis of $\Delta(\lambda)$.
It is then readily seen that 
\[
\Delta(\lambda) \cong \left(S_n^l / (S_n^l)^{> \lambda} \right) \cdot \bar{\Phi}^{\lambda}_{ee},
\]
where ${\Phi}^{\lambda}_{ee}$ is an idempotent in $S_n^l$ and hence also is its image  $\bar{\Phi}^{\lambda}_{ee}\in S_n^l / (S_n^l)^{> \lambda} $.
Since $\partial({\Phi}^{\lambda}_{ee})=0$, it follows that 
$\Delta(\lambda)$ is cofibrant over $S_n^l / (S_n^l)^{> \lambda} $.
\end{proof}

\begin{thm}
\label{qschurpdgquasi}
The quiver Schur algebra $S_n^l$ is a $p$-DG quasi-hereditary cellular algebra.
\end{thm}

\begin{proof}
This follows from Propositions \ref{qschurispdg} and \ref{Deltacofibrant}.
\end{proof}

\begin{cor}
\label{simplenoncontract}
The simple $p$-DG module $L(\lambda)$ is non-acyclic.
\end{cor}

\begin{proof}
This follows from the proof of Proposition \ref{cellnoncontract} since the image of 
$\overline{\Phi}^{\lambda}$
in the quotient $L(\lambda)$ is non-zero under the canonical quotient map $\Delta(\lambda)\twoheadrightarrow L(\lambda)$.
\end{proof}

\subsection{The \texorpdfstring{$p$}{p}-DG envelope of cyclic modules}
In this section, we analyze the $p$-DG filtered envelope (Definition \ref{def-filtered-envelope}) of the $p$-DG module over $\nh_n^l$ 
$$G=\bigoplus_{\lambda\in\mc{P}_n^l}G(\lambda).$$

\begin{defn}
Set $\mathcal{G}$ to be the filtered $p$-DG envelope of $G$.
\end{defn}

We first introduce some auxiliary combinatorially defined modules which will play a role in the next subsection.
 
\begin{defn}\label{defdecompset}
Let $\mc{B}_n^l$ be the set of all decompositions of $n$ into $l$-tuple of natural numbers
\[
\mc{B}_n^l:=\left\{{\bf b}=[b_1,\dots, b_l]| b_1+\cdots+b_l=n,~ b_i\in \N,~i=1,\dots, l \right\}.
\] 
An element ${\bf b}\in \mc{B}_n^l$ will simply be referred to as a \emph{decomposition} if $n$ and $l$ are fixed.

The set $\mc{B}_n^l$ has the lexicographic order where ${\bf b} > {\bf c} $ if and only if there is an $i\in \{1,\dots, l\}$ such that
$b_k = c_k$ for all $k < i$ and $b_i>c_i$ in the $i$th position.

\end{defn}

Any partition $\lambda\in \mc{P}_n^l$ is naturally a decomposition with entries in $\{0,1\}$, and we may regard $\mc{P}_n^l$ as a subset of $\mc{B}_n^l$ in this way. With respect to this inclusion, the lexicographic order restricted to $\mc{P}_n^l$  is a refinement of the dominance order.

\begin{defn}\label{defdecompmodules}
For any decomposition ${\bf b}=[b_1, \ldots, b_l]\in \mc{B}_n^l$,
define the $p$-DG $\nh_n^l$-module 
\begin{equation}
\label{G[lambda]}
G[{\bf b}]:=
(y_1 \ldots y_{b_1})^{l-1} 
(y_{b_1+1} \ldots y_{b_1+b_2})^{l-2} 
\cdots
(y_{b_1+\cdots+b_{l-1}+1} \ldots y_{b_1+\cdots+b_l})^{0} 
\nh_n^l.
\end{equation}
We will refer to the generator $(y_1 \ldots y_{b_1})^{l-1} 
(y_{b_1+1} \ldots y_{b_1+b_2})^{l-2} 
\cdots
(y_{b_1+\cdots+b_{l-1}+1} \ldots y_{b_1+\cdots+b_l})^{0}$ of the module $G[{\bf b}]$ simply as $y^{\bf b}$ when needed.
\end{defn}

Our main goal in this section is to show that $G[{\bf b}]\in \mc{G}$. To do this we will introduce a technical reduction result.
For a fixed ${\bf b}\in \mc{B}_n^l$ whose $i+1$st term is equal to $b\in \N$, i.e., ${\bf b}=[b_1,\dots, b_{i},b,b_{i+2},\dots,b_l]$), we set 
\begin{equation}
{\bf b}^a:=[b_1,\dots, b_{i-1}, b_i+a,b-a,b_{i+2},\dots,b_l], \quad \quad 0\leq a\leq b.
\end{equation}
We also define the idempotent in $\nh_n^l$
\vspace{-0.1in}
\begin{equation}
e_{a,b-a}:~=~
\begin{DGCpicture}[scale=0.9]
\DGCstrand(-1,0)(-1,2)
\DGCcoupon*(-1,.5)(0,1.5){$\cdots$}
\DGCcoupon*(-1.4,-1)(0.4,0){$\underbrace{\hspace{1.2cm}}_{b_1+...+b_{i}}$}
\DGCstrand(0,0)(0,2)
\DGCstrand[Green](1,0)(1,2)[$^{a}$]
\DGCstrand[Green](2,0)(2,2)[$^{b-a}$]
\DGCstrand(3,0)(3,2)
\DGCcoupon*(3,.5)(4,1.5){$\cdots$}
\DGCstrand(4,0)(4,2)
\DGCcoupon*(2.6,-1)(4.4,0){$\underbrace{\hspace{1.2cm}}_{b_{i+2}+...+b_{l}}$}
\DGCcoupon*(2.8,2)(4.2,2.5){}
\end{DGCpicture} \ .
\end{equation}
Note that the idempotent puts $e_a$ on top of the the last $a$ strands in the $i$th group in the new decomposition ${\bf b}^a$, while capping the entire $i+1$st group by $e_{b-a}$ in ${\bf b}^a$.

The proof of the following result is motivated by the proof of \cite[Theorem 6]{Stosicsl3}.  

\begin{prop}
\label{stosicnilhecke}
There is a complex of $p$-DG modules over $\nh_n^l$
\begin{equation*}
\xymatrix{
0 \ar[r] &
e_{0,b} G[{\bf b}] \ar[r]^-{d_0} &
e_{1,b-1}  G[{\bf b}^{1}] \ar[r]^-{d_1} &
\cdots \ar[r]^-{d_{b-1}} &
e_{b,0} G[{\bf b}^{b}] \ar[r] &
0
} ,
\end{equation*}
which is contractible when regarded as a complex of $\nh_n^l$-modules.
Here the differential $d_a$ , $a=0,\dots, b-1$, is given by left multiplication by the element
\begin{equation*}
D_a=
\begin{DGCpicture}[scale=1.1]
\DGCstrand(-1,0)(-1,2)
\DGCcoupon*(-1,.5)(0,1.5){$\cdots$}
\DGCcoupon*(-1.4,-1)(0.4,0){$\underbrace{\hspace{1.2cm}}_{b_1+...+b_{i}}$}
\DGCstrand(0,0)(0,2)
%\DGCPLstrand(2,0.5)(1,1.5)
\DGCstrand/u/(2,.5)(1,1.5)/u/
\DGCdot{1}[d]{$^{b-1}$}
\DGCstrand[Green](1,0)(1,2)[$^{a}$`$_{a+1}$]
\DGCstrand[Green](2,0)(2,2)[$^{b-a}$`$_{b-a-1}$]
\DGCstrand(3,0)(3,2)
\DGCcoupon*(3,.5)(4,1.5){$\cdots$}
\DGCstrand(4,0)(4,2)
\DGCcoupon*(2.6,-1)(4.4,0){$\underbrace{\hspace{1.2cm}}_{b_{i+2}+...+b_{l}}$}
\DGCcoupon*(2.8,2)(4.2,2.5){}
\end{DGCpicture} \ ,
%
%~=~
%\begin{DGCpicture}
%\DGCstrand(-4,0)(-4,2)
%\DGCcoupon*(-3.5,.5)(-2.5,1.5){$\cdots$}
%\DGCcoupon*(-4,-1)(-2,0){$\underbrace{\hspace{.9in}}_{b_1+...+b_{i}}$}
%\DGCstrand(-2,0)(-2,2)
%\DGCstrand[Green](0,0)(0,2)[$^a$`$_{a+1}$]
%\DGCstrand[Green](2,0)(2,2)[$^{b_{i+1}-a}$`$_{b_{i+1}-1-a}$]
%\DGCPLstrand(0,1.5)(2,.5)
%\DGCdot{1.25}[r]{$^{b_{i+1}-1}$}
%\DGCstrand(4,0)(4,2)
%\DGCcoupon*(4.5,.5)(5.5,1.5){$\cdots$}
%\DGCstrand(6,0)(6,2)
%\DGCcoupon*(4,-1)(6,0){$\underbrace{\hspace{.9in}}_{b_{i+2}+...+b_{l}}$}
%\end{DGCpicture}
\end{equation*}
which is an element annihilated under $\partial$.
\end{prop}

\begin{proof}
Using equation \eqref{eqn-d-action-mod-generator}, it is easy to check that $\partial(D_a)=0$, so that left multiplication by $D_a$ is a map of $p$-DG modules.
We next check that $D_{a+1} \cdot D_a=0$, which shows that the above sequence forms a complex of $p$-DG modules over $\nh_n^l$. The computations are independent of the thin strands to the left and right of the diagrams, so we will not depict them in the diagrammatic computations.

\[
D_{a+1}D_a=~
\begin{DGCpicture}[scale=1.1]
\DGCstrand/u/(2,0.2)(1,0.8)/u/
\DGCdot{0.5}[d]{$^{b-1}$}
\DGCstrand/u/(2,1.2)(1,1.8)/u/
\DGCdot{1.5}[d]{$^{b-1}$}
\DGCstrand[Green](1,0)(1,2)[$^{a}$`$_{a+2}$]
\DGCstrand[Green](2,0)(2,2)[$^{b-a}$`$_{b-a-2}$]
\end{DGCpicture}
~=~
\begin{DGCpicture}[scale=1.1]
\DGCPLstrand(1.8,0.5)(1.6,1.2)(1.2,1.5)
\DGCdot{1.15}[u]{$^{b-1}$}
\DGCPLstrand(1.8,0.5)(1.4,0.8)(1.2,1.5)
\DGCdot{0.85}[d]{$^{b-1}$}
\DGCstrand[Green](1,0)(1,2)[$^{a}$`$_{a+2}$]
\DGCstrand[Green](2,0)(2,2)[$^{b-a}$`$_{b-a-2}$]
\DGCPLstrand[Green](1,1.75)(1.2,1.5)
\DGCPLstrand[Green](2,0.25)(1.8,0.5)
\end{DGCpicture}
~=0.
\]
Here in the second equality, we have used the sliding relation \eqref{eqn-thick-relation-sliding}. The last equality holds since, by sliding the symmetric polynomial up, we have
\[
\begin{DGCpicture}
\DGCstrand[Green](0,0)(0,0.25)[$^2$]
\DGCstrand(0,0.25)(-0.5,1)(0,1.75)
\DGCdot{1}[l]{$_{b-1}$}
\DGCstrand(0,0.25)(0.5,1)(0,1.75)
\DGCdot{1}[r]{$_{b-1}$}
\DGCstrand[Green](0,1.75)(0,2)
\end{DGCpicture}
~=~
\begin{DGCpicture}
\DGCstrand[Green](0,0)(0,0.5)[$^2$]
\DGCstrand(0,0.5)(-0.7,1)(0,1.5)
\DGCstrand(0,0.5)(0.7,1)(0,1.5)
\DGCstrand[Green](0,1.5)(0,2)
\DGCcoupon(-0.7,1.6)(0.7,1.9){$_{(y_1y_2)^{b-1}}$}
\end{DGCpicture}
~=0.
\]

For $a=0,\ldots, b_{}-1$, define homotopies 
$h_a \colon e_{a+1,b-a-1}G[{\bf b}^{a+1}]   \longrightarrow e_{a,b-a}G[{\bf b}^{a}] $ 
as left multiplication by the the element
\begin{equation*}
H_a
~=~(-1)^a
\begin{DGCpicture}[scale=1.1]
\DGCstrand(-1,0)(-1,2)
\DGCcoupon*(-1,.5)(0,1.5){$\cdots$}
\DGCcoupon*(-1.4,-1)(0.4,0){$\underbrace{\hspace{1.2cm}}_{b_1+...+b_{i}}$}
\DGCstrand(0,0)(0,2)
%\DGCPLstrand(1,0.5)(2,1.5)
\DGCstrand/u/(1,.5)(2,1.5)/u/
\DGCstrand[Green](1,0)(1,2)[$^{a+1}$`$_{a}$]
\DGCstrand[Green](2,0)(2,2)[$^{b-a-1}$`$_{b-a}$]
\DGCstrand(3,0)(3,2)
\DGCcoupon*(3,.5)(4,1.5){$\cdots$}
\DGCstrand(4,0)(4,2)
\DGCcoupon*(2.6,-1)(4.4,0){$\underbrace{\hspace{1.2cm}}_{b_{i+2}+...+b_{l}}$}
\DGCcoupon*(2.8,2)(4.2,2.5){}
\end{DGCpicture}
%\begin{DGCpicture}
%\DGCstrand(-4,0)(-4,2)
%\DGCcoupon*(-3.5,.5)(-2.5,1.5){$\cdots$}
%\DGCcoupon*(-4,-1)(-2,0){$\underbrace{\hspace{.9in}}_{b_1+...+b_{i}}$}
%\DGCstrand(-2,0)(-2,2)
%\DGCstrand[Green](0,0)(0,2)[$^{a+1}$`$_{a}$]
%\DGCstrand[Green](2,0)(2,2)[$^{b_{i+1}-1-a}$`$_{b_{i+1}-a}$]
%\DGCPLstrand(0,.5)(2,1.5)
%\DGCstrand(4,0)(4,2)
%\DGCcoupon*(4.5,.5)(5.5,1.5){$\cdots$}
%\DGCstrand(6,0)(6,2)
%\DGCcoupon*(4,-1)(6,0){$\underbrace{\hspace{.9in}}_{b_{i+2}+...+b_{l}}$}
%\end{DGCpicture}
\ .
\end{equation*}

We next check $ h_{a} \circ d_a + d_{a-1} \circ h_{a-1} =\mathrm{Id}$ for 
$a=0,\ldots, b$, (where we assume $d_{-1}=h_{-1}=d_{b}=h_{b}=0$).
We again ignore the thin strands on the left and right of the center of the diagrams and compute
\begin{align} \label{HaDa1}
H_{a} D_a + D_{a-1}  H_{a-1} 
 =
(-1)^a~
\begin{DGCpicture}[scale=0.85]
\DGCstrand/u/(2,0.2)(1,0.8)/u/
\DGCdot{0.5}[d]{$^{b-1}$}
\DGCstrand/u/(1,1.2)(2,1.8)/u/
\DGCstrand[Green](1,0)(1,2)[$^{a}$`$_{a}$]
\DGCstrand[Green](2,0)(2,2)[$^{b-a}$`$_{b-a}$]
\end{DGCpicture}
+
(-1)^{a-1}~
\begin{DGCpicture}[scale=0.85]
\DGCstrand/u/(2,1.2)(1,1.8)/u/
\DGCdot{1.5}[d]{$^{b-1}$}
\DGCstrand/u/(1,.2)(2,.8)/u/
\DGCstrand[Green](1,0)(1,2)[$^{a}$`$_{a}$]
\DGCstrand[Green](2,0)(2,2)[$^{b-a}$`$_{b-a}$]
\end{DGCpicture} ~ =
(-1)^{a}~
\begin{DGCpicture}[scale=0.85]
\DGCstrand/u/(2,0.2)(1,1.5)/u/
\DGCdot{0.8}[d]{$^{b-1}$}
\DGCstrand/u/(1,.5)(2,1.8)/u/
\DGCstrand[Green](1,0)(1,2)[$^{a}$`$_{a}$]
\DGCstrand[Green](2,0)(2,2)[$^{b-a}$`$_{b-a}$]
\end{DGCpicture}
+
(-1)^{a-1}~
\begin{DGCpicture}[scale=0.85]
\DGCstrand/u/(2,.5)(1,1.8)/u/
\DGCdot{1.2}[u]{$^{b-1}$}
\DGCstrand/u/(1,.2)(2,1.5)/u/
\DGCstrand[Green](1,0)(1,2)[$^{a}$`$_{a}$]
\DGCstrand[Green](2,0)(2,2)[$^{b-a}$`$_{b-a}$]
\end{DGCpicture}
\end{align}
where the second equality above follows from the sliding relation \eqref{eqn-thick-relation-sliding}.
Repeated use of the dot sliding relations in \eqref{nicheckepicrelations} implies that
\begin{equation}
H_{a}  D_a + D_{a-1} H_{a-1} 
~=~
(-1)^{a-1} \sum_{r+s=b-2} ~
\begin{DGCpicture}
\DGCstrand(1,.25)(1.4,1)(1,1.75)
\DGCdot{1}[ur]{$^{r}$}
\DGCstrand(2.5,.25)(2.1,1)(2.5,1.75)
\DGCdot{1}[ul]{$^{s}$}
\DGCstrand[Green](1,0)(1,2)[$^{a}$`$_{a}$]
\DGCstrand[Green](2.5,0)(2.5,2)[$^{b-a}$`$_{b-a}$]
\end{DGCpicture}
~=~
\begin{DGCpicture}
\DGCstrand[Green](1,0)(1,2)[$^{a}$`$_{a}$]
\DGCstrand[Green](2.5,0)(2.5,2)[$^{b-a}$`$_{b-a}$]
\end{DGCpicture} \ .
\end{equation}
Here the second equality follows from the pairing relation \eqref{eqn-thick-pairing}, which implies that the summands in the middle term are zero unless $r=a-1$ and $s=b-a-1$. In the latter case, one could erase the thin black lines and multiply the expression by $(-1)^{a-1}$.
The homotopy formula follows.
\end{proof}

\begin{example}
Let
${\bf b}=(0,\ldots,0,3)$.
Then there is a contractible complex of $p$-DG $\nh_3^l$-modules
\begin{equation*}
\xymatrix{
0
\ar[r]
&
\begin{DGCpicture}[scale=0.8]
\DGCstrand[Green](0,0)(0,1)[`$_3$]
\DGCstrand/u/(-1,-1)(0,0)/u/
\DGCstrand/u/(0,-1)(0,0)/u/
\DGCstrand/u/(1,-1)(0,0)/u/
\DGCcoupon(-1.2,-1.5)(1.2,-1){$_{\nh_3^l}$}
\end{DGCpicture} 
\ar[r]^{d_0}
&
\begin{DGCpicture}[scale=0.8]
\DGCstrand(-1,0)(-1,1)[`$_1$]
\DGCstrand[Green](.5,0)(.5,1)[`$_2$]
\DGCstrand/u/(-1,-1)(-1,0)/u/
\DGCdot{-.75}
\DGCstrand/u/(0,-1)(.5,0)/u/
\DGCstrand/u/(1,-1)(.5,0)/u/
\DGCcoupon(-1.2,-1.5)(1.2,-1){$_{\nh_3^l}$}
\end{DGCpicture} 
\ar[r]^{d_1}
\ar@<1ex>[l]^{h_0}
&
\begin{DGCpicture}[scale=0.8]
\DGCstrand[Green](-.5,0)(-.5,1)[`$_2$]
\DGCstrand(1,0)(1,1)[`$_1$]
\DGCstrand/u/(-1,-1)(-.5,0)/u/
\DGCdot{-.75}
\DGCstrand/u/(0,-1)(-.5,0)/u/
\DGCdot{-.75}
\DGCstrand/u/(1,-1)(1,0)/u/
\DGCcoupon(-1.2,-1.5)(1.2,-1){$_{\nh_3^l}$}
\end{DGCpicture} 
\ar[r]^{d_2}
\ar@<1ex>[l]^{h_1}
&
\begin{DGCpicture}[scale=0.8]
\DGCstrand[Green](0,0)(0,1)[`$_3$]
\DGCstrand/u/(-1,-1)(0,0)/u/
\DGCdot{-.75}
\DGCstrand/u/(0,-1)(0,0)/u/
\DGCdot{-.75}
\DGCstrand/u/(1,-1)(0,0)/u/
\DGCdot{-.75}
\DGCcoupon(-1.2,-1.5)(1.2,-1){$_{\nh_3^l}$}
\end{DGCpicture} 
\ar@<1ex>[l]^{h_2}
\ar[r]
&
0}
\end{equation*}
where the differentials $d_i$ (which are $p$-DG maps) and the homotopies $h_i$, $i=0,1,2$, are given by the left multiplication on the modules by the corresponding elements
\begin{equation*}
D_0
=
\begin{DGCpicture}[scale=0.75]
\DGCstrand(0,1)(-1,2)[`$_1$]
\DGCdot{1.5}[u]{$_2$}
\DGCstrand[Green](0,0)(0,1)[$^3$]
\DGCstrand[Green](0,1)(1,2)[`$_{2}$]
\end{DGCpicture} \ ,
\quad \quad \quad
D_1
=
\begin{DGCpicture}[scale=0.75]
\DGCstrand[Green](2,0)(2,0.5)[$^2$]
\DGCstrand[Green](0,1.5)(0,2)[`$_{2}$]
\DGCstrand(2,0.5)(2,2)[`$_1$]
\DGCstrand(0,0)(0,1.5)[$^1$]
\DGCstrand/u/(2,0.5)(0,1.5)/u/
\DGCdot{1}[r]{$^2$}
\end{DGCpicture} \ ,
\quad \quad \quad
D_2
=
\begin{DGCpicture}[scale=0.75]
\DGCstrand(1,0)(0,1)[$^1$]
\DGCdot{0.5}[r]{$_2$}
\DGCstrand[Green](0,1)(0,2)[`$_3$]
\DGCstrand[Green](-1,0)(0,1)[$^{2}$]
\end{DGCpicture} \ ,
\end{equation*}
\begin{equation*}
H_0
=
\begin{DGCpicture}[scale=0.75]
\DGCstrand(-1,-2)(0,-1)[$^1$]
\DGCstrand[Green](0,-1)(0,0)[`$_3$]
\DGCstrand[Green](1,-2)(0,-1)[$^{2}$]
\end{DGCpicture} \ ,
\quad \quad \quad
H_1
=
-
\begin{DGCpicture}[scale=0.75]
\DGCstrand[Green](2,-0.5)(2,0)[`$_2$]
\DGCstrand[Green](0,-2)(0,-1.5)[$^{2}$]
\DGCstrand(2,-2)(2,-0.5)[$^1$]
\DGCstrand(0,-1.5)(0,0)[`$_1$]
\DGCstrand/u/(0,-1.5)(2,-0.5)/u/
\end{DGCpicture} \ ,
\quad \quad \quad
H_2
=
\begin{DGCpicture}[scale=0.75]
\DGCstrand(0,-1)(1,0)[`$_1$]
\DGCstrand[Green](0,-2)(0,-1)[$^3$]
\DGCstrand[Green](0,-1)(-1,0)[`$_{2}$]
\end{DGCpicture} \ .
\end{equation*}
\end{example}

\begin{prop}
\label{auxmodsinenv}
Let ${\bf b}$ be a decomposition in $\mc{B}_n^l$.  Then
each $G[{\bf b}]$ is in the filtered $p$-DG envelope $\mc{G}$.
\end{prop}

\begin{proof}
Recall that to the decomposition ${\bf b}$ of $n$, there is a $p$-DG idempotent
$e_{{\bf b}}\in \nh_n^l$ given by equation~\eqref{eqn-sequence-idempotent}.  Since the generator $y^{\bf b}$ (Definition \ref{defdecompmodules}) commutes with the idempotent $e_{\bf b}$, the module $G[{\bf b}]$ has an $\nh_n^l$-split finite $p$-DG filtration whose subquotients are isomorphic to $ e_{{\bf b}} G[{\bf b}]$.  Thus, by repeated use of Proposition  
\ref{propfilteredenvelope2outof3}, it suffices to show that $ e_{{\bf b}} G[{\bf b}] \in \mc{G}$.

Let $\iota({\bf b})$ be the maximal $i$ such that $b_i >1$. We proceed by induction on $\iota({\bf b})$. If no such $i$ exists, then ${\bf b}$ is a partition, and $e_{\bf b}G[{\bf b}]=G[{\bf b}]=G({\bf b})$. Then $e_{\bf b}G[{\bf b}]$ is in $\mc{G}$ by definition.

If $\iota({\bf b})=1$ then $b_1:=b\geq 2$.
The module $e_{{\bf b}}G[{\bf b}]$ is generated by $e_{\bf b}y^{\bf b}$, whose left most portion looks like
\[
\begin{DGCpicture}
\DGCstrand[Green](1.5,0)(1.5,2)[`$_b$]
\DGCcoupon(0.5,0.6)(2.5,1.4){$_{(y_1\dots y_{b})^{l-1}}$}
\DGCcoupon*(2.5,0.5)(3,1.5){$\cdots$}
\end{DGCpicture}
=
\begin{DGCpicture}
\DGCstrand[Green](1.5,1.8)(1.5,2)[`$_b$]
\DGCstrand[Green](1.5,0)(1.5,0.2)
\DGCstrand(1.5,0.2)(0.2,1)(1.5,1.8)
\DGCdot{1}[r]{$_{l+b-2}$}
\DGCstrand(1.5,0.2)(2.8,1)(1.5,1.8)
\DGCdot{1}[r]{$_{l-1}$}
\DGCcoupon*(1,0.5)(2.8,1.5){$\cdots$}
\DGCcoupon*(3.6,0.5)(4.1,1.5){$\cdots$}
\end{DGCpicture}
=0.
\]
The last equality holds by the cyclotomic relation \eqref{eqn-NH-cyclotomicrelation}. Thus $G[{\bf b}]$ is trivially in $\mc{G}$.

Assume that $e_{{\bf c}} G[{\bf c}] \in \mc{G}$ for $\iota({\bf c}) \leq i$. Furthermore suppose that $\iota({\bf b})=i+1$.

By Proposition \ref{propfilteredenvelope2outof3}, the $p$-DG module $e_{{\bf b}} G[{\bf b}]$ lies in $\mc{G}$ if each auxiliary module 
\begin{equation}
\label{smallerG1}
e_{(b_1,\ldots, b_i,a,b_{i+1}-a,\ldots,b_l)} G[b_1, \ldots, b_i+a,b_{i+1}-a,\ldots,b_l]
\end{equation}
appearing in Proposition~\ref{stosicnilhecke} is in the envelope for $a=1,\ldots,b_{i+1}$.
Note that, by the identity decomposition relation \eqref{eq-identitydecomp} applied to the idempotent $e_{b_i,a}$, the object in \eqref{smallerG1} has a filtration whose subquotients are isomorphic to
\begin{equation}
\label{smallerG2}
e_{{\bf d}_a}G[{\bf d}_a]:=e_{(b_1,\ldots, b_i+a,b_{i+1}-a,\ldots,b_l)} G[b_1, \ldots, b_i+a,b_{i+1}-a,\ldots,b_l].
\end{equation}
We claim that the object \eqref{smallerG2} is in $\mc{G}$ for each $a$. Then, by Proposition \ref{propfilteredenvelope2outof3} again, the $p$-DG module \eqref{smallerG1} is in the envelope.

Note that $\iota({\bf d}_a) \leq \iota({\bf b})$.  We will prove the claim by another induction for all decompositions ${\bf b}$ satisfying $\iota(b)=i+1$
on the lexicographic order (see Definition \ref{defdecompset}).  Let ${\bf d}=[d_1,\dots, d_{i+1},\dots, d_l]$ be such a decomposition. When $d_{i+1}=2$, 
${\bf d}$ is the smallest term in the lexicographic order in the collection of decompositions with $\iota$ values equal to $i+1$. Then Proposition~\ref{stosicnilhecke} 
and the identity decomposition relation \eqref{eq-identitydecomp} shows that $e_{\bf d}G[{\bf d}]$ is contained in the filtered $p$-DG envelope of 
\[
e_{(d_1,\dots, d_i+1,1,\dots, d_l)}G[d_1,\dots, d_i+1,1,\dots, d_l] \quad \textrm{and}  \quad e_{(d_1,\dots, d_i+2,0,\dots, d_l)}G[d_1,\dots, d_i+2,0,\dots, d_l] .
\]
Both terms have $\iota $ values less than or equal to $i$, and thus are in $\mc{G}$ by the previous induction hypothesis for modules with $\iota$ values 
less than $i+1$. Then  $e_{\bf d}G[{\bf d}]\in \mc{G}$ for $d_{i+1}=2$ by Proposition \ref{propfilteredenvelope2outof3}. When $d_{i+1}>2$, we consider the 
complex of Proposition \ref{stosicnilhecke} with $e_{\bf d}G[{\bf d}]$ appearing as the left most term. All the other objects appearing in the exact complex 
other than $e_{\bf d}G[{\bf d}]$ are either strictly greater than ${\bf d}$ in the lexicographic order, or has $\iota$ value equal to $i$. Thus they are already in $\mc{G}$ by the previous and new induction 
hypotheses.
Hence $e_{\bf d}G[{\bf d}]\in \mc{G}$ by Proposition \ref{propfilteredenvelope2outof3} again. This finishes the proof of the claim, and thus the proposition follows.
\end{proof}

The proof of the proposition also implies the following.

\begin{cor}\label{corauxiliarytruncmodinG}
For each decomposition ${\bf b}\in \mc{B}_n^l$, the
 module $e_{\bf b}G[{\bf b}]$ is in the filtered $p$-DG envelope $\mc{G}$. \hfill$\square$
\end{cor}

\subsection{Functors \texorpdfstring{$\mf{E}$}{E} and \texorpdfstring{$\mf{F}$}{F}}
Our main aim in this subsection is to show that restriction and induction functors $\mf{E}$ and $\mf{F}$ lift to functors on the category of modules for the quiver Schur algebra. 

First we show that the collection of cyclic modules $G(\lambda)$'s contains a faithful $\nh_n^l$-representation.

\begin{prop}\label{propGisfaithful}
The module $G(\lambda_0)$, where $\lambda_0$ is the minimal partition (Definition \ref{def-order-on-nh-partition}), is a faithful $p$-DG $\nh_n^l$-module.
\end{prop}
\begin{proof}
The module $G(\lambda_0)$ is generated by the monomial $y^{\lambda_0}=y_1^{n-1}y_2^{n-2}\cdots y_{n-1}$. Notice that, since
\[
e_n=y^{\lambda_0}\psi_{w_0}\in G(\lambda_0),
\] 
the projective module $e_n\nh_n^l$ is a $p$-DG submodule in $G(\lambda_0)$. Since $\nh^l_n$ is Frobenius, $e_n\nh_n^l$ is also injective, and thus is a direct summand of $G(\lambda_0)$. The faithfulness follows since $e_n\nh_n^l$ is a faithful $\nh_n^l$-module.
% Notice that there is a direct sum decomposition of $\nh_n^l$-modules
%\[
%G(\lambda_0)=y^{\lambda_0}\nh_n^l\cong e_ny^{\lambda_0}\nh_n^l\oplus (1-e_n)y^{\lambda_0}\nh_n.
%\]
%The submodule $e_ny^{\lambda_0}\nh_n^l$ is isomorphic to the projective module $e_n\nh_n^l$ since we have
%\[
%e_ny^{\lambda_0}\psi_{w_0}=e_n^2=e_n.
%\]
%Therefore  the action of $\nh_n^l$ is faithful. Under the differential the above direct sum decomposition becomes, instead, a filtration of $p$-DG modules 
%\[
%0\lra e_ny^{\lambda_0}\nh_n^l \lra G(\lambda_0)\lra (1-e_n)y^{\lambda_0}\nh_n \lra 0.
%\]
%The faithfulness of the action is unaffected. The result follows.
\end{proof}

Next, we show that $\mf{E}$ and $\mf{F}$ preserve the filtered $p$-DG envelope of $G(\lambda)$'s, so that the general framework of Theorem \ref{thm-pdg-extension} applies in our situation. 
%To do this let us introduce an auxiliary family of cyclic modules that will generate the an equivalent filtered $p$-envelope as the $G(\lambda)$'s, but that are better behaved under the functors $\mf{E}$ and $\mf{F}$.

\begin{prop}
\label{FpreservesGs}
For all $d\in \N$, the functors $\mf{F}^{(d)}$ preserve the filtered $p$-DG envelope $\mc{G}$.
\end{prop}

\begin{proof}
Let $\lambda=(\lambda_1,\ldots,\lambda_l)\in \mc{P}_n^l$ where each $\lambda_j \in \{0,1\}$.
Then we have
\begin{equation}\label{eqn-FG}
\mf{F}^{(d)}G(\lambda) \cong e_{(1^n,d)}G[\lambda_1,\ldots,\lambda_l+d].
\end{equation}
The latter module is in $\mc{G}$ by Proposition \ref{auxmodsinenv} and Corollary \ref{corauxiliarytruncmodinG}
\end{proof}

\begin{rem}
Note that if $\lambda_l=0$ then 
%$\mf{F}^d G(\lambda) \cong G(\lambda_1,\ldots, \lambda_{l-1},1^d) $ 
$\mf{F} G(\lambda) \cong G(\lambda_1,\ldots, \lambda_{l-1},1) $ 
and the proposition above is immediate.
\end{rem}

%\begin{prop}
%\label{FpreservesKar}
%Let $\lambda=(0^{a_1} 1^{b_1} \ldots 0^{a_r} 1^{b_r})$ and let ${\bf b}=(b_1, %\ldots, b_r)$.  Then $\mf{F}G^t(\lambda)$ has a $p$-DG filtration with %subquotients 
%\begin{equation*}
%q^{-b_r} e_{{\bf b}'}G(\lambda^\prime), \quad
%q^{-b_r+2} e_{{\bf b}'}G(\lambda^\prime), \quad \ldots, \quad
%q^{b_r} e_{{\bf b}'}G(\lambda^\prime) \nh_{n+1}^l
%\end{equation*}
%\begin{equation*}
%\mf{F}e_{\bf b} y^{\lambda} \nh_n^l \cong [b_r+1] e_{{\bf b}'} y^{\lambda'} %\nh_{n+1}^l
%\end{equation*}
%where $\lambda^\prime=(0^{a_1} 1^{b_1} \ldots 0^{a_r-1} 1^{b_r+1})$
%and
%${\bf b}'=(b_1, \ldots, b_{r-1}, b_r+1)$.
%\end{prop}

%\begin{proof}
%Let $\hat{\lambda}=(0^{a_1} 1^{b_1} \ldots 0^{a_{r-1}} 1^{b_{r-1}}  0^{a_r} %0^{b_r})$ and
%$\hat{{\bf b}}=(b_1,\ldots,b_{r-1})$.  Then
%\begin{equation*}
%G^t(\lambda) \cong \mf{F}^{(b_r)} e_{\hat{\bf{b}}}y^{\hat{\lambda}}\nh_{n-b_r}^l %\cong \mf{F}^{(b_r)} G^t({\hat{{\bf b}}}) .
%\end{equation*}
%Using relations in the $p$-DG $2$-category $\mathcal{U}$, we get
%\begin{align*}
%\mf{F}G^t(\lambda)  & \cong
%\mf{F} \mf{F}^{(b_r)} G^t(\hat{\lambda}) \\
%& \cong [b_r+1] \mf{F}^{(b_r+1)} G^t(\hat{\lambda}) \\ 
%& \cong [b_r+1] e_{{\bf b}'}G(\lambda^\prime),
%\end{align*}
%where the second and third lines above should be understood as filtered $p$-DG %modules with $b_r+1$ subquotients in various degrees.
%\end{proof}

It is useful to compare module categories for different cyclotomic conditions.  
Let $\delta\leq l$ be a positive integer. Consider the nilHecke algebras $\nh_n^l$ and $\nh_n^{l-\delta}$. The $p$-DG cyclic $\nh_n^l$-module generated by $(y_1\cdots y_n)^\delta$ is a $p$-DG bimodule over $\nh_n^l$ since $(y_1\cdots y_n)^\delta$ is a central element in $\nh_n^l$. Furthermore, the action of $\nh_n^l$ on the module factors through the canonical $p$-DG algebra surjection
\begin{equation}
\pi_\delta \colon \nh_n^l \longrightarrow \nh_n^{l-\delta}.
\end{equation}
since, on the module $(y_1\cdots y_n)^\delta \nh_n^l$, $y_1^{l-\delta}$ always acts as zero. Therefore we may think of $(y_1\cdots y_n)^\delta \nh_n^l$ as a $p$-DG bimodule over $(\nh_n^{l-\delta},\nh_n^l)$.

\begin{defn}\label{defpullbackfunc}
For each $0 \leq \delta \leq l$, define the $p$-DG functor
\begin{equation*}
\mf{P}_\delta \colon (\nh_n^{l-\delta},\dif) \dmod \lra (\nh_n^{l},\dif) \dmod,
\quad \quad \quad
\mf{P}_\delta M = M \otimes_{\nh_n^{l-\delta}}(y_1 \cdots y_n)^{\delta} \nh_n^l.
\end{equation*}
%\begin{equation*}
%\Pi_n^{l-\delta} \colon \nh_n^{l-\delta} \dmod \rightarrow \nh_n^{l} \dmod
%\end{equation*}
%\begin{equation*}
%\Pi_n^{l-\delta} M = (y_1 \cdots y_n)^{\delta} M.
%\end{equation*}
\end{defn}

\begin{lem}
\label{PappliedtoG}
Let $\lambda=(\lambda_1, \ldots , \lambda_{l-\delta}, 0^{\delta})$ and 
$\lambda'=(\lambda_1, \ldots, \lambda_{l-\delta})$.  Then there is an isomorphism of $p$-DG modules
\begin{equation*}
\mf{P}_\delta G(\lambda') \cong G(\lambda).
\end{equation*}
\end{lem}

\begin{proof}
We start by observing that there is an isomorphism
$\nh_n^{l-\delta} \cong (y_1 \cdots y_n)^{\delta} \nh_n^l$
as left $\nh_n^{l-\delta}$-modules. This can be seen as follows.
Define a left $\nh_n^{l-\delta}$-linear map $\phi \colon \nh_n^{l-\delta} \longrightarrow (y_1 \cdots y_n)^{\delta} \nh_n^l $
by setting $\phi(1) :=(y_1 \cdots y_n)^{\delta}$.
A cellular basis element of $\nh_n^{l-\delta}$ is of the form 
$\psi_{\mf{s}}^* y^{\lambda'} \psi_{\mf{t}}$ where $y^{\lambda'}=y_1^{a_1} \cdots y_n^{a_n}$
with $l-\delta > a_1 > \cdots > a_n \geq 0$ and $\mf{s}, \mf{t} \in \mf{S}_n$.
Then 
\begin{equation*}
\phi \colon \psi_{\mf{s}}^* y^{\lambda'} \psi_{\mf{t}} \mapsto 
(y_1 \cdots y_n)^{\delta} \psi_{\mf{s}}^* y^{\lambda'} \psi_{\mf{t}} =
\psi_{\mf{s}}^* y^{\lambda} \psi_{\mf{t}}
\end{equation*}
where $y^{\lambda}=y_1^{b_1} \cdots y_n^{b_n}$ with
$l > b_1 > \cdots > b_n \geq \delta$.  
The elements $\psi_{\mf{s}}^* y^{\lambda} \psi_{\mf{t}}$ are part of the cellular basis of $\nh_n^l$ and hence are linearly independent.  Again, by considering the cellular basis of $\nh_n^l$, elements of the form $\psi_{\mf{s}}^* y^{\lambda} \psi_{\mf{t}}$ span the space $ (y_1 \cdots y_n)^{\delta} \nh_n^l $.

It follows that $\mf{P}_\delta$ is an exact functor on the category of right $\nh_n^{l-\delta}$ modules. Consider the inclusion of right ideals
$
G(\lambda^\prime) \hookrightarrow \nh_n^{l-\delta}.
$
By identifying 
$$\mf{P}_\delta \nh_n^{l-\delta}= \nh_n^{l-\delta} \otimes_{\nh_n^{l-\delta}} (y_1\cdots y_n)^\delta \nh_n^l \cong (y_1\cdots y_n)^\delta \nh_n^l,$$
where the last isomorphism is given by the multiplication map, we have
\begin{align*}
\mf{P}_\delta G(\lambda') &= 
G(\lambda') \otimes_{\nh_n^{l-\delta}} (y_1\cdots y_n)^\delta \nh_n^l \\
&= y^{\lambda'} \nh_n^{l-\delta} \otimes_{\nh_n^{l-\delta}} (y_1\cdots y_n)^\delta \nh_n^l \\
& \cong  y^{\lambda^\prime} \cdot (y_1\cdots y_n)^\delta  \nh_n^l \\
& \cong G(\lambda), 
\end{align*}
where the second to last isomorphism above is again given by the multiplication map.

Note that the isomorphism sends 
\begin{equation}
y^{\lambda'} \otimes (y_1\cdots y_n)^\delta \mapsto y^{\lambda}=y^{\lambda'} (y_1\cdots y_n)^\delta 
\ ,
\end{equation}
which is clearly compatible with the $p$-DG structure.
\end{proof}
We immediately get the following useful fact.
\begin{cor}
The functor $\mf{P}_\delta$ preserves the $p$-DG filtered envelope of the $G(\lambda)$'s. \hfill$\square$
\end{cor}

\begin{prop}
\label{PandFcommute}
For any $d\geq 1$, there is an isomorphism of $p$-DG functors
$ \mf{P}_\delta \circ \mf{E}^{(d)} \cong \mf{E}^{(d)} \circ \mf{P}_\delta$.
\end{prop}

\begin{proof}
We first show the case when $d=1$. It suffices to show that the $p$-DG bimodules representing the composition functors are isomorphic.
We will utilize the diagrammatic interpretation of the functor $\mf{E}$ given in Section \ref{subsec-cat-simples}, in particular equation \eqref{eqnrestrictionfunctor}. 
%First note that $\nh_n^l$ has a basis of the form
%\begin{equation} \label{othernhbasis}
%\{ y_1^{a_1} \cdots y_n^{a_n} \psi_w  | a_i \leq l-i, w \in S_n \}
%\end{equation}
%by moving the monomial in a cellular basis element $\psi_{w'} y^{\lambda} \psi_w  %$ all the way to left and noting that 

Diagramatically, the $p$-DG bimodule representing the functor $\mf{P}_\delta \circ \mf{E}$ is depicted by
\begin{equation} \label{PiE1}
\mf{P}_\delta \circ \mf{E}(\nh_n^{l-\delta})\cong
\begin{DGCpicture}[scale=0.8]
\DGCstrand(0,-1.75)(0,1.5)
\DGCdot*>{1.5}
\DGCstrand(0.5,-1.75)(0.5,1.5)
\DGCdot*>{1.5}
\DGCstrand(3.5,-1.75)(3.5,1.5)
\DGCdot*>{1.5}
\DGCstrand(4,0.5)(4,1.5)
\DGCdot*>{1.5}
\DGCstrand/d/(4,.25)(4.5,0.25)/u/(4.5,1.5)/u/
\DGCcoupon(-0.1,0.25)(4.1,1){$_{\nh_n^{l-\delta}}$}
\DGCcoupon*(0,-0.5)(4,0.25){$\cdots$}
\DGCcoupon*(0,-1.75)(4,-1.25){$\cdots$}
\DGCcoupon*(0,1.2)(4,1.4){$\cdots$}
\DGCcoupon(-0.1,-1.25)(3.6,-0.5){$_{(y_1 \cdots y_{n-1})^{\delta} \nh_{n-1}^{l}}$}
\end{DGCpicture} \ .
\end{equation}
Recall from Proposition \ref{dotbottombasis} (ii) that,
as a right module over $\nh_{n-1}^{l-\delta}$, the nilHecke algebra
$\nh_n^{l-\delta}$ has rank $n(l-\delta-n+1)$ with a basis 
\begin{equation} 
\left\{ \begin{DGCpicture}
\DGCstrand(0,0.5)(0,1.5)[`$_1$]
\DGCdot*>{1.5}
\DGCstrand(1,0.5)(1,1.5)[`$_{i-1}$]
\DGCdot*>{1.5}
\DGCstrand(3,0.5)(3,0.75)(3.5,1.5)[`$_n$]
\DGCdot*>{1.5}
\DGCstrand(1.5,0.5)(1.5,0.75)(2,1.5)[`$_{i+1}$]
\DGCdot*>{1.5}
\DGCstrand(3.5,0.5)(3.5,0.75)(1.5,1.5)[`$_i$]
\DGCdot*>{1.5}
\DGCdot{0.8}[urr]{$_r$}
\DGCcoupon*(1.5,0.5)(3,1.2){$\cdots$}
\DGCcoupon*(0,0.5)(1,1.4){$\cdots$}
\end{DGCpicture} 
\Bigg| r\leq l-\delta-n
\right\}
%\quad
%\begin{DGCpicture}
%\DGCstrand(0,-.5)(0,1.5)
%%\DGCdot*>{1.5}
%\DGCstrand(3.5,-.5)(3.5,1.5)
%%\DGCdot*>{1.5}
%\DGCstrand(4,-.5)(4,1)(1.75,1.5)
%%\DGCdot*>{1.5}
%%\DGCdot{1}
%\DGCdot{1}[urr]{$ \leq l-\delta-n$}
%%\DGCstrand/d/(4,.25)(4.5,0.25)/u/(4.5,1.5)/u/
%\DGCcoupon(0,.25)(3.5,1){$\nh_{n-1}^{l-\delta}$}
%\DGCcoupon*(0,-.4)(3.75,-.2){$\cdots$}
%\DGCcoupon*(1.75,1.3)(3.75,1.5){$\cdots$}
%\end{DGCpicture}
\end{equation}
Therefore, the $p$-DG bimodule in equation \eqref{PiE1} decomposes into a filtered sum
\begin{equation} \label{PiE2}
\begin{DGCpicture}[scale=0.8]
\DGCstrand(0,-1.75)(0,1.5)
\DGCdot*>{1.5}
\DGCstrand(0.5,-1.75)(0.5,1.5)
\DGCdot*>{1.5}
\DGCstrand(3.5,-1.75)(3.5,1.5)
\DGCdot*>{1.5}
\DGCstrand(4,0.5)(4,1.5)
\DGCdot*>{1.5}
\DGCstrand/d/(4,0.25)(4.5,0.25)/u/(4.5,1.5)/u/
\DGCcoupon(-0.1,0.25)(4.1,1){$_{\nh_n^{l-\delta}}$}
\DGCcoupon*(0,-0.5)(4,0.25){$\cdots$}
\DGCcoupon*(0,1.2)(4,1.4){$\cdots$}
\DGCcoupon*(0,-1.75)(4,-1.25){$\cdots$}
\DGCcoupon(-0.1,-1.25)(3.6,-0.5){$_{(y_1 \cdots y_{n-1})^{\delta} \nh_{n-1}^{l}}$}
\end{DGCpicture}
\cong
\bigoplus_{i=1}^n\bigoplus_{r=0}^{l-\delta-n}~
\begin{DGCpicture}
\DGCstrand(0,-1)(0,1.5)[`$_1$]
\DGCdot*>{1.5}
\DGCstrand(1,-1)(1,1.5)[`$_{i-1}$]
\DGCdot*>{1.5}
\DGCstrand(3,-1)(3,0.75)(3.5,1.5)[`$_n$]
\DGCdot*>{1.5}
\DGCstrand(1.5,-1)(1.5,0.75)(2,1.5)[`$_{i+1}$]
\DGCdot*>{1.5}
\DGCstrand/d/(4,1.5)(4,0.6)(3.5,0.6)/u/(3.5,0.75)(1.5,1.5)[`$_i$]
\DGCdot*>{1.5,2}
\DGCdot{0.6,2}[urr]{$_r$}
\DGCcoupon(-0.25,-0.3)(3.25,0.4){$_{(y_1\cdots y_{n-1})^\delta \nh_{n-1}^{l}}$}
\DGCcoupon*(1.5,-1)(3,-0.25){$\cdots$}
\DGCcoupon*(0,-1)(1,-0.25){$\cdots$}
\DGCcoupon*(0,1.25)(1,1.5){$\cdots$}
\DGCcoupon*(1.75,1.3)(3.75,1.5){$\cdots$}
\end{DGCpicture} 
\end{equation}
%\begin{equation}
%\begin{DGCpicture}
%\DGCstrand(0,-.5)(0,1.5)
%%\DGCdot*>{1.5}
%\DGCstrand(3.5,-.5)(3.5,1.5)
%%\DGCdot*>{1.5}
%\DGCstrand(4,.5)(4,1)(1.75,1.5)
%%\DGCdot*>{1.5}
%%\DGCdot{1}
%\DGCstrand/d/(4,.5)(4.5,0.5)/u/(4.5,1.5)/u/
%\DGCdot{1}[urr]{$ \gamma$}
%\DGCcoupon(0,.25)(3.5,1){$ \nh_{n-1}^{l-\delta}$}
%\DGCcoupon*(0,-.4)(3.75,-.2){$\cdots$}
%\DGCcoupon*(1.75,1.3)(3.75,1.5){$\cdots$}
%\DGCcoupon(0,-1.25)(3.5,-.5){$(y_1 \cdots y_{n-1})^{\delta} \nh_{n-1}^{l}$}
%\DGCstrand(0,-2)(0,-1.25)
%%\DGCdot*>{1.5}
%\DGCstrand(3.5,-2)(3.5,-1.25)
%%\DGCdot*>{1.5}
%\DGCcoupon*(0,-1.65)(3.75,-1.45){$\cdots$}
%\end{DGCpicture}
%=
%\begin{DGCpicture}
%\DGCstrand(0,-.5)(0,1.5)
%%\DGCdot*>{1.5}
%\DGCstrand(3.5,-.5)(3.5,1.5)
%%\DGCdot*>{1.5}
%\DGCstrand(4,.5)(4,1)(1.75,1.5)
%%\DGCdot*>{1.5}
%%\DGCdot{1}
%\DGCstrand/d/(4,.5)(4.5,0.5)/u/(4.5,1.5)/u/
%\DGCdot{1}[urr]{$ \gamma$}
%%\DGCcoupon(0,.25)(3.5,1){$ \nh_{n-1}^{l-\delta}$}
%\DGCcoupon*(0,-.4)(3.75,-.2){$\cdots$}
%\DGCcoupon*(1.75,1.3)(3.75,1.5){$\cdots$}
%\DGCcoupon(0,-1.25)(3.5,-.5){$(y_1 \cdots y_{n-1})^{\delta} \nh_{n-1}^{l}$}
%\DGCstrand(0,-2)(0,-1.25)
%%\DGCdot*>{1.5}
%\DGCstrand(3.5,-2)(3.5,-1.25)
%%\DGCdot*>{1.5}
%\DGCcoupon*(0,-1.65)(3.75,-1.45){$\cdots$}
%\end{DGCpicture}
%\end{equation}
%where $0 \leq \gamma \leq l-n-\delta$.

On the other hand, we compute that the bimodule
$\mf{E} \circ \mf{P}_\delta (\nh_n^{l-\delta})$ can be diagramatically represented by
\begin{equation} \label{EPi1}
\mf{E} \circ \mf{P}_\delta (\nh_n^{l-\delta})\cong
\begin{DGCpicture}[scale=0.8]
\DGCstrand(0,-1.75)(0,1.5)
\DGCdot*>{1.5}
\DGCstrand(0.5,-1.75)(0.5,1.5)
\DGCdot*>{1.5}
\DGCstrand(3.5,-1.75)(3.5,1.5)
\DGCdot*>{1.5}
\DGCstrand(4,-1.25)(4,1.5)
\DGCdot*>{1.5}
\DGCstrand/d/(4,-1.25)(4.5,-1.25)/u/(4.5,1.5)/u/
\DGCcoupon(-0.1,0.25)(4.1,1){$_{(y_1 \cdots y_{n})^{\delta}\nh_n^{l}}$}
\DGCcoupon*(0,-0.5)(4,0.25){$\cdots$}
\DGCcoupon*(0,1.2)(4,1.4){$\cdots$}
\DGCcoupon*(0,-1.75)(4,-1.25){$\cdots$}
\DGCcoupon(-0.1,-1.25)(4.1,-0.5){$_{ \nh_{n}^{l}}$}
\end{DGCpicture}
~\cong~
\begin{DGCpicture}[scale=0.8]
\DGCstrand(0,-1.75)(0,1.5)
\DGCdot*>{1.5}
\DGCstrand(0.5,-1.75)(0.5,1.5)
\DGCdot*>{1.5}
\DGCstrand(3.5,-1.75)(3.5,1.5)
\DGCdot*>{1.5}
\DGCstrand(4,-1.25)(4,1.5)
\DGCdot*>{1.5}
\DGCstrand/d/(4,-1.25)(4.5,-1.25)/u/(4.5,1.5)/u/
\DGCcoupon(-0.1,-0.7)(4.1,0.4){$_{(y_1 \cdots y_{n})^{\delta}\nh_n^{l}}$}
\DGCcoupon*(0,-1.7)(4,-0.9){$\cdots$}
\DGCcoupon*(0,0.7)(4,1.4){$\cdots$}
\end{DGCpicture} \ .
\end{equation}
which, after sliding the central $(y_1\cdots y_n)^\delta$ to the bottom part and using Proposition \ref{dotbottombasis} again, we can rewrite as
\begin{equation} \label{EPi2}
\begin{DGCpicture}[scale=0.8]
\DGCstrand(0,-1.75)(0,1.5)
\DGCdot*>{1.5}
\DGCstrand(0.5,-1.75)(0.5,1.5)
\DGCdot*>{1.5}
\DGCstrand(3.5,-1.75)(3.5,1.5)
\DGCdot*>{1.5}
\DGCstrand(4,-1.25)(4,1.5)
\DGCdot*>{1.5}
\DGCstrand/d/(4,-1.25)(4.5,-1.25)/u/(4.5,1.5)/u/
\DGCcoupon(-0.1,-0.7)(4.1,0.4){$_{(y_1 \cdots y_{n})^{\delta}\nh_n^{l}}$}
\DGCcoupon*(0,-1.7)(4,-0.9){$\cdots$}
\DGCcoupon*(0,0.7)(4,1.4){$\cdots$}
\end{DGCpicture}
~\cong~
\bigoplus_{i=1}^n\bigoplus_{t=\delta}^{l-n}~
\begin{DGCpicture}
\DGCstrand(0,-1)(0,1.5)[`$_1$]
\DGCdot*>{1.5}
\DGCstrand(1,-1)(1,1.5)[`$_{i-1}$]
\DGCdot*>{1.5}
\DGCstrand(3,-1)(3,0.75)(3.5,1.5)[`$_n$]
\DGCdot*>{1.5}
\DGCstrand(1.5,-1)(1.5,0.75)(2,1.5)[`$_{i+1}$]
\DGCdot*>{1.5}
\DGCstrand/d/(4,1.5)(4,-0.5)(3.5,-0.5)/u/(3.5,0.75)(1.5,1.5)[`$_i$]
\DGCdot*>{1.5,2}
\DGCdot{0.3}[urr]{$_t$}
\DGCcoupon(-0.25,-0.1)(3.25,0.7){$_{(y_1\cdots y_{n-1})^\delta \nh_{n-1}^{l}}$}
\DGCcoupon*(1.5,-1)(3,-0.25){$\cdots$}
\DGCcoupon*(0,-1)(1,-0.25){$\cdots$}
\DGCcoupon*(0,1.25)(1,1.5){$\cdots$}
\DGCcoupon*(1.75,1.3)(3.75,1.5){$\cdots$}
\end{DGCpicture} 
%\quad
%\begin{DGCpicture}
%\DGCstrand(0,-.5)(0,1.5)
%%\DGCdot*>{1.5}
%\DGCstrand(3.5,-.5)(3.5,1.5)
%%\DGCdot*>{1.5}
%\DGCstrand(4,.5)(4,1)(1.75,1.5)
%%\DGCdot*>{1.5}
%%\DGCdot{1}
%\DGCstrand/d/(4,.5)(4.5,0.5)/u/(4.5,1.5)/u/
%\DGCdot{1}[urr]{$ \gamma'$}
%\DGCcoupon(0,.25)(3.5,1){$ (y_1 \cdots y_{n-1})^{\delta} \nh_{n-1}^{l}$}
%\DGCcoupon*(0,-.4)(3.75,-.2){$\cdots$}
%\DGCcoupon*(1.75,1.3)(3.75,1.5){$\cdots$}
%%\DGCcoupon(0,-1.25)(3.5,-.5){$(y_1 \cdots y_{n-1})^{\delta} \nh_{n-1}^{l}$}
%%\DGCstrand(0,-2)(0,-1.25)
%%\DGCdot*>{1.5}
%%\DGCstrand(3.5,-2)(3.5,-1.25)
%%\DGCdot*>{1.5}
%%\DGCcoupon*(0,-1.65)(3.75,-1.45){$\cdots$}
%\end{DGCpicture}
\end{equation}
where $\delta \leq t \leq l-n$.

Now it is clear from the right hand side of \eqref{PiE2} and \eqref{EPi2} that
$ \mf{P}_\delta \circ \mf{E}(\nh_n^{l-\delta})$ and $ \mf{E} \circ \mf{P}_{\delta}(\nh_n^{l-\delta})$ are isomorphic as $(\nh_n^{l-\delta}, \nh_{n-1}^l)$-bimodules where we identify summand-wise
\begin{equation} \label{PiE=EPiab}
\begin{DGCpicture}
\DGCstrand(0,-1)(0,1.5)[`$_1$]
\DGCdot*>{1.5}
\DGCstrand(1,-1)(1,1.5)[`$_{i-1}$]
\DGCdot*>{1.5}
\DGCstrand(3,-1)(3,0.75)(3.5,1.5)[`$_n$]
\DGCdot*>{1.5}
\DGCstrand(1.5,-1)(1.5,0.75)(2,1.5)[`$_{i+1}$]
\DGCdot*>{1.5}
\DGCstrand/d/(4,1.5)(4,0.6)(3.5,0.6)/u/(3.5,0.75)(1.5,1.5)[`$_i$]
\DGCdot*>{1.5,2}
\DGCdot{0.6,2}[urr]{$_r$}
\DGCcoupon(-0.25,-0.3)(3.25,0.4){$_{(y_1\cdots y_{n-1})^\delta \nh_{n-1}^{l}}$}
\DGCcoupon*(1.5,-1)(3,-0.25){$\cdots$}
\DGCcoupon*(0,-1)(1,-0.25){$\cdots$}
\DGCcoupon*(0,1.25)(1,1.5){$\cdots$}
\DGCcoupon*(1.75,1.3)(3.75,1.5){$\cdots$}
\end{DGCpicture} 
~\cong~
\begin{DGCpicture}
\DGCstrand(0,-1)(0,1.5)[`$_1$]
\DGCdot*>{1.5}
\DGCstrand(1,-1)(1,1.5)[`$_{i-1}$]
\DGCdot*>{1.5}
\DGCstrand(3,-1)(3,0.75)(3.5,1.5)[`$_n$]
\DGCdot*>{1.5}
\DGCstrand(1.5,-1)(1.5,0.75)(2,1.5)[`$_{i+1}$]
\DGCdot*>{1.5}
\DGCstrand/d/(4,1.5)(4,-0.5)(3.5,-0.5)/u/(3.5,0.75)(1.5,1.5)[`$_i$]
\DGCdot*>{1.5,2}
\DGCdot{0.3}[urr]{$_{r+\delta}$}
\DGCcoupon(-0.25,-0.1)(3.25,0.7){$_{(y_1\cdots y_{n-1})^\delta \nh_{n-1}^{l}}$}
\DGCcoupon*(1.5,-1)(3,-0.25){$\cdots$}
\DGCcoupon*(0,-1)(1,-0.25){$\cdots$}
\DGCcoupon*(0,1.25)(1,1.5){$\cdots$}
\DGCcoupon*(1.75,1.3)(3.75,1.5){$\cdots$}
\end{DGCpicture} \ .
\end{equation}

We would next like to show that the isomorphism in \eqref{PiE=EPiab} is an isomorphism of $p$-DG modules. To do this let us emphasize that the cup diagrams in equations \eqref{PiE2} and \eqref{EPi2} arise from two different categorical $\mf{E}$ and $\mf{F}$ actions: the former cup in a relatively higher position is an adjunction map for the categorical quantum $\mf{sl}_2$ action on $\oplus_{n=0}^{l-\delta} \nh_n^{l-\delta}$, while the latter relatively lower cup is that for the action on $\oplus_{n=0}^l \nh_n^l$. See the discussion in Section \ref{subsec-cat-simples}.

By equation \eqref{eqnrestrictionfunctor}, the differential action on the bimodule generator of the restriction functor 
$$\mf{E} \colon (\nh_n^{l},\dif) \dmod \lra (\nh_{n-1}^{l},\dif) \dmod$$
is given diagramatically by
\begin{equation} \label{twistedres}
\dif\left(~
\begin{DGCpicture}
\DGCstrand(0,0)(0,1.5)
\DGCdot*>{1.5}
\DGCstrand(0.5,0)(0.5,1.5)
\DGCdot*>{1.5}
\DGCstrand(1.5,0.5)(1.5,1.5)
\DGCdot*>{1.5}
\DGCstrand/d/(1.5,0.5)(2,0.5)/u/(2,1.5)/u/
\DGCcoupon*(2.1,0.4)(3.2,1.1){$_{l-2n+2}$}
\DGCcoupon*(0.6,0.1)(1.4,0.3){$\cdots$}
\DGCcoupon*(0.6,1.2)(1.4,1.4){$\cdots$}
\end{DGCpicture}
~\right)
=(2n-l-1)
\begin{DGCpicture}
\DGCstrand(0,0)(0,1.5)
\DGCdot*>{1.5}
\DGCstrand(0.5,0)(0.5,1.5)
\DGCdot*>{1.5}
\DGCstrand(1.5,0.5)(1.5,1.5)
\DGCdot*>{1.5}
\DGCdot{1}
\DGCstrand/d/(1.5,0.5)(2,0.5)/u/(2,1.5)/u/
\DGCcoupon*(2.1,0.4)(3.2,1.1){$_{l-2n+2}$}
\DGCcoupon*(0.6,0.1)(1.4,0.3){$\cdots$}
\DGCcoupon*(0.6,1.2)(1.4,1.4){$\cdots$}
\end{DGCpicture} \ .
\end{equation}
Computing the differentials on both sides of \eqref{PiE=EPiab} produces a sum of terms.  It is clear that all terms match up except for possibly the terms resulting from applying $\partial$ to the cups with dots on them.   We use \eqref{twistedres} to calculate these terms.  Applying the differential to the cup on the left hand side of \eqref{PiE=EPiab} produces a term with an extra dot with coefficient $2n-(l-\delta)-1+r$.
Applying $\partial$ to the cup on the right hand side of \eqref{PiE=EPiab} produces a term with an extra dot with coefficient $2n-l-1+r+\delta$.  Thus the isomorphism in \eqref{PiE=EPiab} is that of $p$-DG modules.

The higher divided powers case follows from the $d=1$ case by iterating the previous isomorphism $d$ times, and capping the right most $d$ downward strands by a thickness $d$ strand representing $\mf{F}^{(d)}$:
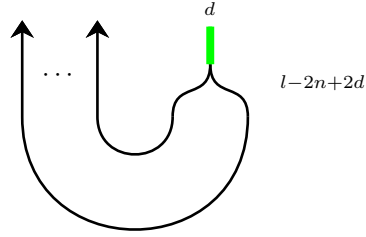
\[
\begin{DGCpicture}
\DGCstrand/d/(1.75,1.5)(1.75,0.3)(2.75,0.3)/u/
\DGCdot<{1.5}
\DGCstrand/d/(0.75,1.5)(0.75,0.3)(3.75,0.3)/u/
\DGCdot<{1.5}
\DGCstrand(2.75,0.3)(3.25,1)
\DGCstrand(3.75,0.3)(3.25,1)
\DGCstrand[Green](3.25,1)(3.25,1.5)[`$_d$]
\DGCcoupon*(3.25,0.4)(6.25,1.1){$_{l-2n+2d}$}
\DGCcoupon*(0.75,0.7)(1.75,1){$\cdots$}
\end{DGCpicture}
\]
The result follows.
\end{proof}

\begin{prop}
\label{proprestrictofG}
For each $\lambda\in\mc{P}_n^l$ and $d\geq 1$, the module
$\mf{E}^{(d)} G(\lambda)$ is in the filtered $p$-DG envelope $\mc{G}$.
\end{prop}

\begin{proof}
 We will prove this by induction on $l$.  It is easy to check the base case $l=1$.  In that case $n=1$ and we are just restricting to modules over $\Bbbk$.

Write $\lambda=(0^{a_1} 1^{b_1} \ldots 0^{a_r} 1^{b_r})$ with $a_1, b_r\geq 0$ and the other $a_i$ and $b_i$ are positive.

First assume $b_r=0$.  Then set
$\lambda'=(0^{a_1} 1^{b_1} \ldots 0^{a_{r-1}} 1^{b_{r-1}} 0^{\delta}) $
where $\delta=a_r$.  By Lemma \ref{PappliedtoG} and Proposition \ref{PandFcommute},
\begin{equation*}
\mf{E}^{(d)}G(\lambda) \cong 
\mf{E}^{(d)} \mf{P}_{\delta} G(\lambda') \cong
\mf{P}_{\delta} \mf{E}^{(d)} G(\lambda').
\end{equation*}
By the inductive hypothesis, $\mf{E}^{(d)} G(\lambda')$ is in the filtered $p$-DG envelope since $\lambda' \in \mathcal{P}_{n-b_r}^{l-\delta}$.  Then 
$\mf{P}_{\delta} \mf{E}^{(d)} G(\lambda')$ is in the filtered $p$-DG envelope as well which verifies this case.

Now assume $b_r \neq 0$.
Let $\mu=(0^{a_1} 1^{b_1} \ldots 0^{a_{r}} 1^{b_{r}-1} 0)$.
Then $G(\lambda)=\mf{F}G(\mu)$ so $ \mf{E}^{(d)}G(\lambda) \cong \mf{E}^{(d)} \mf{F} G(\mu)$.  
Also let $\mu'=(0^{a_1} 1^{b_1} \ldots 0^{a_{r}} 1^{b_{r}-1})$.
There are two cases to consider here depending upon inequalities involving
$n$, $l$, and $d$.  For details, see \cite[Sections 5.2 and 6.2]{EQ2}.  The two cases lead to
the following two possible short exact sequences of $p$-DG modules which split when forgetting the differential
\begin{equation} \label{2EFpossibilities}
0 \lra \mf{E}^{(d)} \mf{F} G(\mu) \lra \mf{F} \mf{E}^{(d)}  G(\mu) \lra X \lra 0,
\quad \quad \quad
0 \lra \mf{F} \mf{E}^{(d)} G(\mu)   \lra  \mf{E}^{(d)} \mf{F} G(\mu) \lra X \lra 0
\end{equation}
where $X$ has a $p$-DG filtration whose subquotients are isomorphic to graded shifts of $\mf{E}^{(d-1)}  G(\mu)$.

By Lemma \ref{PappliedtoG} and Proposition \ref{PandFcommute}, 
\begin{equation*}
\mf{E}^{(d)} G(\mu) \cong 
\mf{E}^{(d)} \mf{P}_1 G(\mu') \cong 
\mf{P}_1 \mf{E}^{(d)}  G(\mu'). 
\end{equation*}
By the inductive hypothesis $\mf{P}_1 \mf{E}^{(d)}  G(\mu')$ is in the filtered $p$-DG envelope.  Then by Proposition \ref{FpreservesGs}, the module $\mf{F} \mf{E}^{(d)}  G(\mu)$ is in the filtered $p$-DG envelope as well.
By the inductive hypothesis, one deduces in a similar way that $\mf{E}^{(d-1)}  G(\mu)$ is in the filtered $p$-DG envelope.  By repeated use of Proposition \ref{propfilteredenvelope2outof3}, 
$X$ is in the filtered $p$-DG envelope as well.

Using Proposition \ref{propfilteredenvelope2outof3} again, we see that since the objects 
$\mf{F} \mf{E}^{(d)}  G(\mu)$ and $X$ 
in \eqref{2EFpossibilities} are in the filtered $p$-DG envelope, so is 
$ \mf{E}^{(d)}G(\lambda) \cong \mf{E}^{(d)} \mf{F} G(\mu)$.

%Case I: $(l-n+1) \geq n-1$.
%In this case
%\begin{align*}
% \mf{E} \mf{F} G(\mu) & \cong \mf{F} \mf{E} G(\mu) \bigoplus [l-2n+2] G(\mu)\\
% & \cong \mf{F} \mf{E} \mf{P}_1 G(\mu')  \bigoplus [l-2n+2] G(\mu)
%\end{align*}
%By Lemma \ref{PappliedtoG}, Proposition \ref{PandFcommute} and the inductive hypothesis, $\mf{E} \mf{P}_1 G(\mu')$ is in the filtered $p$-DG envelope.  Since $\mf{F}$ preserves the envelope by Proposition \ref{FpreservesGs}, we get that $\mf{E} \mf{F} G(\mu)$ and hence $\mf{E} G(\lambda)$ is also in the $p$-DG Karoubi envelope.

%Case II: $n-1 \geq l-n+1$.
%In this case
%\begin{equation*}
%\mf{F} \mf{E} G(\mu) \cong \mf{E} \mf{F} G(\mu)  \bigoplus [l-2n+2] G(\mu).
%\end{equation*}
%As in the previous case $\mf{F} \mf{E} G(\mu)$ is in the filtered $p$-DG envelope and thus so is $\mf{E} \mf{F} G(\mu) \cong \mf{E} G(\lambda)$.
\end{proof}

We are now ready to prove our main result of this section.  Recall the definition of $Z(\lambda)$ from Definition \ref{defZlambda}.

\begin{thm}\label{thmtensorcatfn}
There is an action of the derived $p$-DG category $\mc{D}(\dot{\mathcal{U}})$ on $\oplus_{n=0}^l  \mathcal{D}(S_n^l)$.  As modules over $ \dot{U}_{\mathbb{O}_p}$, there is an isomorphism
\begin{equation*}
K_0\left(\bigoplus_{n=0}^l  \mathcal{D}^c(S_n^l)\right) \cong V_1^{\otimes l},
\end{equation*}
where the right hand side is the tensor product of the  $\dot{U}_{\mathbb{O}_p}$ Weyl module $V_1$.
The identification is given by, for any $\lambda\in \mc{P}_n^l$,
\begin{equation}\label{eqn-Z-basis}
[Z(\lambda)] \mapsto 
F^{}( \cdots  {F}^{}  ({F}^{} (v_0^{\otimes j_1}) \otimes v_0^{\otimes (j_2-j_1)}) \otimes \cdots \otimes  v_0^{\otimes (j_n-j_{n-1})}),
\end{equation}
where the entries of $\lambda$ in positions $j_1,\ldots,j_n$ are $1$ with $j_1 < \cdots < j_n$ and zero otherwise.
\end{thm}

\begin{proof}
We have shown that $S_n^l$ is a $p$-DG quasi-hereditary cellular algebra (Theorem \ref{qschurpdgquasi}). By Corollary \ref{corquasihereditarycase}, the $p$-DG Grothendieck group of $S_n^l$ is a free $\mathbb{O}_p$-module of rank $|\mc{P}_n^l| = {l \choose n}$.

By Propositions \ref{FpreservesGs} and \ref{proprestrictofG} the (divided) induction and restriction functors restrict to the filtered $p$-DG envelope $\mc{G}$ of the cyclic modules. Then the fact that there is an action of the $p$-DG category $\dot{\mathcal{U}}$ on
$\oplus_{n=0}^l  \mathcal{D}^c(S_n^l)$ follows from combining these two propositions with
Theorem \ref{thm-pdg-extension} and Proposition \ref{propGisfaithful}.

Define a map of $\mathbb{O}_p$-modules by sending $[Z(\lambda)]$ to the term prescribed above. By the proofs of Propositions \ref{FpreservesGs} and \ref{proprestrictofG}, one can check, on the decategorified level, that this map intertwines the $\dot{U}_{\mathbb{O}_p}(\mf{sl}_2)$ action on both sides. This is done by matching the combinatorics of the $E$, $F$ action from equations \eqref{eqn-FG} and \eqref{2EFpossibilities} with the $E$, $F$ action on the right hand side of \eqref{eqn-Z-basis}. Since the weight space $1_{l-2n}V_1^{\otimes l}$ also has rank ${l\choose n}$ (see Section \ref{subsec-rep}), the map above is a surjection of projective $\mathbb{O}_p$-modules, and thus is an isomorphism.
\end{proof}

\begin{rem}
In Theorem \ref{thmqschurwebmoritaequiv}, we prove that there is a $p$-DG Morita equivalence between the $p$-DG quiver Schur algebra and the $p$-DG Webster algebra.

A graphical description of the categorical actions of $E$ and $F$ on $p$-DG quiver Schur algebras may be obtained by using Theorem \ref{catofVl} and applying the general categorical extension result (Theorem \ref{thm-pdg-extension}).  The next subsection describes this in more detail.

Using this graphical description of the functors, and   
\cite[Theorem 4.38]{Webcombined}, one obtains an alternative proof of the categorification result above (Theorem \ref{thmtensorcatfn}).
\end{rem}

\subsection{A comparison of functors}
In this subsection we review a graded Schur functor introduced in \cite[Section 4.3]{HuMathas} and show that it agrees with the Soergel functor in the general setting of Definition \ref{def-of-commutant}.

First extend the set  $\mc{P}_n^l $ to $\dot{\mc{P}}_n^l=\mc{P}_n^l \cup \{ \omega \}$ for some symbol dummy symbol $\omega$.
Let $G(\omega)=\nh_n^l$ and then define the extended quiver Schur algebra
\begin{equation*}
\dot{S}_n^l:=\END_{\nh_n^l}\left(\bigoplus_{\lambda\in \dot{\mc{P}}_n^l} G(\lambda)\right).
\end{equation*}
Note that 
$1_{S}  := \sum_{\lambda \in \mc{P}_n^l}\Psi^{\lambda \lambda}_{e e \lambda}$ 
is the identity element in $S_n^l$ and an idempotent in $\dot{S}_n^l$.
Let $\Psi^{\omega \omega}_{e e \omega} \in \dot{S}_n^l$ be the endomorphism which is the identity on $\nh_n^l$ and is zero on $G(\lambda)$ for $\lambda \in \mc{P}_n^l$.

Then the \emph{graded Schur functor} introduced by Hu and Mathas is defined as
\begin{equation}
\label{HMSchur}
\mathcal{S} := 
\HOM_{\dot{S}_n^l}\left(\dot{S}_n^l\Psi^{\omega \omega}_{e e \omega} , \dot{S}_n^l 1_{S} \otimes_{S_n^l}(\mbox{-})\right) 
\colon S_n^l \dmod \longrightarrow \nh_n^l \dmod.
\end{equation}
It is shown by Hu-Mathas (\cite[Corollary 4.36]{HuMathas}) that the functor $\mc{S}$ is fully-faithful on projective modules. The functor extends naturally to a $p$-DG functor on the corresponding $p$-DG module categories, where we equip $\nh_n^l$ with the natural right regular $p$-DG module structure over $(\nh_n^l,\dif)$. In this subsection, we will show that the $p$-DG graded Schur functor agrees with the $p$-DG Soergel functor $\mc{V}$ (Definition \ref{def-of-commutant}). The general framework of Section \ref{sec-double} will then show that the functor is fully-faithful on cofibrant $p$-DG modules.

\begin{thm}\label{thmisooffunctors}
There is an isomorphism of $p$-DG functors
\[
\mathcal{S} \cong \mc{V}= \HOM_{S_n^l}\left(G, \mbox{-}\right) \colon (S_n^l,\dif) \dmod \longrightarrow (\nh_n^l ,\dif)\dmod.
\]
\end{thm}

\begin{proof}
Via Proposition \ref{prop-p-dg-Morita}, the functor 
\[
\dot{S}_n^l 1_{S} \otimes_{S_n^l}( \mbox{-} )\colon 
({S}_n^l,\dif) \dmod \longrightarrow (\dot{S}_n^l,\dif)\dmod
\]
is a $p$-DG Morita equivalence of categories with inverse
\begin{equation}
\label{dummyinverse}
\mc{T} \colon (\dot{S}_n^l,\dif )\dmod \longrightarrow (S_n^l,\dif )\dmod, \quad \quad 
N \mapsto 1_S N.
\end{equation}
This is because, by Proposition \ref{propGisfaithful}, the left regular $p$-DG module $\nh_n^l$ is already contained in the $p$-DG filtered envelope of $G(\lambda_0)$ (see also \cite[Lemma 4.28]{HuMathas}).  Thus applying $\mc{T}$ to the morphism space defining $\mc{S}$ in \eqref{HMSchur} gives an isomorphism
\begin{equation}
\label{HMSchurHOM}
\HOM_{\dot{S}_n^l}\left(\dot{S}_n^l\Psi^{\omega \omega}_{e e \omega} , \dot{S}_n^l 1_{S} \otimes_{S_n^l} (\mbox{-})\right) 
\cong
\HOM_{{S}_n^l}\left(1_S \dot{S}_n^l\Psi^{\omega \omega}_{e e \omega} ,  \mbox{-}\right). 
\end{equation}
By the proof of \cite[Proposition 5.6]{HuMathas},
\[
1_S \dot{S}_n^l\Psi^{\omega \omega}_{e e \omega} 
\cong
G.
\]
This is a $p$-DG isomorphism since $\dif(1_S)=0$ and $\dif(\Psi_{eew}^{ww})=0$. The theorem now follows from the isomorphism in \eqref{HMSchurHOM}. 
\end{proof}

Recall that $P(\lambda)$ is the indecomposable projective cover of the simple module $L(\lambda)$ (see Definition \ref{defofLandP}).

\begin{prop}
\label{Sonproj}
There are isomorphisms of $(\nh_n^l,\dif)$-modules
\[
\mc{V}(Z(\lambda)) \cong G(\lambda) \quad \quad
\mc{V}(P(\lambda)) \cong Y(\lambda).
\]
\end{prop}

\begin{proof}
This follows from \cite[Proposition 5.6]{HuMathas} and its proof for the $\mc{S}$ functor.
\end{proof}

\begin{prop}
\label{SonG}
The $p$-DG module $G$ has an $n!$-step $p$-DG filtration, whose subquotients are isomorphic to $P(\lambda_0)$ up to grading shifts. Furthermore, as left $p$-DG modules over $\nh_n^l$, there is an isomorphism $\mc{V}(G) \cong \nh_n^l$.
\end{prop}

\begin{proof}
We first show the second statement. By Propostion \ref{propGisfaithful}, $G$ is a faithful $p$-DG right module over the Frobenius $p$-DG algebra $\nh_n^l$ (Proposition \ref{symm}). Thus, by Theorem \ref{thm-double-centralizer}, $\nh_n^l$ and $S_n^l$ are double centralizers of each other. It follows that
\[
\mc{V}(G)\cong \HOM_{S_n^l}(G,G)=\nh_n^l.
\]

To prove the first claim, we use the double centralizer property again, and compute the $p$-DG endomorphism algebra of $G$ is isomorphic to 
$
\END_{S_n^l}(G,G)\cong \nh_n^l.
$
Since $\nh_n^l$ has a left $p$-DG module filtration over itself whose subquotients are isomorphic to $\nh_n^le_n^\prime$ (see equation \eqref{eqnenprime}), it follows that $G$ has an $n!$-step filtration whose subquotients are $p$-DG isomorphic to grading shifts of $G\otimes_{\nh_n^l}\nh_n^l e_n^\prime$. Notice that the latter is an $S_n^l$-direct summand of $G$ ignoring the differential. The fully-faithfulness of $\mc{V}$ now shows that
\[
\END_{S_n^l}\left(G\otimes_{\nh_n^l}\nh_n^l e_n^\prime\right)\cong \END_{\nh_n^l}\left(\mc{V}(G\otimes_{\nh_n^l}\nh_n^l e_n^\prime)\right)\cong \END_{\nh_n^l}\left(\nh_n^le_n^\prime\right)\cong \mH^\bullet(\mathrm{Gr}(n,l),\Bbbk).
\] 
Thus $G\otimes_{\nh_n^l}\nh_n^l e_n^\prime$ is an indecomposable module over $S_n^l$ and the first claim follows.
\end{proof}

%%%%%%%%%%%%%%%%%%%%%%%%%%%%%%%%%%%%%%%%%%%%%%%%%%%%%%%%%%%%%%%%%%%%%%%%%%%%%%%%%%%%%%%%%%%%%%%%%%%%%%%%%%%%%%%%%%%%%%
%%%%%%%%%%%%%%%%%%%%%%%%%%%%%%%%%%%%%%%%%%%%%%%%%%%%%%%%%%%%%%%%%%%%%%%%%%%%%%%%%%%%%%%%%%%%%%%%%%%%%%%%%%%%%%%%%%%%%%

\section{Webster algebras}
\label{sec-Web}
Webster \cite{Webcombined} constructed a family of diagrammatically defined algebras which categorify tensor products of representations of quantum groups.  In type $A$, he showed that quiver Schur algebras are Morita equivalent to subalgebras of Webster algebras.  We extend this to a $p$-DG Morita equivalence (in the $\mf{sl}_2$ case).  
We further show that the cellular basis of the quiver Schur algebra (again in the $\mf{sl}_2$ case) constructed by Hu and Mathas agrees with a truncation of a cellular basis of Webster algebras constructed by Stroppel and Webster \cite{SWSchur}.
\subsection{Definitions}
\label{subsec-Def-Web}
We begin by recalling the definition of a particular Webster algebra. More general versions of these algebras, which are associated with arbitrary finite Cartan data, can be found in \cite{Webcombined}.

\begin{defn}
\label{def-Webster-algebra}
The Webster algebra $W_n^l$ is an algebra with $l$ red strands and $n$ black strands. Far away generators commute with each other. The black strands are allowed to carry dots, and red strands are not allowed to cross each other. We depict the local generators of this algebra by
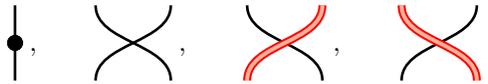
\[
\begin{DGCpicture}
\DGCstrand(0,0)(0,1)
\DGCdot{0.5}
\end{DGCpicture}
\ , \quad \quad
\begin{DGCpicture}
\DGCstrand(1,0)(0,1)
\DGCstrand(0,0)(1,1)
\end{DGCpicture}
\ , \quad \quad
\begin{DGCpicture}
\DGCstrand(1,0)(0,1)
\DGCstrand[Red](0,0)(1,1)
\end{DGCpicture}
\ , \quad \quad
\begin{DGCpicture}
\DGCstrand(0,0)(1,1)
\DGCstrand[Red](1,0)(0,1)
\end{DGCpicture}
\ .
\]
There is a $\Z$-grading on this algebra where the generators above are assigned degrees $2, -2, 1, 1$ respectively.
Multiplication in this algebra is vertical concatenation of diagrams.  When the colors of the boundary points of one diagram do not match the colors of the boundary points of  a second diagram, their product is taken to be zero.
The relations between the local generators are given by the usual nilHecke algebra relations among black strands
\begin{subequations}
\begin{gather}
\begin{DGCpicture}[scale=0.55]
\DGCstrand(1,0)(0,1)(1,2)
\DGCstrand(0,0)(1,1)(0,2)
\end{DGCpicture}
~= 0 \,  \quad \quad
\begin{DGCpicture}[scale=0.55]
\DGCstrand(0,0)(2,2)
\DGCstrand(1,0)(0,1)(1,2)
\DGCstrand(2,0)(0,2)
\end{DGCpicture}
~=~
\begin{DGCpicture}[scale=0.55]
\DGCstrand(0,0)(2,2)
\DGCstrand(1,0)(2,1)(1,2)
\DGCstrand(2,0)(0,2)
\end{DGCpicture}
\ , \\
\begin{DGCpicture}
\DGCstrand(0,0)(1,1)
\DGCdot{0.25}
\DGCstrand(1,0)(0,1)
\end{DGCpicture}
-
\begin{DGCpicture}
\DGCstrand(0,0)(1,1)
\DGCdot{0.75}
\DGCstrand(1,0)(0,1)
\end{DGCpicture}
~=~
\begin{DGCpicture}
\DGCstrand(0,0)(0,1)
\DGCstrand(1,0)(1,)
\end{DGCpicture}
~=~
\begin{DGCpicture}
\DGCstrand(1,0)(0,1)
\DGCdot{0.75}
\DGCstrand(0,0)(1,1)
\end{DGCpicture}
-
\begin{DGCpicture}
\DGCstrand(1,0)(0,1)
\DGCdot{0.25}
\DGCstrand(0,0)(1,1)
\end{DGCpicture} \ ,
\end{gather}
and local relations among red-black strands
\begin{equation}
\begin{DGCpicture}[scale=0.55]
\DGCstrand(1,0)(0,1)(1,2)
\DGCstrand[Red](0,0)(1,1)(0,2)
\end{DGCpicture}
~=~
\begin{DGCpicture}[scale=0.55]
\DGCstrand(1,0)(1,2)
\DGCdot{1}
\DGCstrand[Red](0,0)(0,2)
\end{DGCpicture}
\ , \quad \quad \quad
\begin{DGCpicture}[scale=0.55]
\DGCstrand(0,0)(1,1)(0,2)
\DGCstrand[Red](1,0)(0,1)(1,2)
\end{DGCpicture}
~=~
\begin{DGCpicture}[scale=0.55]
\DGCstrand(0,0)(0,2)
\DGCdot{1}
\DGCstrand[Red](1,0)(1,2)
\end{DGCpicture}
\ , \label{eqn-Webster-double-crossing}
\end{equation}
\begin{equation}\label{eqn-Webster-triple-crossing1}
\begin{DGCpicture}[scale=0.5]
\DGCstrand(0,0)(2,2)
\DGCstrand(1,0)(0,1)(1,2)
\DGCstrand[Red](2,0)(0,2)
\end{DGCpicture}
~=~
\begin{DGCpicture}[scale=0.5]
\DGCstrand(0,0)(2,2)
\DGCstrand(1,0)(2,1)(1,2)
\DGCstrand[Red](2,0)(0,2)
\end{DGCpicture}
\ , \quad \quad
\begin{DGCpicture}[scale=0.5]
\DGCstrand(1,0)(0,1)(1,2)
\DGCstrand(2,0)(0,2)
\DGCstrand[Red](0,0)(2,2)
\end{DGCpicture}
~=~
\begin{DGCpicture}[scale=0.5]
\DGCstrand(1,0)(2,1)(1,2)
\DGCstrand(2,0)(0,2)
\DGCstrand[Red](0,0)(2,2)
\end{DGCpicture}
\ , 
\end{equation}
\begin{equation}
\begin{DGCpicture}
\DGCstrand(0,0)(1,1)
\DGCdot{0.25}
\DGCstrand[Red](1,0)(0,1)
\end{DGCpicture}
~=~
\begin{DGCpicture}
\DGCstrand(0,0)(1,1)
\DGCdot{0.75}
\DGCstrand[Red](1,0)(0,1)
\end{DGCpicture}
\ , \quad \quad
\begin{DGCpicture}
\DGCstrand(1,0)(0,1)
\DGCdot{0.25}
\DGCstrand[Red](0,0)(1,1)
\end{DGCpicture}
~=~
\begin{DGCpicture}
\DGCstrand(1,0)(0,1)
\DGCdot{0.75}
\DGCstrand[Red](0,0)(1,1)
\end{DGCpicture}
, 
\end{equation}
\begin{equation}\label{eqn-Webster-triple-crossing2}
\begin{DGCpicture}[scale=0.5]
\DGCstrand(1,0)(3,2)
\DGCstrand(3,0)(1,2)
\DGCstrand[Red](2,0)(1,1)(2,2)
\end{DGCpicture}
\ - \
\begin{DGCpicture}[scale=0.5]
\DGCstrand(1,0)(3,2)
\DGCstrand(3,0)(1,2)
\DGCstrand[Red](2,0)(3,1)(2,2)
\end{DGCpicture}
\ = \
\begin{DGCpicture}[scale=0.5]
\DGCstrand(1,0)(1,2)
\DGCstrand(3,0)(3,2)
\DGCstrand[Red](2,0)(2,2)
\end{DGCpicture}
\ ,
\end{equation}
together with the \emph{cyclotomic relation} that a black strand, appearing on the far left of any diagram, annihilates the entire picture:
\begin{equation}\label{eqn-Web-cyclotomic}
\begin{DGCpicture}
\DGCstrand(1,0)(1,1)
\DGCcoupon*(1.25,0.25)(1.75,0.75){$\cdots$}
\end{DGCpicture}
~=~0.
\end{equation}
\end{subequations}
There is an anti-involution $*:W_n^l\lra W_n^l$ defined by a horizontal diagram flipping. This is an analogue of the corresponding map for nilHecke algebras (see Section \ref{subsec-def-nilHecke}).
\end{defn}

A family of differentials was introduced on $W_n^l$ in \cite{KQ}. Among the family, a unique differential, up to conjugation by (anti)-automorphisms of $W_n^l$, is determined in \cite{QiSussan}. This is the differential which is compatible with the natural categorical half-quantum $\mf{sl}_2$ action.  We give this choice in the lemma below.  The fact that $\partial$ is indeed a $p$-differential is a straightforward calculation.

\begin{lem}
\label{lem-dif-on-Web}
The Webster algebra has a $p$-DG structure given by
\begin{gather}
\label{diffonWeb}
\dif\left(~
\begin{DGCpicture}
\DGCstrand(0,0)(1,1)
\DGCstrand[Red](1,0)(0,1)
\end{DGCpicture}
~\right)=0 , 
\quad \quad \quad
\dif\left(~
\begin{DGCpicture}
\DGCstrand(1,0)(0,1)
\DGCstrand[Red](0,0)(1,1)
\end{DGCpicture}
~\right)=~
\begin{DGCpicture}
\DGCstrand(1,0)(0,1)
\DGCdot{0.75}
\DGCstrand[Red](0,0)(1,1)
\end{DGCpicture}
\ 
\end{gather}
and extended to the full algebra by the Leibniz rule.
\end{lem}
\begin{proof}
See \cite[Section 4.2]{QiSussan}.
\end{proof}

To a sequence $\kappa$ of vertical lines read from left to right,
we associate an idempotent $\e(\kappa) \in W_n^l$. We will occasionally emphasize this element by a coupon labeled by $\e(\kappa)$ if the sequence is not explicitly written out:
\[
\e(\kappa)=
\begin{DGCpicture}[scale=0.75]
\DGCstrand[Red](-1,0)(-1,2)
\DGCstrand(0,0)(0,2)
\DGCstrand[Red](1,0)(1,2)[$^{}$`{\ }]
\DGCstrand[Red](2,0)(2,2)
\DGCstrand(3,0)(3,2)
\DGCstrand(5,0)(5,2)
\DGCcoupon*(3.1,0)(4.9,0.5){$\cdots$}
\DGCcoupon*(3.1,1.5)(4.9,2){$\cdots$}
\DGCcoupon(-1.2,0.5)(5.2,1.5){$\e(\kappa)$}
\end{DGCpicture}
\]
A special sequence $\kappa_0$, having all the black strands to the right of the red strands, will play an important role later:
\begin{equation}
\e(\kappa_0)
=
\begin{DGCpicture}[scale=0.75]
\DGCstrand[Red](-2,-1)(-2,1)[$^{1}$]
\DGCcoupon*(-1.9,-0.8)(-1.1,0.8){$\cdots$}
\DGCstrand[Red](-1,-1)(-1,1)[$^{l}$]
\DGCPLstrand(-0.5,-1)(-0.5,1)[$^1$]
\DGCPLstrand(0,-1)(0,1)[$^{2}$]
\DGCPLstrand(1,-1)(1,1)[$^{n}$]
\DGCcoupon*(0.2,-0.8)(0.8,0.8){$\cdots$}
\end{DGCpicture} \ .
\end{equation}
We will drop the coupon for this idempotent as its diagrammatic description is unambiguous.

\begin{lem}\label{lem-Web-big-block}
There is a $p$-DG isomorphism
\[
\begin{array}{ccc}
\nh_n^l & \lra & \e(\kappa_0)W_n^l\e(\kappa_0),\\ && \\
\begin{DGCpicture}
\DGCstrand(0,0)(0,1)[$^1$]
\DGCstrand(1,0)(1,1)[$^n$]
\DGCcoupon(-0.1,0.3)(1.1,0.7){$x$}
\DGCcoupon*(0.1,0)(0.9,0.2){$\cdots$}
\DGCcoupon*(0.1,0.8)(0.9,1){$\cdots$}
\end{DGCpicture}
& \mapsto &
\begin{DGCpicture}
\DGCstrand[Red](-2,0)(-2,1)[$^{1}$]
\DGCcoupon*(-1.9,0.2)(-1.1,.8){$\cdots$}
\DGCstrand[Red](-1,0)(-1,1)[$^{l}$]
\DGCstrand(-0.5,0)(-0.5,1)[$^1$]
\DGCstrand(0.5,0)(0.5,1)[$^n$]
\DGCcoupon(-0.6,0.3)(0.6,0.7){$x$}
\DGCcoupon*(-0.4,0)(0.4,0.2){$\cdots$}
\DGCcoupon*(-0.4,0.8)(0.4,1){$\cdots$}
\end{DGCpicture} \ .
\end{array} 
\]
\end{lem}
\begin{proof}
See \cite[Proposition 5.31]{Webcombined}. The compatibility with $p$-differentials is clear from the definitions of the differentials on both sides.
\end{proof}

\begin{defn}\label{def-projective-Webster}
To any sequence $\kappa$ there is a projective right module $W_n^l$-module\footnote{The reader may wonder why we still consider certain right $p$-DG modules over $W_n^l$, as we would like to compare the left $p$-DG module categories over the Webster algebra and the quiver Schur algebra. The reason is that, for any ($p$-DG) algebra $A$, the natural isomorphism $A\cong \END_A(A)$ holds only when $A$ is considered as the right regular module. Considered as a left module, $\END_A(A)\cong A^\circ$, the opposite ($p$-DG) algebra of $A$.}
\begin{equation*}
Q(\kappa):=\e(\kappa) W_n^l .
\end{equation*}
The projective module $Q(\kappa)$ carries a $p$-DG structure inherited from that of $W_n^l$, which makes it into a cofibrant $p$-DG right module since $\dif(\e(\kappa))=0$.
\end{defn}

Below, we will use some special idempotents arising from $\nh_n^l$-multipartitions (Definition \ref{def-multipartition}) along with decompositions of $n$ into $l$ parts (Definition \ref{defdecompset}).

\begin{defn}\label{defpartitionsequencedictionary}
\begin{enumerate}
\item[(1)]
Let $\lambda \in \mathcal{P}_n^{l}$ be a partition. To such a $\lambda$ we associate a sequence $\kappa_{\lambda}$ 
%and $\tau_{\lambda}$  
of vertical lines as follows.
Reading from left to right in the sequence $\kappa_\lambda$, each box of $\lambda$ contributes a red strand followed by a black strand while any empty position contributes just a red strand.
\item[(2)] 
Let ${\bf b}=[b_1,\ldots,b_l]\in \mc{B}_n^l$ be a decomposition of $n$ into non-negative integers with $l$ parts.
To such a ${\bf b}$ we associate a sequence $\kappa_{{\bf b}}$ 
of vertical lines as follows.
Reading from left to right in the sequence $\kappa_{\bf b}$, we associate a red strand followed by $b_i$ black strands for $i$ ranging from $1$ to $l$.
\end{enumerate}
\end{defn}

The sequences in Definition \ref{defpartitionsequencedictionary} can in turn be regarded as Webster idempotents $\e(\kappa_\lambda)$ and $\e(\kappa_{\bf b})$.
We will define 
\begin{equation}
Q(\lambda):=Q(\kappa_\lambda)=\e(\kappa_\lambda)W_n^l, \quad \quad \quad
Q[{\bf b}]:=Q(\kappa_{\bf b})=\e(\kappa_{\bf b})W_n^l.
\end{equation}
When no confusion can be caused, we will always abbreviate $\e(\lambda):=\e(\kappa_\lambda)$ and $\e({\bf b}):= \e(\kappa_{\bf b})$.

\begin{rem}\label{rmkWebtautisoend}
Not every decomposition in $\mc{B}_n^l$ produces a non-zero element in the Webster algebra. One can show, using the Webster relations and \cite[Lemma 4.3]{KK}, that a decomposition ${\bf b}=[b_1,\dots b_l]\in \mc{B}_n^l$ has its corresponding idempotent $\e({\bf b})=0$ if and only if there exists a $k\in \{1,\dots, l\}$ such that $b_1+\cdots +b_k > k$. Furthermore, we have a tautological isomorphism of ($p$-DG) algebras
\begin{equation}
W_n^l \cong \END_{W_n^l}\left(\bigoplus_{\mathbf{b} \in \mc{B}_n^l} Q[\mathbf{b}]\right).
\end{equation}
\end{rem}

\begin{example} 
For $l=3,n=2$, the three partitions
\begin{equation*}
\lambda = (~\yng(1)~, ~\yng(1)~, ~\emptyset~)  \quad \quad
\mu = (~\yng(1)~, ~\emptyset~, ~\yng(1)~) \quad \quad
\nu = (~\emptyset~, ~\yng(1)~, ~\yng(1)~).
\end{equation*}
in $\mc{P}_2^3$ correspond to the sequences
\[
\kappa_\lambda=
\begin{DGCpicture}[scale=0.5]
\DGCstrand[Red](0,-2)(0,0)
\DGCstrand(1,-2)(1,0)
\DGCstrand[Red](2,-2)(2,0)
\DGCstrand(3,-2)(3,0)
\DGCstrand[Red](4,-2)(4,0)
\end{DGCpicture}
\quad \quad
\kappa_\mu=
\begin{DGCpicture}[scale=0.5]
\DGCstrand[Red](0,-2)(0,0)
\DGCstrand(1,-2)(1,0)
\DGCstrand[Red](2,-2)(2,0)
\DGCstrand[Red](3,-2)(3,0)
\DGCstrand(4,-2)(4,0)
\end{DGCpicture}
\quad \quad
\kappa_\nu=
\begin{DGCpicture}[scale=0.5]
\DGCstrand[Red](0,-2)(0,0)
\DGCstrand[Red](1,-2)(1,0)
\DGCstrand(2,-2)(2,0)
\DGCstrand[Red](3,-2)(3,0)
\DGCstrand(4,-2)(4,0)
\end{DGCpicture}
\ .
\]
The decomposition set $\mc{B}_2^3$ produces five non-zero idempotents, which are
\begin{equation*}
(0,2,0), \quad (1,0,1), \quad (0,0,2), \quad (1,1,0), \quad (0,1,1)
\end{equation*}
and they correspond to sequences
\[
\kappa_{(0,2,0)}=
\begin{DGCpicture}[scale=0.5]
\DGCstrand[Red](0,-2)(0,0)
\DGCstrand[Red](1,-2)(1,0)
\DGCstrand(2,-2)(2,0)
\DGCstrand(3,-2)(3,0)
\DGCstrand[Red](4,-2)(4,0)
\end{DGCpicture}
\quad \quad
\kappa_{(1,0,1)}=
\begin{DGCpicture}[scale=0.5]
\DGCstrand[Red](0,-2)(0,0)
\DGCstrand(1,-2)(1,0)
\DGCstrand[Red](2,-2)(2,0)
\DGCstrand[Red](3,-2)(3,0)
\DGCstrand(4,-2)(4,0)
\end{DGCpicture}
\quad \quad
\kappa_{(0,0,2)}=
\begin{DGCpicture}[scale=0.5]
\DGCstrand[Red](0,-2)(0,0)
\DGCstrand[Red](1,-2)(1,0)
\DGCstrand[Red](2,-2)(2,0)
\DGCstrand(3,-2)(3,0)
\DGCstrand(4,-2)(4,0)
\end{DGCpicture}
\]
\[
\kappa_{(1,1,0)}=
\begin{DGCpicture}[scale=0.5]
\DGCstrand[Red](0,-2)(0,0)
\DGCstrand(1,-2)(1,0)
\DGCstrand[Red](2,-2)(2,0)
\DGCstrand(3,-2)(3,0)
\DGCstrand[Red](4,-2)(4,0)
\end{DGCpicture}
\quad \quad
\kappa_{(0,1,1)}=
\begin{DGCpicture}[scale=0.5]
\DGCstrand[Red](0,-2)(0,0)
\DGCstrand[Red](1,-2)(1,0)
\DGCstrand(2,-2)(2,0)
\DGCstrand[Red](3,-2)(3,0)
\DGCstrand(4,-2)(4,0)
\end{DGCpicture}
\ .
\]
\end{example}

\begin{defn}\label{defsweepingelements}
Let $\kappa$ be a sequence of vertical lines.
\begin{enumerate}
\item[(1)] The element  ${\theta}_{\kappa} \in W_n^l$ is obtained from the diagram with minimal number of crossings and no dots, which takes the sequence of black and red boundary points at the bottom of the diagram governed by $\kappa$ to the sequence $\e(\kappa_0)$ at the top.
\item[(2)] The element $\theta_{\kappa}^* \in W_n^l$ is obtained by reflecting the diagram for  ${\theta}_{\kappa}$ in the horizontal axis at the bottom of the diagram.
\end{enumerate}
\end{defn}
When we concatenate a diagram with ${\theta}_{\kappa}$ (or $\theta^*_{\kappa}$), we say that we {\it sweep} the diagram to the upper (or lower) right.

\begin{example}
Suppose $\kappa$ is the sequence: red strand, followed by a black strand, followed by a red strand, followed by a black strand.
Then
\[
{\kappa} =
\begin{DGCpicture}[scale=0.5]
\DGCstrand[Red](1,0)(1,2)[$^{}$`{\ }]
\DGCstrand(2,0)(2,2)
\DGCstrand[Red](3,0)(3,2)
\DGCstrand(4,0)(4,2)
\end{DGCpicture}
\ , \quad \quad
{\theta}_{\kappa} =
\begin{DGCpicture}[scale=0.5]
%\DGCstrand[Red](0,0)(0,2)
\DGCstrand[Red](1,0)(1,2)[$^{}$`{\ }]
\DGCstrand(2,0)(3,2)
\DGCstrand[Red](3,0)(2,2)
\DGCstrand(4,0)(4,2)
%\DGCstrand[Red]/u/(6,0)(2,2)/u/[$^{}$`{\ }]
\end{DGCpicture}
\ , \quad \quad
\theta^*_{\kappa} =
\begin{DGCpicture}[scale=0.5]
\DGCstrand(3,0)(2,2)
\DGCstrand(4,0)(4,2)
\DGCstrand[Red](1,0)(1,2)[$^{}$`{\ }]
\DGCstrand[Red](2,0)(3,2)
\end{DGCpicture}
\ .
\]
\end{example}

\subsection{Diagrammatic quiver Schur algebras}
\label{subsec-conn-Web}
Our next goal is to relate the quiver Schur algebras to a special block of the Webster algebra $W_n^l$. 

We start by constructing a collection of submodules of the $p$-DG Webster algebra that restrict to the direct sum of modules 
$
G:=\oplus_{\lambda\in \mc{P}_n^{l}} G(\lambda)
$,
as a $p$-DG module over $\nh_n^l$.

Now, consider
\[
\HOM_{W_n^l}(Q(\kappa_0),Q(\lambda))\cong \e({\lambda})W_n^l\e(\kappa_0).
\]
The right hand side is naturally a module over $\e({\kappa_0})W_n^l\e(\kappa_0)\cong \nh_n^l$ (Lemma \ref{lem-Web-big-block}). Diagrammatically, the module consists of diagrams of the form
\begin{equation}
\e({\lambda})W_n^l\e(\kappa_0)
\cong
\left\{~
\begin{DGCpicture}
\DGCstrand[Red](0,-2)(0,0)[]
\DGCstrand(1,-2)(0.75,0)
\DGCstrand[Red](0.5,-2)(1.5,0)[]
\DGCstrand(2.25,-2)(2,0)
\DGCcoupon*(.2,-1.3)(.6,-1.4){$\cdots$}
\DGCcoupon*(1.35,-1.3)(1.75,-1.4){$\cdots$}
\DGCcoupon(0.8,-1.8)(2.4,-1.5){$x$}
\DGCcoupon(-0.2,-0.4)(2.45,-0.1){$_{\e({\lambda})}$}
\end{DGCpicture}~\Bigg|
x\in \nh_n^l\right\}
\end{equation}
where the $\e({\lambda})$ in the upper part of the diagram indicates the sequence of red and black boundary points at the top of the picture.

For $\lambda \in \mc{P}_n^{l}$, we have an isomorphism of right $\nh_n^l$-modules
\[
\e({\lambda})W_n^l\e(\kappa_0)\cong G(\lambda). 
\]
We illustrate the first isomorphism as follows. We sweep the black strands on the top all the way to the right, i.e., multiply on top of the diagram by the element ${\theta}_\kappa$, so that the top of the diagram ends at $\e(\kappa_0)$. Then we simplify the diagrams obtained using relation \eqref{eqn-Webster-double-crossing}. Finally we utilize the isomorphism of Lemma~\ref{lem-Web-big-block} to produce
\begin{equation}\label{eqnsweepingmap}
\begin{DGCpicture}
\DGCstrand[Red](0,-2)(0,0)[]
\DGCstrand(1,-2)(0.75,0)
\DGCstrand[Red](0.5,-2)(1.5,0)[]
\DGCstrand(2.25,-2)(2,0)
\DGCcoupon*(1.35,-1.3)(1.75,-1.4){$\cdots$}
\DGCcoupon(0.8,-1.8)(2.4,-1.5){$x$}
\DGCcoupon*(.2,-1.3)(.6,-1.4){$\cdots$}
\DGCcoupon(-0.2,-0.4)(2.45,-0.1){$_{\e({\lambda})}$}
\end{DGCpicture}
\mapsto
\begin{DGCpicture}
\DGCstrand[Red](0,-2)(0,0)[]
%\DGCstrand(0.5,-.7)(1.25,0)[]
\DGCstrand/u/(1,-2)(1,-1.75)(0.5,-0.8)/u/(1.25,0)/u/
\DGCstrand[Red]/u/(0.5,-2)(0.5,-1.75)/u/(1.25,-0.8)/u/(0.5,0)/u/[]
\DGCstrand(2,-0.7)(2,0)[]
\DGCstrand(2.25,-2)(2,-.7)
\DGCcoupon*(1.35,-1.3)(1.75,-1.4){$\cdots$}
\DGCcoupon(0.7,-1.8)(2.3,-1.5){$x$}
\DGCcoupon*(.2,-1.3)(.6,-1.4){$\cdots$}
\DGCcoupon(-0.2,-1)(2.3,-0.7){$_{\e({\lambda})}$}
\end{DGCpicture}
~=~
\begin{DGCpicture}
\DGCstrand[Red](-.5,-2)(-.5,0)[]
\DGCstrand(1,-2)(1,0)
\DGCstrand(2.25,-2)(2.25,0)
\DGCcoupon(0.75,-0.9)(2.4,-0.6){$_{y^{\lambda}}$}
%\DGCstrand(1.5,-2)(1.5,0)
\DGCstrand[Red](0.5,-2)(0.5,0)[]
%\DGCstrand(1.75,-2)(2,-.7)
\DGCcoupon*(1.45,-1.2)(1.85,-1.3){$\cdots$}
\DGCcoupon(0.8,-1.8)(2.4,-1.5){$x$}
\DGCcoupon*(-.3,-1.3)(.3,-1.4){$\cdots$}
\end{DGCpicture}
~\mapsto~
\begin{DGCpicture}
\DGCstrand(1,-2)(1,0)
\DGCstrand(2.25,-2)(2.25,0)
\DGCcoupon(0.75,-0.5)(2.4,-0.2){$_{y^{\lambda}}$}
%\DGCstrand(1.5,-2)(1.5,0)
%\DGCstrand(1.75,-2)(2,-.7)
\DGCcoupon*(1.45,-1.2)(1.85,-1.3){$\cdots$}
\DGCcoupon(0.8,-1.8)(2.4,-1.5){$x$}
\end{DGCpicture} \ .
\end{equation}
The middle equality holds by repeated applications of relation \eqref{eqn-Webster-double-crossing}. The sweeping map is always an injection \cite[Lemma 5.25]{Webcombined}. Furthermore, it is a $p$-DG homomorphism since, by Lemma \ref{lem-dif-on-Web}, we have
\[
\dif\left(~
\begin{DGCpicture}
\DGCstrand(0,0)(1,1)
\DGCstrand[Red](1,0)(0,1)
\end{DGCpicture}
~\right)
= 0.
\]
%The second isomorphism follows in a similar fashion using the sweeping element $\dot{\theta}_{\tau_\lambda}$.

The discussion above may be summarized by the following proposition.  Also
see \cite[Proposition 9.9]{KQS} and the discussion in \cite[Section 10.3]{KQS}.

\begin{prop}\label{propGlambdaisotoWebblock}
For each $\lambda \in \mc{P}_n^{l}$, there is an isomorphism of right $p$-DG modules over $\nh_n^l$
\[
\HOM_{W_n^l}(Q(\kappa_0),Q({\lambda}))\cong \e({\lambda})W_n^l\e(\kappa_0)\cong G(\lambda).
\]
\end{prop}

Summing over $\lambda\in \mc{P}_n^{l}$, we set
\begin{equation}
Q:=\bigoplus_{\lambda\in \mc{P}_n^{l}}Q(\lambda). 
\end{equation}
We have shown that
\begin{equation}
\HOM_{W_n^l}(Q(\kappa_0), Q)
\cong \bigoplus_{\lambda\in \mc{P}_n^{l}} G(\lambda)=G. 
\end{equation}

\begin{thm}\label{thm-iso-to-Webster-block}
There is an isomorphism of $p$-DG algebras 
\begin{equation*}
S_n^l \cong \END_{W_n^l}(Q).
\end{equation*}
\end{thm}

\begin{proof}
The result follows from the general framework of Theorem \ref{thm-pdg-extension}. 

More specifically, observe that the $\nh_n^l$-module $G=\oplus_{\lambda\in \mc{P}_n^l} G(\lambda)$ is a faithful $p$-DG module by Proposition \ref{propGisfaithful}. By Theorem \ref{thm-pdg-extension}, the functor
\[
\HOM_{S_n^l}(G,\mbox{-}):(S_n^l,\dif)\dmod\lra (\nh_n^l,\dif)\dmod
\]
is fully faithful on cofibrant summands of $S_n^l$. 
Thus, using Propositions \ref{Sonproj} and \ref{SonG}, we have
\[
\HOM_{S_n^l}(G,S_n^l)\cong \HOM_{\nh_n^l}(\mathcal{V}(G),\mathcal{V}(S_n^l))\cong \HOM_{\nh_n^l}(\nh_n^l,G)  \cong G.
\]
Now the result follows from the previous proposition identifying the space with $\HOM_{W_n^l}(Q(\kappa_0),Q)$.
\end{proof}

\subsection{A \texorpdfstring{$p$}{p}-DG Morita equivalence}
Our goal in this section is to establish a $p$-DG Morita equivalence between $S_n^l$ and $W_n^l$. 

We first construct a thick version of the idemptent $\e({\bf b})$ as follows. We place a diagram for the thick idempotent $e_{b_i}$ (equation \eqref{eqn-thick-idemp}) to the right of $i$th red strand  for $i=1,\dots, l$.
\begin{equation}\label{eqn-Webster-idem-2}
\e^\tau({\bf b})
=
\begin{DGCpicture}[scale=0.75]
\DGCstrand[Red](-1,-1)(-1,1)
\DGCstrand[Green](0,-1)(0,1)[`$_{b_1}$]
\DGCstrand[Red](1,-1)(1,1)
\DGCstrand[Green](2,-1)(2,1)[`$_{b_2}$]
\DGCcoupon*(2,-1)(3,1){$\cdots$}
\DGCstrand[Red](3,-1)(3,1)
\DGCstrand[Green](4,-1)(4,1)[`$_{b_l}$]
\end{DGCpicture} \ .
\end{equation}
We then set
\begin{equation}
Q^\tau[{\bf b}]:=\e^\tau({\bf b})Q[{\bf b}]=\e^\tau({\bf b})W_n^l.
\end{equation}
which is a $p$-DG right submodule of $Q[{\bf b}]$.

The proof of the next result is directly motivated by  \cite[Theorem 6]{Stosicsl3} originally established in the setting of KLR-algebras of type $A_2$. The result extends to the particular case of Webster algebra we are interested in without much difficulty, and we further show that it is true in the $p$-DG setting.  The proof is also closely related and similar to the proof of Proposition \ref{stosicnilhecke}.

\begin{prop}
\label{stosicweb}
There is an exact complex of right $p$-DG $W_n^l$-modules
\begin{equation*}
\xymatrix{
0 \ar[r] &
\begin{DGCpicture}
\DGCcoupon*(-0.5,.25)(0,.75){$_{...}$}
\DGCstrand[Red](0,0)(0,1)
\DGCstrand[Green](.5,0)(.5,1)[`$_a$]
\DGCcoupon*(.5,.25)(1,.75){$_{...}$}
\DGCcoupon(-0.5,0)(1,-.5){$W_n^l$}
\end{DGCpicture}
\ar[r]^-{d_0} &
\cdots
\ar[r]^-{d_{i-1}} &
\begin{DGCpicture}
\DGCcoupon*(-1,.25)(-.5,.75){$_{...}$}
\DGCstrand[Green](-.5,0)(-.5,1)[`$_{i}$]
\DGCstrand[Red](0,0)(0,1)
\DGCstrand[Green](.5,0)(.5,1)[`$_{a-i}$]
\DGCcoupon*(.5,.25)(1,.75){$_{...}$}
\DGCcoupon(-1,0)(1,-.5){$W_n^l$}
\end{DGCpicture} 
\ar[r]^-{ d_{i} } &
\cdots 
\ar[r]^-{d_{a-1}} &
\begin{DGCpicture}
\DGCcoupon*(-1,.25)(-.5,.75){$_{...}$}
\DGCstrand[Green](-.5,0)(-.5,1)[`$_{a}$]
\DGCstrand[Red](0,0)(0,1)
\DGCcoupon*(0,.25)(0.5,.75){$_{...}$}
\DGCcoupon(-1,0)(0.5,-.5){$W_n^l$}
\end{DGCpicture}
\ar[r] &
0
}
\end{equation*}
that is contractible as a complex of modules over $\nh_n^l$.
The differential $d_i$  is given, locally, by juxtaposition on top of the depicted strands by the element
\begin{equation*}
D_i =
\begin{DGCpicture}[scale=0.85]
\DGCstrand/u/(1.5,.5)(-.5,1.5)/u/
\DGCdot{1.1}[u]{$^{a-2}$}
\DGCstrand[Green](-.5,0)(-.5,2)[$^{i}$`$_{i+1}$]
\DGCstrand[Red](.5,0)(.5,2)
\DGCstrand[Green](1.5,0)(1.5,2)[$^{a-i}$`$_{a-i-1}$]
\end{DGCpicture} .
\end{equation*}
\end{prop}

\begin{proof}
The fact that $d_i^2=0$ follows in the same way as the proof of
Proposition \ref{stosicnilhecke}. Furthermore, one computes that $\partial(D_i)=0$ using the differential actions \eqref{eqn-d-action-mod-generator} and \eqref{diffonWeb}, so that the complex is indeed a complex of $p$-DG modules.

Define the homotopy map $h_i$, $i=0,\dots, a-1$, by
\begin{equation*}
h_i
:
\begin{DGCpicture}
\DGCcoupon*(-1,.25)(-.5,.75){$_{...}$}
\DGCstrand[Green](-.5,0)(-.5,1)[`$_{i+1}$]
\DGCstrand[Red](0,0)(0,1)
\DGCstrand[Green](.5,0)(.5,1)[`$_{a-i-1}$]
\DGCcoupon*(.5,.25)(1,.75){$_{...}$}
\DGCcoupon(-1,0)(1,-.5){$x$}
\end{DGCpicture} 
\mapsto
(-1)^{a-i} 
\begin{DGCpicture}
\DGCstrand/u/(-0.5,0.15)(0.5,0.85)
\DGCcoupon*(-1,.25)(-.5,.75){$_{...}$}
\DGCstrand[Green](-.5,0)(-.5,1)[`$_i$]
\DGCstrand[Red](0,0)(0,1)
\DGCstrand[Green](.5,0)(.5,1)[`$_{a-i}$]
\DGCcoupon*(.5,.25)(1,.75){$_{...}$}
\DGCcoupon(-1,0)(1,-.5){$x$}
\end{DGCpicture} 
\ .
\end{equation*}
Again, as in the proof of Proposition \ref{stosicnilhecke}, one checks that these maps $h_i$ provide the necessary homotopies to contract the complex.
\end{proof}

We remark that, in view of Theorem \ref{thmtensorcatfn} and the next result, the complex above is a categorical lifting of the equality in Lemma \ref{lemSerrelikerelation}.

\begin{thm}\label{thmqschurwebmoritaequiv}
The algebras $S_n^l$ and $W_n^l$ are $p$-DG Morita equivalent. The $p$-DG Morita equivalence is realized by tensoring with the $p$-DG bimodules
\[
\HOM_{W_n^l}\left(Q, \bigoplus_{ {\bf b}\in \mc{B}_n^l} Q[{\bf b}]\right)
\cong \bigoplus_{\lambda\in \mc{P}_n^l}W_n^l \e(\lambda)
, \quad \quad
\HOM_{W_n^l} \left( \bigoplus_{{\bf b}\in \mc{B}_n^l} Q[{\bf b}],Q\right)\cong
\bigoplus_{\lambda\in \mc{P}_n^l}\e(\lambda) W_n^l.
\]
\end{thm}

\begin{proof}
By Theorem 
\ref{thm-iso-to-Webster-block} and Remark \ref{rmkWebtautisoend}
there are isomorphisms of $p$-DG algebras 
\begin{equation*}
S_n^l \cong \END_{W_n^l}\left(\bigoplus_{\lambda\in \mc{P}_n^l}Q(\lambda)\right), \quad \quad W_n^l \cong \END_{W_n^l}\left(\bigoplus_{{\bf b}\in \mc{B}_n^l}Q[{\bf b}]\right).
\end{equation*}
By iterated application of Proposition \ref{prop-p-dg-Morita}, it suffices to show that each $Q[{\bf b}]$ is contained in the filtered $p$-DG envelope of $Q=\oplus_{\lambda\in \mc{P}_n^l}Q(\lambda)$.

It is clear that $Q[{\bf b}]$ has a finite-step $p$-DG filtration, whose subquotients are isomorphic to $Q^{\tau}[{\bf b}]$. Therefore it suffices to show that each $Q^{\tau}[{\bf b}]$ is contained in the envelope of $Q$. This is proved in a manner similar to what we have done in the proof of Proposition \ref{auxmodsinenv}. 

The last statement follows by an inductive application on $\mc{P}_n^l$ via Proposition \ref{prop-p-dg-Morita}. We leave the details to the reader as an exercise.
\end{proof}

\subsection{A diagrammatic cellular basis}
In this subsection, we first recall a diagrammatic cellular basis for the Webster algebra $W_n^l$ defined in \cite{SWSchur}. Then, utilizing Theorem \ref{thm-iso-to-Webster-block}, we provide a diagrammatic description of Hu-Mathas' cellular basis for the quiver Schur algebra $S_n^l$.

\begin{defn}\label{defWebstercomb}
Let $\lambda$ be an $\nh_n^l$-partition. Fix $l$ alphabets
$ (1_i, 2_i, \ldots,   )_{i=1}^l $.  A filling $\mf{t}$ of $\lambda$ by the alphabets is called a \emph{multi-standard tableau} if it satisfies the following two conditions.
\begin{enumerate}
\item[(1)] The $i$th alphabet can only appear in the first $i$ parts of $\lambda$.
\item[(2)] No gap should occur within each alphabet, that is, if the letter $a_i$ occurs in $\mf{t}$, then $1_i, \ldots, (a-1)_i$ must have already been filled in $\mf{t}$ (or the entire $i$th alphabet may be skipped as well).
%\item[(3)] The entire $i$th alphabet must occur before the $(i+1)$st alphabet.
\end{enumerate}
Denote set of all multi-standard tableaux of shape $\lambda$ by $\T^\prime_\lambda$.
\end{defn}

Suppose the boxes of $\lambda$ occur in positions $j_1, \ldots, j_n$ with $j_1 < \cdots < j_n$, which we refer to as the \emph{box numbers} below.  Define the \emph{standard filling $\mf{t}^\lambda$} to be the tableau with letter $1_{j_k}$ in the box in position $j_k$.

To tableaux $\mf{t}, \mf{s}\in \T^\prime_\lambda$, Stroppel and Webster \cite{SWSchur} associate elements $B^{{\mf{t}}}_\lambda$, $B^\lambda_{\mf{t}}$ and $B^\lambda_{\mf{t}\mf{s}}$ of $W_n^l$ as follows.

\begin{defn}\label{defStroppelWebsterbasis}
Let $\mf{t} $ a be multi-standard tableau in $ \T^\prime_\lambda$. We first read off two sequences of red and labeled black dots as follows.
\begin{itemize}
\item[(a)](\emph{The entry sequence}). For each letter in the $i$th alphabet entered into $\mf{t}$, read off in increasing order the box number it is entered, and draw a black dot decorated by the box number. Collect, in increasing order of alphabets, the labeled groups of black dots, with a red dot in between two neighboring groups. 
\item[(b)](\emph{The shape sequence}). Draw a red dot for an empty slot, and a red dot followed by a black dot for a box.  Record the black dot with the box number it corresponds to.
\end{itemize}
We label all red dots its position number counted from the left (omitted if clear from context). It is then easy to see that the shape sequence of $\mf{t}$ is equal to the entry sequence for the standard filling $\mf{t}^\lambda\in \T^\prime_\lambda$.
\begin{enumerate}
\item[(1)] For each $\mf{t}\in \T^\prime_\lambda$, place its entry sequence on a top horizontal line and the shape sequence on a lower horizontal line. The element $B_\lambda^{\mf{t}}$ is the element of $W_n^l$ obtained by connecting the entry and shape sequences in a minimal way, with red and black dots connected by matching their labels.
\item[(2)] Define the element $B^\lambda_{\mf{t}}:=(B_\lambda^{\mf{t}})^*$.
\item[(3)] For a pair of tableaux $\mf{s},\mf{t}\in \T^\prime_\lambda$, set $B^\lambda_{\mf{s}\mf{t}}:=B^{\mf{s}}_\lambda B^\lambda_{\mf{t}}$.
\end{enumerate}
\end{defn}

\begin{example}\label{egWebstercomb}
For $l=3,n=2$, the three partitions are
\begin{equation*}
\lambda = (~\yng(1)~, ~\yng(1)~, ~\emptyset~)  \quad \quad
\mu = (~\yng(1)~, ~\emptyset~, ~\yng(1)~) \quad \quad
\nu = (~\emptyset~, ~\yng(1)~, ~\yng(1)~).
\end{equation*}
The corresponding sets of multi-standard tableaux are given by
\begin{center}
\begin{tabular}{|c|c|c|}
\hline 
$\T^\prime_\lambda$ & $\T^\prime_\mu$ & $\T^\prime_\nu$ \\ 
\hline 
$\mf{r}_1=\left(~\fbox{$1_1$}~,~\fbox{$1_2$}~,~\emptyset~\right)$ & $\mf{s}_1=\left(~\fbox{$1_1$}~,~\emptyset~,~\fbox{$1_3$}~\right)$ & $\mf{t}_1=\left(~\emptyset~,~\fbox{$1_2$},~\fbox{$1_3$}~\right)$  \\ 
$\mf{r}_2=\left(~\fbox{$1_1$}~,~\fbox{$1_3$}~,~\emptyset~\right)$ & $\mf{s}_2=\left(~\fbox{$1_2$}~,~\emptyset~,~\fbox{$1_3$}~\right)$ & $\mf{t}_2=\left(~\emptyset~,~\fbox{$1_3$},~\fbox{$2_3$}~\right)$  \\  
$\mf{r}_3=\left(~\fbox{$1_2$}~,~\fbox{$1_3$}~,~\emptyset~\right)$ & $\mf{s}_3=\left(~\fbox{$1_3$}~,~\emptyset~,~\fbox{$2_3$}~\right)$ & $\mf{t}_3=\left(~\emptyset~,~\fbox{$2_3$},~\fbox{$1_3$}~\right)$   \\ 
$\mf{r}_4=\left(~\fbox{$1_2$}~,~\fbox{$2_2$}~,~\emptyset~\right)$ & $\mf{s}_4=\left(~\fbox{$2_3$}~,~\emptyset~,~\fbox{$1_3$}~\right)$ &   \\ 
$\mf{r}_5=\left(~\fbox{$2_2$}~,~\fbox{$1_2$}~,~\emptyset~\right)$ & &    \\  
$\mf{r}_6=\left(~\fbox{$1_3$}~,~\fbox{$1_2$}~,~\emptyset~\right)$ &  &   \\  
$\mf{r}_7=\left(~\fbox{$1_3$}~,~\fbox{$2_3$}~,~\emptyset~\right)$ &  &  \\ 
$\mf{r}_8=\left(~\fbox{$2_3$}~,~\fbox{$1_3$}~,~\emptyset~\right)$ &  &   \\ 
\hline   
\end{tabular}
\end{center}
We first give the entry sequences for the tableaux $\mf{r}_6$, $\mf{r}_8$, $\mf{s}_2$ and $\mf{t}_3$ as examples, which are
\[
\mf{r}_6:~
\begin{DGCpicture}
\DGCstrand[Red](0,0)(0,0.25)
\DGCstrand[Red](0.5,0)(0.5,0.25)
\DGCstrand(1,0)(1,0.25)[{}`$_{2}$]
\DGCstrand[Red](1.5,0)(1.5,0.25)
\DGCstrand(2,0)(2,0.25)[`$_{1}$]
\end{DGCpicture}
\quad \quad 
\mf{r}_8:~
\begin{DGCpicture}
\DGCstrand[Red](0,0)(0,0.25)
\DGCstrand[Red](0.5,0)(0.5,0.25)
\DGCstrand[Red](1,0)(1,0.25)
\DGCstrand(1.5,0)(1.5,0.25)[{}`$_{2}$]
\DGCstrand(2,0)(2,0.25)[`$_{1}$]
\end{DGCpicture}
\quad \quad 
\mf{s}_2:~
\begin{DGCpicture}
\DGCstrand[Red](0,0)(0,0.25)
\DGCstrand[Red](0.5,0)(0.5,0.25)
\DGCstrand(1,0)(1,0.25)[{}`$_{1}$]
\DGCstrand[Red](1.5,0)(1.5,0.25)
\DGCstrand(2,0)(2,0.25)[`$_{3}$]
\end{DGCpicture}
\quad \quad 
\mf{t}_3:~
\begin{DGCpicture}
\DGCstrand[Red](0,0)(0,0.25)
\DGCstrand[Red](0.5,0)(0.5,0.25)
\DGCstrand[Red](1,0)(1,0.25)
\DGCstrand(1.5,0)(1.5,0.25)[{}`$_{3}$]
\DGCstrand(2,0)(2,0.25)[`$_{2}$]
\end{DGCpicture} \ .
\] 
Their shape sequences are given respectively by
\[
\mf{r}_6:~
\begin{DGCpicture}
\DGCstrand[Red](0,0)(0,0.25)
\DGCstrand(0.5,0)(0.5,0.25)[$^{1}$]
\DGCstrand[Red](1,0)(1,0.25)
\DGCstrand(1.5,0)(1.5,0.25)[$^{2}$]
\DGCstrand[Red](2,0)(2,0.25)
\end{DGCpicture}
\quad \quad 
\mf{r}_8:~
\begin{DGCpicture}
\DGCstrand[Red](0,0)(0,0.25)
\DGCstrand(0.5,0)(0.5,0.25)[$^{1}$]
\DGCstrand[Red](1,0)(1,0.25)
\DGCstrand(1.5,0)(1.5,0.25)[$^{2}$]
\DGCstrand[Red](2,0)(2,0.25)
\end{DGCpicture}
\quad \quad 
\mf{s}_2:~
\begin{DGCpicture}
\DGCstrand[Red](0,0)(0,0.25)
\DGCstrand(0.5,0)(0.5,0.25)[$^{1}$]
\DGCstrand[Red](1,0)(1,0.25)
\DGCstrand[Red](1.5,0)(1.5,0.25)
\DGCstrand(2,0)(2,0.25)[$^{3}$]
\end{DGCpicture}
\quad \quad 
\mf{t}_3:~
\begin{DGCpicture}
\DGCstrand[Red](0,0)(0,0.25)
\DGCstrand[Red](0.5,0)(0.5,0.25)
\DGCstrand(1,0)(1,0.25)[$^{2}$]
\DGCstrand[Red](1.5,0)(1.5,0.25)
\DGCstrand(2,0)(2,0.25)[$^{3}$]
\end{DGCpicture} \ .
\] 
Now we connect the entry and shape sequences to obtain $B^{\mf{r}_6}_{\lambda}$, $B^{\mf{r}_8}_{\lambda}$, $B^{\mf{s}_2}_{\mu}$ and $B^{\mf{t}_3}_{\nu}$:
\[
B^{\mf{r}_6}_\lambda=~
\begin{DGCpicture}
\DGCstrand[Red](0,0)(0,1)
\DGCstrand(0.5,0)(2,1)[$^{1}$]
\DGCstrand[Red](1,0)(0.5,1)
\DGCstrand(1.5,0)(1,1)[$^{2}$]
\DGCstrand[Red](2,0)(1.5,1)
\end{DGCpicture}
\quad \quad 
B^{\mf{r}_8}_\lambda=~
\begin{DGCpicture}[scale= 0.5]
\DGCstrand[Red](0,0)(0,2)
\DGCstrand(1,0)(4,2)[$^{1}$]
\DGCstrand[Red](2,0)(1,2)
\DGCstrand(3,0)(4,1)(3,2)[$^{2}$]
\DGCstrand[Red](4,0)(2,2)
\end{DGCpicture}
\quad \quad 
B^{\mf{s}_2}_\mu=~
\begin{DGCpicture}
\DGCstrand[Red](0,0)(0,1)
\DGCstrand(0.5,0)(1,1)[$^{1}$]
\DGCstrand[Red](1,0)(0.5,1)
\DGCstrand[Red](1.5,0)(1.5,1)
\DGCstrand(2,0)(2,1)[$^{3}$]
\end{DGCpicture}
\quad \quad 
B^{\mf{t}_3}_\nu=~
\begin{DGCpicture}
\DGCstrand[Red](0,0)(0,1)
\DGCstrand[Red](0.5,0)(0.5,1)
\DGCstrand(1,0)(2,1)[$^{2}$]
\DGCstrand[Red](1.5,0)(1,1)
\DGCstrand(2,0)(1.5,1)[$^{3}$]
\end{DGCpicture} \ .
\] 

As examples of $B^\lambda_{\mf{t}}$ and $B^\lambda_{\mf{s}\mf{t}}$, we give
\[
B_{\mf{r}_8}^\lambda=~
\begin{DGCpicture}[scale=0.5]
\DGCstrand[Red](0,-2)(0,0)
\DGCstrand(4,-2)(1,0)[`$_{1}$]
\DGCstrand[Red](1,-2)(2,0)
\DGCstrand(3,-2)(4,-1)(3,0)[`$_{2}$]
\DGCstrand[Red](2,-2)(4,0)
\end{DGCpicture} \ ,
\quad 
\quad 
\quad
B_{\mf{r}_6\mf{r}_8}^\lambda=~
\begin{DGCpicture}[scale= 0.5, ratio=0.8]
\DGCstrand[Red](0,0)(0,2)
\DGCstrand(1,0)(4,2)
\DGCstrand[Red](2,0)(1,2)
\DGCstrand(3,-2)(4,-1)(3,0)(2,2)
\DGCstrand[Red](4,0)(3,2)
\DGCstrand[Red](0,-2)(0,0)
\DGCstrand(4,-2)(1,0)
\DGCstrand[Red](1,-2)(2,0)
\DGCstrand[Red](2,-2)(4,0)
\end{DGCpicture} \ .
\]
\end{example}

\begin{defn}\label{defcellchainWebster}
With respect to the dominance order (Definition \ref{def-order-on-nh-partition}) on $\mc{P}_n^l$, let $(W_n^l)^{\geq \lambda}$ be the $\Bbbk$-linear span of
\begin{equation*}
\{B^\mu_{\mf{s}\mf{t}} | \mf{s},\mf{t} \in \T^\prime_\mu, \mu \geq \lambda \}.
\end{equation*}
Similarly, let $(W_n^l)^{> \lambda}$ be the $\Bbbk$-linear span of
\begin{equation*}
\{B^\mu_{\mf{s}\mf{t}} | \mf{s},\mf{t} \in \T^\prime_\mu, \mu > \lambda \}.
\end{equation*}
\end{defn}

\begin{thm}\label{thmWebstercellular}
The algebra $W_n^l$ is a $p$-DG quasi-hereditary cellular algebra with a cellular basis
\begin{equation*}
\bigsqcup_{\lambda\in \mc{P}_n^l}\{B_{\mf{t}\mf{s}}^\lambda  | \mf{t},\mf{s} \in \T^\prime_\lambda \}
\end{equation*}
Moreover, the chain of ideals $\{(W_n^l)^{\geq \lambda}|\lambda\in \mc{P}_n^l\}$ constitutes a $p$-DG cellular chain of ideals in $W_n^l$. 
\end{thm}
\begin{proof}
In the absence of the $p$-differential, the combinatorial construction of the cellular basis is due to Stroppel and Webster \cite{SWSchur}. We will only show that the chain of ideals is indeed a $p$-DG quasi-hereditary cellular chain (Definition \ref{defcellalgquasihereditary}) utilizing Theorem \ref{thmqschurwebmoritaequiv}.

Each ideal in the chain of cellular ideals $\{ (S_n^l)^{\geq \lambda}|\lambda\in \mc{P}_n^l\} $ on $S_n^l$, given in Proposition \ref{qschurispdg}, is isomorphic to the two-sided ideal generated by the idempotents $\sum_{\mu \geq \lambda} \e(\mu)$ under the isomorphism of Theorem \ref{thm-iso-to-Webster-block}. Thus, under the $p$-DG Morita equivalence of Theorem \ref{thmqschurwebmoritaequiv}, they translate into two-sided ideals
\[
\left(\bigoplus_{\mu\in \mc{P}_n^l}W_n^l\e(\mu)\right) \otimes_{S_n^l} (S_n^l)^{\geq \lambda} \otimes_{S_n^l}
\left( \bigoplus_{\nu\in \mc{P}_n^l}\e(\nu)W_n^l \right)
\cong
W_{n}^l\left(\sum_{\mu\geq \lambda} \e(\mu)\right) W_n^l = (W_n^l)^{\geq \lambda}, \quad \lambda\in \mc{P}_n^l.
\]
The result follows.
\end{proof}

Motivated by the cellular combinatorics of Hu-Mathas (Definition \ref{defcellularbasisofSchur}) and Stroppel-Webster (Definition \ref{defStroppelWebsterbasis}), we make the following definition.

\begin{defn}\label{defcellbasisdiagrams}
Fix a partition $\lambda \in \mathcal{P}_n^l$, and consider $\T_\lambda=\sqcup_{\mu\in \mc{P}_n^l} \Tab_\lambda(\mu)$. 
\begin{enumerate}
\item[(1)]To a tableau $\mf{t} \in \Tab_{\lambda}(\mu)$ we will associate an element $D^{\mf{t}}_\lambda \in \END_{W_n^l}(Q)$ in the following procedure.
\begin{enumerate}
\item[(1.a)] The partitions $\lambda$ and $\mu$ each create a sequence of $l$ red dots and $n$ black dots.
Reading from left to right, each box contributes a red dot followed by a black dot while any empty position contributes just a red dot. In other words, we assign to $\lambda$, $\mu$ the idempotents $\e(\lambda)$ and $\e(\mu)$ respectively.
\item[(1.b)] The bottom boundary of the diagram $D^{\mf{t}}_\lambda$ is given by the sequence associated to $\lambda$, while the top boundary of the diagram $D^{\mf{t}}_\lambda$ is given by the sequence associated to $\mu$. In other words, the element $D^{\mf{t}}_\lambda\in \e(\mu) W_n^l \e(\lambda)$.
\item[(1.c)] The filling $\mf{t}$ of $\mu$ determines a ${\mu \choose \lambda}$-permissible permutation $w_{\mf{t}}\in \mf{S}_n$ (see Definition \ref{defallowablepermutation}) and, in turn, an element $\psi_\mf{t}$ of $\nh_n$.  Connect the black dots on the top to the black dots on the bottom using the diagrammatic description of $\psi_\mf{t}$. 
\item[(1.d)] Connect the remaining red dots in the top and bottom by red strands in a minimal way such that no red strands cross, and there are no triple intersections with black strands.
\end{enumerate}
\item[(2)] Set $D^\lambda_{\mf{t}}:=(D^{\mf{t}}_\lambda)^*$, the diagram of which is equal to that of $D^{\mf{t}}_\lambda$ flipped upside-down.
\item[(3)] For any $\mf{s}, \mf{t} \in \T_{\lambda}$, let
$D^{\lambda}_{\mf{t}\mf{s}}=D^{\mf{t}}_\lambda \cdot D_{\mf{s}}^\lambda$.
\end{enumerate}
\end{defn}

Let us illustrate the above definition with the following example, continued from Example \ref{eg-23partitions} and \ref{eg-23-tableaux}.

\begin{example} \label{eg-23-cellularbasis}
For $l=3,n=2$, the three partitions are
\begin{equation*}
\lambda = (~\yng(1)~, ~\yng(1)~, ~\emptyset~)  \quad \quad
\mu = (~\yng(1)~, ~\emptyset~, ~\yng(1)~) \quad \quad
\nu = (~\emptyset~, ~\yng(1)~, ~\yng(1)~).
\end{equation*}
We have determined the sets $\T_\lambda$, $\T_\mu$ and $\T_\nu$ in Example \ref{eg-23-tableaux}. Let us now follow Definition \ref{defcellbasisdiagrams} to give some examples in these cases.
\begin{enumerate}
\item[(1)] For $\lambda$ we have
\[
\T_\lambda=\Tab_\lambda(\lambda)\sqcup \Tab_\lambda(\mu) \sqcup \Tab_\lambda(\nu).
\]
Each tableau in $\T_\lambda$ has its corresponding diagram with bottom boundary equal to $\e(\lambda)$. Their diagrams are given as follows.

\[
\mathrm{Tab}_{\lambda}(\lambda)=\left\{
\mf{r}=
\left(~\young(1)~,~\young(2)~,~\emptyset~\right)
\right\} \quad \quad
\mathrm{Tab}_{\lambda}(\mu)=\left\{
\mf{s}=
\left(~\young(1)~,~\emptyset~,~\young(2)~\right)
\right\} 
\]

\[
D^{\mf{r}}_\lambda=
\begin{DGCpicture}[scale=0.5]
\DGCstrand[Red](0,-2)(0,0)
\DGCstrand(1,-2)(1,0)
\DGCstrand[Red](2,-2)(2,0)
\DGCstrand(3,-2)(3,0)
\DGCstrand[Red](4,-2)(4,0)
\end{DGCpicture}
\quad \quad \quad \quad
D^{\mf{s}}_\lambda=
\begin{DGCpicture}[scale=0.5]
\DGCstrand[Red](0,-2)(0,0)
\DGCstrand(1,-2)(1,0)
\DGCstrand[Red](2,-2)(2,0)
\DGCstrand(3,-2)(4,0)
\DGCstrand[Red](4,-2)(3,0)
\end{DGCpicture}
\]

\[
\mathrm{Tab}_\lambda(\nu)=\left\{
\mf{t}=
\left(~\emptyset~,~\young(1)~,~\young(2)~\right),\quad \quad \quad
\mf{u}=
\left(~\emptyset~,~\young(2)~,~\young(1)~\right)
\right\}
\]

\[ 
D^{\mf{t}}_\lambda=
\begin{DGCpicture}[scale=0.5]
\DGCstrand[Red](0,-2)(0,0)
\DGCstrand(1,-2)(2,0)
\DGCstrand[Red](2,-2)(1,0)
\DGCstrand(3,-2)(4,0)
\DGCstrand[Red](4,-2)(3,0)
\end{DGCpicture}
\quad \quad \quad \quad
D^{\mf{u}}_\lambda=
\begin{DGCpicture}[scale=0.5]
\DGCstrand[Red](0,-2)(0,0)
\DGCstrand(3,-2)(2,0)
\DGCstrand(1,-2)(4,0)
\DGCstrand[Red](4,-2)(3,0)
\DGCstrand[Red](2,-2)(1,0)
\end{DGCpicture}
\ .
\]
\item[(2)]As an instance of elements in (2) and (3) of Definition \ref{defcellbasisdiagrams}, we have:
\[
D_{\mf{u}}^\lambda=(D^{\mf{u}}_\lambda)^*=~
\begin{DGCpicture}[scale=0.5]
\DGCstrand[Red](0,-2)(0,0)
\DGCstrand(2,-2)(3,0)
\DGCstrand(4,-2)(1,0)
\DGCstrand[Red](1,-2)(2,0)
\DGCstrand[Red](3,-2)(4,0)
\end{DGCpicture}
\ , \quad \quad \quad
D_{\mf{t}\mf{u}}^{\lambda}
= D^{\mf{t}}_\lambda\cdot D^{\lambda}_{\mf{u}}
=~
\begin{DGCpicture}[scale=0.35]
\DGCstrand[Red](0,-2)(0,0)
\DGCstrand(2,-4)(3,-2)(4,0)
\DGCstrand(1,-2)(2,0)
\DGCstrand[Red](4,-2)(3,0)
\DGCstrand[Red](2,-2)(1,0)
\DGCstrand[Red](0,-4)(0,-2)
\DGCstrand(4,-4)(1,-2)
\DGCstrand[Red](1,-4)(2,-2)
\DGCstrand[Red](3,-4)(4,-2)
\end{DGCpicture} 
\ .
\]
\item[(3)] For $\mu$ there is
\[
\T_\mu=\Tab_\mu(\mu)\sqcup \Tab_\mu(\nu).
\]
The tableaux and their corresponding diagrams are matched as follows:
\[
\mathrm{Tab}_{\mu}(\mu)=\left\{
\mf{v}=
\left(~\young(1)~,~\emptyset~,~\young(2)~\right)
\right\} \quad  \quad \quad
\mathrm{Tab}_{\mu}(\nu)=\left\{
\mf{w}=
\left(~\emptyset~,~\young(1)~,~\young(2)~\right)
\right\} 
\]

\[
D^{\mf{v}}_\mu=
\begin{DGCpicture}[scale=0.5]
\DGCstrand[Red](0,-2)(0,0)
\DGCstrand(1,-2)(1,0)
\DGCstrand[Red](2,-2)(2,0)
\DGCstrand[Red](3,-2)(3,0)
\DGCstrand(4,-2)(4,0)
\end{DGCpicture}
\quad \quad \quad
D^{\mf{w}}_\mu=
\begin{DGCpicture}[scale=0.5]
\DGCstrand[Red](0,-2)(0,0)
\DGCstrand(1,-2)(2,0)
\DGCstrand[Red](2,-2)(1,0)
\DGCstrand[Red](3,-2)(3,0)
\DGCstrand(4,-2)(4,0)
\end{DGCpicture}
\ .
\]
\item[(4)] Finally, for $\nu$, we have only one tableau and its corresponding diagram

\[
\T_\nu=\mathrm{Tab}_{\nu}(\nu)=\left\{
\mf{z}=
\left(~\emptyset~,~\young(1)~,~\young(2)~\right)
\right\} 
\]

\[
D^{\mf{z}}_\nu=
\begin{DGCpicture}[scale=0.5]
\DGCstrand[Red](0,-2)(0,0)
\DGCstrand[Red](1,-2)(1,0)
\DGCstrand(2,-2)(2,0)
\DGCstrand[Red](3,-2)(3,0)
\DGCstrand(4,-2)(4,0)
\end{DGCpicture}
\ .
\]
\end{enumerate}
\end{example}

\begin{rem}\label{rmk-cellcombinatorics-Webster}
In the proof of Lemma \ref{bijlem}, we have associated with each tableau in $\Tab_\lambda(\mu)$ a ${\mu \choose \lambda}$-permissible permutation by turning each box into a black dot, and connecting the dots according to the correspondence determined by the tableaux. Here, we have instead enhanced the algorithm by turning boxes into a red dot followed by a black dot, while making empty slots just red dots. We still match the black dots  according to the given tableau, and the red dots are then just matched up according to their numbering counted from the left. For instance, for the tableau $\mf{u}$ in the example above, we may regard $D^{\mf{u}}_\lambda$ as built in the following process:
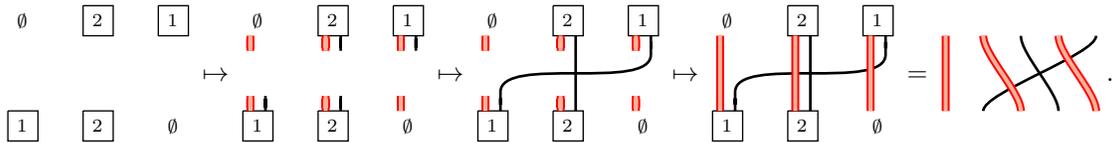
\[
\begin{DGCpicture}
\DGCcoupon(0,0)(0.4,-0.4){$_1$}
\DGCcoupon(1,0)(1.4,-0.4){$_2$}
\DGCcoupon*(2,0)(2.4,-0.4){$_\emptyset$}
\DGCcoupon*(0,1.4)(0.4,1){$_\emptyset$}
\DGCcoupon(1,1.4)(1.4,1){$_2$}
\DGCcoupon(2,1.4)(2.4,1){$_1$}
\end{DGCpicture}~\mapsto~
\begin{DGCpicture}
\DGCcoupon(0,0)(0.4,-0.4){$_1$}
\DGCcoupon(1,0)(1.4,-0.4){$_2$}
\DGCcoupon*(2,0)(2.4,-0.4){$_\emptyset$}
\DGCcoupon*(0,1.4)(0.4,1){$_\emptyset$}
\DGCcoupon(1,1.4)(1.4,1){$_2$}
\DGCcoupon(2,1.4)(2.4,1){$_1$}
\DGCstrand(1.3,0)(1.3,0.2)
\DGCstrand(0.3,0)(0.3,0.2)
\DGCstrand[Red](1.1,0)(1.1,0.2)
\DGCstrand[Red](0.1,0)(0.1,0.2)
\DGCstrand[Red](2.1,0)(2.1,0.2)
\DGCstrand(2.3,0.8)(2.3,1)
\DGCstrand(1.3,0.8)(1.3,1)
\DGCstrand[Red](1.1,0.8)(1.1,1)
\DGCstrand[Red](0.1,0.8)(0.1,1)
\DGCstrand[Red](2.1,0.8)(2.1,1)
\end{DGCpicture}
~\mapsto~
\begin{DGCpicture}
\DGCcoupon(0,0)(0.4,-0.4){$_1$}
\DGCcoupon(1,0)(1.4,-0.4){$_2$}
\DGCcoupon*(2,0)(2.4,-0.4){$_\emptyset$}
\DGCcoupon*(0,1.4)(0.4,1){$_\emptyset$}
\DGCcoupon(1,1.4)(1.4,1){$_2$}
\DGCcoupon(2,1.4)(2.4,1){$_1$}
\DGCstrand(1.3,0)(1.3,1)
\DGCstrand(0.3,0)(0.3,0.25)(2.3,0.75)(2.3,1)
\DGCstrand[Red](1.1,0)(1.1,0.2)
\DGCstrand[Red](0.1,0)(0.1,0.2)
\DGCstrand[Red](2.1,0)(2.1,0.2)
\DGCstrand[Red](1.1,0.8)(1.1,1)
\DGCstrand[Red](0.1,0.8)(0.1,1)
\DGCstrand[Red](2.1,0.8)(2.1,1)
\end{DGCpicture}
~\mapsto~
\begin{DGCpicture}
\DGCcoupon(0,0)(0.4,-0.4){$_1$}
\DGCcoupon(1,0)(1.4,-0.4){$_2$}
\DGCcoupon*(2,0)(2.4,-0.4){$_\emptyset$}
\DGCcoupon*(0,1.4)(0.4,1){$_\emptyset$}
\DGCcoupon(1,1.4)(1.4,1){$_2$}
\DGCcoupon(2,1.4)(2.4,1){$_1$}
\DGCstrand(1.3,0)(1.3,1)
\DGCstrand(0.3,0)(0.3,0.25)(2.3,0.75)(2.3,1)
\DGCstrand[Red](1.1,0)(1.1,1)
\DGCstrand[Red](0.1,0)(0.1,1)
\DGCstrand[Red](2.1,0)(2.1,1)
\end{DGCpicture}
~=~
\begin{DGCpicture}[scale=0.5]
\DGCstrand[Red](0,-2)(0,0)
\DGCstrand(3,-2)(2,0)
\DGCstrand(1,-2)(4,0)
\DGCstrand[Red](4,-2)(3,0)
\DGCstrand[Red](2,-2)(1,0)
\end{DGCpicture} \ .
\]
\end{rem}

\begin{lem}\label{lem-notriple}
In any diagram for $D_\lambda^{\mf{t}}$, where $\mf{t}\in \T_\lambda$, there are no two black strands and a red strand making either of the following configurations
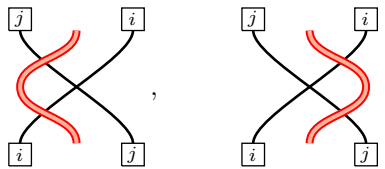
\[
\begin{DGCpicture}[scale=0.75]
\DGCstrand(1,0)(3,2)
\DGCstrand(3,0)(1,2)
\DGCstrand[Red](2,0)(1,1)(2,2)
\DGCcoupon(0.8,-0.4)(1.2,0){$_i$}
\DGCcoupon(2.8,-0.4)(3.2,0){$_j$}
\DGCcoupon(0.8,2)(1.2,2.4){$_j$}
\DGCcoupon(2.8,2)(3.2,2.4){$_i$}
\end{DGCpicture} \ , 
\quad \quad \quad
\begin{DGCpicture}[scale=0.75]
\DGCstrand(1,0)(3,2)
\DGCstrand(3,0)(1,2)
\DGCstrand[Red](2,0)(3,1)(2,2)
\DGCcoupon(0.8,-0.4)(1.2,0){$_i$}
\DGCcoupon(2.8,-0.4)(3.2,0){$_j$}
\DGCcoupon(0.8,2)(1.2,2.4){$_j$}
\DGCcoupon(2.8,2)(3.2,2.4){$_i$}
\end{DGCpicture} \ .
\]
\end{lem}
\begin{proof}
By the definition of permissible permutations (Definition \ref{defallowablepermutation}), the position of upper $i$ and $j$ letters must occur to the right of the corresponding lower ones.
Now, by the above Remark \ref{rmk-cellcombinatorics-Webster}, if such configurations of two black strands and a red strand were to occur, the red strand would be connecting top and bottom red dots with different position numbers counted from the left, in violation of the algorithm for building the diagram of $D_\lambda^{\mf{t}}$.
\end{proof}

Under the isomorphism of Theorem \ref{thm-iso-to-Webster-block}, we have the following relationship between the elements $D^{\lambda}_{\mf{t}\mf{s}}$ of $W_n^l$ and the cellular basis elements and $\Phi_{\mf{t}\mf{s}}^{\lambda}$ of $S_n^l$.

\begin{thm}
\label{thmcellbascompat}
For any $ \mf{t}, \mf{s}\in \T_\lambda$, the cellular basis element $\Phi^{\lambda}_{\mf{t}\mf{s}}$ is equal, diagrammatically, to the element $D^{\lambda}_{\mf{t}\mf{s}}$:
\[
\Phi^{\lambda}_{\mf{t}\mf{s}}=
\begin{DGCpicture}[scale=0.5]
\DGCstrand[Red](0,0)(0,4)
\DGCstrand(2,0)(1,2)(4,4)
\DGCstrand(4,0)(3,2)(2,4)
\DGCstrand[Red](1,0)(2,2)(1,4)
\DGCstrand[Red](3,0)(4,2)(3,4)
\DGCcoupon(-0.3,1.7)(4.3,2.3){$_{\e(\lambda)}$}
\DGCcoupon(-0.3,2.7)(4.3,3.3){$_{\e(\mu)\psi_{\mf{t}}}$}
\DGCcoupon(-0.3,0.7)(4.3,1.3){$_{\psi_{\mf{s}}^*\e(\nu)}$}
\end{DGCpicture} \ .
\]
\end{thm}
\begin{proof}
Recall that the $p$-DG isomorphism of Proposition \ref{propGlambdaisotoWebblock} 
\[
{\theta}_{\kappa_{\nu}}:
\e({\nu}) W_n^l\e(\kappa_0) =
\left\{
\begin{DGCpicture}
\DGCstrand[Red](0,-2)(0,0)[]
\DGCstrand(1,-2)(0.75,0)
\DGCstrand[Red](0.5,-2)(1.5,0)[]
\DGCstrand(2.25,-2)(2,0)
\DGCcoupon*(1.35,-1.3)(1.75,-1.4){$\cdots$}
\DGCcoupon(0.8,-1.8)(2.4,-1.5){$x$}
\DGCcoupon*(.2,-1.3)(.6,-1.4){$\cdots$}
\DGCcoupon(-0.2,-0.4)(2.45,-0.1){$_{\e({\nu})}$}
\end{DGCpicture}
~\Bigg| x\in \nh_n^l
\right\} \stackrel{\cong}{\lra} G(\nu)
\]
is given by the injective ``sweeping'' map \eqref{eqnsweepingmap}. On the other hand, the element $\Phi^{\lambda}_{\mf{t}\mf{s}}$ is a map from $G(\nu)$ to $G(\mu)$ (equation \eqref{eqndefPsielements}). To prove the theorem, it therefore suffices to verify the commutativity of the following diagram of right $\nh_n^l$-modules:
\begin{equation}\label{cellbasissame}
\begin{gathered}
\xymatrix{
\e({\nu}) W_n^l\e(\kappa_0)  \ar[rr]^{\hspace{.4in} {\theta}_{\kappa_{\nu}}}_{\hspace{.4in} \cong} \ar[d]_{D^{\lambda}_{\mf{t}\mf{s}}} & & G(\nu) \ar[d]^{\Phi^{\lambda}_{\mf{t}\mf{s}}}
\\ \e({\mu}) W_n^l\e(\kappa_0)  \ar[rr]^{\hspace{.4in} {\theta}_{\kappa_{\mu}}}_{\hspace{.4in} \cong} & & G(\mu)
}
\end{gathered} \ .
\end{equation}

Now, we compute that the composition $\Phi^{\lambda}_{\mf{t}\mf{s}} \circ {\theta}_{\kappa_{\nu}} $ has the effect 
\begin{equation}\label{WebGgen1}
\begin{DGCpicture}
\DGCstrand[Red](0,-2)(0,0)[]
\DGCstrand(1,-2)(0.75,0)
\DGCstrand[Red](0.5,-2)(1.5,0)[]
\DGCstrand(2.25,-2)(2,0)
\DGCcoupon*(1.35,-1.3)(1.75,-1.4){$\cdots$}
\DGCcoupon(0.8,-1.8)(2.4,-1.5){$x$}
\DGCcoupon*(.2,-1.3)(.6,-1.4){$\cdots$}
\DGCcoupon(-0.2,-0.4)(2.45,-0.1){$_{\e({\nu})}$}
\end{DGCpicture}
\mapsto
\begin{DGCpicture}
\DGCstrand[Red](0,-2)(0,0)[]
%\DGCstrand(0.5,-.7)(1.25,0)[]
\DGCstrand/u/(1,-2)(1,-1.75)(0.5,-0.8)/u/(1.25,0)/u/
\DGCstrand[Red]/u/(0.5,-2)(0.5,-1.75)/u/(1.25,-0.8)/u/(0.5,0)/u/[]
\DGCstrand(2,-0.7)(2,0)[]
\DGCstrand(2.25,-2)(2,-.7)
\DGCcoupon*(1.35,-1.3)(1.75,-1.4){$\cdots$}
\DGCcoupon(0.7,-1.8)(2.3,-1.5){$x$}
\DGCcoupon*(.2,-1.3)(.6,-1.4){$\cdots$}
\DGCcoupon(-0.2,-1)(2.3,-0.7){$_{\e({\nu})}$}
\end{DGCpicture}
~=~
\begin{DGCpicture}
\DGCstrand[Red](-.5,-2)(-.5,0)[]
\DGCstrand(1,-2)(1,0)
\DGCstrand(2.25,-2)(2.25,0)
\DGCcoupon(0.75,-0.9)(2.4,-0.6){$_{y^{\nu} }$}
\DGCcoupon(0.75,-1.3)(2.4,-1){$_{x}$}
\DGCstrand[Red](0.5,-2)(0.5,0)[]
\DGCcoupon*(1.45,-2)(1.85,-1.3){$\cdots$}
%\DGCcoupon(0.8,-1.8)(2.4,-1.5){$x$}
\DGCcoupon*(-.3,-2)(.3,-1.4){$\cdots$}
\end{DGCpicture}
~\mapsto~
\begin{DGCpicture}
\DGCstrand[Red](-.5,-2)(-.5,0)[]
\DGCstrand(1,-2)(1,0)
\DGCstrand(2.25,-2)(2.25,0)
\DGCcoupon(0.85,-0.9)(2.4,-0.6){$_{y^{\lambda} }$}
\DGCcoupon(0.85,-1.3)(2.4,-1){$_{\psi_\mf{s}^*}$}
\DGCstrand[Red](0.5,-2)(0.5,0)[]
\DGCcoupon*(1.45,-2)(1.85,-1.7){$\cdots$}
\DGCcoupon(0.85,-1.7)(2.4,-1.4){$_{x}$}
\DGCcoupon(0.85,-0.2)(2.4,-0.5){$_{\psi_\mf{t}}$}
\DGCcoupon*(-.3,-1.3)(.3,-1.4){$\cdots$}
\end{DGCpicture} \ .
\end{equation}

Next, by construction, the element  $D^{\lambda}_{\mf{t}\mf{s}}$ first maps
\begin{equation*}
\label{WebGgen2}
\begin{DGCpicture}
\DGCstrand[Red](0,-2)(0,0)[]
\DGCstrand(1,-2)(0.75,0)
\DGCstrand[Red](0.5,-2)(1.5,0)[]
\DGCstrand(2.25,-2)(2,0)
\DGCcoupon*(1.35,-1.3)(1.75,-1.4){$\cdots$}
\DGCcoupon(0.8,-1.8)(2.4,-1.5){$x$}
\DGCcoupon*(.2,-1.3)(.6,-1.4){$\cdots$}
\DGCcoupon(-0.2,-0.4)(2.45,-0.1){$_{\e({\nu})}$}
\end{DGCpicture}
~\mapsto~
\begin{DGCpicture}
\DGCstrand[Red](0,-2)(0,0)[]
\DGCstrand/u/(1,-2)(1,-1.75)(0.5,-0.9)/u/(0.5,0)/u/
\DGCstrand[Red]/u/(0.5,-2)(0.5,-1.75)/u/(1,-0.8)/u/(1,0)/u/[]
\DGCstrand(2.25,-2)(2,-1.2)/u/(2,-0.9)(2,0)/u/[]
\DGCcoupon*(1.35,-1.3)(1.75,-1.4){$\cdots$}
\DGCcoupon(0.7,-1.8)(2.3,-1.5){$_{x}$}
\DGCcoupon*(.2,-1.3)(.6,-1.4){$\cdots$}
\DGCcoupon(-0.2,-1.2)(2.3,-0.9){$_{\psi_{\mf{s}}^*\e({\nu})}$}
\DGCcoupon(-0.2,-0.4)(2.3,-0.1){$_{\e({\mu})\psi_\mf{t}}$}
\DGCcoupon(-0.2,-0.8)(2.3,-0.5){$_{\e({\lambda})}$}
\end{DGCpicture} \ ,
\end{equation*}
Then, we apply the sweeping map ${\theta}_{\kappa_\mu}$. In the obtained diagram, any of the black-black crossing in $\psi_\mf{t}$ can be pushed all the way to the upper right corner, while similarly the black-black crossings in $\psi_{\mf{s}}^*$ can be moved to the lower right corner. This is executed via repeated applications of relation \eqref{eqn-Webster-triple-crossing1} since, by Lemma \ref{lem-notriple}, there are no two black strands a red strand making configurations occurring in equation \eqref{eqn-Webster-triple-crossing2}. Diagrammatically, we have
\[
\begin{DGCpicture}
\DGCstrand[Red](0,-2)(0,0)[]
\DGCstrand/u/(1,-2)(1,-1.75)(0.5,-0.9)/u/(0.5,0)/u/
\DGCstrand[Red]/u/(0.5,-2)(0.5,-1.75)/u/(1,-0.8)/u/(1,0)/u/[]
\DGCstrand(2.25,-2)(2,-1.2)/u/(2,-0.9)(2,0)/u/[]
\DGCcoupon*(1.35,-1.3)(1.75,-1.4){$\cdots$}
\DGCcoupon(0.7,-1.8)(2.3,-1.5){$_{x}$}
\DGCcoupon*(.2,-1.3)(.6,-1.4){$\cdots$}
\DGCcoupon(-0.2,-1.2)(2.3,-0.9){$_{\psi_{\mf{s}}^*\e({\nu})}$}
\DGCcoupon(-0.2,-0.4)(2.3,-0.1){$_{\e({\mu})\psi_\mf{t}}$}
\DGCcoupon(-0.2,-0.8)(2.3,-0.5){$_{\e({\lambda})}$}
\end{DGCpicture}
~\mapsto~
\begin{DGCpicture}
\DGCstrand[Red](0,-2)(0,0)(0,0.5)[]
\DGCstrand/u/(1,-2)(1,-1.75)(0.5,-0.9)/u/(0.5,0)/u/(1.5,0.5)
\DGCstrand(2.25,-2)(2,-1.2)/u/(2,-0.9)(2,0)/u/(2,0.5)[]
\DGCstrand[Red]/u/(0.5,-2)(0.5,-1.75)/u/(1,-0.8)/u/(1,0)/u/(0.5,0.5)[]
\DGCcoupon*(1.35,-1.3)(1.75,-1.4){$\cdots$}
\DGCcoupon(0.7,-1.8)(2.3,-1.5){$_{x}$}
\DGCcoupon*(.2,-1.3)(.6,-1.4){$\cdots$}
\DGCcoupon(-0.2,-1.2)(2.3,-0.9){$_{\psi_{\mf{s}}^*\e({\nu})}$}
\DGCcoupon(-0.2,-0.4)(2.3,-0.1){$_{\e({\mu})\psi_\mf{t}}$}
\DGCcoupon(-0.2,-0.8)(2.3,-0.5){$_{\e({\lambda})}$}
\end{DGCpicture}
~=~
\begin{DGCpicture}
\DGCstrand[Red](0,-2)(0,0)(0,0.5)[]
\DGCstrand/u/(1,-2)(1,-1.75)(0.5,-0.9)/u/(1.5,0)/u/(1.5,0.5)
\DGCstrand(2.25,-2)(2,-1.2)/u/(2,-0.9)(2,0.5)[]
\DGCstrand[Red]/u/(0.5,-2)(0.5,-1.7)/u/(1.25,-0.8)(0.5,0.5)/u/[]
\DGCcoupon*(1.35,-1.3)(1.75,-1.4){$\cdots$}
\DGCcoupon(0.7,-1.8)(2.3,-1.5){$_{\psi_{\mf{s}}^* x}$}
\DGCcoupon*(.2,-1.3)(.6,-1.4){$\cdots$}
\DGCcoupon(-0.2,-1)(2.3,-0.7){$_{\e({\lambda})}$}
%\DGCcoupon(-0.2,-0.55)(2.3,-0.25){$_{\e({\nu})}$}
\DGCcoupon(1.3,0.1)(2.2,0.4){$_{\psi_{\mf{t}}}$}
\end{DGCpicture}
~=~
\begin{DGCpicture}
\DGCstrand[Red](-.5,-2)(-.5,0)[]
\DGCstrand(1,-2)(1,0)
\DGCstrand(2.25,-2)(2.25,0)
\DGCcoupon(0.85,-1.1)(2.4,-0.8){$_{y^{\lambda} }$}
\DGCcoupon(0.85,-1.5)(2.4,-1.2){$_{\psi_\mf{s}^*x}$}
\DGCstrand[Red](0.5,-2)(0.5,0)[]
\DGCcoupon*(1.45,-2)(1.85,-1.7){$\cdots$}
\DGCcoupon(0.85,-0.4)(2.4,-0.7){$_{\psi_\mf{t}}$}
\DGCcoupon*(-.3,-1.3)(.3,-1.4){$\cdots$}
\end{DGCpicture}
\ .
\]
Here we have used that, since $\lambda\geq \mu$, sweeping from $\e(\lambda)$ to $\e(\mu)$ followed by sweeping from $\e(\mu)$ to $\e(\kappa_0)$ is equal to directly sweeping from $\e(\lambda)$ to $\e(\kappa_0)$. Comparing the expressions gives the desired result.
\end{proof}

\begin{cor}\label{corcellbasesagree}
Under the isomorphism
$S_n^l\cong \END_{W_n^l}(Q)$,
the cellular basis $\sqcup_{\lambda\in \mc{P}_n^l}\{ D^\lambda_{\mf{s}\mf{t}} \}$ for $S_n^l$ is identified with the cellular basis for $W_n^l$ where tableaux are only allowed to be filled by the letters $\{1_i|i=1,\dots,l\}$.
\end{cor}

\begin{proof}
Define $\T^{\prime\prime}_\lambda\subset \T^\prime_\lambda$ to be the subset consisting of multi-standard tableaux (Definition \ref{defWebstercomb}) where we only allow the entries in the set $\{1_i|i=1,\dots, l\}$.
We will show that there is a bijection between $\T_\lambda$ and $\T^{\prime\prime}_\lambda$ by identifying both sets with the auxiliary  permissible permutations (Definition \ref{defallowablepermutation}):
\[
\mf{S}_\lambda:=\sqcup_{\mu \in \mc{P}_n^l} \mf{S}_\lambda^{\mu}.
\]

Let $\mf{t}\in \T^{\prime \prime}_\lambda$ be a multi-standard tableaux of shape $\lambda$. We associate to $\mf{t}$ a permissible permutation as follows. As in the proof of Lemma \ref{bijlem}, we depict the total sequence of boxes and empty slots in $\lambda$ on a lower horizontal line. On an upper horizontal line, we draw $l$ markings directly over the $l$ symbols below.  Suppose $\lambda$ has a box in positions $j_1<j_2<\dots < j_n$.  If $1_i$ is filled into box $j_k$, then necessarily $i\geq j_k$ since the $i$th alphabet can only occur in the first $i$ parts of $\lambda$. We the define the symmetric group element for $\mf{t}$ by requiring that its trajectory connects the bottom $j_k$-box with the top $i$th marking. Put a box at each $i$ on the top line where it is connected to a bottom box. It is clear that we obtain a partition $\mu$ with a label $k$ in it. In this way we obtain an element of $\mf{S}_\lambda^\mu$ (Definition \ref{defallowablepermutation}) and a tableau in $\Tab_\lambda(\mu)$. Below is an example explaining the process for the tableau $\mf{r}_6$ of Example \ref{egWebstercomb}:
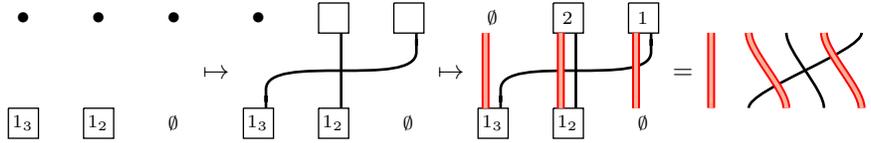
\[
\begin{DGCpicture}
\DGCcoupon(0,0)(0.4,-0.4){$_{1_3}$}
\DGCcoupon(1,0)(1.4,-0.4){$_{1_2}$}
\DGCcoupon*(2,0)(2.4,-0.4){$_\emptyset$}
\DGCcoupon*(0,1.4)(0.4,1){$\bullet$}
\DGCcoupon*(1,1.4)(1.4,1){$\bullet$}
\DGCcoupon*(2,1.4)(2.4,1){$\bullet$}
\end{DGCpicture}~\mapsto~
\begin{DGCpicture}
\DGCcoupon(0,0)(0.4,-0.4){$_{1_3}$}
\DGCcoupon(1,0)(1.4,-0.4){$_{1_2}$}
\DGCcoupon*(2,0)(2.4,-0.4){$_\emptyset$}
\DGCcoupon*(0,1.4)(0.4,1){$\bullet$}
\DGCcoupon(1,1.4)(1.4,1){}
\DGCcoupon(2,1.4)(2.4,1){}
\DGCstrand(1.3,0)(1.3,1)
\DGCstrand(0.3,0)(0.3,0.25)(2.3,0.75)(2.3,1)
\end{DGCpicture}
~\mapsto~
\begin{DGCpicture}
\DGCcoupon(0,0)(0.4,-0.4){$_{1_3}$}
\DGCcoupon(1,0)(1.4,-0.4){$_{1_2}$}
\DGCcoupon*(2,0)(2.4,-0.4){$_\emptyset$}
\DGCcoupon*(0,1.4)(0.4,1){$_\emptyset$}
\DGCcoupon(1,1.4)(1.4,1){$_2$}
\DGCcoupon(2,1.4)(2.4,1){$_1$}
\DGCstrand(1.3,0)(1.3,1)
\DGCstrand(0.3,0)(0.3,0.25)(2.3,0.75)(2.3,1)
\DGCstrand[Red](1.1,0)(1.1,1)
\DGCstrand[Red](0.1,0)(0.1,1)
\DGCstrand[Red](2.1,0)(2.1,1)
\end{DGCpicture}
~=~
\begin{DGCpicture}[scale=0.5]
\DGCstrand[Red](0,-2)(0,0)
\DGCstrand(3,-2)(2,0)
\DGCstrand(1,-2)(4,0)
\DGCstrand[Red](4,-2)(3,0)
\DGCstrand[Red](2,-2)(1,0)
\end{DGCpicture} \ .
\]
The obtained tableau on the top line of the third picture corresponds to ${\mf{u}}\in \T_\lambda$ of Example \ref{eg-23-cellularbasis}.

Conversely, any permissible permutation in $\mf{S}_\lambda$ can be translated back into an element of $\T^{\prime \prime}_\lambda$. The result follows.
\end{proof}

\begin{rem}
In the absence of the differential, the above corollary generalizes without much difficulty to the case of type $A$ KLR algebras.
\end{rem}

\subsection{\texorpdfstring{$p$}{p}-DG structure on standard modules}
In this subsection, we give a diagrammatic algorithm for computing the $p$-DG structure on cell modules over $S_n^l$ and $W_n^l$. We will use Theorem \ref{thm-iso-to-Webster-block} to identify $S_n^l$ with the corresponding block of $W_n^l$ throughout.

For each $\lambda\in \mc{P}_n^l$, by Proposition \ref{propdifkilledidempotentideals}, the $p$-DG cell module $\Delta(\lambda)$ and its opposite $\Delta^\circ(\lambda)$ over $S_n^l$ are isomorphic to
\begin{equation}
\Delta(\lambda):=
\Bbbk\left\langle
D_\lambda^{\mf{t}}|\mf{t}\in \T_\lambda
\right\rangle \cong \dfrac{S_n^l \e(\lambda) +(S_n^l)^{>\lambda}}{(S_n^l)^{>\lambda}} , \quad \quad
\Delta^\circ(\lambda):=
\Bbbk\left\langle
D^\lambda_{\mf{t}}|\mf{t}\in \T_\lambda
\right\rangle
\cong \dfrac{\e(\lambda)S_n^l  +(S_n^l)^{>\lambda}}{(S_n^l)^{>\lambda}} \ ,
\end{equation}
where $\e(\lambda)=D^{\lambda}_{\mf{t}^\lambda \mf{t}^{\lambda}}$ is an idempotent annihilated by $\dif$.
Likewise, the $p$-DG cellular modules over $W_n^l$, denoted $V(\lambda)$ and $V^\circ(\lambda)$ respectively, are isomorphic to
\begin{equation}
V(\lambda):=
\Bbbk\left\langle
B_\lambda^{\mf{t}}|\mf{t}\in \T^\prime_\lambda
\right\rangle
\cong \dfrac{W_n^l \e(\lambda) +(W_n^l)^{>\lambda}}{(W_n^l)^{>\lambda}}
, \quad \quad
V^\circ(\lambda):=
\Bbbk\left\langle
B^\lambda_{\mf{t}}|\mf{t}\in \T^\prime_\lambda
\right\rangle
\cong \dfrac{ \e(\lambda) W_n^l +(W_n^l)^{>\lambda}}{(W_n^l)^{>\lambda}}.
\end{equation}

We would like to make the $p$-DG module structure of $\Delta(\lambda)$ more explicit. To do this we will need the following.

\begin{lem} \label{lemdotreduction}
For any $\lambda\in \mc{P}_n^l$, if an element in $(S_n^l)^{\geq \lambda}$ (resp.~$(W_n^l)^{\geq \lambda}$) contains as a portion of its diagram
\begin{equation} \label{dotcellpic}
\begin{DGCpicture}[scale=0.65]
\DGCstrand[Red](1,-1.5)(1,1)
\DGCstrand(2,-1.5)(2,1)
\DGCdot{0.35}
\DGCcoupon(-0.2,-0.5)(3.2,0){$_{\e(\lambda)}$}
\DGCcoupon*(0,0)(0.8,1){$\cdots$}
\DGCcoupon*(2.2,0)(3,1){$\cdots$}
\DGCcoupon*(0,-0.5)(0.8,-1.5){$\cdots$}
\DGCcoupon*(2.2,-0.5)(3,-1.5){$\cdots$}
\end{DGCpicture} \ ,
\end{equation}
then the element is in $(S_n^l)^{>\lambda}$ (resp. $(W_n^l)^{>\lambda}$).
\end{lem}
\begin{proof}
We prove the lemma by induction on $k$, which counts the red strands starting from the left. When the red strand occurs in \eqref{dotcellpic} as the first one on the left, and the black strand immediately to its right carries a dot, the diagram is annihilated by the double crossing relation \eqref{eqn-Webster-double-crossing} and the cyclotomic relation \eqref{eqn-Web-cyclotomic}.

Inductively, suppose the result holds when the red strand is the $k-1$st one counted from the left in $\e(\lambda)$. For the $k$th red strand and a black strand right next to it carrying a dot,  there are two cases we need to consider: either the $k-1$st position of $\lambda$ is empty or is a box.

In the first case, we have, as part of the given diagram, the left hand term in the equation
\[
\begin{DGCpicture}[scale=0.65]
\DGCstrand[Red](0,0)(0,2)[{}`$_{k-1}$]
\DGCstrand[Red](1,0)(1,2)[{}`$_{k}$]
\DGCstrand(2,0)(2,2)
\DGCdot{1}
\DGCcoupon(-0.2,-0.5)(2.2,0){$_{\e(\lambda)}$}
\end{DGCpicture}
~=~
\begin{DGCpicture}[scale=0.65]
\DGCstrand[Red](0,0)(0,2)[{}`$_{k-1}$]
\DGCstrand(2,0)(1,1)(2,2)
\DGCstrand[Red](1,0)(2,1)(1,2)[{}`$_{k}$]
\DGCcoupon(-0.2,-0.5)(2.2,0){$_{\e(\lambda)}$}
\end{DGCpicture} \ .
\]
The middle part of the diagram on the right corresponds to the partition obtained from $\lambda$ by moving the box to the empty slot immediately to its left. 
Thus it factors through a partition larger than $\lambda$ in the dominance order.

For case two, we have the following as a part of the given diagram
\begin{align}\label{eqntwostrandreduction}
\begin{DGCpicture}[scale=0.5]
\DGCstrand[Red](0,0)(0,2)[{}`$_{k-1}$]
\DGCstrand(1,0)(1,2)
\DGCstrand[Red](2,0)(2,2)[{}`$_{k}$]
\DGCstrand(3,0)(3,2)
\DGCdot{1}
\DGCcoupon(-0.2,-0.5)(3.2,0){$_{\e(\lambda)}$}
\end{DGCpicture}
 ~ =~
\begin{DGCpicture}[scale=0.5]
\DGCstrand[Red](0,0)(0,2)[{}`$_{k-1}$]
\DGCstrand(1,0)(1,2)
\DGCstrand(3,0)(2,1)(3,2)
\DGCstrand[Red](2,0)(3,1)(2,2)[{}`$_{k}$]
\DGCcoupon(-0.2,-0.5)(3.2,0){$_{\e(\lambda)}$}
\end{DGCpicture}
~=~
\begin{DGCpicture}[scale=0.5]
\DGCstrand[Red](0,0)(0,2)[{}`$_{k-1}$]
\DGCstrand(1,0)(3,2)
\DGCdot{0.75}
\DGCstrand(3,0)(1,2)
\DGCstrand[Red](2,0)(3,1)(2,2)[{}`$_{k}$]
\DGCcoupon(-0.2,-0.5)(3.2,0){$_{\e(\lambda)}$}
\end{DGCpicture}
-
\begin{DGCpicture}[scale=0.5]
\DGCstrand[Red](0,0)(0,2)[{}`$_{k-1}$]
\DGCstrand(1,0)(3,2)
\DGCdot{1.2}
\DGCstrand(3,0)(1,2)
\DGCstrand[Red](2,0)(3,1)(2,2)[{}`$_{k}$]
\DGCcoupon(-0.2,-0.5)(3.2,0){$_{\e(\lambda)}$}
\end{DGCpicture}
~=~
\begin{DGCpicture}[scale=0.5]
\DGCstrand[Red](0,0)(0,2)[`$_{k-1}$]
\DGCstrand(1,0)(3,2)
\DGCdot{0.5}
\DGCstrand(3,0)(1,2)
\DGCstrand[Red](2,0)(3,1)(2,2)[`$_{k}$]
\DGCcoupon(-0.2,-0.5)(3.2,0){$_{\e(\lambda)}$}
\end{DGCpicture}
-
\begin{DGCpicture}[scale=0.5]
\DGCstrand(1,0)(2,1)(1,2)
\DGCstrand(3,0)(0,1)(3,2)
\DGCdot{1.8}
\DGCstrand[Red](0,0)(1,1)(0,2)[`$_{k-1}$]
\DGCstrand[Red](2,0)(3,1)(2,2)[`$_{k}$]
\DGCcoupon(-0.2,-0.5)(3.2,0){$_{\e(\lambda)}$}
\end{DGCpicture} \ .
\end{align}
By the induction hypothesis, the first diagram in the last term is in $(S_n^l)^{> \lambda}$ already, and we will show that so is the second diagram in that term. If there is a red strand immediately to the left of the last term, we are then done as the middle portion of the diagram corresponds to a partition greater than $\lambda$:
\begin{equation}\label{term1}
\begin{DGCpicture}[scale=0.65]
\DGCstrand(1,0)(2,1)(1,2)
\DGCstrand(3,0)(0,1)(3,2)
\DGCdot{1.8}
\DGCstrand[Red](-1,0)(-1,2)[`$_{k-2}$]
\DGCstrand[Red](0,0)(1,1)(0,2)[`$_{k-1}$]
\DGCstrand[Red](2,0)(3,1)(2,2)[`$_{k}$]
\DGCcoupon(-1.2,-0.5)(3.2,0){$_{\e(\lambda)}$}
\end{DGCpicture}
\end{equation}
Otherwise, there is a black strand immediately to the left of the $k-1$st red strand, and we have two black strands sandwiched between the $k-2$nd and $k-1$st red strands in the middle above $\e(\lambda)$:
\begin{equation}\label{term2}
\begin{DGCpicture}[scale=0.65]
\DGCstrand(1,0)(2,1)(1,2)
\DGCstrand(3,0)(0,1)(3,2)
\DGCdot{1.8}
\DGCstrand(-1,0)(-1,2)
\DGCstrand[Red](-2,0)(-2,2)[`$_{k-2}$]
\DGCstrand[Red](0,0)(1,1)(0,2)[`$_{k-1}$]
\DGCstrand[Red](2,0)(3,1)(2,2)[`$_{k}$]
\DGCcoupon(-2.2,-0.5)(3.2,0){$_{\e(\lambda)}$}
\end{DGCpicture}
\end{equation} 
Now we may repeat the step used in the second equality in equation \eqref{eqntwostrandreduction}. In finitely many steps, we either push the right most black strand in between two red strands, thus obtaining a higher partition-idempotent in the middle; or it may be pushed all the way to the far left of the entire diagram, and thus killing the corresponding term. The lemma now follows.
\end{proof}

The lemma gives us a constructive way to compute the $p$-differential on the modules $\Delta(\lambda)$, $V (\lambda)$ and their opposite modules for all $\lambda\in\mc{P}_n^l$. Namely, for  $\Delta(\lambda)$, after applying $\dif$ to an element $D_\lambda^{\mf{t}}$, we obtain a linear combination of terms that are either $D_{\lambda}^{\mf{s}}$ for some $\mf{s}\in \T^\lambda$, or $D_{\lambda}^{\mf{t}}$ carrying dots. Now move the dots down to reach $\e(\lambda)$ using the Webster relations, we eventually only keep the obtained terms without dots. This can be similarly carried out for $V(\lambda)$ and the opposite modules $\Delta^\circ(\lambda)$, $V^\circ(\lambda)$.

\begin{example}\label{eg-Delta-Delta-op}
Let us consider the $p$-complex structure of  $\Delta(\lambda)$ for $\lambda=(~\yng(1)~, ~\yng(1)~, ~\emptyset~)$. By Example \ref{eg-23-cellularbasis} (1), the cellular basis for $\Delta(\lambda)$ is given by
\[
\left\{
D^{\mf{r}}_\lambda=
\begin{DGCpicture}[scale=0.5]
\DGCstrand[Red](0,-2)(0,0)
\DGCstrand(1,-2)(1,0)
\DGCstrand[Red](2,-2)(2,0)
\DGCstrand(3,-2)(3,0)
\DGCstrand[Red](4,-2)(4,0)
\end{DGCpicture} \ ,
~
D^{\mf{s}}_\lambda=
\begin{DGCpicture}[scale=0.5]
\DGCstrand[Red](0,-2)(0,0)
\DGCstrand(1,-2)(1,0)
\DGCstrand[Red](2,-2)(2,0)
\DGCstrand(3,-2)(4,0)
\DGCstrand[Red](4,-2)(3,0)
\end{DGCpicture} \ ,
~
D^{\mf{t}}_\lambda=
\begin{DGCpicture}[scale=0.5]
\DGCstrand[Red](0,-2)(0,0)
\DGCstrand(1,-2)(2,0)
\DGCstrand[Red](2,-2)(1,0)
\DGCstrand(3,-2)(4,0)
\DGCstrand[Red](4,-2)(3,0)
\end{DGCpicture} \ ,
~ 
D^{\mf{u}}_\lambda=
\begin{DGCpicture}[scale=0.5]
\DGCstrand[Red](0,-2)(0,0)
\DGCstrand(3,-2)(2,0)
\DGCstrand(1,-2)(4,0)
\DGCstrand[Red](4,-2)(3,0)
\DGCstrand[Red](2,-2)(1,0)
\end{DGCpicture}~
\right\}.
\]
It is readily seen, using Lemma \ref{lem-dif-on-Web}, that 
$$\dif(D_\lambda^{\mf{r}})=\dif(D_\lambda^{\mf{s}})=\dif(D_\lambda^{\mf{t}})=0,$$
and
\[
\dif(D^{\mf{u}}_\lambda)= 
-
\begin{DGCpicture}[scale=0.5]
\DGCstrand[Red](0,-2)(0,0)
\DGCstrand(1,-2)(2,0)
\DGCstrand[Red](2,-2)(1,0)
\DGCstrand(3,-2)(4,0)
\DGCstrand[Red](4,-2)(3,0)
\end{DGCpicture}
-2~
\begin{DGCpicture}[scale=0.5]
\DGCstrand[Red](0,-2)(0,0)
\DGCstrand(3,-2)(2,0)
\DGCdot{-1.5}
\DGCstrand(1,-2)(4,0)
\DGCstrand[Red](4,-2)(3,0)
\DGCstrand[Red](2,-2)(1,0)
\end{DGCpicture}
\equiv
-
\begin{DGCpicture}[scale=0.5]
\DGCstrand[Red](0,-2)(0,0)
\DGCstrand(1,-2)(2,0)
\DGCstrand[Red](2,-2)(1,0)
\DGCstrand(3,-2)(4,0)
\DGCstrand[Red](4,-2)(3,0)
\end{DGCpicture}
=-D_\lambda^{\mf{t}} \ .
\]
On the other hand, the opposite module $\Delta^\circ(\lambda)$ has as a basis
\[
\left\{
D_{\mf{r}}^\lambda=
\begin{DGCpicture}[scale=0.5]
\DGCstrand[Red](0,0)(0,2)
\DGCstrand(1,0)(1,2)
\DGCstrand[Red](2,0)(2,2)
\DGCstrand(3,0)(3,2)
\DGCstrand[Red](4,0)(4,2)
\end{DGCpicture} \ ,
~
D_{\mf{s}}^\lambda=
\begin{DGCpicture}[scale=0.5]
\DGCstrand[Red](0,0)(0,2)
\DGCstrand(1,0)(1,2)
\DGCstrand[Red](2,0)(2,2)
\DGCstrand(4,0)(3,2)
\DGCstrand[Red](3,0)(4,2)
\end{DGCpicture} \ ,
~
D_{\mf{t}}^\lambda=
\begin{DGCpicture}[scale=0.5]
\DGCstrand[Red](0,0)(0,2)
\DGCstrand(2,0)(1,2)
\DGCstrand[Red](1,0)(2,2)
\DGCstrand(4,0)(3,2)
\DGCstrand[Red](3,0)(4,2)
\end{DGCpicture} \ ,
~ 
D_{\mf{u}}^\lambda=
\begin{DGCpicture}[scale=0.5]
\DGCstrand[Red](0,0)(0,2)
\DGCstrand(2,0)(3,2)
\DGCstrand(4,0)(1,2)
\DGCstrand[Red](3,0)(4,2)
\DGCstrand[Red](1,0)(2,2)
\end{DGCpicture}~
\right\}.
\]
It can be computed, using Lemma \ref{lem-dif-on-Web} and Lemma \ref{lemdotreduction}, that the differential $\dif$ vanishes on $\Delta^\circ(\lambda)$. This shows that the $p$-complex structure on $\Delta$ and $\Delta^\circ$ are, in general, different.
\end{example} 

In the Appendix \ref{apdx-pdgcellmod}, we also present the $p$-DG cellular structure of a larger block consisting of four red and two black strands.

Finally, using the diagrammatic description of cell modules, we give a $p$-DG interpretation of an important multiplicity space in the parabolic Kazhdan-Lusztig algorithms in \cite[Section 5]{HuMathas}.

\begin{prop}\label{propmultiplicitycomplex}
For two partitions $\mu,\lambda\in \mc{P}_n^l$, we have the following for the graded multiplicity space
\[
[Z(\lambda):\Delta(\mu)]_q= \mathrm{dim}_q(\Delta^\circ(\mu) \e(\lambda))=
\mathrm{dim}_q(\Bbbk\langle \Tab_\mu(\lambda) \rangle).
\]
In particular, it is non-zero if and only if 
%$\lambda\geq \mu$.
$ \lambda \leq \mu$.
\end{prop}
\begin{proof}
By Proposition \ref{pdgcellfiltration}, we know that $Z(\lambda)=S_n^l\e(\lambda)$ has a $p$-DG cell filtration with subquotients isomorphic to cell modules. 
Let us analyze which $\Delta(\mu)$'s figure inside $Z(\lambda)$ with their corresponding multiplicities.

Using the cellular $\Phi$-basis for $S_n^l$ (Theorem \ref{cellularschuralgprop} and Lemma \ref{lemPhiequalsPsi}), together with the diagrammatic translation (Theorem \ref{thmcellbascompat}),  $Z(\lambda)$ decomposes
\[
Z(\lambda)=(S_n^l)^{\geq \lambda} \cdot \e(\lambda) = \bigoplus_{\mu\geq \lambda} \Delta(\mu)\otimes \Delta^\circ(\mu) \e(\lambda).
\]
Therefore, $\Delta(\mu)$ occurs in $Z(\lambda)$ with graded multiplicity space $\Delta^\circ(\mu) \e(\lambda)$, which carries a natural $p$-complex structure. Now by Lemma \ref{bijlem} and Remark \ref{rmk-cellcombinatorics-Webster}, this $p$-complex has a graded basis parametrized by $\Tab_\mu(\lambda)$. Here any element in $\Tab_\mu(\lambda)$ is interpreted as a Webster diagram as in Remark \ref{rmk-cellcombinatorics-Webster} and turned upside-down. This explains the first part of the proposition.

%The last statement follows from the fact that, when $\lambda\geq \mu$ in the dominance order, $\mf{t}^{\lambda}$ can be obtained from $\mf{t}^{\mu}$ by a ${\lambda \choose \mu}$-permissible permutation. 
The last statement follows from the fact that $\Tab_\mu(\lambda)$ is non-empty if and only if $\lambda \leq \mu$.
\end{proof}

\begin{rem}
It remains an interesting question to investigate the $p$-complex structure on the multiplicity space $\Bbbk\langle\Tab_\mu(\lambda)\rangle$. This space plays an important role in the study of a certain block of singular category $\mathcal{O}$ for the universal enveloping algebra of $\mf{gl}_l$. We give an example of this $p$-complex in Appendix \ref{apdx-pdgcellmod}.
\end{rem}

\appendix
\section{Example: the algebra \texorpdfstring{$S_2^4$}{S24}} \label{sec-appendix}
\subsection{The cellular combinatorics}
There are six partitions in this case to consider.
\[
\lambda_1= \left(~\yng(1)~,~\yng(1)~,~\emptyset~,~\emptyset~\right),\quad \quad
\lambda_2= \left(~\yng(1)~,~\emptyset~,~\yng(1)~,~\emptyset~\right),\quad \quad
\lambda_3= \left(~\yng(1)~,~\emptyset~,~\emptyset~,~\yng(1)~\right),
\]
\[
\lambda_4=\left(~\emptyset~, ~\yng(1)~,~\yng(1)~,~\emptyset~\right),\quad \quad
\lambda_5=\left(~\emptyset~, ~\yng(1)~,~\emptyset~,~\yng(1)~\right), \quad \quad
\lambda_6=\left(~\emptyset~,~\emptyset~,~\yng(1)~,~\yng(1)~\right).
\]
The partial ordering (Definition \ref{def-order-on-nh-partition}) in this case reads:
\[
\lambda_1>\lambda_2>\{\lambda_3,\lambda_4\}>\lambda_5>\lambda_6,
\]
with $\lambda_3$ and $\lambda_4$ being incomparable.

We list the set of tableaux for each partition, abbreviating $\T_i:=\T_{\lambda_i}$ ($i=1,\dots,6$).

\begin{center}
\begin{tabular}{|c|c|c|}
\hline 
$\T_1$ & $\T_2$ & $\T_3$ \\ 
\hline 
$\mf{r}_1=\left(~\fbox{$1$}~,~\fbox{$2$}~,~\emptyset~,~\emptyset~\right)$ & $\mf{s}_1=\left(~\fbox{$1$}~,~\emptyset~,~\fbox{$2$}~,~\emptyset~\right)$ & $\mf{t}_1=\left(~\fbox{$1$}~,~\emptyset~,~\emptyset~,~\fbox{$2$}~\right)$  \\ 
$\mf{r}_2=\left(~\fbox{$1$}~,~\emptyset~,~\fbox{$2$}~,~\emptyset~\right)$ & $\mf{s}_2=\left(~\fbox{$1$}~,~\emptyset~,~\emptyset~,~\fbox{$2$}~\right)$ & $\mf{t}_2=\left(~\emptyset~,~\fbox{$1$}~,~\emptyset~,~\fbox{$2$}~\right)$  \\  
$\mf{r}_3=\left(~\fbox{$1$}~,~\emptyset~,~\emptyset~,~\fbox{$2$}~\right)$ & $\mf{s}_3=\left(~\emptyset~,~\fbox{$1$}~,~\fbox{$2$}~,~\emptyset~\right)$ & $\mf{t}_3=\left(~\emptyset~,~\emptyset~,~\fbox{$1$}~,~\fbox{$2$}~\right)$   \\ 
$\mf{r}_4=\left(~\emptyset~,~\fbox{$1$}~,~\fbox{$2$}~,~\emptyset~\right)$ & $\mf{s}_4=\left(~\emptyset~,~\fbox{$1$}~,~\emptyset~,~\fbox{$2$}~\right)$ &   \\ 
$\mf{r}_5=\left(~\emptyset~,~\fbox{$2$}~,~\fbox{$1$}~,~\emptyset~\right)$ & $\mf{s}_5=\left(~\emptyset~,~\emptyset~,~\fbox{$1$}~,~\fbox{$2$}~\right)$ &    \\  
$\mf{r}_6=\left(~\emptyset~,~\fbox{$1$}~,~\emptyset~,~\fbox{$2$}~\right)$ & $\mf{s}_6=\left(~\emptyset~,~\emptyset~,~\fbox{$2$}~,~\fbox{$1$}~\right)$ &   \\  
$\mf{r}_7=\left(~\emptyset~,~\fbox{$2$}~,~\emptyset~,~\fbox{$1$}~\right)$ &  &  \\ 
$\mf{r}_8=\left(~\emptyset~,~\emptyset~,~\fbox{$1$}~,~\fbox{$2$}~\right)$ &  &   \\ 
$\mf{r}_9=\left(~\emptyset~,~\emptyset~,~\fbox{$2$}~,~\fbox{$1$}~\right)$ &  &   \\
\hline   
\end{tabular}
\end{center}

\begin{center}
\begin{tabular}{|c|c|c|}
\hline 
$\T_4$ & $\T_5$ & $\T_6$ \\ 
\hline 
$\mf{u}_1=\left(~\emptyset~,~\fbox{$1$}~,~\fbox{$2$}~,~\emptyset~\right)$ & $\mf{v}_1=\left(~\emptyset~,~\fbox{$1$}~,~\emptyset~,~\fbox{$2$}~\right)$ & $\mf{w}_1=\left(~\emptyset~,~\emptyset~,~\fbox{$1$}~,~\fbox{$2$}~\right)$  \\ 
$\mf{u}_2=\left(~\emptyset~,~\fbox{$1$}~,~\emptyset~,~\fbox{$2$}~\right)$ & $\mf{v}_2=\left(~\emptyset~,~\emptyset~,~\fbox{$1$}~,~\fbox{$2$}~\right)$ &   \\  
$\mf{u}_3=\left(~\emptyset~,~\emptyset~,~\fbox{$1$}~,~\fbox{$2$}~\right)$ &  &   \\ 
$\mf{u}_4=\left(~\emptyset~,~\emptyset~,~\fbox{$2$}~,~\fbox{$1$}~\right)$ &  &   \\ 
\hline   
\end{tabular}
\end{center}

\subsection{Cyclic nilHecke modules}

The corresponding six cyclic modules and their decompositions into indecomposable modules are
\begin{align*}
G(\lambda_1)&=y_1^3 y_2^2 \nh_2^4 \cong Y(\lambda_1) \\
G(\lambda_2)&=y_1^3 y_2^1 \nh_2^4 \cong Y(\lambda_2) \\
G(\lambda_3)&=y_1^3 y_2^0 \nh_2^4 \cong Y(\lambda_3) \\
G(\lambda_4)&=y_1^2 y_2^1 \nh_2^4 \cong Y(\lambda_4) \oplus Y(\lambda_1)
\cong e_2 G(\lambda_4) \oplus Y(\lambda_1) \\
G(\lambda_5)&=y_1^2 y_2^0 \nh_2^4 \cong Y(\lambda_5)  \\
G(\lambda_6)&=y_1^1 y_2^0 \nh_2^4 \cong Y(\lambda_6) \oplus Y(\lambda_4)
\cong e_2G(\lambda_6) \oplus Y(\lambda_4).
\end{align*}

As $p$-DG modules, $G(\lambda_4)$ and $G(\lambda_6)$ are filtered as follows:
\begin{equation}\label{eqnfilteredG46}
Y(\lambda_4) \hookrightarrow G(\lambda_4) \twoheadrightarrow Y(\lambda_1)
\quad \quad 
Y(\lambda_6) \hookrightarrow G(\lambda_6) \twoheadrightarrow Y(\lambda_4).
\end{equation}

\subsection{Cellular \texorpdfstring{$p$}{p}-DG modules}\label{apdx-pdgcellmod}
We now present a basis of the cell modules in the terms of their pre-images in projective modules along with their $p$-DG structures and the corresponding diagrammatic descriptions coming from Theorem \ref{thmcellbascompat}.

$\Delta(\lambda_1)$:

\begin{align*}
\partial(\Phi_{\mf{r}_1})&=0
&
\dif\left(~
\begin{DGCpicture}[scale=0.5]
\DGCstrand[Red](0,0)(0,2)
\DGCstrand(1,0)(1,2)
\DGCstrand[Red](2,0)(2,2)
\DGCstrand(3,0)(3,2)
\DGCstrand[Red](4,0)(4,2)
\DGCstrand[Red](5,0)(5,2)
\end{DGCpicture}
~\right)&=0 \\
\partial(\Phi_{\mf{r}_2})&=0
&
\dif\left(~
\begin{DGCpicture}[scale=0.5]
\DGCstrand[Red](0,0)(0,2)
\DGCstrand(1,0)(1,2)
\DGCstrand[Red](2,0)(2,2)
\DGCstrand(3,0)(4,2)
\DGCstrand[Red](4,0)(3,2)
\DGCstrand[Red](5,0)(5,2)
\end{DGCpicture}
~\right)&=0 \\
\partial(\Phi_{\mf{r}_3})&=0
&
\dif\left(~
\begin{DGCpicture}[scale=0.5]
\DGCstrand[Red](0,0)(0,2)
\DGCstrand(1,0)(1,2)
\DGCstrand[Red](2,0)(2,2)
\DGCstrand(3,0)(5,2)
\DGCstrand[Red](4,0)(3,2)
\DGCstrand[Red](5,0)(4,2)
\end{DGCpicture}
~\right)&=0 
\\
\partial(\Phi_{\mf{r}_4})&=0
&
\dif\left(~
\begin{DGCpicture}[scale=0.5]
\DGCstrand[Red](0,0)(0,2)
\DGCstrand(1,0)(2,2)
\DGCstrand[Red](2,0)(1,2)
\DGCstrand(3,0)(4,2)
\DGCstrand[Red](4,0)(3,2)
\DGCstrand[Red](5,0)(5,2)
\end{DGCpicture}
~\right)&=0 \\
\partial(\Phi_{\mf{r}_5})&=-\Phi_{\mf{r}_4}
&
\dif\left(~
\begin{DGCpicture}[scale=0.5]
\DGCstrand[Red](0,0)(0,2)
\DGCstrand(1,0)(4,2)
\DGCstrand[Red](2,0)(1,2)
\DGCstrand(3,0)(2,2)
\DGCstrand[Red](4,0)(3,2)
\DGCstrand[Red](5,0)(5,2)
\end{DGCpicture}
~\right) &= -
\begin{DGCpicture}[scale=0.5]
\DGCstrand[Red](0,0)(0,2)
\DGCstrand(1,0)(2,2)
\DGCstrand[Red](2,0)(1,2)
\DGCstrand(3,0)(4,2)
\DGCstrand[Red](4,0)(3,2)
\DGCstrand[Red](5,0)(5,2)
\end{DGCpicture} \\
\partial(\Phi_{\mf{r}_6})&=0
&
\dif\left(~
\begin{DGCpicture}[scale=0.5]
\DGCstrand[Red](0,0)(0,2)
\DGCstrand(1,0)(2,2)
\DGCstrand[Red](2,0)(1,2)
\DGCstrand(3,0)(5,2)
\DGCstrand[Red](4,0)(3,2)
\DGCstrand[Red](5,0)(4,2)
\end{DGCpicture}
~\right)&=0 
\\
\partial(\Phi_{\mf{r}_7})&=-\Phi_{\mf{r}_6}
&
\dif\left(~
\begin{DGCpicture}[scale=0.5]
\DGCstrand[Red](0,0)(0,2)
\DGCstrand(1,0)(5,2)
\DGCstrand[Red](2,0)(1,2)
\DGCstrand(3,0)(2,2)
\DGCstrand[Red](4,0)(3,2)
\DGCstrand[Red](5,0)(4,2)
\end{DGCpicture}
~\right) &= -
\begin{DGCpicture}[scale=0.5]
\DGCstrand[Red](0,0)(0,2)
\DGCstrand(1,0)(2,2)
\DGCstrand[Red](2,0)(1,2)
\DGCstrand(3,0)(5,2)
\DGCstrand[Red](4,0)(3,2)
\DGCstrand[Red](5,0)(4,2)
\end{DGCpicture} \\
\partial(\Phi_{\mf{r}_8})&=0
&
\dif\left(~
\begin{DGCpicture}[scale=0.5]
\DGCstrand[Red](0,0)(0,2)
\DGCstrand(1,0)(3,2)
\DGCstrand[Red](2,0)(1,2)
\DGCstrand(3,0)(5,2)
\DGCstrand[Red](4,0)(2,2)
\DGCstrand[Red](5,0)(4,2)
\end{DGCpicture}
~\right)&=0 \\
\partial(\Phi_{\mf{r}_9})&=-\Phi_{\mf{r}_8}
&
\dif\left(~
\begin{DGCpicture}[scale=0.5]
\DGCstrand[Red](0,0)(0,2)
\DGCstrand(1,0)(5,2)
\DGCstrand[Red](2,0)(1,2)
\DGCstrand(3,0)(2.5,1)(3,2)
\DGCstrand[Red](4,0)(2,2)
\DGCstrand[Red](5,0)(4,2)
\end{DGCpicture}
~\right)&= -
\begin{DGCpicture}[scale=0.5]
\DGCstrand[Red](0,0)(0,2)
\DGCstrand(1,0)(3,2)
\DGCstrand[Red](2,0)(1,2)
\DGCstrand(3,0)(5,2)
\DGCstrand[Red](4,0)(2,2)
\DGCstrand[Red](5,0)(4,2)
\end{DGCpicture}
\end{align*}

We remark that the computations of $\dif(\Phi_{\mf{r}_i})$, $i=5,7,9$ follows from the differential action on the crossing (equivalent to the second equation in \eqref{eqndifactionnilHeckegen}):
\begin{equation}
\dif
\left(~
\begin{DGCpicture}
\DGCstrand(0,0)(1,1)
\DGCstrand(1,0)(0,1)
\end{DGCpicture}
~\right)
=-~
\begin{DGCpicture}
\DGCstrand(0,0)(0,1)
\DGCstrand(1,0)(1,1)
\end{DGCpicture}
-2
\begin{DGCpicture}
\DGCstrand(0,0)(1,1)
\DGCstrand(1,0)(0,1)
\DGCdot{0.25}
\end{DGCpicture}
\end{equation}
and applying Lemma \ref{lemdotreduction}. Alternatively, one sees that the terms involving a dot annihilates the resulting diagram by the Webster relations.

$\Delta(\lambda_2)$:

\begin{align*}
\partial(\Phi_{\mf{s}_1})&=0
& 
\dif\left(~
\begin{DGCpicture}[scale=0.5]
\DGCstrand[Red](0,0)(0,2)
\DGCstrand(1,0)(1,2)
\DGCstrand[Red](2,0)(2,2)
\DGCstrand[Red](3,0)(3,2)
\DGCstrand(4,0)(4,2)
\DGCstrand[Red](5,0)(5,2)
\end{DGCpicture}
~\right)&= 0
\\
\partial(\Phi_{\mf{s}_2})&=0
& 
\dif\left(~
\begin{DGCpicture}[scale=0.5]
\DGCstrand[Red](0,0)(0,2)
\DGCstrand(1,0)(1,2)
\DGCstrand[Red](2,0)(2,2)
\DGCstrand[Red](3,0)(3,2)
\DGCstrand(4,0)(5,2)
\DGCstrand[Red](5,0)(4,2)
\end{DGCpicture}
~\right)&= 0
\\
\partial(\Phi_{\mf{s}_3})&=0
& 
\dif\left(~
\begin{DGCpicture}[scale=0.5]
\DGCstrand[Red](0,0)(0,2)
\DGCstrand(1,0)(2,2)
\DGCstrand[Red](2,0)(1,2)
\DGCstrand[Red](3,0)(3,2)
\DGCstrand(4,0)(4,2)
\DGCstrand[Red](5,0)(5,2)
\end{DGCpicture}
~\right)&= 0
\\
\partial(\Phi_{\mf{s}_4})&=0
& 
\dif\left(~
\begin{DGCpicture}[scale=0.5]
\DGCstrand[Red](0,0)(0,2)
\DGCstrand(1,0)(2,2)
\DGCstrand[Red](2,0)(1,2)
\DGCstrand[Red](3,0)(3,2)
\DGCstrand(4,0)(5,2)
\DGCstrand[Red](5,0)(4,2)
\end{DGCpicture}
~\right)&= 0
\\
\partial(\Phi_{\mf{s}_5})&=0
&
\dif\left(~
\begin{DGCpicture}[scale=0.5]
\DGCstrand[Red](0,0)(0,2)
\DGCstrand(1,0)(3,2)
\DGCstrand[Red](2,0)(1,2)
\DGCstrand[Red](3,0)(2,2)
\DGCstrand(4,0)(5,2)
\DGCstrand[Red](5,0)(4,2)
\end{DGCpicture}
~\right)&= 0
\\
\partial(\Phi_{\mf{s}_6})&=-\Phi_{\mf{s}_5}
&
\dif\left(~
\begin{DGCpicture}[scale=0.5]
\DGCstrand[Red](0,0)(0,2)
\DGCstrand(1,0)(5,2)
\DGCstrand[Red](2,0)(1,2)
\DGCstrand[Red](3,0)(2,2)
\DGCstrand(4,0)(3,2)
\DGCstrand[Red](5,0)(4,2)
\end{DGCpicture}
~\right)&= -
\begin{DGCpicture}[scale=0.5]
\DGCstrand[Red](0,0)(0,2)
\DGCstrand(1,0)(3,2)
\DGCstrand[Red](2,0)(1,2)
\DGCstrand[Red](3,0)(2,2)
\DGCstrand(4,0)(5,2)
\DGCstrand[Red](5,0)(4,2)
\end{DGCpicture}
\end{align*}

$\Delta(\lambda_3)$:
\begin{align*}
\partial(\Phi_{\mf{t}_1})&=0
&
\dif\left(~
\begin{DGCpicture}[scale=0.5]
\DGCstrand[Red](0,0)(0,2)
\DGCstrand(1,0)(1,2)
\DGCstrand[Red](2,0)(2,2)
\DGCstrand[Red](3,0)(3,2)
\DGCstrand[Red](4,0)(4,2)
\DGCstrand(5,0)(5,2)
\end{DGCpicture}
~\right)&= 0
\\
\partial(\Phi_{\mf{t}_2})&=0
&
\dif\left(~
\begin{DGCpicture}[scale=0.5]
\DGCstrand[Red](0,0)(0,2)
\DGCstrand(1,0)(2,2)
\DGCstrand[Red](2,0)(1,2)
\DGCstrand[Red](3,0)(3,2)
\DGCstrand[Red](4,0)(4,2)
\DGCstrand(5,0)(5,2)
\end{DGCpicture}
~\right)&= 0
\\
\partial(\Phi_{\mf{t}_3})&=0
&
\dif\left(~
\begin{DGCpicture}[scale=0.5]
\DGCstrand[Red](0,0)(0,2)
\DGCstrand(1,0)(3,2)
\DGCstrand[Red](2,0)(1,2)
\DGCstrand[Red](3,0)(2,2)
\DGCstrand[Red](4,0)(4,2)
\DGCstrand(5,0)(5,2)
\end{DGCpicture}
~\right)&= 0
\end{align*}

$\Delta(\lambda_4)$:
\begin{align*}
\partial(\Phi_{\mf{u}_1})&=0
&
\dif\left(~
\begin{DGCpicture}[scale=0.5]
\DGCstrand[Red](0,0)(0,2)
\DGCstrand[Red](1,0)(1,2)
\DGCstrand(2,0)(2,2)
\DGCstrand[Red](3,0)(3,2)
\DGCstrand(4,0)(4,2)
\DGCstrand[Red](5,0)(5,2)
\end{DGCpicture}
~\right)&= 0
\\
\partial(\Phi_{\mf{u}_2})&=0
&
\dif\left(~
\begin{DGCpicture}[scale=0.5]
\DGCstrand[Red](0,0)(0,2)
\DGCstrand[Red](1,0)(1,2)
\DGCstrand(2,0)(2,2)
\DGCstrand[Red](3,0)(3,2)
\DGCstrand(4,0)(5,2)
\DGCstrand[Red](5,0)(4,2)
\end{DGCpicture}
~\right)&= 0
\\ 
\partial(\Phi_{\mf{u}_3})&=0
&
\dif\left(~
\begin{DGCpicture}[scale=0.5]
\DGCstrand[Red](0,0)(0,2)
\DGCstrand[Red](1,0)(1,2)
\DGCstrand(2,0)(3,2)
\DGCstrand[Red](3,0)(2,2)
\DGCstrand(4,0)(5,2)
\DGCstrand[Red](5,0)(4,2)
\end{DGCpicture}
~\right)&= 0
\\
\partial(\Phi_{\mf{u}_4})&=-\Phi_{\mf{u}_3}
&
\dif\left(~
\begin{DGCpicture}[scale=0.5]
\DGCstrand[Red](0,0)(0,2)
\DGCstrand[Red](1,0)(1,2)
\DGCstrand(2,0)(5,2)
\DGCstrand[Red](3,0)(2,2)
\DGCstrand(4,0)(3,2)
\DGCstrand[Red](5,0)(4,2)
\end{DGCpicture}
~\right)&= -
\begin{DGCpicture}[scale=0.5]
\DGCstrand[Red](0,0)(0,2)
\DGCstrand[Red](1,0)(1,2)
\DGCstrand(2,0)(3,2)
\DGCstrand[Red](3,0)(2,2)
\DGCstrand(4,0)(5,2)
\DGCstrand[Red](5,0)(4,2)
\end{DGCpicture}
\end{align*}
The last equality above follows by applying Lemma \ref{lemdotreduction}.

$\Delta(\lambda_5)$:
\begin{align*}
\partial(\Phi_{\mf{v}_1})&=0
&
\dif\left(~
\begin{DGCpicture}[scale=0.5]
\DGCstrand[Red](0,0)(0,2)
\DGCstrand[Red](1,0)(1,2)
\DGCstrand(2,0)(2,2)
\DGCstrand[Red](3,0)(3,2)
\DGCstrand[Red](4,0)(4,2)
\DGCstrand(5,0)(5,2)
\end{DGCpicture}
~\right)&= 0
\\
\partial(\Phi_{\mf{v}_2})&=0
&\dif\left(~
\begin{DGCpicture}[scale=0.5]
\DGCstrand[Red](0,0)(0,2)
\DGCstrand[Red](1,0)(1,2)
\DGCstrand(2,0)(3,2)
\DGCstrand[Red](3,0)(2,2)
\DGCstrand[Red](4,0)(4,2)
\DGCstrand(5,0)(5,2)
\end{DGCpicture}
~\right)&= 0
\end{align*}

$\Delta(\lambda_6)$:
\begin{align*}
\partial(\Phi_{\mf{w}_1})&=0
&
\dif\left(~
\begin{DGCpicture}[scale=0.5]
\DGCstrand[Red](0,0)(0,2)
\DGCstrand[Red](1,0)(1,2)
\DGCstrand[Red](2,0)(2,2)
\DGCstrand(3,0)(3,2)
\DGCstrand[Red](4,0)(4,2)
\DGCstrand(5,0)(5,2)
\end{DGCpicture}
~\right)&= 0
\end{align*} 

By Proposition \ref{propmultiplicitycomplex}, the projective modules $Z(\lambda_i)$ have $p$-DG filtrations whose subquotients are isomorphic to $p$-DG cell modules.  In the Grothendieck group we have the following relations.
\begin{align*}
[Z(\lambda_1)]&=[\Delta(\lambda_1)] \\
[Z(\lambda_2)]&=[\Delta(\lambda_2)]+q[\Delta(\lambda_1)] \\
[Z(\lambda_3)]&=[\Delta(\lambda_3)]+q[\Delta(\lambda_2)]+q^2 [\Delta(\lambda_1)] \\
[Z(\lambda_4)]&=[\Delta(\lambda_4)]+q[\Delta(\lambda_2)]+(q^2+1)[\Delta(\lambda_1)] \\
[Z(\lambda_5)]&=[\Delta(\lambda_5)]+q[\Delta(\lambda_4)]+q [\Delta(\lambda_3)]+q^2 [\Delta(\lambda_2)]+(q^3+q)[\Delta(\lambda_1)] \\
[Z(\lambda_6)]&=[\Delta(\lambda_6)]+q[\Delta(\lambda_5)]+(q^2+1) [\Delta(\lambda_4)]+q^2 [\Delta(\lambda_3)]+(q^3+q)[\Delta(\lambda_2)]+(q^4+q^2)[\Delta(\lambda_1)].
\end{align*}

The usual Kazhdan-Lusztig algorithm tells us that, from this data, that $Z(\lambda_i)$ is indecomposable projective for $i=1,2,3, 5$, while $Z(\lambda_4)$ and $Z(\lambda_6)$ are decomposable:
\[
Z(\lambda_4)\cong P(\lambda_4)\oplus Z(\lambda_1),\quad \quad
Z(\lambda_6)\cong P(\lambda_6)\oplus P(\lambda_4).
\]
Here $P(\lambda_4)$, $P(\lambda_6)$ are the projective covers of $L(\lambda_4)$ and $L(\lambda_6)$ respectively.

In the $p$-DG setting, the $p$-complex structure on $\Delta^\circ(\lambda_1)\e(\lambda_4)$ and $\Delta^\circ(\lambda_1)\e(\lambda_6)$ tells us that the above direct sum decompositions are adapted to $p$-DG filtrations instead:
\[
0\lra Z(\lambda_1)\lra Z(\lambda_4)\lra P(\lambda_4)\lra 0, \quad \quad
0\lra Z(\lambda_4)\lra Z(\lambda_6)\lra P(\lambda_6)\lra 0.
\]
Alternatively, the filtrations arise by applying $\HOM_{\nh_n^l}(\mbox{-},G)$ to the sequences in equation \eqref{eqnfilteredG46}.

\subsection{A basic \texorpdfstring{$p$}{p}-DG algebra}
We provide here a basic $p$-DG algebra $p$-DG Morita equivalent to 
$S_2^4$.

By Proposition \ref{prop-p-dg-Morita}, the algebra $S_2^4$  and the endormorphism algebra 
$\END_{\nh_2^4}(\oplus_{i=1}^6 Y(\lambda_i))$ are $p$-DG Morita equivalent. An explicit computation shows that the latter endomorphism algebra admits a quiver description as follows. Consider the path algebra of the quiver \ref{quiverex} :
\begin{equation}\label{quiverex}
\begin{gathered}
\xymatrix{
&&  \overset{3}{~\circ~} && &&  \\
&& && &&\\
2 {~\circ~} \uurrtwocell{'} &&
 \overset{1}{~\circ~}\lltwocell{'}&&
 \overset{5}{\circ}\lltwocell{'} \uulltwocell{'} &&
 \overset{6}{~\circ~}\lltwocell{'} \\
&& && &&\\
 &&  \overset{4}{~\circ~} \uurrtwocell{'} \uulltwocell{'} && && 
}
\end{gathered}
\end{equation}
modulo relations
\[
(1|2|1) = 0 \quad \quad \quad (1|5|1) = 0 
\]
\[
(6|5|1) = 0 \quad \quad \quad (1|5|6) = 0
\]
\[
(3|5|3) = 0 \quad \quad \quad (4|5|4) = 0 
\]
\[
(3|2|1) = (3|5|1) \quad \quad \quad (4|2|1) = (4|5|1)
\]
\[
(2|3|2) = (2|1|2) \quad \quad \quad (2|1|2) = (2|4|2) 
\]
\[
(1|2|3) = (1|5|3) \quad \quad \quad (1|2|4) = (1|5|4) 
\]
\[
(4|2|3) = (4|5|3) \quad \quad \quad (3|2|4) = (3|5|4) 
\]
\[
(2|4|5) + (2|1|5)+(2|3|5)=0 \quad \quad \quad (5|4|2) + (5|3|2)+(5|1|2)=0  
\]
\[
(5|6|5) = (5|3|5)+(5|4|5). 
\]

The $p$-DG structure is given by
\begin{align*}
\partial (2|1) &=0 && \partial (1|2) = 0 \\
\partial (5|3) &= 0 && \partial (3|5) = 0 \\
\partial (5|4) &=0 &&\partial (4|5)=0 \\
\partial (3|2) &= 0 && \partial (2|3) = (2|3|2|3) \\
\partial (4|2) &=-(4|2|1|2) && \partial (2|4)=0 \\
\partial (5|1) &= -(5|3|5|1) && \partial (1|5)= - (1|5|3|5) \\
\partial (6|5) &= -(6|5|4|5) && \partial (5|6)= (5|3|5|6) \\
\end{align*}

\addcontentsline{toc}{section}{References}

% ====================================================================
% REFERENCES

\bibliographystyle{alpha}
\bibliography{qy-bib}

%
% ====================================================================

\vspace{0.1in}

\noindent Y.~Q.: { \sl \small Department of Mathematics, University of Virginia, Charlottesville, VA 22904, USA} \newline \noindent {\tt \small email: yq2dw@virginia.edu}

\vspace{0.1in}

\noindent J.~S.:  {\sl \small Department of Mathematics, CUNY Medgar Evers, Brooklyn, NY, 11225, USA}\newline \noindent {\tt \small email: jsussan@mec.cuny.edu}

% ==============================================================================
%
\end{document}